\documentclass[10pt,twoside]{amsart}
\usepackage{dsliheader9}

\title{The Local Langlands Correspondence for $\GL_n$ over function fields}
\author{Siyan Daniel Li-Huerta}

\usepackage{times}
\usepackage{setspace}
\theoremstyle{plain}
\newtheorem*{lemfirst}{First Inductive Lemma}
\newtheorem*{lemsecond}{Second Inductive Lemma}
\newtheorem*{thmD}{Theorem D}

\begin{document}
\begin{abstract}
Let $F$ be a local field of characteristic $p>0$. By adapting methods of Scholze \cite{Sch13}, we give a new proof of the local Langlands correspondence for $\GL_n$ over $F$. More specifically, we construct $\ell$-adic Galois representations associated with many discrete automorphic representations over global function fields, which we use to construct a map $\pi\mapsto\rec(\pi)$ from isomorphism classes of irreducible smooth representations of $\GL_n(F)$ to isomorphism classes of $n$-dimensional semisimple continuous representations of $W_F$. Our map $\rec$ is characterized in terms of a local compatibility condition on traces of a certain test function $f_{\tau,h}$, and we prove that $\rec$ equals the usual local Langlands correspondence (after forgetting the monodromy operator). 
\end{abstract}

\maketitle
\tableofcontents
\section*{Introduction}
We start by recalling the \emph{local Langlands correspondence for $\GL_n$} over non-archimedean local fields. Let $F$ be a nonarchimedean local field, write $\ka$ for its residue field, and fix a prime number $\ell\neq\pchar\ka$. The local Langlands correspondence for $\GL_n$ over $F$ posits a collection of \emph{canonical} bijections
\begin{align*}
\left\{
  \begin{tabular}{c}
    isomorphism classes of irreducible \\
    smooth representations of $\GL_n(F)$ over $\ov\bQ_\ell$
  \end{tabular}
\right\}\longleftrightarrow\left\{
  \begin{tabular}{c}
  isomorphism classes of $n$-dimensional continuous \\
  Frobenius-semisimple representations of $W_F$ over $\ov\bQ_\ell$
  \end{tabular}
\right\},
\end{align*}
where $W_F$ denotes the Weil group of $F$ with respect to a fixed separable closure of $F$, and $n$ ranges over all positive integers. By canonical, we mean they are characterized by recovering local class field theory when $n=1$, being compatible with central characters and duals, and preserving $L$-functions and $\eps$-factors of pairs. 

To prove such a correspondence, work of Bernstein--Zelevinsky allows us to reduce to the case of cuspidal representations on the automorphic side and irreducible representations on the Galois side. More precisely, they show {\cite[10.3]{Zel80}} that one can uniquely reconstruct bijections as above from their restrictions to cuspidal representations
\begin{align}\label{eq:llc}
\left\{
  \begin{tabular}{c}
    isomorphism classes of irreducible \\
    cuspidal representations of $\GL_n(F)$ over $\ov\bQ_\ell$
  \end{tabular}
\right\}\longleftrightarrow\left\{
  \begin{tabular}{c}
  isomorphism classes of irreducible continuous  \\
  $n$-dimensional representations of $W_F$ over $\ov\bQ_\ell$
  \end{tabular}
\right\}\tag{$\diamond$}.
\end{align}
As alluded to above, work of Henniart \cite[Theorem 1.2]{Hen85} indicates that there is at most one collection of bijections as in Equation (\ref{eq:llc}) satisfying our canonicity requirements. This ensures the \emph{uniqueness} of the local Langlands correspondence for $\GL_n$ over $F$.

As for its \emph{existence}, such bijections were first proved for function fields by Laumon--Rapoport--Stuhler \cite[(15.7)]{LRS93} and for $p$-adic fields by Harris--Taylor \cite[Theorem A]{HT01} and Henniart \cite[1.2]{Hen00}. When $F$ is a $p$-adic field, Scholze gave a new proof and characterization \cite[Theorem 1.2]{Sch13} of the local Langlands correspondence for $\GL_n$ over $F$, simplifying arguments of Harris--Taylor.

By adapting the methods of Scholze's proof, the goal of this paper is to give a new proof and characterization of the local Langlands correspondence for $\GL_n$ when $F$ is a function field. Thus, let us henceforth assume that $p\deq\pchar{F}>0$. Grothendieck's $\ell$-adic monodromy theorem \cite[Appendix]{ST68} implies that irreducible continuous $n$-dimensional representations of $W_F$ over $\ov\bQ_\ell$ are smooth, so by fixing a field isomorphism $\ov\bQ_\ell=\bC$, Equation (\ref{eq:llc}) remains unchanged if we replace $\ov\bQ_\ell$ with $\bC$. Unless otherwise specified, all subsequent representations shall be taken over $\bC$.

We begin by motivating this new characterization of the local Langlands correspondence for $\GL_n$ over $F$, as in the work of Scholze. Since $n$-dimensional semisimple continuous representations of $W_F$ are determined by their traces, it would be natural to characterize the correspondence via a trace condition. More precisely, one would like to construct a map $\pi\mapsto\rho(\pi)$ from isomorphism classes of irreducible smooth representations of $\GL_n(F)$ to isomorphism classes of $n$-dimensional semisimple continuous representations of $W_F$ satisfying the following condition: for all $\tau$ in $W_F$, there exists a test function $f_\tau$ in $C^\infty_c(\GL_n(F))$ such that
\begin{align*}
\tr(f_\tau|\pi) = \tr\left(\tau|\rho(\pi)\right)
\end{align*}
for all irreducible smooth representations $\pi$ of $\GL_n(F)$.

However, this is too much to ask for, as noted by Scholze in \cite{Sch13}. To see this, note that we want $\rho(\pi)$ to be (a Tate twist of) the Weil representation corresponding to $\pi$ under the local Langlands correspondence. But if this were the case, then the above equation implies that $f_\tau$ has nonzero trace on \emph{every} component of the Bernstein center of $\GL_n(F)$. This is impossible because $f_\tau$ is locally constant.\footnote{However, such an $f_\tau$ \emph{does} exist as an element of the Bernstein center and hence as a distribution. See Proposition \ref{ss:functiondesire}.}

We smooth out this issue by convolving $f_\tau$ with a cut-off function. To elaborate, let us introduce some notation: write $q$ for $\#\ka$, and write $v:W_F\rar\bZ$ for the unramified homomorphism sending geometric $q$-Frobenii to $1$. For any $\tau$ in $W_F$ with $v(\tau)>0$ and $h$ in $C^\infty_c(\GL_n(\cO))$, we construct a test function $f_{\tau,h}$ in $C^\infty_c(\GL_n(F))$ satisfying the following theorem.

\begin{thmA}Let $n$ be a positive integer, and let $\pi$ be any irreducible smooth representation of $\GL_n(F)$. 
\begin{enumerate}[(i)]
\item There exists a unique $n$-dimensional semisimple continuous representation $\rho(\pi)$ of $W_F$ satisfying the following property:
  \begin{align*}
    \mbox{for all }\tau\mbox{ in }W_F\mbox{ with }v(\tau)>0\mbox{ and }h\mbox{ in }C^\infty_c(\GL_n(\cO))\mbox{, we have }\tr(f_{\tau,h}|\pi) = \tr\left(\tau|\rho(\pi)\right)\tr(h|\pi).
  \end{align*}
Write $\rec(\pi)$ for $\rho(\pi)(\frac{1-n}2)$, where $(s)$ denotes the $s$-th Tate twist.
\item Suppose $\pi$ is isomorphic to a subquotient of the normalized parabolic induction of 
  \begin{align*}
    \pi_1\otimes\dotsb\otimes\pi_t,
  \end{align*}
where the $\pi_i$ are irreducible smooth representations of $\GL_{n_i}(F)$ such that $n_1+\dotsb+n_t=n$. Then
\begin{align*}
  \rec(\pi) = \rec(\pi_1)\oplus\dotsb\oplus\rec(\pi_t).
\end{align*}
\end{enumerate}
\end{thmA}
Theorem A.(i) indicates that $f_{\tau,h}$ satisfies the trace compatibility property we would expect from the convolution $f_\tau* h$. As our notation suggests, we shall see in Theorem C that $\pi\mapsto\rec(\pi)$ equals the usual local Langlands correspondence for $\GL_n$ over $F$ (after forgetting the monodromy operator). Thus this gives a new characterization of the local Langlands correspondence in this case.

How can we find such an $f_{\tau,h}$? Let $\ov\ka$ be a fixed separable closure of $\ka$. The Deligne--Carayol conjecture (which was proven for function fields by Boyer \cite[Theorem 3.2.4]{Boy99} and for $p$-adic fields by Harris--Taylor \cite[Theorem B]{HT01}) indicates that, roughly speaking, the local Langlands correspondence for $\GL_n$ as in Equation (\ref{eq:llc}) can be found in the cohomology of deformation spaces of certain $1$-dimensional formal $\cO$-modules over $\ov\ka$. By the Dieudonn\'e equivalence, $1$-dimensional formal $\cO$-modules correspond to special examples of \emph{$1$-dimensional effective minuscule local shtukas}. The deformation spaces of general $1$-dimensional effective minuscule local shtukas can be pieced together from those of $1$-dimensional formal $\cO$-modules (see \S\ref{s:deformationspaces}), and this mirrors how general irreducible smooth representations of $\GL_n(F)$ can be pieced together from cuspidal ones via parabolic induction (see \S\ref{s:parabolicinduction}).

Therefore, to find a test function $f_{\tau,h}$ satisfying Theorem A, we study deformation spaces of $1$-dimensional effective minuscule local shtukas. We start by parameterizing these objects as follows. Write $r\deq v(\tau)$, write $F_r$ for the $r$-th degree unramified extension of $F$, write $\cO_r$ for its ring of integers, and write $\ka_r$ for its residue field. Then the Cartan decomposition for $\GL_n(F_r)$ shows that isomorphism classes of $1$-dimensional effective minuscule local shtukas over $\ka_r$ correspond to elements
\begin{align*}
\de\in\GL_n(\cO_r)\diag(\vpi,1,\dotsc,1)\GL_n(\cO_r)
\end{align*}
up to $\GL_n(\cO_r)$-$\sg$-conjugacy, where $\vpi$ is a uniformizer of $F$. Form the deformation space of the corresponding $1$-dimensional effective minuscule local shtuka with Drinfeld level-$m$ structure, and write $R^i\psi_{\de,m}$ for the $i$-th $\ell$-adic cohomology group of the adic generic fiber of this deformation space. As $m$ varies, we use these cohomology groups to construct representations
\begin{align*}
R^i\psi_\de\deq\dirlim_m R^i\psi_{\de,m}\mbox{ and }[R\psi_\de]\deq\sum_{i=0}^\infty(-1)^iR^i\psi_\de,
\end{align*}
which have commuting actions of $W_{F_r}$ and $\GL_n(\cO)$. From here, we can first define a function $\phi_{\tau,h}$ in $C^\infty_c(\GL_n(F_r))$ by sending
\begin{align*}
 \de\mapsto
\begin{cases}
\tr(\tau\times h|[R\psi_\de]) & \mbox{if }\de\mbox{ is in }\GL_n(\cO_r)\diag(\vpi,1,\dotsc,1)\GL_n(\cO_r),\\
0 & \mbox{otherwise,}
\end{cases}
\end{align*}
and then we let $f_{\tau,h}$ be a transfer of $\phi_{\tau,h}$ to $\GL_n(F)$.

Now that we have our test function $f_{\tau,h}$, let's discuss the proof of Theorem A. After constructing $\rho(\pi)$ in special cases, we use the geometry of our deformation spaces, along with work of Bernstein--Zelevinsky \cite{Zel80} on the local automorphic side, to inductively piece together $\rho(\pi)$ in the general case. The process of piecing together $\rho(\pi)$ amounts entirely to local deformation theory and nonarchimedean harmonic analysis. This proceeds as in \cite{Sch13}, except the necessary harmonic analysis is more difficult in characteristic $p$. Therefore we carefully give the argument and supply references for the relevant harmonic analysis in characteristic $p$.

We use global techniques to construct $\rho(\pi)$ in special cases. But instead of Shimura varieties as in \cite{Sch13}, we use moduli spaces of \emph{$\sD$-elliptic sheaves}, the latter of which is an equi-characteristic analogue of abelian varieties equipped with certain endomorphisms. For good reduction, these moduli spaces were first considered by Laumon--Rapoport--Stuhler \cite{LRS93} in their proof of the local Langlands correspondence for $\GL_n$ over $F$, and we shall also crucially use a version of these moduli spaces with bad reduction (which corresponds to using Drinfeld level structures) as considered by Boyer \cite{Boy99} in his proof of the Deligne--Carayol conjecture for $F$. 

We explicitly define $\rho(\pi)$ as a multiplicity space in the cohomology of the moduli space of $\sD$-elliptic sheaves. To show that $\rho(\pi)$ satisfies our desired trace compatibility condition, we adapt Scholze's version of the Langlands--Kottwitz method to compute traces of Frobenius and Hecke operators on the aforementioned cohomology groups. These traces are related to our test function $f_{\tau,h}$ via a Serre--Tate theorem for $\sD$-elliptic sheaves, which equates deformations of $\sD$-elliptic sheaves to deformations of certain associated local shtukas. For this, it is crucial to carry out the Langlands--Kottwitz method at bad reduction, which is new in positive characteristic.

This global work yields the following construction of $\ell$-adic Galois representations associated with certain discrete automorphic representations. Let $\bf{F}$ be a global function field, write $\bA$ for its ring of adeles, and let $\{x_1,x_2,\infty\}$ be three distinct places of $\bf{F}$. Write $G_{\bf{F}}$ for the absolute Galois group of $\bf{F}$ with respect to a fixed separable closure of $\bf{F}$.

\begin{thmB}
Let $\Pi$ be an irreducible discrete automorphic representation of $\GL_n(\bA)$ whose components at $x_1$, $x_2$, and $\infty$ are either irreducible $L^2$ representations or Speh modules. Then there exists a unique $n$-dimensional semisimple continuous representation $R(\Pi)$ of $G_\bf{F}$ over $\ov\bQ_\ell$ such that, for all places $o$ of $\bf{F}$ not lying in $\{x_1,x_2,\infty\}$, the restriction of $R(\Pi)$ to $W_{\bf{F}_o}$ satisfies
\begin{align*}
\res{R(\Pi)}_{W_{\bf{F}_o}} = \rho(\Pi_o),
\end{align*}
where we identify $\ov\bQ_\ell$ with $\bC$.
\end{thmB}
We finish the construction of $\rho(\pi)$ by finding a $\Pi$ as above such that $\pi$ is isomorphic to $\Pi_o$, and then applying Theorem B. In general, this is more difficult in characteristic $p$ because the trace formula is not as developed, but we circumvent this using M\oe glin--Waldspurger's description of the discrete automorphic spectrum. This concludes our construction of $\rho(\pi)$ in special cases and thus our proof of Theorem A.

From here, we prove the following bijectivity result.
\begin{thmC}
The map $\pi\mapsto\rec(\pi)$ yields a bijection from isomorphism classes of irreducible cuspidal representations of $\GL_n(F)$ to isomorphism classes of $n$-dimensional irreducible continuous representations of $W_F$.
\end{thmC}
The key ingredient is an explicit calculation of the inertia invariants of nearby cycles due to Scholze \cite[Theorem 5.3]{Sch13b}, which relies on a case of Grothendieck's purity conjecture as proved by Thomason \cite[Corollary 3.9]{Tho84}. We use this calculation to deduce that if $\rec(\pi)$ is unramified, then $\pi$ must be as well. By passing to the Galois side and using our work on local-global compatibility, we show that after applying cyclic base change \cite[(II.1.4)]{HL11} finitely many times, the representation $\pi$ becomes unramified. Passing to this unramified representation ultimately allows us to prove Theorem C. As in \cite{Sch13}, this argument bypasses the need to appeal to Henniart's numerical local Langlands \cite[Theorem 1.2]{Hen88}.

Finally, we show that $\pi\mapsto\rec(\pi)$ satisfies the usual canonicity requirements of the local Langlands correspondence for $\GL_n$.
\begin{thmD}
The bijections
\begin{align*}
\rec:\left\{
  \begin{tabular}{c}
    isomorphism classes of irreducible  \\
    cuspidal representations of $\GL_n(F)$
  \end{tabular}
\right\}\rar^\sim\left\{
  \begin{tabular}{c}
  isomorphism classes of $n$-dimensional  \\
  irreducible continuous representations of $W_F$
  \end{tabular}
\right\}
\end{align*}
satisfy the following properties:
\begin{enumerate}[(i)]
\item for all irreducible cuspidal representations of $\GL_1(F)$, that is, smooth characters $\chi:F^\times\rar\bC^\times$, we have
\begin{align*}
\rec(\chi) = \chi\circ\Art^{-1}, 
\end{align*}
where $\Art$ denotes the Artin isomorphism $\Art:F^\times\rar^\sim W^\ab_F$ that sends uniformizers to geometric $q$-Frobenii.

\item for all irreducible cuspidal representations $\pi$ of $\GL_n(F)$ and smooth characters $\chi:F^\times\rar\bC^\times$, we have
  \begin{align*}
    \rec(\pi\otimes(\chi\circ\det)) = \rec(\pi)\otimes\rec(\chi).
  \end{align*}

\item for all irreducible cuspidal representations $\pi$ of $\GL_n(F)$ with central character $\om_\pi:F^\times\rar\bC^\times$, we have
\begin{align*}
\rec(\om_\pi) = \det\circ\rec(\pi)\mbox{ and }\rec(\pi^\vee) = \rec(\pi)^\vee.
\end{align*}
\item for all irreducible cuspidal representations $\pi$ of $\GL_n(F)$ and $\pi'$ of $\GL_{n'}(F)$, we have
\begin{align*}
L(\pi\times\pi',s) = L(\rec(\pi)\otimes\rec(\pi'),s)\mbox{ and } \eps(\pi\times\pi',\psi,s) = \eps(\rec(\pi)\otimes\rec(\pi'),\psi,s)
\end{align*}
for all nontrivial continuous homomorphisms $\psi:F\rar\bC^\times$.
\end{enumerate}
\end{thmD}
By \cite[Theorem 1.2]{Hen85}, Theorem D shows that our construction indeed equals the usual local Langlands correspondence for $\GL_n$.

We now discuss the proof of Theorem D. Compatibility with local class field theory follows from earlier work, as our use of $1$-dimensional formal $\cO$-modules recovers the Lubin--Tate description of local class field theory when $n=1$. Compatibility with twists of characters follows by embedding into the global setting, using Theorem B, and applying the strong multiplicity one theorem \cite[Theorem 3.3.(b)]{BR17}. We prove compatibility with central characters via similar global means---this time, we reduce to the situation of induced representations by Brauer induction, and then we prove some cases of non-Galois automorphic induction in order to pass to the automorphic side. From here, Henniart's trick of twisting by highly ramified characters \cite[Lemma 4.2]{Hen86} implies compatibility with $L$-functions and $\eps$-factors, and then compatibility with duals follows from the decomposition of $L$-functions of pairs in terms of $L$-functions of characters.

\subsection*{Outline} In \S\ref{s:deformationspaces}, we introduce local shtukas, the geometry of their deformation spaces, and the test function $f_{\tau,h}$. Next, in \S\ref{s:firstinductivelemma}, we recall the Bernstein--Zelevinsky classification, the Bernstein center, and Schneider--Zink tempered types. We use this to prove Lemma \ref{ss:firstinductivelemma}, which serves as the framework for proving Theorem A. The primary goal of \S\ref{s:parabolicinduction}--\S\ref{s:localglobal} is to prove the ingredients needed for running Lemma \ref{ss:firstinductivelemma}. In \S\ref{s:parabolicinduction}, we use the geometry of deformation spaces along with the Bernstein--Zelevinsky induction-restriction formula to prove compatibility with parabolic induction. In \S\ref{s:lubintatetower}, we compare our deformation spaces with the Lubin--Tate tower. With the exception of Lemma \ref{ss:nearbycyclesalgebraization}, the entirety of \S\ref{s:deformationspaces}--\S\ref{s:lubintatetower} is local.

At this point, we switch to a global setup. In \S\ref{s:modulispaces}, we introduce $\sD$-elliptic sheaves, their moduli spaces, and their relationship with local shtukas. In \S\ref{s:nearbycycles}, we present an explicit calculation of nearby cycle sheaves on the special fibers of these moduli spaces at bad reduction. In \S\ref{s:langlandskottwitz}, we employ a variant of the Langlands--Kottwitz method, using the Serre--Tate theorem as introduced in \ref{ss:extendedserretate} to relate the cohomology of the moduli space of $\sD$-elliptic sheaves with the local test function $f_{\tau,h}$. In \S\ref{s:localglobal}, we deduce Theorem B from \S\ref{s:langlandskottwitz}, and then we embed our local situation into our global one to finish the proof of Theorem A.

We now turn to the proofs of Theorem C and Theorem D. In \S\ref{s:secondinductivelemma}, we recall cyclic base change, and we combine this with local-global embedding theorems to prove Theorem C. In \S\ref{s:llc}, we use Theorem B (as well as compatibility with automorphic induction) to prove some cases of non-Galois automorphic induction. We tie these results together to prove Theorem D, concluding our proof of the local Langlands correspondence for $\GL_n$ over $F$.
\subsection*{Notation}
Let $\ell\neq p$ be a prime number. We fix an identification $\ov\bQ_\ell=\bC$ of fields. Unless otherwise specified, all representations are over $\bC$.

Throughout \S\ref{s:deformationspaces}--\S\ref{s:lubintatetower}, let $F$ be a local field of positive characteristic. Write $\cO$ for its ring of integers, fix a uniformizer $\vpi$ of $\cO$, and write $\ka$ for the residue field $\cO/\vpi$. This choice of $\vpi$ yields an identification $\cO=\ka\llb\vpi$. We denote $\#\ka$ using $q$. We write $v$ for the normalized discrete valuation on $F$, and we denote the normalized valuation on $F$ using $\abs-$. We fix a separable closure $F^\sep$ in which our separable extensions of $F$ lie, and we write $\ov\ka$ for the residue field of $F^\sep$, which is a separable closure of $\ka$. We write $\bC_\vpi$ for the completion of $F^\sep$, and we view completions of separable extensions of $F$ as closed subfields of $\bC_\vpi$.

We write $W_F$ for the Weil group of $F$ with respect to this choice of $F^\sep$, and we view Weil groups of separable extensions of $F$ as living inside $W_F$. We denote the inertia subgroup of $F$ using $I_F$. We have a canonical short exact sequence
\begin{align*}
1\rar I_F\rar W_F\rar^v\bZ\rar1,
\end{align*}
and we identify $1$ in $\bZ$ with the geometric $q$-Frobenius in $\Gal(\ov\ka/\ka)$. We write $\Art:F^\times\rar^\sim W_F^\ab$ for the local Artin isomorphism normalized by sending uniformizers to geometric $q$-Frobenii.

\subsection*{Acknowledgments} The author is tremendously indebted to Sophie Morel for her advice on this project, suggesting this topic of research, and her encouragement. The author extends his gratitude to Richard Taylor for his careful reading of an earlier iteration of this paper as well as clarifying some misconceptions. Many of the ideas in this paper are due to Peter Scholze, and the author thanks him for answering some questions about \cite{Sch13}. The author would also like to thank Thomas Haines for some helpful conversations on nonarchimedean harmonic analysis, and to thank Alain Genestier for directing him to Boyer's thesis \cite{Boy99}.

\section{Deformation spaces of local shtukas}\label{s:deformationspaces}

In this section, we introduce \emph{local shtukas}, which are the equi-characteristic analogue of isocrystals. We then narrow our scope to \emph{effective minuscule} local shtukas, which are the analogue of Dieudonn\'e modules. A similar Dieudonn\'e equivalence relates $1$-dimensional connected local shtukas to $1$-dimensional formal $\cO$-modules, which allows us to use results of Drinfeld to study their deformation spaces. From here, we deduce finitude properties about the cohomology of these deformation spaces from a global algebraization result, and we conclude by defining the test functions $\phi_{\tau,h}$ and $f_{\tau,h}$ using these cohomology groups.

\subsection{}
Given a scheme $S$ over $\Spec\cO$, we write $\ze$ for the image of $\vpi$ in $\sO_S$. Note that requiring $\ze$ to be locally nilpotent is the same as requiring $S\rar\Spec\cO$ to factor as $S\rar\Spf\cO\rar\Spec\cO$. We consider $\sO_S\llb{\vpi}$ and $\sO_S\llp{\vpi}$ as sheaves of formal power and Laurent series on $\abs{S}$, respectively \cite[p.~4]{HS16}.
\begin{defn*}
Let $S$ be a scheme over $\Spf\cO$. An \emph{local shtuka} over $S$ is a pair $(\sM,\sF)$, where
\begin{enumerate}[$\bullet$]
\item $\sM$ is a locally free $\sO_S\llb\vpi$-module of finite rank,
\item $\sF:\sg^*\sM[\frac1\vpi]\rar\sM[\textstyle\frac1\vpi]$ is an $\sO_S\llp\vpi$-module isomorphism,
\end{enumerate}
where $\sg$ denotes the absolute $q$-Frobenius on $S$ and its canonical lifts to $\sO_S\llb\vpi$ and $\sO_S\llp\vpi$. A \emph{morphism} of local shtukas is a morphism $f:\sM\rar\sM'$ of locally free $\sO_S\llb\vpi$-modules satisfying $f[\frac1\vpi]\circ\sF = \sF'\circ f[\frac1\vpi]$. A \emph{quasi-isogeny} of local shtukas is an invertible element in $\Hom(\sM,\sM')[\frac1\vpi]$. When such an element exists, we say $\sM$ and $\sM'$ are \emph{isogenous}.
\end{defn*}
\subsection{}\label{ss:localshtukaconjugacy}For any field $\la$ over $\ka$, we see that isomorphism classes of local shtukas over $\Spec\la$ of rank $n$ correspond to $\GL_n(\la\llb\vpi)$-$\sg$-conjugacy classes in $\GL_n(\la\llp\vpi)$ via sending $\de$ in $\GL_n(\la\llp\vpi)$ to $(\la\llb\vpi^{ n},\de\circ\sg^{\oplus n})$, where $\de$ acts via left multiplication on column vectors. In further analogy with isocrystals, we also have a \emph{Dieudonn\'e--Manin classification} \cite[(2.4.5)]{Lau96} for isogeny classes of local shtukas over algebraically closed fields.
\begin{defn}
Let $\sM$ be a local shtuka over $S$. We say that $\sM$ is \emph{effective} if $\sF$ is the localization of an $\sO_S\llb\vpi$-module morphism $\sg^*\sM\rar\sM$, which we henceforth refer to as $\sF$ by abuse of notation. For an effective local shtuka $\sM$, one can show that $\coker\sF$ is a locally free $\sO_S$-module of finite rank \cite[Lemma 2.3]{HS16}, and we say the \emph{dimension} of $\sM$ is the rank of $\coker\sF$. We say $\sM$ is \emph{\'etale} if its dimension equals zero.

We say $\sM$ is \emph{effective minuscule} if $\vpi-\ze$ annihilates $\coker\sF$. For an effective minuscule local shtuka $\sM$, we say $\sM$ is \emph{connected} if $\sF$ is $\vpi$-adically nilpotent.
\end{defn}
\subsection{}\label{ss:effectiveminusculelocalshtukaconjugacy}
In our above description \ref{ss:localshtukaconjugacy} of isomorphism classes of local shtukas over $\Spec\la$, we see that the local shtuka corresponding to $\de$ is effective if and only if $\de$ lies in $\M_n(\la\llb\vpi)$. By the Cartan decomposition, it is effective minuscule if and only if $\de$ lies in
\begin{align*}
\coprod_{d=0}^n\GL_n(\la\llb\vpi)\diag(\underbrace{\vpi,\dotsc,\vpi}_{d\,\text{times}},1,\dotsc,1)\GL_n(\la\llb\vpi),
\end{align*}
where $d$ corresponds to the dimension of $(\la\llb\vpi^{n},\de\circ\sg^{\oplus n})$.

\subsection{}
Just as with Dieudonn\'e modules, effective minuscule local shtukas have a good theory of
\begin{enumerate}[$\bullet$]
\item truncated variants, called \emph{finite $\ka$-shtukas} \cite[Definition 2.6]{HS16}, which we remark can be defined over any scheme $S$ over $\Spec\cO$,
\item a contravariant \emph{Dieudonn\'e equivalence} \cite[Theorem 5.2]{HS16} between finite $\ka$-shtukas and certain finite module schemes called \emph{strict $\ka$-modules} \cite[Definition 4.8]{HS16} that extends to an anti-equivalence \cite[Theorem 8.3]{HS16} between effective minuscule local shtukas and certain module sheaves called \emph{$\vpi$-divisible local Anderson modules} \cite[Definition 7.1]{HS16},
\item a \emph{connected-\'etale} short exact sequence over local Artinian rings that splits over perfect fields \cite[Proposition 2.9]{HS16}.
\end{enumerate}

\subsection{}\label{ss:localshtukaalgebraicclosure}
Let $\br{F}$ be the completion of the maximal unramified extension of $F$, and write $\br\cO$ for its ring of integers. Our choice of $\vpi$ yields an identification $\br\cO=\ov\ka\llb\vpi$.

We begin by examining the case $S=\Spec\ov\ka$. Let $\de$ be an element of $\GL_n(\br\cO)\diag(\vpi,1,\dotsc,1)\GL_n(\br\cO)$, and write $\br{H}_\de$ for the associated effective minuscule local shtuka over $\Spec\ov\ka$ of rank $n$ and dimension $1$. The connected-\'etale sequence yields a decomposition
\begin{align*}
\br{H}_\de = \br{H}_\de^\circ\oplus\br{H}_\de^\et,
\end{align*}
where $\br{H}_\de^\circ$ and $\br{H}_\de^\et$ are the connected and \'etale parts of $\br{H}_\de$, respectively. Write $k$ for the rank of $\br{H}_\de^\circ$.

As $\br{H}_\de^\circ$ is the connected part, it must have dimension $1$, so under the Dieudonn\'e equivalence it corresponds to the unique formal $\cO$-module of height $k$ and dimension $1$, i.e. the Lubin--Tate module \cite[Proposition 1.7.1)]{Dri74}. By applying the Dieudonn\'e equivalence to $\br{H}_\de^\et$ and using the algebraic closedness of $\Spec\ov\ka$, we also see that $\br{H}_\de^\et$ is the unique \'etale effective minuscule local shtuka over $\Spec\ov\ka$ of rank $n-k$, i.e. the constant $\vpi$-divisible local Anderson module $\ul{(F/\cO)}^{n-k}$. Thus under \ref{ss:localshtukaconjugacy}, we see that $\br{H}_\de^\et$ corresponds to the $\GL_{n-k}(\br\cO)$-$\sg$-conjugacy class of $1$, and $H_\de^\circ$ corresponds to a basic element $\de^\circ$ in $\GL_k(\br\cO)\diag(\vpi,1,\dotsc,1)\GL_k(\br\cO)$, up to $\GL_k(\br\cO)$-$\sg$-conjugation.

\subsection{}\label{ss:localshtukaunramifiedextension}
Let $F_r$ be the $r$-th degree unramified extension of $F$, write $\cO_r$ for its ring of integers, and write $\ka_r$ for $\cO_r/\vpi$, which is the $r$-degree extension of $\ka$. Our choice of $\vpi$ yields an identification $\cO_r=\ka_r\llb\vpi$.

We now turn to the case of $S=\Spec\ka_r$. Let $\de$ be an element of $\GL_n(\cO_r)\diag(\vpi,1,\dotsc,1)\GL_n(\cO_r)$, and write $H_\de$ for the associated effective minuscule local shtuka over $\Spec\ka_r$ of rank $n$ and dimension $1$. The connected-\'etale sequence yields a decomposition
\begin{align*}
H_\de = H_\de^\circ\oplus H_\de^\et,
\end{align*}
where $H_\de^\circ$ and $H^\et_\de$ are the connected and \'etale parts of $H_\de$, respectively. Write $k$ for the rank of $H_\de^\circ$. Under \ref{ss:localshtukaconjugacy}, we see that $H_\de^\et$ corresponds to some $\de^\et$ in $\GL_{n-k}(\cO_r)$, up to $\GL_{n-k}(\cO_r)$-$\sg$-conjugation, and one can show that $H_\de^\circ$ corresponds to some $\de^\circ$ in $\GL_k(\cO_r)\diag(\vpi,1,\dotsc,1)\GL_k(\cO_r)$, up to $\GL_k(\cO_r)$-$\sg$-conjugation, whose norm 
\begin{align*}
\N(\de^\circ)\deq\de^\circ\sg(\de^\circ)\dotsm\sg^{r-1}(\de^\circ)
\end{align*}
is $\GL_k(F_r)$-conjugate to an elliptic element in $\GL_k(F)$. We denote the set of such $\de$ by $B_{n,k}$, and we write $B_k$ for $B_{k,k}$. Note that the pullback of $H_\de$ to $\Spec\ov\ka$ is isomorphic to $\br{H}_\de$.

\subsection{}
To define our deformation spaces of local shtukas (with additional data), we need to introduce a notion of \emph{Drinfeld level structures} for effective local shtukas. For any effective local shtuka $\sM$, its $m$-th level truncation is $\sM/\vpi^m$. For any finite $\ka$-shtuka $M$, we write $\Dr(M)$ for the corresponding finite $\ka$-strict module \cite[p.~16]{HS16}, which is an $\cO/\vpi^m$-module scheme over $S$.
\begin{defn*}
Let $\sM$ be an effective minuscule local shtuka over $S$ of constant rank $n$. We say a \emph{Drinfeld level-$m$ structure} on $\sM$ is a Drinfeld level-$m$ structure on $\Dr(\sM/\vpi^m)$, that is, an $\cO/\vpi^m$-module morphism
\begin{align*}
\al:(\vpi^{-m}\cO/\cO)^n\rar\Dr(\sM/\vpi^m)(S)
\end{align*}
such that the collection of all $\al(x)$ for $x$ in $(\vpi^{-m}\cO/\cO)^n$ forms a full set of sections of $\Dr(\sM/\vpi^m)$ as in \cite[(1.8.2)]{KM85}. For any Drinfeld level-$m$ structure $\al$, its restriction to $(\vpi^{-m'}\cO/\cO)^n$ is a Drinfeld level-$m'$ structure, which we denote by $\res\al_{m'}$. Furthermore, if $S$ is the spectrum of a local Artinian ring, then the restriction of $\al$ to $\ker\al$ is a Drinfeld level-$m$ structure on the connected part $\sM^\circ$ of $\sM$. We denote this by $\al^\circ$.
\end{defn*}

\subsection{}
We now initiate our study of deformation spaces. Write $\wh\cC$ for the opposite category of the category whose
\begin{enumerate}[$\bullet$]
\item objects are complete Noetherian local $\br\cO$-algebras $A$ with residue field $\ov\ka$,
\item morphisms are local $\br\cO$-algebra morphisms,
\end{enumerate}
and write $\cC$ for the full subcategory of $\wh\cC$ consisting of Artinian rings. We identify $\wh\cC$ with a full subcategory of formal schemes over $\Spf\br\cO$. Note that for $A$ in $\cC$, we have $\Spf{A}=\Spec{A}$.

For any contravariant functor $E:\cC\rar(\text{Set})$, we write $\wh{E}:\wh\cC\rar(\text{Set})$ for the extension of $E$ to $\wh\cC$ given by sending
\begin{align*}
  A\mapsto\invlim_i E(A/\fm^i_A),
\end{align*}
where $\fm_A$ is the maximal ideal of $A$. We say $E$ \emph{has a deformation space} if $\wh{E}$ is representable by a finite disjoint union $\fX$ of formal schemes in $\wh\cC$. By Yoneda's lemma, such an $\fX$ is unique up to isomorphism.

\subsection{}\label{ss:connecteddeformationalgebraicclosure}
Return to the situation in \ref{ss:localshtukaalgebraicclosure}, and let $\al$ be a Drinfeld level-$m$ structure on $\br{H}_\de$. Write $\br{E}_{\de,\al}:\cC\rar(\text{Set})$ for the functor sending
\begin{align*}
A\mapsto\{\mbox{triples }(H',\al',\io')\}/\sim,
\end{align*}
where $H'$ is a local shtuka over $\Spec{A}$, $\al'$ is a Drinfeld level-$m$ structure on $H'$, and $\io'$ is an isomorphism $\br{H}_\de\rar^\sim H'_{\ov\ka}$ such that $(\al')_{\ov\ka} = \io'\circ\al$. In other words, $\br{E}_{\de,\al}$ parametrizes deformations of $(\br{H}_\de,\al)$. Note that $\br{E}_{\de,\al}$ has a right action of $\GL_n(\cO/\vpi^m)$ given by sending $(H',\al',\io')$ to $(H',\al'\circ\ga,\io')$ for any $\ga$ in $\GL_n(\cO/\vpi^m)$. 

By applying the Dieudonn\'e equivalence to rephrase the problem in terms of $\vpi$-divisible local Anderson modules, we immediately deduce the following from classical results of Drinfeld.
\begin{prop*}[{\cite[Proposition 4.2]{Dri74}, \cite[Proposition 4.3]{Dri74}, \cite[2.1.2.(ii)]{Str08}}]\hfill
\begin{enumerate}[(i)]
\item The functor $\br{E}_{\de,\al}$ has a deformation space, which we denote using $\br\fX_{\de,\al}=\Spf\br{R}_{\de,\al}$. This $\br{R}_{\de,\al}$ is a regular complete local Noetherian $\br\cO$-algebra.
\item Write $\br\fX_\de=\Spf\br{R}_\de$ for the deformation space with Drinfeld level-$0$ structure, that is, no Drinfeld level structure. Choosing a coordinate on $\br{H}_\de^\circ$ and an isomorphism $\br{H}_\de\rar^\sim\br{H}_\de^\circ\oplus\ul{(F/\cO)}^{n-k}$ induces an isomorphism from $\br{R}_\de$ to the formal power series ring $\br\cO\llb{s_1,\dotsc,s_{n-1}}$.

\item The same choice as in (ii) induces an isomorphism from $\br{R}_{\de,\al}$ to the formal power series ring $\br{R}_{\de^\circ,\al^\circ}\llb{t_1,\dotsc,t_{n-k}}$, where the $\br{R}_{\de^\circ,\al^\circ}$-algebra structure is given by the restriction morphism $\br{R}_{\de^\circ,\al^\circ}\rar\br{R}_{\de,\al}$.

\item For all non-negative integers $m'\leq m$, the restriction morphism $\br{R}_{\de,\res\al_{m'}}\rar\br{R}_{\de,\al}$ is finite flat. Its generic fiber $\br{R}_{\de,\res\al_{m'}}[\frac1\vpi]\rar\br{R}_{\de,\al}[\textstyle\frac1\vpi]$ is a Galois extension of rings, and the left action of
  \begin{align*}
\ker\left(\GL_n(\cO/\vpi^m)\rar\GL_n(\cO/\vpi^{m'})\right)
  \end{align*}
on $\br{R}_{\de,\res\al_{m'}}\rar\br{R}_{\de,\al}$ yields the Galois action on $\br{R}_{\de,\res\al_{m'}}[\frac1\vpi]\rar\br{R}_{\de,\al}[\textstyle\frac1\vpi]$.
\end{enumerate}
\end{prop*}
In particular, parts (ii) and (iv) imply that $\br\fX_{\de,\al}$ has dimension $n-1$ over $\Spf\br\cO$.

\subsection{}\label{ss:drinfeldparameters}
Let $e_1,\dotsc,e_k$ be an $\cO/\vpi^m$-basis of $\ker\al$. After choosing a coordinate on $\br{H}_{\de^\circ}$, we may identify its $A$-points with $\fm_A$. Classical results of Drinfeld yield certain local parameters of $\br{R}_{\de,\al}$ which satisfy the following relationship with Drinfeld level-$m$ structures.
\begin{prop*}[{\cite[Proposition 4.3.2)]{Dri74}}]
There exist local parameters $x_1,\dotsc,x_k$ of the regular local ring $\br{R}_{\de^\circ,\al^\circ}$ such that, for any morphism $f:\br{R}_{\de^\circ,\al^\circ}\rar A$ in $\wh\cC$, the image of $x_i$ under $f$ equals the element of $\fm_A$ corresponding to $\al'(e_i)$, where $(H',\al',\io')$ is the deformation of $(\br{H}^\circ_\de,\al^\circ)$ corresponding to $f$.
\end{prop*}

\subsection{}\label{ss:deformationalgebraicclosure}
Next, we allow the Drinfeld level-$m$ structure to vary. Write $\br{E}_{\de,m}:\cC\rar(\text{Set})$ for the functor sending
\begin{align*}
A\mapsto\{\mbox{triples }(H',\al',\io')\}/\sim,
\end{align*}
where $H'$ is a local shtuka over $\Spec{A}$, $\al'$ is a Drinfeld level-$m$ structure on $H'$, and $\io'$ is an isomorphism $\br{H}_\de\rar^\sim H'_{\ov\ka}$. In other words, $\br{E}_{\de,m}$ parametrizes deformations of $\br{H}_\de$ along with a Drinfeld level-$m$ structure. As in \ref{ss:connecteddeformationalgebraicclosure}, our functor $\br{E}_{\de,m}$ has a right action of $\GL_n(\cO/\vpi^m)$.

Because every $\al'$ yields a Drinfeld level-$m$ structure $\al'_{\ov\ka}$ of $\br{H}_\de$, we have
\begin{align*}
\br{E}_{\de,m} = \coprod_\al\br{E}_{\de,\al},
\end{align*}
where $\al$ ranges over all Drinfeld level-$m$ structures on $\br{H}_\de$. This disjoint union respects both restriction to the connected component as well as restriction to level-$m'$, where $m'\leq m$. By decomposing Drinfeld level-$m$ structures in terms of the connected-\'etale sequence, we deduce the following result from Proposition \ref{ss:connecteddeformationalgebraicclosure}.
\begin{prop*}
The functor $\br{E}_{\de,m}$ has a deformation space, which we denote using $\br\fX_{\de,m}$. We have an identification 
\begin{align*}
\br\fX_{\de,m} = \coprod_\al\br\fX_{\de,\al} = \coprod_V\coprod_{\ker\al=V}\Spf\br{R}_{\de^\circ,\al^\circ}\llb{t_1,\dotsc,t_{n-k}},
\end{align*}
where $V$ ranges over all $\cO/\vpi^m$-linear direct summands of $(\vpi^{-m}\cO/\cO)^n$ with rank $k$. Under the first identification, the restriction morphisms $\br\fX_{\de,m}\rar\br\fX_{\de,m'}$ equal the disjoint union of the restriction morphisms $\br\fX_{\de,\al}\rar\br\fX_{\de,\res\al_{m'}}$.
\end{prop*}
Fix such a $V$. Since Drinfeld level-$m$ structures on \'etale group schemes are precisely group isomorphisms \cite[(1.8.3)]{KM85}, the set of all $\al$ satisfying $\ker\al=V$ is a right principal homogeneous space for $\GL((\vpi^{-m}\cO/\cO)^n/V)$.

\subsection{}\label{ss:deformationunramifiedextension}
Now suppose that $\de$ lies in $\GL_n(\cO_r)\diag(\vpi,1,\dotsc,1)\GL_n(\cO_r)$. This will allow us to descend $\br\fX_{\de,m}$ to a formal scheme $\fX_{\de,m}$ over $\Spf\cO_r$ as follows. Let $\sg$ be any element of $\Gal(\ov\ka/\ka_r)$, which we identify with its canonical lift to $\Aut(\br\cO/\cO_r)$, and write $\sg$ for $\Spec\sg$ by abuse of notation. The automorphism of $\br{H}_\de = H_{\de,\ov\ka}$ given by $\id_{H_\de}\times_{\ka_r}\sg$ lies over $\sg$ and hence induces an isomorphism $\sg^*\br{H}_\de\rar^\sim\br{H}_\de$ over $\Spec\ov\ka$, which we denote using $f_\sg$.

Note that $\sg^*\br{H}_\de$ is isomorphic to $\br{H}_{\sg(\de)}$ and that $\sg^*\br\fX_{\de,m}$ is isomorphic to $\br\fX_{\sg(\de),m}$. We obtain an isomorphism $\vp_\sg:\br\fX_{\de,m}\rar^\sim\sg^*\br\fX_{\de,m}$ by sending $(H',\al',\io')$ to $(H',\al',\io'\circ f_\sg)$. Since the morphisms $f_\sg$ satisfy the cocycle condition, the morphisms $\vp_\sg$ do as well (albeit contravariantly), and thus the $\vp_\sg$ provide Weil descent datum for $\br\fX_{\de,m}$ as in \cite[Definition (3.5)]{RZ96}\footnote{The definition here is stated for mixed characteristic, but it adapts to equal characteristic by using formal power series instead of Witt vectors.} One can show this Weil descent datum is effective, and we write $\fX_{\de,m}$ for the resulting formal scheme over $\Spf\cO_r$. The right action of $\GL_n(\cO/\vpi^m)$ on $\br\fX_{\de,m}$ commutes with $\vp_\sg$ and hence also descends to a right action on $\fX_{\de,m}$. The restriction morphisms similarly descend to morphisms $\fX_{\de,m}\rar\fX_{\de,m'}$.

\subsection{}\label{ss:nearbycyclesalgebraization}
Let $\fX$ be a special formal scheme over $\Spf{A}$ as in \cite[p.~370]{Ber96} for a complete valuation ring $A$. For any complete nonarchimedean field $K$ containing $A$, we write $\fX_K$ for the adic pullback of $\fX$ to $K$ \cite[p.~370]{Ber96}. Writing $\la$ for the residue field of $K$, we denote the adic pullback\footnote{These are referred to in \cite{Ber96} as the generic and special fibers, respectively. Our terminology stems from interpreting them as pullbacks in the category of adic spaces.} of $\fX$ to $\la$ by $\fX_\la$ \cite[p.~370]{Ber96}. These constructions correspond to generic and special fibers, respectively, in the setting of nonarchimedean analytic geometry.

We shall now introduce the cohomology of our deformation spaces. Write $R^i\psi_{\de,m}$ for the $\ov\bQ_\ell$-vector space $H^i(\fX_{\de,m,\bC_\vpi},\ov\bQ_\ell)$. The zero-dimensionality of $\fX_{\de,m,\ov\ka}$ and the nearby cycles spectral sequence \cite[Corollary 2.5]{Ber96} show that
\begin{align*}
R^i\psi_{\de,m} = H^0(\fX_{\de,m,\ov\ka},R^i\Psi_{\fX_{\de,m,\bC_\vpi}}\ov\bQ_\ell),
\end{align*}
where $R^i\Psi_{\fX_{\de,m,\bC_\vpi}}$ denotes the $i$-th nearby cycles functor\footnote{This is referred to in \cite{Ber96} and \cite{HT01} as the vanishing cycles functor.} \cite[p.~373]{Ber96} on $\fX_{\de,m,\cO_{\bC_\vpi}}$. This equality explains our choice of notation for $R^i\psi_{\de,m}$. To prove the finite-dimensionality of $R^i\psi_{\de,m}$ as well as eventually prove the admissibility of a certain $\GL_n(\cO)$-action, we take recourse to the following algebraization result.
\begin{lem*}\label{lem:alglol}
There exists a projective scheme $\cM$ over $\Spec\br\cO$, a zero-dimensional closed subscheme $\cZ$ of $\cM_{\ov\ka}$, and a right action of $\GL_n(\cO/\vpi^m)$ on the pair $(\cM,\cZ)$ such that the completion of $\cM$ at $\cZ$ is isomorphic to $\br\fX_{\de,m}$ with its right action of $\GL_n(\cO/\vpi^m)$.
\end{lem*}
The desired algebraization comes from a moduli space of \emph{$\sD$-elliptic sheaves with bad reduction}, but we won't introduce these moduli spaces until \S\ref{s:modulispaces}, and we won't explain how to relate them to local shtukas until \S\ref{s:langlandskottwitz}. Although we record the proof here, one can safely take this algebraization result as a black box.
\begin{proof}[Proof of Lemma \ref{lem:alglol}]
Let $C=\bP^1_\ka$ be our curve of interest, and let $\infty$ and $o$ be distinct $\ka$-points of $C$. Let $D$ be any central division algebra over $\ka(C)$ of dimension $n^2$ that splits at $o$ and $\infty$, and let $\sD$ be a maximal order of $D$, which can be constructed using \ref{ss:orders} because division algebras split at cofinitely many places. Proposition \ref{ss:dellipticsheavesisomorphismclasses} allows us to find a $\sD$-elliptic sheaf $(\sE_i,t_i,j_i)_i$ over $\ov\ka$ whose local shtuka $\sM_o'$ at $o$ as in \ref{ss:dellipticsheavestolocalshtukazero} is isomorphic to $\br{H}_\de$.

Denote the corresponding $\ov\ka$-point of $\cM_{\varnothing,\br\cO}$ using $z$. Write $o^m$ for the finite closed subscheme of $C$ supported on $o$ with multiplicity $m$, and write $\pi:\cM_{o^m,\br\cO}\rar\cM_{\varnothing,\br\cO}$ for the restriction morphism. Then \ref{ss:extendedserretate} says that $\cM=\cM_{o^m,\br\cO}$ and $\cZ=\pi^{-1}(z)$ yield the desired algebraization of $\br\fX_{\de,m}$.
\end{proof}

\subsection{}\label{ss:fundamentalrepresentation}
With Lemma \ref{ss:nearbycyclesalgebraization} in hand, Berkovich's nearby cycles comparison theorem \cite[Theorem 3.1]{Ber96} implies that $R^i\psi_{\de,m}$ is finite-dimensional over $\ov\bQ_\ell$ and vanishes for $i>n-1$ by Proposition \ref{ss:connecteddeformationalgebraicclosure}. Furthermore, it has commuting continuous left actions of $W_{F_r}$ and $\GL_n(\cO/\vpi^m)$. Write $R^i\psi_\de$ for the direct limit
\begin{align*}
R^i\psi_\de \deq \dirlim_mR^i\psi_{\de,m},
\end{align*}
where the transition maps $R^i\psi_{\de,m'}\rar R^i\psi_{\de,m}$ are induced by the restriction morphisms $\fX_{\de,m}\rar\fX_{\de,m'}$ for $m'\leq m$. Now $R^i\psi_\de$ has commuting left actions of $W_{F_r}$ and $\GL_n(\cO)$, and Proposition \ref{ss:connecteddeformationalgebraicclosure}.(iv) shows that
\begin{align*}
(R^i\psi_\de)^{1+\vpi^m\!\M_n(\cO)} = R^i\psi_{\de,m}.
\end{align*}
Therefore $R^i\psi_\de$ is a $\GL_n(\cO)\times I_F$-admissible/continuous representation as in \cite[p.~24]{HT01} of $\GL_n(\cO)\times W_{F_r}$ over $\ov\bQ_\ell$. Write $[R\psi_\de]$ for the virtual representation
\begin{align*}
[R\psi_\de] \deq \sum_{i=0}^\infty(-1)^iR^i\psi_\de.
\end{align*}

\subsection{}\label{ss:sigmaconjugationisopen}
At this point, we can finally begin to define the test function $f_{\tau,h}$ mentioned in Theorem A. We shall start by defining $\phi_{\tau,h}$, which will end up being a transfer of $f_{\tau,h}$. Let $\tau$ be an element of $W_F$ with positive $v(\tau)$, write $r=v(\tau)$, and let $h$ be a function in $C^\infty_c(\GL_n(\cO))$. We write $\phi_{\tau,h}:\GL_n(F_r)\rar\bC$ for the function
\begin{align*}
\de\mapsto
\begin{cases}
\tr(\tau\times h|[R\psi_\de]) & \mbox{if }\de\mbox{ is in }\GL_n(\cO_r)\diag(\vpi,1,\dotsc,1)\GL_n(\cO_r),\\
0 & \mbox{otherwise,}
\end{cases}
\end{align*}
where we have identified $\ov\bQ_\ell$ with $\bC$.
\begin{lem*}
The map $\GL_n(\br{F})\rar\GL_n(\br{F})$ given by $g\mapsto g^{-1}\de\sg(g)$ is open, where $\sg$ denotes the lift of $q$-Frobenius.
\end{lem*}
\begin{proof}
This follows from reducing to the Lie algebra situation, using the Dieudonn\'e--Manin classification to casework on $\de$, and applying the nonarchimedean Banach open mapping theorem, c.f. \cite[Lemma 4.4]{Sch13c}.
\end{proof}
Lemma \ref{ss:sigmaconjugationisopen} implies that the subsets $B_{n,k}$ of $\GL_n(\cO_r)\diag(\vpi,1,\dotsc,1)\GL_n(\cO_r)$ are open. Because
\begin{align*}
\GL_n(\cO_r)\diag(\vpi,1,\dotsc,1)\GL_n(\cO_r) = \coprod_{k=1}^nB_{n,k},
\end{align*}
we see that the subsets $B_{n,k}$ are also closed.

One use Lemma \ref{ss:sigmaconjugationisopen} and proceed as in \cite[Proposition 4.3]{Sch13c} to show that $\phi_{\tau,h}$ is locally constant. As $\phi_{\tau,h}$ is supported on the compact subset $\GL_n(\cO_r)\diag(\vpi,1,\dotsc,1)\GL_n(\cO_r)$ by definition, we see that $\phi_{\tau,h}$ is in $C^\infty_c(\GL_n(\cO_r))$.

\subsection{}\label{ss:ftauhdefinition}
Let $f_{\tau,h}$ be a transfer of $\phi_{\tau,h}$, i.e. a function in $C^\infty_c(\GL_n(F))$ such that $f_{\tau,h}$ and $\phi_{\tau,h}$ have matching twisted orbital integrals \cite[(I.2.5, prop.)]{HL11}. We remark that while $f_{\tau,h}$ is not uniquely determined as a function in $C^\infty_c(\GL_n(F))$, its orbital integrals are well-defined by the fact that they equal the corresponding twisted orbital integrals of $\phi_{\tau,h}$. Hence the Weyl integration formula \cite[(II.2.10, formula (1)]{HL11} implies that the trace of $f_{\tau,h}$ on admissible representations of $\GL_n(F)$ is also well-defined. 

Since we have now defined $f_{\tau,h}$, the statement of Theorem A makes sense.

\section{The first inductive lemma: trace identities and parabolic induction}\label{s:firstinductivelemma}
In this section, we begin by recalling the \emph{Bernstein--Zelevinsky classification} of irreducible smooth representations of $\GL_n(F)$, which we use to prove a lemma on checking equalities of traces. Then, we use Schneider--Zink's theory of \emph{tempered types} for $\GL_n(F)$, which is a version of Bushnell--Kutzko types for tempered representations, to construct test functions that pick out certain representations.

With these preliminaries, we proceed to the main result of this section: a lemma which allows us to prove Theorem A by inducting on $n$, provided that we verify certain conditions. The proof of our lemma uses theory of the \emph{Bernstein center}, which describes the center of the category of smooth representations of $\GL_n(F)$. In subsequent sections, we will focus on verifying the necessary conditions for our inductive lemma.

\subsection{}\label{ss:essentiallyl2}
First, we recall Bernstein--Zelevinsky's description of irreducible essentially $L^2$ representations of $\GL_n(F)$. Let $\De$ be a segment 
\begin{align*}
\De = \{\pi_0[\textstyle\frac{1-m}2], \pi_0[\frac{3-m}2],\dotsc,\pi_0[\frac{m-1}2]\}
\end{align*}
as in \cite[3.1]{Zel80}, where $m$ is a positive divisor of $n$, $\pi_0$ is an irreducible cuspidal representation of $\GL_{n/m}(F)$, and $[s]$ denotes twisting by the unramified character $\abs\det^s$. Consider the normalized parabolic induction
\begin{align*}
\nInd_{P(F)}^{\GL_n(F)}(\pi_0[\textstyle\frac{1-m}2]\otimes\pi_0[\frac{3-m}2]\otimes\dotsb\otimes\pi_0[\frac{m-1}2]),
\end{align*}
where $P$ is the standard parabolic subgroup of $\GL_n$ with block sizes $(m,\dotsc,m)$. This representation of $\GL_n(F)$ has a unique irreducible quotient \cite[9.1]{Zel80}, which we denote by $Q(\De)$.

Recall that the irreducible essentially $L^2$ representations of $\GL_n(F)$ are those isomorphic to $Q(\De)$ \cite[9.3]{Zel80}. We see that $\De$ is the cuspidal support of $Q(\De)$, and note that $Q(\De)$ has unitary central character if and only if $\pi_0$ does.

\subsection{}\label{ss:bernsteinzelevinsky}
Next, we proceed to arbitrary irreducible smooth representations of $\GL_n(F)$. Let $\{\De_1,\dotsc,\De_t\}$ be a collection of segments such that $\De_i$ does not precede $\De_j$ as in \cite[4.1]{Zel80} for $i<j$. Each $Q(\De_i)$ is an irreducible essentially $L^2$ representation of $\GL_{n_i}(F)$, and we may form the normalized parabolic induction
\begin{align*}
\nInd_{P(F)}^{\GL_n(F)}\left(Q(\De_1)\otimes\dotsb\otimes Q(\De_t)\right),
\end{align*}
where $n=n_1+\dotsb+n_t$, and $P$ is the standard parabolic subgroup of $\GL_n$ with block sizes $(n_1,\dotsc,n_t)$. When none of the $\De_i$ are linked, this representation of $\GL_n(F)$ is irreducible \cite[9.7.(a)]{Zel80}. For general $\De_i$, it has a unique irreducible quotient, which we denote by $Q(\De_1,\dotsc,\De_t)$. Furthermore, the induced representation above is irreducible precisely when $Q(\De_1,\dotsc,\De_t)$ is a generic representation of $\GL_n(F)$ \cite[9.7.(b)]{Zel80}.

The isomorphism class of $Q(\De_1,\dotsc,\De_t)$ does not depend on the ordering of the $\{\De_1,\dotsc,\De_t\}$, as long as it still satisfies the condition that $\De_i$ does not precede $\De_j$ for $i<j$. Recall that every irreducible smooth representation of $\GL_n(F)$ is isomorphic to $Q(\De_1,\dotsc,\De_t)$ for a unique choice of $\{\De_1,\dotsc,\De_t\}$ \cite[9.7.(b)]{Zel80}, up to reordering. We see that the multiset $\De_1\coprod\dotsb\coprod\De_t$ is the cuspidal support of $Q(\De_1,\dotsc,\De_t)$. Furthermore, recall that $Q(\De_1,\dotsc,\De_t)$ is tempered if and only if every $Q(\De_i)$ is $L^2$ \cite{Jac77}.

\subsection{}\label{ss:generalizedsteinbergspeh}
The following terminology generalizes the Steinberg and trivial representations, respectively.
\begin{defn*}
Let $t$ be a positive divisor of $n$, and let $\pi_0$ be an irreducible cuspidal representation of $\GL_{n/t}(F)$. We write $\St_t(\pi_0)$ for the representation
\begin{align*}
\St_t(\pi_0)\deq Q\left(\{\pi_0[\textstyle\frac{1-t}2], \pi_0[\frac{3-t}2],\dotsc,\pi_0[\frac{t-1}2]\}\right),
\end{align*}
and we write $\Sp_t(\pi_0)$ for the representation
\begin{align*}
\Sp_t(\pi_0) \deq Q\left(\{\pi_0[\textstyle\frac{t-1}2]\},\{\pi_0[\frac{t-3}2]\},\dotsc,\{\pi_0[\frac{1-t}2]\}\right).
\end{align*}
We say that a representation of $\GL_n(F)$ is a \emph{Steinberg module} if it is isomorphic to some $\Sp_t(\pi_0)$ for some $\pi_0$ with unitary central character, and we say it is a \emph{Speh module}\footnote{What we call a Steinberg module is often called a generalized Steinberg representation. However, what we call a Speh module is actually less general than the usual definition of Speh representations.} if it is isomorphic to $\Sp_t(\pi_0)$ for some $\pi_0$ with unitary central character. Note that \ref{ss:essentiallyl2} indicates that being a Steinberg module means the same thing as being irreducible $L^2$.
\end{defn*}
We recover the usual Steinberg and trivial representations by taking $t=n$ and letting $\pi_0$ be the trivial representation.

\subsection{}\label{ss:kazhdanvariant}
We now turn to the following lemma on checking equalities of traces.
\begin{lem*}
Let $f$ be a function in $C^\infty_c(\GL_n(F))$, and suppose that $\tr(f|\pi)=0$ whenever $\pi$ is an irreducible smooth representation of $\GL_n(F)$ that is
\begin{enumerate}[(a)]
\item tempered but not $L^2$,
\item a Speh module.
\end{enumerate}
Then $\tr(f|\pi)=0$ for all irreducible smooth representations $\pi$ of $\GL_n(F)$.
\end{lem*}
\begin{proof}
Kazhdan's density theorem \cite[Theorem 0]{Kaz86}\footnote{The proof given here is stated for $p$-adic fields, but it only uses the Langlands classification and hence carries over to any nonarchimedean local field.} indicates that it suffices to check $\tr(f|\pi)=0$ for all irreducible tempered $\pi$, and (a) above further reduces it to checking $\tr(f|\pi)=0$ whenever $\pi$ is irreducible $L^2$. Thus we can write $\pi$ as $\St_t(\pi_0)$, where $t$ is some positive divisor of $n$, we set $d\deq n/t$, and $\pi_0$ is an irreducible cuspidal representation of $\GL_d(F)$ with unitary central character.

If $t=1$, then $\pi=\pi_0=\Sp_1(\pi_0)$, so (b) shows that $\tr(f|\pi)=0$. Next, suppose that $t\geq2$. A routine calculation \cite[Lemma I.3.2]{HT01} using the graph-theoretic description of the Jordan--H\"older factors of
\begin{align*}
\nInd_{P(F)}^{\GL_n(F)}(\pi_0[\textstyle\frac{1-t}2]\otimes\pi_0[\frac{3-t}2]\otimes\dotsb\otimes\pi_0[\frac{t-1}2]),
\end{align*}
where $P$ is the standard parabolic subgroup of $\GL_n$ with block sizes $(d,\dotsb,d)$, yields the following equality of virtual smooth representations of $\GL_n(F)$:
\begin{align*}
\Sp_t(\pi_0)+(-1)^t\pi = \sum_{j=1}^{t-1}(-1)^{j-1}\nInd_{P_j(L)}^{\GL_n(L)}\left(\Sp_{t-j}(\pi_0[\textstyle\frac{t-1+j}2])\otimes \St_j(\pi_0[\frac{j-1}2])\right),
\end{align*}
where $P_j$ is the standard parabolic subgroup of $\GL_n$ with block sizes $(d(t-j),dj)$. Note that $\Sp_t(\pi_0)$ is a Speh module. Therefore once we prove the following lemma, taking traces and using (b) yields the desired result.
\end{proof}
\begin{lem}\label{lem:tempguys}
Let $f$ be a function in $C^\infty_c(\GL_n(F))$, and suppose that $\tr(f|\pi)=0$ whenever $\pi$ is irreducible tempered but not $L^2$. Then for all partitions $n=n_1+n_2$ and irreducible smooth representations $\pi_i$ of $\GL_{n_i}(F)$, we have
\begin{align*}
\tr\left(f|\nInd_{P(F)}^{\GL_n(F)}(\pi_1\otimes\pi_2)\right)=0,
\end{align*}
where $P$ is the standard parabolic subgroup of $\GL_n$ with block sizes $(n_1,n_2)$.
\end{lem}
For any function $f$ in $C_c^\infty(\GL_n(F))$, we write $f^{P,\GL_n(\cO)}$ for its normalized $\GL_n(\cO)$-invariant constant term along a parabolic subgroup $P$ as in \cite[p.~80]{Lem16}. It is a function in $C_c^\infty(M(F))$, where $M$ is the standard Levi subgroup of $P$.
\begin{proof}[Proof of Lemma \ref{lem:tempguys}]
Van Dijk's formula \cite[Theorem 5.9]{Lem16} implies that
\begin{align*}
\tr\left(f|\nInd_{P(F)}^{\GL_n(F)}(\pi_1\otimes\pi_2)\right) = \tr(f^{P,\GL_n(\cO)}|\pi_1\otimes\pi_2).
\end{align*}
Recall from \ref{ss:bernsteinzelevinsky} that, if $\pi_1$ and $\pi_2$ are tempered, then
\begin{align*}
\nInd_{P(F)}^{\GL_n(F)}(\pi_1\otimes\pi_2)
\end{align*}
is irreducible tempered but not $L^2$. In this situation, the left-hand side of the above equation vanishes by assumption. As $\pi_1$ and $\pi_2$ run over tempered representations of $\GL_{n_1}(F)$ and $\GL_{n_2}(F)$, respectively, $\pi_1\otimes\pi_2$ runs over irreducible tempered representations of $M=\GL_{n_1}(F)\times\GL_{n_2}(F)$. Thus applying Kazhdan's density theorem \cite[Theorem 0]{Kaz86} to $f^{P,\GL_n(\cO)}$ indicates that the right-hand and hence left-hand side of the above equation vanishes for all $\pi_1$ and $\pi_2$, as desired.
\end{proof}

\subsection{}\label{ss:schneiderzinktestfunctions}
By using Schneider--Zink's tempered types, we can construct the following test functions that pick out Speh modules.
\begin{lem*}
  Let $\pi=\Sp_t(\pi_0)$ be a Speh module, where $t$ is a positive divisor of $n$, and $\pi_0$ is an irreducible cuspidal representation of $\GL_{n/t}(F)$ with unitary central character. Then there exists a function $h$ in $C^\infty_c(\GL_n(\cO))$ such that
  \begin{enumerate}[(i)]
  \item $\tr(h|\pi)=1$,
    \item if $\pi'$ is an irreducible tempered representation of $\GL_n(F)$ that is not isomorphic to $Q\left(\{\pi_0[s_1]\},\dotsc,\{\pi_0[s_t]\}\right)$ for some purely imaginary $s_1,\dotsc,s_t$, then $\tr(h|\pi')=0$.
  \end{enumerate}
\end{lem*}
\begin{proof}
  We reduce this to tempered type theory as follows. Suppose we could find an irreducible smooth representation $\la$ of $\GL_n(\cO)$ such that
  \begin{enumerate}[(i)]
  \item $\pi$ contains $\la$ with multiplicity $1$,
    \item if $\pi'$ is an irreducible tempered representation of $\GL_n(F)$ that is not isomorphic to $Q\left(\{\pi_0[s_1]\},\dotsc,\{\pi_0[s_t]\}\right)$ for some purely imaginary $s_1,\dotsc,s_t$, then $\pi'$ does not contain $\la$.
    \end{enumerate}
    Then Schur orthogonality shows that the function $g\mapsto\ov{\tr(g|\la)}$ yields the desired $h$. As for finding such a $\la$, we use \cite{SZ99}. More precisely, let $\la$ be the representation $\sg_\cP(\la)$ as in \cite[p.~30]{SZ99} for $\cP=t\cdot\De(\pi_0,1)$, where we adopt the notation of \cite{SZ99}. Note that $\pi$ lies in $\im Q_\cP$, so (i) follows from \cite[Prop. 11, i.]{SZ99}. Furthermore, we see that $\cP$ is a maximal partition-valued function on $\cC$, so (ii) follows from \cite[Prop. 11, iii.]{SZ99}.
\end{proof}

\subsection{}\label{ss:firstinductivelemma}
Finally, we can introduce the main result of this section: our inductive lemma.
\begin{lemfirst}
Assume that the following conditions hold for all admissible representations $\pi$ of $\GL_n(F)$:
\begin{enumerate}[(a)]
\item Theorem A is true for all $n'<n$,
\item if $\pi$ is isomorphic to the normalized parabolic induction
  \begin{align*}
    \nInd^{\GL_n(F)}_{P(F)}(\pi_1\otimes\dotsb\otimes\pi_t),
  \end{align*}
where $t\geq2$, the $\pi_i$ are irreducible smooth representations of $\GL_{n_i}(F)$, we have $n=n_1+\dotsb+n_t$, and $P$ is the standard parabolic subgroup of $\GL_n$ with block sizes $(n_1,\dotsc,n_t)$, then we have an equality of traces
\begin{align*}
  \tr(f_{\tau,h}|\pi) = \tr\left(\tau|\rho(\pi_1)(\textstyle\frac{n-n_1}2)\oplus\dotsb\oplus\rho(\pi_t)(\frac{n-n_t}2)\right)\tr(h|\pi)
\end{align*}
for all $\tau$ in $W_F$ with $v(\tau)>0$ and $h$ in $C^\infty_c(\GL_n(\cO))$,
\item if $\pi$ is a Speh module or irreducible essentially $L^2$, there exists an $n$-dimensional $\bQ_{\geq0}$-virtual continuous representation $\rho(\pi)$ of $W_F$ satisfying the trace condition
  \begin{align*}
    \tr(f_{\tau,h}|\pi) = \tr\left(\tau|\rho(\pi)\right)\tr(h|\pi)
  \end{align*}
for all $\tau$ in $W_F$ with $v(\tau)>0$ and $h$ in $C^\infty_c(\GL_n(\cO))$,
\item if $\pi$ is irreducible cuspidal, then the representation $\rho(\pi)$ from condition (c) is actually a $\bZ$-virtual continuous representation of $W_F$.
\end{enumerate}
Then Theorem A is true for $n$.
\end{lemfirst}
Since semisimple representations of $W_F$ are determined by their traces, any $\rho(\pi)$ satisfying the equality of traces in Theorem A.(i) is unique up to isomorphism. In particular, the $\rho(\pi)$ whose existence is posited by Theorem A.(i) is isomorphic to the $\rho(\pi)$ whose existence is provided by condition (c).

We now proceed to prove the first inductive lemma. To this end, for the remainder of this section we shall assume that conditions (a)--(d) hold. Later, we shall verify these conditions in \S\ref{s:parabolicinduction}, \S\ref{s:lubintatetower}, and \S\ref{s:localglobal}. 

\subsection{}\label{ss:cuspidalrho}
We begin with the following observation. Suppose that $\pi$ is irreducible cuspidal. Then condition (c) provides an $n$-dimensional $\bQ$-virtual continuous representation $\rho(\pi)$ of $W_F$ such that, when expanded in terms of the $\bQ$-basis of irreducible continuous representations of $W_F$, every coefficient is non-negative. Condition (d) implies that these coefficients are integers, so $\rho(\pi)$ corresponds to an actual $n$-dimensional semisimple continuous representation of $W_F$.

\subsection{}\label{ss:bernsteincenter}
Now let $\pi$ be any irreducible smooth representation of $\GL_n(F)$. Before proceeding, we gather a few recollections on the Bernstein center. For any $f$ in the Bernstein center $\wh\cZ$ of $\GL_n(F)$, we write $\tr(f|\pi)$ for the scalar by which $f$ acts on $\pi$. We identify $\wh\cZ$ with the product of the rings of regular functions on the Bernstein components of $\GL_n(F)$ \cite[2.10]{Ber84}.

Write $\pi_1,\dotsc,\pi_t$ for the cuspidal support of $\pi$, where the $\pi_i$ are irreducible cuspidal representations of $\GL_{n_i}(F)$ such that $n=n_1+\dotsb+n_t$. Write $P$ for the standard parabolic subgroup of $\GL_n$ with block sizes $(n_1,\dotsc,n_t)$. For any function $h$ in $C^\infty_c(\GL_n(F))$, van Dijk's formula \cite[Theorem 5.9]{Lem16} implies that the assignment
\begin{align*}
\pi\mapsto\tr\left(h|\nInd_{P(F)}^{\GL_n(F)}(\pi_1\otimes\dotsb\otimes\pi_t)\right)
\end{align*}
yields an element of $\wh\cZ$, which we shall denote using $\wh{h}$. When $\nInd_{P(F)}^{\GL_n(F)}(\pi_1\otimes\dotsb\otimes\pi_t)$ is irreducible and hence isomorphic to $\pi$, we have $\tr(\wh{h}|\pi)=\tr(h|\pi)$. However, we stress that this is false for general $\pi$.

\subsection{}\label{ss:functiondesire}
We construct an element $f_\tau$ of $\wh\cZ$ as follows. Write $r(\pi)$ for the representation
\begin{align*}
r(\pi)\deq\rho(\pi_1)(\textstyle\frac{1-n_1}2)\oplus\dotsb\oplus\rho(\pi_t)(\frac{1-n_t}2),
\end{align*}
which is well-defined for $t=1$ by \ref{ss:cuspidalrho} and for $t\geq2$ by Theorem A.(i) for the $\pi_i$. Note that $r(\pi)$ is semisimple. We shall see in Proposition \ref{ss:convolutiontraceequality} that $r(\pi)$ is always isomorphic to $\rec(\pi)$, where the latter is defined as in Theorem A.
\begin{prop*}
The function $f_\tau$ whose action on $\pi$ is given by $\tr\left(\tau|\sg(\pi)(\frac{n-1}2)\right)$ is an element of $\wh\cZ$.
\end{prop*}
This fulfills the desire to have a function-like object (in this case, an element of the Bernstein center, which is a distribution) associated with $\tau$ whose trace equals the action of $\tau$ on $r(\pi)$ and hence on $\rec(\pi)$.

\begin{proof}
It suffices to prove that $\pi\mapsto\tr\left(\tau|r(\pi)(\frac{n-1}2)\right)$ is a regular function on every Bernstein component, so we may focus on one particular $\pi$. First, use Schur orthogonality to obtain a function $h_1$ in $C^\infty_c(\GL_n(\cO))$ that satisfies
\begin{align*}
\tr\left(h_1|\nInd_{P(F)}^{\GL_n(F)}(\pi_1\otimes\dotsb\otimes\pi_t)\right)=1,
\end{align*}
where $P$ is the standard parabolic subgroup of $\GL_n$ with block sizes $(n_1,\dotsc,n_t)$. Now any other irreducible representation $\pi'$ in the Bernstein component of $\pi$ is a subquotient of
\begin{align*}
\nInd_{P(F)}^{\GL_n(F)}(\pi_1[s_1]\otimes\dotsb\otimes\pi_t[s_t]),
\end{align*}
where the $s_i$ are complex numbers. The Iwasawa decomposition implies that $h_1^{P,\GL_n(\cO)}$ is supported on $M(\cO)$, where $M$ is the standard Levi subgroup of $P$, so van Dijk's formula \cite[Theorem 5.9]{Lem16} shows that
\begin{align*}
\tr\left(h_1|\nInd_{P(F)}^{\GL_n(F)}(\pi_1[s_1]\otimes\dotsb\otimes\pi_t[s_t])\right)=  \tr\left(h_1|\nInd_{P(F)}^{\GL_n(F)}(\pi_1\otimes\dotsb\otimes\pi_t)\right)=1.
\end{align*}
Because $\pi_1[s_1],\dotsc,\pi_t[s_t]$ is the cuspidal support of $\pi'$, we obtain
\begin{align*}
\tr\left(\tau|r(\pi')(\textstyle\frac{n-1}2)\right) &= \tr\left(\tau|\rho(\pi_1[s_1])(\textstyle\frac{n-n_1}2)\oplus\dotsb\oplus\rho(\pi_t[s_t])(\frac{n-n_t}2)\right) \\
&= \tr\left(\tau|\rho(\pi_1[s_1])(\textstyle\frac{n-n_1}2)\oplus\dotsb\oplus\rho(\pi_t[s_t])(\frac{n-n_t}2)\right)\tr\left(h_1|\nInd_{P(F)}^{\GL_n(F)}(\pi_1[s_1]\otimes\dotsb\otimes\pi_t[s_t])\right).
\end{align*}
This equals 
\begin{align*}
\tr\left(f_{\tau,h_1}|\nInd_{P(F)}^{\GL_n(F)}(\pi_1[s_1]\otimes\dotsb\otimes\pi_t[s_t])\right)
\end{align*}
for $t=1$ by condition (c) and for $t\geq2$ by condition (b). Altogether, we see that $f_\tau$ acts on the Bernstein component of $\pi$ via $\wh{f_{\tau,h_1}}$, so it is a regular function.
\end{proof}

\subsection{}\label{ss:convolutiontraceequality}
We can relate the distribution $f_\tau$ to the test function $f_{\tau,h}$ as follows. Recall that the convolution product $f_\tau*h$  is a function in $C^\infty_c(\GL_n(F))$.
\begin{prop*}
We have $\tr(f_{\tau,h}|\pi) = \tr(f_\tau*h|\pi)$.
\end{prop*}
Note that, since $f_\tau$ lies in the Bernstein center of $\GL_n(F)$, we have
\begin{align*}
\tr(f_\tau*h|\pi) = \tr(f_\tau|\pi)\tr(h|\pi) = \tr\left(\tau|r(\pi)(\textstyle\frac{n-1}2)\right)\tr(h|\pi).
\end{align*}
By taking $\rho(\pi)=r(\pi)(\textstyle\frac{n-1}2)$, we see that this proposition gives the existence part of Theorem A.(i) for $n$.
\begin{proof}[Proof of Proposition \ref{ss:convolutiontraceequality}]
By Lemma \ref{ss:kazhdanvariant}, it suffices to consider $\pi$ that are either Speh modules or irreducible tempered but not $L^2$. In the latter case, the descriptions given in \ref{ss:bernsteinzelevinsky} imply that $\pi$ is isomorphic to a representation of the form considered in condition (b). Therefore condition (b) yields the desired result in this case.

Turning to the case where $\pi$ is a Speh module, suppose that $\pi$ is isomorphic to $\Sp_t(\pi_0)$ for some positive divisor $t$ of $n$ and irreducible cuspidal representation $\pi_0$ of $\GL_{n/t}(F)$ with unitary central character. If $t=1$, then $\pi=\pi_0$ is cuspidal, and condition (c) gives the desired equality. Therefore suppose that $t\geq2$, and let $h_2$ be a function in $C^\infty_c(\GL_n(\cO))$ as in Lemma \ref{ss:schneiderzinktestfunctions}. We start by proving the case where $h=h_2$.
\begin{lem}\label{lem:nestedkazhdan}
For all irreducible smooth representations $\pi'$ of $\GL_n(F)$, we have $\tr\left(f_{\tau,h_2}|\pi'\right) = \tr\left(f_\tau*h_2|\pi'\right)$.
\end{lem}
\begin{proof}
Kazhdan's density theorem \cite[Theorem 0]{Kaz86} indicates that it suffices to check this for irreducible tempered $\pi'$. When $\pi'$ is not $L^2$, this identity is a case of Proposition \ref{ss:convolutiontraceequality} we proved above, so we now tackle the case when $\pi'$ is $L^2$. Since $t\geq2$, the descriptions given in \ref{ss:essentiallyl2} and \ref{ss:bernsteinzelevinsky} indicate that $\pi'$ cannot be of the form $Q\left(\{\pi_0[s_1]\},\dotsc,\{\pi_0[s_t]\}\right)$ for purely imaginary $s_1,\dotsc,s_t$. Thus condition (c) and Lemma \ref{ss:schneiderzinktestfunctions}.(ii) imply that
\begin{align*}
\tr\left(f_{\tau,h_2}|\pi'\right) = \tr\left(\tau|\rho(\pi')\right)\tr(h_2|\pi') = 0 = \tr\left(\tau|r(\pi)(\textstyle\frac{n-1}2)\right)\tr(h_2|\pi) = \tr\left(f_\tau*h_2|\pi'\right),
\end{align*}
which concludes the proof of Lemma \ref{lem:nestedkazhdan}.
\end{proof}
Return to the proof of Proposition \ref{ss:convolutiontraceequality}. By plugging $\pi'=\pi$ into Lemma \ref{lem:nestedkazhdan} and using condition (c) and Lemma \ref{ss:schneiderzinktestfunctions}.(i), we see that
\begin{align*}
\tr\left(\tau|\rho(\pi)\right) = \tr\left(f_{\tau,h_2}|\pi\right) = \tr\left(f_\tau*h_2|\pi\right) = \tr\left(\tau|r(\pi)(\textstyle\frac{n-1}2)\right).
\end{align*}
Because semisimple representations of $W_F$ are determined by their traces, this indicates that $\rho(\pi)$ is isomorphic to $r(\pi)(\textstyle\frac{n-1}2)$. In particular, applying condition (c) again verifies the desired equality when $\pi$ is a Speh module.

\end{proof}
We conclude this section by finishing the proof of Lemma \ref{ss:firstinductivelemma}, that is, by proving that Theorem A holds for $n$.
\begin{proof}[Proof of Lemma \ref{ss:firstinductivelemma}]
We have already noted that Proposition \ref{ss:convolutiontraceequality} gives the existence part of Theorem A.(i) for $n$. Since semisimple representations of $W_F$ are determined by their traces, any such $\rho(\pi)$ satisfying the equality of traces in Theorem A.(i) is unique. As for Theorem A.(ii), it immediately follows from the associativity of normalized parabolic induction as well as the definition of $\sg(\pi)$ in terms of the cuspidal support of $\pi$.
\end{proof}
The next several sections will focus on verifying that conditions (b)--(d) hold, sometimes while assuming the inductive condition (a). More precisely, we shall prove condition (b) in Proposition \ref{ss:conditionb} under the assumption that the inductive condition (a) holds, we will prove condition (d) in Proposition \ref{ss:conditiond} while black-boxing condition (c), and we shall prove condition (c) itself in Proposition \ref{s:conditionc}. We will also prove the $n=1$ case of Theorem A in Proposition \ref{ss:theoremabasecase}, which serves as the base case for applying Lemma \ref{ss:firstinductivelemma}.

\section{Geometry of parabolic induction}\label{s:parabolicinduction}
Our goal in this section is to prove that condition (b) in Lemma \ref{ss:firstinductivelemma} holds, under the assumption that the inductive condition (a) holds. To do this, we start by using \'etale local shtukas to prove the existence of transfers for $\GL_n(\cO)$-conjugation-invariant functions in $C^\infty_c(\GL_n(\cO))$. We then use work from this proof, along with our description of the deformation spaces $\br\fX_{\de,m}$ from \S\ref{s:deformationspaces}, to prove an identity relating $\tr(f_{\tau,h}|\pi)$ with the traces of $\phi_{\tau,h'}$ on certain Jacquet restrictions of $\pi$ (where $h'$ now ranges over functions in $C^\infty_c(M(\cO))$ for certain Levi blocks $M$ in $\GL_n$). The proof of this identity uses \emph{Casselman's theorem} on characters of Jacquet modules. We finish the proof that condition (b) holds, under the assumption that the inductive condition (a) holds, by using the Bernstein--Zelevinsky \emph{induction-restriction formula} to relate this aforementioned trace identity to a reformulated version of condition (b).

\subsection{}\label{ss:normbijection}
We begin by using two descriptions of isomorphism classes of \'etale local shtukas over $\Spec\ka_r$ to prove the following bijection. In the end, this amounts to an application of completed unramified descent.
\begin{lem*}
The norm map $\de\mapsto\de\sg(\de)\dotsm\sg^{r-1}(\de)$ induces a bijection
\begin{align*}
\N:\{\GL_n(\cO_r)\mbox{-}\sg\mbox{-conjugacy classes in }\GL_n(\cO_r)\}\rar^\sim\{\GL_n(\cO)\mbox{-conjugacy classes in }\GL_n(\cO)\}.
\end{align*}
\end{lem*}
\begin{proof}
We shall construct bijections between both the left-hand as well as right-hand sides and the set
\begin{align*}
\{\mbox{isomorphism classes of \'etale local shtukas over }\Spec\ka_r\mbox{ of rank }n\}
\end{align*}
such that the composed bijection is equal to $\N$. On the left-hand side, \ref{ss:effectiveminusculelocalshtukaconjugacy} yields the desired bijection. 

We now turn to the right-hand side. Let $\sM$ be an \'etale local shtuka over $\Spec\ka_r$. Completed unramified descent implies that the isomorphism class of $\sM$ is determined by the isomorphism class of $\br\sM\deq\sM_{\ov\ka}$ along with a descent isomorphism $f:\sg^{-r,*}\br\sM\rar^\sim\br\sM$ corresponding to the action of the geometric $q^r$-Frobenius map $\sg^{-r}$. As $\br\sM$ must be isomorphic to $(\ov\ka\llb\vpi^n,\sg^{\oplus n})$ by \ref{ss:localshtukaalgebraicclosure}, any such isomorphism identifies $\Aut(\br\sM)$ with $\GL_n(\cO)$. Write $\br\sF:\sg^*\br\sM\rar\br\sM$ for the Frobenius of $\br\sM$, which is an isomorphism because $\br\sM$ is \'etale, and write $\ga$ for the automorphism
\begin{align*}
\ga\deq\br\sF\circ\dotsb\circ\sg^{r-1,*}\br\sF\circ\sg^{r,*}f
\end{align*}
of $\br\sM$, viewed as an element of $\GL_n(\cO)$. Since any other choice of isomorphism $\br\sM\rar^\sim (\ov\ka\llb\vpi^n,\sg^{\oplus n})$ preserves the $\GL_n(\cO)$-conjugacy class of $\ga$, this yields a bijection $\ga\leftrightarrow\sM$ between $\GL_n(\cO)$-conjugacy classes in $\GL_n(\cO)$ and isomorphism classes of \'etale local shtukas over $\Spec\ka_r$ of rank $n$.

Recall that $\br\sF$ corresponds to $\de\circ\sg^{\oplus n}$ under \ref{ss:effectiveminusculelocalshtukaconjugacy}. Since $f$ corresponds to $(\sg^{\oplus n})^{-r}$, we see that $\ga$ corresponds to
\begin{align*}
\underbrace{(\de\circ\sg^{\oplus n})\circ\dotsb\circ(\de\circ\sg^{\oplus n})}_{r\,\text{times}}\circ (\sg^{\oplus n})^{-r} = \left(\de\sg(\de)\dotsm\sg^{r-1}(\de)\right)\circ(\sg^{\oplus n})^r\circ (\sg^{\oplus n})^{-r} = \N\de,
\end{align*}
as desired.
\end{proof}

\subsection{}\label{ss:etalelocalshtukatransfer}
We now construct certain functions with matching twisted orbital integrals, that is, certain transfers of functions. Let $h$ be a $\GL_n(\cO)$-conjugation-invariant function in $C^\infty_c(\GL_n(\cO))$, and form the function $\phi:\GL_n(\cO_r)\rar\bC$ by sending $\de\mapsto h(\N\de)$. One can show that $\phi$ is locally constant \cite[Corollary 2.3]{Sch13b}\footnote{The proof given here is stated for mixed characteristic, but it works verbatim for any nonarchimedean field.}.
\begin{lem*}
The functions $h$ and $\phi$ have matching twisted orbital integrals.
\end{lem*}
\begin{proof}
We need to prove that $\O_{\N\de}(h)=\TO_{\de,\sg}(\phi)$ for all regular $\de$ in $\GL_n(\cO_r)$, where $\O_{\N\de}$ and $\TO_{\de,\sg}$ denote the orbital integral on $\N\de$ and twisted orbital integral on $\de$, respectively, with respect to compatible Haar measures. We can prove this as follows: write $\sM$ for the \'etale local shtuka corresponding to $\de$ under \ref{ss:effectiveminusculelocalshtukaconjugacy}, and write $X$ for the set of isomorphism classes of quasi-isogenies $\be:\sM\dasharrow\sM'$ between \'etale local shtukas. Note that the group $\Ga$ of self-quasi-isogenies $\sM\dasharrow\sM$ has a left action on $X$ given by sending $\be$ to $\be\circ g^{-1}$ for any $g$ in $\Ga$.

One can equate both $\O_{\N\de}(h)$ as well as $\TO_{\de,\sg}(\phi)$ to the sum
\begin{align*}
\sum_{(\sM',\be)}\wt{h}(\sM',\be),
\end{align*}
where $(\sM',\be)$ ranges over elements in $\Ga\bs X$, and $\wt{h}$ is the function sending $(\sM',\be)$ to $h(\de')$, where $\de'$ corresponds to $\sM'$ under \ref{ss:effectiveminusculelocalshtukaconjugacy}. In a line of reasoning similar to the proof of Lemma \ref{ss:normbijection}, we relate $\TO_{\de,\sg}(\phi)$ to the above sum by using \ref{ss:effectiveminusculelocalshtukaconjugacy}, and we relate $\O_{\N\de}(h)$ to the above sum by using completed unramified descent, c.f. \cite[Proposition 4.3]{Sch13}
\end{proof}

\subsection{}\label{ss:deformationtestfunction}
Next, we recast our results from \S\ref{s:deformationspaces} in terms of our test functions. For this, recall that for any $\de$ in
\begin{align*}
\GL_n(\cO_r)\diag(\vpi,1,\dotsc,1)\GL_n(\cO_r),
\end{align*}
we constructed an associated effective minuscule local shtuka $H_\de$ over $\Spec\ka_r$ of dimension $1$. Its connected--\'etale decomposition \cite[Proposition 2.9]{HS16} is $H_{\de^\circ}\oplus H_{\de^\et}$, where $H_{\de^\circ}$ has rank $k$, and $B_{n,k}$ is the set of $\de$ such that $H_\de$ decomposes in this way. By considering $\de$ as an element of $\GL_n(\br\cO)\diag(\vpi,1,\dotsc,1)\GL_n(\br\cO)$, we obtain the pullback $\br{H}_\de$ of $H_\de$ to $\Spec\ov\ka$, and the space $\br\fX_{\de,m}$ parametrizes deformations of $\br{H}_\de$ along with a Drinfeld level-$m$ structure. We defined a virtual representation $[R\psi_\de]$ in terms of the cohomology of $\br\fX_{\de,m}$ as $m$ varies, and we defined $\phi_{\tau,h}(\de)$ as a trace on $[R\psi_\de]$.
\begin{prop*}
We have an equality
\begin{align*}
\phi_{\tau,h}(\de) = \tr\left((\tau\times h)|\Ind_{P_k(\cO)}^{\GL_n(\cO)}\left([R\psi_{\de^\circ}]\otimes C^\infty_c(\GL_{n-k}(\cO))\right)\right),
\end{align*}
where $P_k$ is the standard parabolic subgroup of $\GL_n$ with block sizes $(k,n-k)$, $\GL_{n-k}(\cO)$ acts on $C^\infty_c(\GL_{n-k}(\cO))$ by left inverse multiplication, and $W_{F_r}$ acts on $C^\infty_c(\GL_{n-k}(\cO))$ via the unramified action sending $\tau$ to the action of right multiplication by $\N\de^\et$.
\end{prop*}
\begin{proof}
Let $f_{\sg^{-r}}:\sg^{-r,*}\br{H}_\de\rar^\sim\br{H}_\de$ and $\vp_{\sg^{-r}}:\br\fX_{\de,m}\rar\sg^{-r,*}\br\fX_{\de,m}$ be as in \ref{ss:deformationunramifiedextension}. Using the Dieudonn\'e equivalence \cite[Theorem 8.3]{HS16} to pass to $\vpi$-divisible local Anderson modules, the proof of Lemma \ref{ss:normbijection} shows that the triple $(H',\al',\io'\circ f_{\sg^{-r}})$ is isomorphic to the triple $(H',\N\de^\et\circ\al',\io')$. Therefore $\vp_{\sg^{-r}}$ sends $\br\fX_{\de,\al}$ to $\sg^{-r,^*}\br\fX_{\de,\N\de^\et\circ\al}$. Taking cohomology of both sides in Proposition \ref{ss:deformationalgebraicclosure} and using \cite[Lemma I.5.6]{HT01} give us
\begin{align*}
R\psi^i_{\de,m} = \Ind_{P_k(\cO/\vpi^m)}^{\GL_n(\cO/\vpi^m)}\left(R\psi^i_{\de^\circ,m}\otimes\ov\bQ_\ell[\GL_{n-k}(\cO/\vpi^m)]\right),
\end{align*}
where $W_{F_r}$ acts on $\ov\bQ_\ell[\GL_{n-k}(\cO)]$ via the unramified action sending $\sg^{-r}I_F$ and hence $\tau$ to the action of right multiplication by $N\de^\et$. Taking the direct limit over $m$ and forming the alternating sum over $i$ gives an equality
\begin{align*}
[R\psi_\de] = \Ind_{P_k(\cO)}^{\GL_n(\cO)}\left([R\psi_{\de^\circ}]\otimes C^\infty_c(\GL_{n-k}(\cO))\right)
\end{align*}
of virtual representations, where we have identified $\ov\bQ_\ell$ with $\bC$. Taking the trace of $\tau\times h$ yields the desired result.
\end{proof}

\subsection{}\label{ss:twistedcharacters}
Before proceeding, we gather a few facts about twisted characters and automorphic base change. Recall that the norm map $\de\mapsto\de\sg(\de)\dotsm\sg^{r-1}(\de)$ induces an injection
\begin{align*}
\N:\{\GL_n(F_r)\mbox{-}\sg\mbox{-conjugacy classes in }\GL_n(F_r)\}\inj{}\{\GL_n(F_r)\mbox{-conjugacy classes in }\GL_n(F_r)\}
\end{align*}
\cite[Lemma 4.2]{Lan80}. For any finite length smooth representation $\pi$ of $\GL_n(F)$, its character distribution $\tr(-|\pi)$ is represented by a locally integrable conjugation-invariant function $\Te_\pi:\GL_n(F)\rar\bC$ \cite[Theorem (7.1)]{Lem04}, which we call the \emph{character} of $\pi$. Naturally, taking characters is additive. The function $\Te_{\pi/F_r}^{\sg}:\GL_n(F_r)\rar\bC$ that sends $\de\mapsto\Te_\pi(\N\de)$ is locally integrable \cite[(II.2.8)]{HL11}, and we denote the associated distribution by $\tr\left((-,\sg)|\pi/F_r\right):C^\infty_c(\GL_n(F_r))\rar\bC$.

The twisted Weyl integration formula \cite[(II.2.10, formula (1)]{HL11} implies that, if $f$ and $\phi$ are functions in $C^\infty_c(\GL(F))$ and $C^\infty_c(\GL_n(F_r))$, respectively, such that $f$ and $\phi$ have matching twisted orbital integrals, then we have
\begin{align*}
\tr\left((\phi,\sg)|\pi/F_r\right) = \tr(f|\pi).
\end{align*}
As suggested by our notation, if $\pi$ is generic, then the twisted character of its base change lift $\pi_{F_r}$ to $\GL_n(F_r)$ equals $\Te_{\pi/F_r}^{\sg}$, and we have $\tr((\phi,\sg)|\pi_{F_r})=\tr\left((\phi,\sg)|\pi/F_r\right)$ \cite[(II.2.8)]{HL11}.

\subsection{}\label{ss:unramifiedtwistsofrho}
We use \ref{ss:twistedcharacters} to obtain the following compatibility result for $\rho$ and unramified twists.
\begin{prop*}
Let $\pi$ be an irreducible smooth representation of $\GL_n(F)$, and suppose that Theorem A.(i) holds for $\pi$. Then for any complex number $s$, Theorem A.(i) holds for $\pi[s]$. That is, there exists a unique $n$-dimensional semisimple continuous representation $\rho(\pi[s])$ of $W_F$ satisfying the following property:
\begin{align*}
\mbox{for all }\tau\mbox{ in }W_F\mbox{ with }v(\tau)>0\mbox{ and }h\mbox{ in }C^\infty_c(\GL_n(\cO))\mbox{, we have }\tr\left(f_{\tau,h}|\pi[s]\right) = \tr\left(\tau|\rho(\pi[s])\right)\tr\left(h|\pi[s]\right).
\end{align*}
This $\rho(\pi[s])$ is given by $\rho(\pi[s])=\rho(\pi)(s)$.
\end{prop*}
\begin{proof}
We immediately have $\tr\left(f_{\tau,h}|\pi[s]\right) = \tr\left((\phi_{\tau,h},\sg)|\pi[s]/F_r\right)$ by \ref{ss:twistedcharacters}. Unpacking definitions shows that the right hand side equals
\begin{align*}
\int_{\GL_n(\cO_r)\diag(\vpi,1,\dotsc,1)\GL_n(\cO_r)}\!\dif\de\,\tr\left(\tau\times h|[R\psi_\de]\right)\Te_{\pi[s]}(\N\de).
\end{align*}
Since $\pi[s]$ is the twist of $\pi$ by the character $\abs\det^s$, we see that $\Te_{\pi[s]}(\N\de) = \Te_\pi(\N\de)\abs{\det\N\de}^s$. As $\de$ lies in 
\begin{align*}
  \GL_n(\cO_r)\diag(\vpi,1,\dotsc,1)\GL_n(\cO_r),
\end{align*}
its determinant has valuation $1$, which makes $\det\N\de$ have valuation $r$. Therefore the above integral becomes
\begin{align*}
\frac1{q^{rs}}\int_{\GL_n(\cO_r)\diag(\vpi,1,\dotsc,1)\GL_n(\cO_r)}\!\dif\de\,\tr\left(\tau\times h|[R\psi_\de]\right)\Te_\pi(\N\de) = \frac1{q^{rs}}\tr\left(\tau|\rho(\pi)\right)\tr\left(h|\pi\right)
\end{align*}
by applying Theorem A.(i) for $\pi$. Because $\tau$ acts through $(s)$ via $1/q^{rs}$, the desired result follows.
\end{proof}

\subsection{}\label{ss:peterweyl}
In order to state the trace identity involving $f_{\tau,h}$ and Jacquet restriction, we first need to recall the Peter--Weyl theorem for profinite groups $G$. One version of it says that we have a canonical isomorphism of smooth representations of $G\times G$
\begin{align*}
C^\infty_c(G)&\rar^\sim\bigoplus_\la\End_\bC(\la)=\bigoplus_\la\la^\vee\otimes\la\\
f&\longmapsto\bigoplus_\la\la(f),
\end{align*}
where $\la$ ranges over isomorphism classes of irreducible smooth representations of $G$, and $G\times G$ acts on $C^\infty_c(G)$ via letting $(g_1,g_2)$ send $f$ to the function $x\mapsto f(g_2^{-1}xg_1)$. This implies, along with Schur orthogonality, that the trace map induces a $\bC$-linear bijection
\begin{align*}
\tr:\{\bC\mbox{-virtual admissible representations of }G\}\rar^\sim\{\mbox{conjugation-invariant distributions on }G\}.
\end{align*}
We shall use this bijection to view conjugation-invariant distributions as $\bC$-virtual admissible representations.

\subsection{}\label{ss:jacquetmoduletraceidentity}
We may now introduce an identity that relates the trace of $f_{\tau,h}$ on $\pi$ with traces on certain Jacquet restrictions of $\pi$. The proof shall critically use Proposition \ref{ss:deformationtestfunction}, which itself is derived from the geometry of $\br\fX_{\de,m}$ as described in \S\ref{s:deformationspaces}. 

For all integers $1\leq k \leq n$, write $N_k$ for the unipotent radical of $P_k$. Write $\pi_{N_k}$ for the unnormalized Jacquet restriction of $\pi$, and suppose that we have an equality of virtual representations
\begin{align*}
\pi_{N_k} = \sum_{i=1}^{t_k}\pi_{1,i}^k\otimes\pi_{2,i}^k,
\end{align*}
where $t_k$ is a non-negative integer, and the $\pi_{1,i}^k$ and $\pi_{2,i}^k$ are finite length smooth representations of $\GL_k(F)$ and $\GL_{n-k}(F)$, respectively. Note that we do not require the $\pi_{1,i}^k$ and $\pi_{2,i}^k$ to be irreducible. For any subset $X$, we denote the indicator function on that subset using $\bf1_X$.
\begin{prop*}
We have an equality of traces
\begin{align*}
\tr(f_{\tau,h}|\pi) = \sum_{k=1}^nq^{(n-k)r}\sum_{i=1}^{t_k}\tr\left(h|\Ind_{P_k(\cO)}^{\GL_n(\cO)}\left(\tr((\bf1_{B_k}\cdot\phi_{\tau,-},\sg)|\pi_{1,i}^k/F_r)\otimes\pi_{2,i}^k\right)\right),
\end{align*}
where we view the conjugation-invariant distribution $\tr((\bf1_{B_k}\cdot\phi_{\tau,-},\sg)|\pi_{1,i}^k/F_r)$ as a $\bC$-virtual admissible representation of $\GL_k(\cO)$ via \ref{ss:peterweyl}.
\end{prop*}
In the proof, we take our Haar measures compatibly whenever possible.
\begin{proof}
We immediately have $\tr(f_{\tau,h}|\pi)=\tr\left((\phi_{\tau,h},\sg)|\pi/F_r\right)$. Because $\phi_{\tau,h}$ is supported on
\begin{align*}
\GL_n(\cO_r)\diag(\vpi,1,\dotsc,1)\GL_n(\cO_r) = \coprod_{k=1}^nB_{n,k},
\end{align*}
it suffices to show that
\begin{align*}
\tr\left((\bf1_{B_{n,k}}\cdot\phi_{\tau,h},\sg)|\pi/F_r\right) = q^{(n-k)r}\sum_{i=1}^{t_k}\tr\left(h|\Ind_{P_k(\cO)}^{\GL_n(\cO)}\left(\tr((\bf1_{B_k}\cdot\phi_{\tau,-},\sg)|\pi_{1,i}^k/F_r)\otimes\pi_{2,i}^k\right)\right)
\end{align*}
for all $k$. We begin with left-hand side. Applying the twisted Weyl integration formula \cite[(II.2.10, formula (1)]{HL11} to the function $\bf1_{B_{n,k}}\cdot\phi_{\tau,h}$ in $C^\infty_c(\GL_n(F_r))$ gives
\begin{align*}
\tr\left((\bf1_{B_{n,k}}\cdot\phi_{\tau,h},\sg)|\pi/F_r\right) = \sum_T\frac1{\#W(T,\GL_n)}\int_{T(F_r)^{\sg-1}\bs T(F_r)}\!\dif\de\,\abs{D_{\GL_n}(\N\de)}\TO_{\de,\sg}(\bf1_{B_{n,k}}\cdot\phi_{\tau,h})\Te_{\pi/F_r}^{\sg}(\de),
\end{align*}
where $T$ runs over conjugacy classes of maximal tori of $\GL_n$ over $F$, and $D_{\GL_n}:\GL_n(F)\rar F$ is the regular function defined by the equation
\begin{align*}
\det(t-\ad\ga+1|\fg\fl_n(F))\equiv D_{\GL_n}(\ga)t^n\pmod{t^{n+1}}.
\end{align*}
The decomposition $\de=\de^\circ\oplus\de^\et$, along with the description of $\de^\circ$ and $\de^\et$ in \ref{ss:localshtukaunramifiedextension}, shows that the only terms remaining in the above sum are
\begin{align*}
&\sum_{T_k}\frac1{\#W(T_k,\GL_k)}\sum_{T_{n-k}}\frac1{\#W(T_{n-k},\GL_{n-k})}\\
&\cdot\int_{(T_k(F_r)^{\sg-1}\bs T_k(F_r)_1)\times (T_{n-k}(F_r)^{\sg-1}\bs T_{n-k}(\cO_r))}\!\dif\de\,\abs{D_{\GL_n}(\N\de)}\TO_{\de,\sg}(\bf1_{B_{n,k}}\cdot\phi_{\tau,h})\Te_\pi(\N\de),
\end{align*}
where $T_k$ runs over conjugacy classes of anisotropic mod center maximal tori of $\GL_k$, $T_{n-k}$ runs over conjugacy classes of maximal tori of $\GL_{n-k}$, and we have $T=T_k\times T_{n-k}$. 

With visible product decompositions beginning to form, let us make the change of variables $\de=(\de_k,\de_{n-k})$. Our goal is to break up the integrand in terms of $\GL_k$ and $\GL_{n-k}$. By using block matrices to expand $\abs{D_{\GL_n}(\N\de)}$, we see that
\begin{align*}
\abs{D_{\GL_n}(\N\de)} = q^{(n-k)r}\abs{D_{\GL_k}(\N\de_k)}\cdot\abs{D_{\GL_{n-k}}(\N\de_{n-k})}.
\end{align*}
To obtain a similar decomposition of $\TO_{\de,\sg}(\bf1_{B_{b,k}}\cdot\phi_{\tau,h})$, we shall use the following result.
\begin{lem}
We have an equality of twisted orbital integrals
\begin{align*}
\TO_{\de,\sg}(\bf1_{B_{n,k}}\cdot\phi_{\tau,h}) = \TO_{(\de_k,\de_{n-k}),\sg}((\bf1_{B_{n,k}}\cdot\phi_{\tau,h})|_{M_k(F_r)}),
\end{align*}
where $M_k$ is the standard Levi subgroup of $P_k$, $(\bf1_{B_{n,k}}\cdot\phi_{\tau,h})|_{M_k(F_r)}$ is the restriction of $\bf1_{B_{n,k}}\cdot\phi_{\tau,h}$ to $M_k(F_r)$, and the right-hand side is the twisted orbital integral of $(\bf1_{B_{n,k}}\cdot\phi_{\tau,h})|_{M_k(F_r)}$ on $(\de_k,\de_{n-k})$.
\end{lem}
\begin{proof}
Using Nakayama's lemma over $\cO$ to reduce to the residue field, we can explicitly calculate the normalized $\GL_n(\cO_r)$-$\sg$-invariant constant term $(\bf1_{B_{n,k}}\cdot\phi_{\tau,h})^{P_k,\sg,\GL_n(\cO_r)}$ of $\bf1_{B_{n,k}}\cdot\phi_{\tau,h}$ along $P_k$ to be
\begin{align*}
(\bf1_{B_{n,k}}\cdot\phi_{\tau,h})^{P_k,\sg,\GL_n(\cO_r)} = q^{\frac12(n-k)r}(\bf1_{B_{n,k}}\cdot\phi_{\tau,h})|_{M_k(F_r)}
\end{align*}
\cite[Lemma 6.6]{Sch13b}. The desired equality follows by applying a standard result \cite[(4.4.9)]{Lau96} relating the twisted orbital integrals of $\bf1_{B_{n,k}}\cdot\phi_{\tau,h}$ and $(\bf1_{B_{n,k}}\cdot\phi_{\tau,h})^{P_k,\sg,\GL_n(\cO_r)}$. 
\end{proof}
Return to the proof of Proposition \ref{ss:jacquetmoduletraceidentity}. Casselman's theorem \cite[5.2]{Cas77} shows that
\begin{align*}
\Te_\pi(\N\de) = \Te_{\pi_{N_k}}(\N\de) = \sum_{i=1}^{t_k}\Te_{\pi_{1,i}^k}(\N\de_k)\Te_{\pi_{2,i}^k}(\N\de_{n-k}),
\end{align*}
and by putting all of this together, we can rewrite our sum of integrals as
\begin{align*}
&q^{(n-k)r}\sum_{i=1}^{t_k}\sum_{T_k}\frac1{\#W(T_k,\GL_k)}\int_{T_k(F_r)^{\sg-1}\bs T_k(F_r)_1}\!\dif\de_k\,\abs{D_{\GL_k}(\N\de_k)}\Te_{\pi_{1,i}^k}(\N\de_k)\sum_{T_{n-k}}\frac1{\#W(T_{n-k},\GL_{n-k})}\\
&\cdot\int_{T_{n-k}(F_r)^{\sg-1}\bs T_{n-k}(\cO_r)}\!\dif\de_{n-k}\,\abs{D_{\GL_{n-k}}(\N\de_{n-k})}\TO_{(\de_k,\de_{n-k}),\sg}(\phi_{\tau,h}|_{M_k(F_r)})\Te_{\pi_{2,i}^k}(\N\de_{n-k}),
\end{align*}
where we drop $\bf1_{B_{n,k}}$ since our $\de$ all lie in $B_{n,k}$.

Before proceeding further, we apply Proposition \ref{ss:deformationtestfunction} as follows. Write $\wt{f}_{\tau,h}(\de_k,-):\GL_{n-k}(\cO)\rar\bC$ for the function that sends
\begin{align*}
\ga\mapsto\tr\left((\ga,h)|\Ind_{P_k(\cO)}^{\GL_n(\cO)}\left(\phi_{\tau,-}(\de_k)\otimes C^\infty_c(\GL_{n-k}(\cO))\right)\right),
\end{align*}
where we view the conjugation-invariant distribution $\phi_{\tau,-}(\de_k)$ as a $\bC$-virtual admissible representation of $\GL_k(\cO)$ via \ref{ss:peterweyl}, $\ga$ acts trivially on $\phi_{\tau,-}(\de_k)$, and $\ga$ acts on $C^\infty_c(\GL_{n-k}(\cO))$ via right multiplication. Note that, with Proposition \ref{ss:deformationtestfunction} in mind, $\wt{f}_{\tau,h}$ heavily resembles $\phi_{\tau,h}$.
\begin{lem}\label{lem:ftilde}
The function $\wt{f}_{\tau,h}(\de_k,-)$ is in $C^\infty_c(\GL_{n-k}(\cO))$, and it satisfies the equalities
\begin{align*}
\tr((\phi_{\tau,h}(\de_k,-),\sg)|\pi_{2,i}^k/F_r) = \tr(\wt{f}_{\tau,h}(\de_k,-)|\pi_{2,i}^k) = \tr\left(h|\Ind_{P_k(\cO)}^{\GL_n(\cO)}(\phi_{\tau,-}(\de_k)\otimes\pi_{2,i}^k)\right).
\end{align*}
\begin{proof}
Proposition \ref{ss:deformationtestfunction} indicates that $\phi_{\tau,h}(\de_k,\de)=\wt{f}_{\tau,h}(\de_k,\N\de)$. Thus \ref{ss:etalelocalshtukatransfer} reduces the first claim to the statement that $\phi_{\tau,h}$ is in $C^\infty_c(\GL_n(\cO_r))$, which holds by \ref{ss:sigmaconjugationisopen}. For the second claim, note that Lemma \ref{ss:etalelocalshtukatransfer} implies $\wt{f}_{\tau,h}(\de_k,-)$ and $\phi_{\tau,h}(\de_k,-)$ have matching twisted orbital integrals. Therefore \ref{ss:twistedcharacters} yields $\tr((\phi_{\tau,h}(\de_k,-),\sg)|\pi_{2,i}^k/F_r) = \tr(\wt{f}_{\tau,h}(\de_k,-)|\pi_{2,i}^k)$.

From here, applying the Peter--Weyl theorem and Schur orthogonality to the definition of $\wt{f}_{\tau,h}(\de_k,-)$ show that $\tr(\wt{f}_{\tau,h}(\de_k,-)|\pi_{2,i}^k)=\tr\left(h|\Ind_{P_k(\cO)}^{\GL_n(\cO)}(\phi_{\tau,-}(\de_k)\otimes\pi_{2,i}^k)\right)$.
\end{proof}
\end{lem}
Return to the proof of Proposition \ref{ss:jacquetmoduletraceidentity}. Applying the twisted Weyl integration formula \cite[(II.2.10, formula (1)]{HL11} to the function $\TO_{\de_k,\sg}(\phi_{\tau,h}(-,-))$ in $C^\infty_c(\GL_{n-k}(F_r))$ allows us to collapse the sum over $T_{n-k}$ and simplify our expression into
\begin{align*}
q^{(n-k)r}\sum_{i=1}^{t_k}\sum_{T_k}\frac1{\#W(T_k,\GL_k)}\int_{T_k(F_r)^{\sg-1}\bs T_k(F_r)_1}\!\dif\de_k\,\abs{D_{\GL_k}(\N\de_k)}\Te_{\pi_{1,i}^k}(\N\de_k)\TO_{\de_k,\sg}(\tr(\wt{f}_{\tau,h}(-,-)|\pi_{2,i}^k)),
\end{align*}
where we first used $\tr((\TO_{\de_k,\sg}(\phi_{\tau,h}(-,-)),\sg)|\pi_{2,i}^k/F_r)=\TO_{\de_k,\sg}(\tr((\phi_{\tau,h}(-,-),\sg)|\pi_{2,i}^k/F_r))$ and then the first equality in Lemma \ref{lem:ftilde}. Applying the twisted Weyl integration formula to the function $\bf1_{B_k}\cdot\tr(\wt{f}_{\tau,h}(-,-)|\pi_{2,i}^k)$ in $C^\infty_c(\GL_k(F_r))$ lets us collapse the sum over $T_k$, and we get
\begin{align*}
q^{(n-k)r}\sum_{i=1}^{t_k}\tr\left((\bf1_{B_k}\cdot\tr(\wt{f}_{\tau,h}(-,-)|\pi_{2,i}^k),\sg)|\pi_{1,i}^k/F_r)\right).
\end{align*}
Now the second equality in Lemma \ref{lem:ftilde} turns this into
\begin{align*}
q^{(n-k)r}\sum_{i=1}^{t_k}\tr\left(h|\Ind_{P_k(\cO)}^{\GL_n(\cO)}\left(\tr((\bf1_{B_k}\cdot\phi_{\tau,-},\sg)|\pi_{1,i}^k/F_r)\otimes\pi_{2,i}^k\right)\right),
\end{align*}
which completes the proof of Proposition \ref{ss:jacquetmoduletraceidentity}.
\end{proof}

\subsection{}\label{ss:properlyinducedsupport}
Before moving on to the proof of condition (b) in Lemma \ref{ss:firstinductivelemma}, we record a useful consequence of van Djik's formula. Let $P$ be a proper standard parabolic subgroup of $\GL_n$, write $M$ for the standard Levi subgroup of $P$, let $\pi'$ be an irreducible smooth representation of $M(F)$, and write $\pi$ for the induced representation
\begin{align*}
\pi\deq\nInd_{M(F)}^{\GL_n(F)}\pi'.
\end{align*}
As $\pi'$ has finite length, so does $\pi$.
\begin{lem*}
Let $\phi$ be a function in $C^\infty_c(\GL_n(F_r))$ that is supported in elements of $\GL_n(F_r)$ whose norm is elliptic. Then we have $\tr((\phi,\sg)|\pi/F_r)=0$.
\end{lem*}
\begin{proof}
Let $f$ be a transfer of $\phi$, i.e. a function in $C^\infty_c(\GL_n(F))$ such that $f$ and $\phi$ have matching twisted orbital integrals. Then $f$ is supported in elliptic elements of $\GL_n(F)$, and no such element is contained in $P(F)$ because $P$ is a proper parabolic subgroup. Therefore the normalized $\GL_n(\cO)$-invariant constant term $f^{P,\GL_n(\cO)}$ of $f$ along $P$ is identically zero on $M(F)$. Our remark from \ref{ss:twistedcharacters} and van Djik's formula \cite[Theorem 5.9]{Lem16} together yield
\begin{align*}
\tr((\phi,\sg)|\pi/F_r) = \tr(f|\pi) =\tr(f^{P,\GL_n(\cO)}|\pi')=0,
\end{align*}
as desired.
\end{proof}
We will apply the above lemma to $\bf1_{B_k}\cdot\phi_{\tau,h}$ in the proof of Proposition \ref{ss:conditionb}.

\subsection{}\label{ss:conditionb}
Finally, we use Proposition \ref{ss:jacquetmoduletraceidentity} to prove condition (b), assuming condition (a).
\begin{prop*}
Assume that the inductive condition (a) in Lemma \ref{ss:firstinductivelemma} holds, that is, Theorem A is true for $n'<n$. Let $\pi$ be a smooth representation of $\GL_n(F)$, and suppose it is of the form
\begin{align*}
\pi=\nInd_{P(F)}^{\GL_n(F)}(\pi_1\otimes\dotsb\otimes\pi_t),
\end{align*}
where $t\geq2$, the $\pi_i$ are some irreducible smooth representations of $\GL_{n_i}(F)$ with $n=n_1+\dotsb+n_t$, and $P$ is the standard parabolic subgroup of $\GL_n$ with block sizes $(n_1,\dotsc,n_t)$. Then we have an equality of traces
\begin{align*}
  \tr(f_{\tau,h}|\pi) = \tr\left(\tau|\rho(\pi_1)(\textstyle\frac{n-n_1}2)\oplus\dotsb\oplus\rho(\pi_t)(\frac{n-n_t}2)\right)\tr(h|\pi)
\end{align*}
for all $\tau$ in $W_F$ with $v(\tau)>0$ and $h$ in $C^\infty_c(\GL_n(\cO))$.
\end{prop*}

In the proof of Proposition \ref{ss:conditionb}, we shall frequently pass between normalized and un-normalized induction. For this, one immediately checks that
\begin{align*}
\nInd_{Q(F)}^{\GL_m(F)}(\la_1\otimes\la_2) = \Ind_{Q(F)}^{\GL_m(F)}(\la_1[\textstyle\frac{m_2}2]\otimes\la_2[-\frac{m_1}2]),
\end{align*}
where the $\la_i$ are irreducible smooth representations of $\GL_{m_i}(F)$ with $m=m_1+m_2$, and $Q$ is the standard parabolic subgroup of $\GL_m$ with block sizes $(m_1,m_2)$.

We also recall that the two representations
\begin{align*}
\nInd_{Q(F)}^{\GL_m(F)}(\la_1\otimes\la_2)\mbox{ and }\nInd_{Q'(F)}^{\GL_m(F)}(\la_2\otimes\la_1),
\end{align*}
where $Q'$ denotes the standard parabolic subgroup of $\GL_n$ with block sizes $(m_2,m_1)$, have the same Jorder--H\"older series \cite[2.9]{BZ77}. This enables us to identify them as virtual representations of $\GL_n(F)$.
\begin{proof}[Proof of Proposition \ref{ss:conditionb}]
We start with some immediate reductions. By using the transitivity of parabolic induction \cite[1.7]{Zel80}, it suffices to prove the proposition for $t=2$. Furthermore, because swapping $\pi_1$ and $\pi_2$ leaves the virtual representation $\pi$ unchanged, we may assume that $n_1\leq n_2$. Finally, by taking the trace of $h$ afterwards, Proposition \ref{ss:jacquetmoduletraceidentity} shows that it suffices to find, for all integers $1\leq k\leq n$, a decomposition of virtual representations
\begin{align*}
\pi_{N_k} = \sum_{i=1}^{t_k}\pi_{1,i}^k\otimes\pi_{2,i}^k
\end{align*}
of $\GL_k(F)\times\GL_{n-k}(F)$, where the $\pi_{1,i}^k$ and $\pi_{2,i}^k$ have finite length, such that we have an equality of $\bC$-virtual representations
\begin{align}\label{eq:goal}
  \sum_{k=1}^nq^{(n-k)r}\sum_{i=1}^{t_k}\Ind_{P_k(\cO)}^{\GL_n(\cO)}\left(\tr((\bf1_{B_k}\cdot\phi_{\tau,-},\sg)|\pi_{1,i}^k/F_r)\otimes\pi_{2,i}^k\right) = \tr\left(\tau|\rho(\pi_1)(\textstyle\frac{n-n_1}2)\oplus\rho(\pi_2)(\frac{n-n_2}2)\right)\pi\tag{$\star$}
\end{align}
of $\GL_n(\cO)$. We will choose these $\pi_{1,i}^k$ and $\pi_{2,i}^k$ in a way that relates them to $\pi_1$ and $\pi_2$. To do so, we shall use the following \emph{induction-restriction formula} of Bernstein--Zelevinsky.

\subsection{}\label{ss:bernsteinzelevinskyinductionrestriction}
Write $T$ for the standard maximal torus of $\GL_n$, and write $B$ for the standard Borel subgroup of $\GL_n$. Let $P$ and $Q$ be standard parabolic subgroups of $\GL_n$, and denote their standard Levi subgroups using $M$ and $N$, respectively. We identify the Weyl group $W(T,\GL_n)$ with the symmetric group $\fS_n$ acting by permuting entries. Write $W_{M,N}$ for the subset
\begin{align*}
W_{M,N}\deq\{w\in W(T,\GL_n)|w(M\cap B)\subseteq B\mbox{ and }w^{-1}(N\cap B)\subseteq B\}
\end{align*}
of $W(T,\GL_n)$.
\begin{lem*}[{\cite[2.12]{BZ77}}]
Let $\pi$ be an irreducible smooth representation of $\GL_n(F)$. Then we have an equality of virtual representations
\begin{align*}
\nRes_{Q(F)}^{\GL_n(F)}(\nInd_{P(F)}^{\GL_n(F)}\pi) = \sum_w\nInd_{Q'(F)}^{N(F)}(w(\nRes_{P'(F)}^{M(F)}\pi))
\end{align*}
of $N(L)$, where $\nRes$ denotes normalized Jacquet restriction, $w$ runs over all elements of $W_{M,N}$, $P'$ denotes the standard parabolic subgroup of $M$ whose standard Levi subgroup is $M\cap w^{-1}(N)$, and $Q'$ denotes the standard parabolic subgroup of $N$ whose standard Levi subgroup is $w(M)\cap N$.
\end{lem*}
Return to the proof of Proposition \ref{ss:conditionb}. Our goal will be to apply $\tr((\bf1_{B_k}\cdot\phi_{\tau,-},\sg)|-/F_r)$ to the $\pi_{1,i}^k$. At this point, Lemma \ref{ss:properlyinducedsupport} indicates that the terms corresponding to properly parabolically induced $\pi_{1,i}^k$ will vanish. Therefore, when using the induction-restriction formula for $M=\GL_{n_1}\times\GL_{n_2}$ and $N=\GL_k\times\GL_{n-k}$, we will only be interested in the terms for which $w(M)\cap N$ contains $\GL_k$.

With these choices of $M$ and $N$, now $W_{M,N}$ corresponds precisely to the set of permutations $w$ in $\fS_n$ for which
\begin{itemize}
\item $w$ is order-preserving on $\{1,\dotsc,n_1\}$ and $\{n_1+1,\dotsc,n\}$,
\item $w^{-1}$ is order-preserving on $\{1,\dotsc,k\}$ and $\{k+1,\dotsc,n\}$.
\end{itemize}
Furthermore, the $w(M)\cap N\supseteq\GL_k$ condition is equivalent to asking that
\begin{align*}
w(\{1,\dotsc,n_1\})\supseteq\{1,\dotsc,k\}\mbox{ or }w(\{n_1+1,\dotsc,n\})\supseteq\{1,\dotsc,k\}.
\end{align*}
If $k\leq n_1$, then the only such $w$ in $W_{M,N}$ are the identity permutation and the permutation $\te$ that sends $\{1,\dotsc,n_1\}$ to $\{n_2+1,\dotsc,n\}$ and $\{n_1+1,\dotsc,n\}$ to $\{1,\dotsc,n_2\}$ while preserving their internal orders. If instead $n_1< k\leq n_2$, then the identity no longer satisfies $w(M)\cap N\supseteq\GL_k$, although $\te$ continues to do so. Finally, for $n_2<k$, no element of $W_{M,N}$ has $w(M)\cap N\supseteq\GL_k$. Altogether, Lemma \ref{ss:bernsteinzelevinskyinductionrestriction} yields an equality 
\begin{align}\label{eq:jacquetrestriction}
\pi_{N_k} = &\Ind_{P_{n_1}(F)\cap(\GL_k(F)\times\GL_{n-k}(F))}^{\GL_k(F)\times\GL_{n-k}(F)}(\pi'_{N_{k_1}}\otimes\pi'')\nonumber\\
&+\Ind_{P_{n_2}(F)\cap(\GL_k(F)\times\GL_{n-k}(F))}^{\GL_k(F)\times\GL_{n-k}(F)}(\pi''_{N_{k_2}}[n_1]\otimes\pi'[-n_2])+\sum_{j=1}^{u_k}\tau_{1,j}^k\otimes\tau_{2,j}^k
\end{align}
of virtual representations of $\GL_k(F)\times\GL_{n-k}(F)$, where $\pi'$ denotes $\pi_1[\frac{n_2}2]$, $\pi''$ denotes $\pi_2[-\frac{n_1}2]$, $N_{k_1}$ denotes the unipotent radical of $P_{n_1}\cap\GL_k$, $N_{k_2}$ denotes the unipotent radical of $P_{n_2}\cap\GL_k$, and the $\tau_{1,1}^k,\dotsc,\tau_{1,u_k}^k$ (respectively $\tau_{2,1}^k,\dotsc,\tau_{2,u_k}^k$) are smooth representations of $\GL_k(F)$ (respectively $\GL_{n-k}(F)$) such that the $\tau_{1,j}^k$ are properly induced representations of $\GL_k(F)$. As remarked above, if $k>n_1$ or $k>n_2$, then we can absorb the first or second term on the right hand side into the third summation term, respectively. 

So suppose that $k\leq n_1$ and $k\leq n_2$. We shall begin by studying the first term on the right hand side of Equation (\ref{eq:jacquetrestriction}). Because the unnormalized Jacquet restriction $\pi'_{N_{k_1}}$ of $\pi'$ is a smooth representation of $\GL_k(F)\times\GL_{n_1-k}(F)$ of finite length, we can write its virtual representation as
\begin{align*}
\pi'_{N_{k_1}} = \sum_{i=1}^{t_k'}\pi'^k_{1,i}\otimes\pi'^k_{2,i},
\end{align*}
where $t_k'$ is a non-negative integer, and the $\pi'^k_{1,i}$ and $\pi'^k_{2,i}$ are irreducible smooth representations of $\GL_k(F)$ and $\GL_{n_1-k}(F)$, respectively. Write $P_{k,n_1}$ for the standard parabolic subgroup of $\GL_n$ with block sizes $(k,n_1-k,n_2)$. By applying $\Ind_{P_{n_1}(F)\cap(\GL_k(F)\times\GL_{n-k}(F))}^{\GL_k(F)\times\GL_{n-k}(F)}(-\otimes\pi'')$ to the above decomposition of $\pi'_{N_{k_1}}$, we obtain the equality of virtual representations
\begin{align}\label{eq:virtual1}
\Ind_{P_{n_1}(F)\cap(\GL_k(F)\times\GL_{n-k}(F))}^{\GL_k(F)\times\GL_{n-k}(F)}(\pi'_{N_{k_1}}\otimes\pi'') = \sum_{i=1}^{t_k'}\pi_{1,i}'^k\otimes\Ind_{P_{k,n_1}(F)\cap\GL_{n-k}(F)}^{\GL_{n-k}(F)}(\pi_{2,i}'^k\otimes\pi''),
\end{align}
giving an alternate description of the first term on the right hand side of Equation (\ref{eq:jacquetrestriction}). 

Now Proposition \ref{ss:jacquetmoduletraceidentity} for $\pi'$ indicates that
\begin{align*}
\sum_{k=1}^{n_1}q^{(n_1-k)r}\sum_{i=1}^{t_k'}\tr\left(h_1|\Ind_{P_k(\cO)\cap\GL_{n_1}(\cO)}^{\GL_{n_1}(\cO)}\left(\tr((\bf1_{B_k}\cdot\phi_{\tau,-},\sg)|\BC(\pi_{1,i}'^k))\otimes\pi_{2,i}'^k\right)\right) = \tr(f_{\tau,h_1}|\pi')
\end{align*}
for all functions $h_1$ in $C^\infty_c(\GL_{n_1}(\cO))$. Theorem A.(i) for $\pi'$ implies that $\tr(f_{\tau,h_1}|\pi') = \tr\left(\tau|\rho(\pi')\right)\tr(h_1|\pi')$, so \ref{ss:peterweyl} yields an equality of $\bC$-virtual representations
\begin{align*}
\sum_{k=1}^{n_1}q^{(n_1-k)r}\sum_{i=1}^{t_k'}\Ind_{P_k(\cO)\cap\GL_{n_1}(\cO)}^{\GL_{n_1}(\cO)}\left(\tr((\bf1_{B_k}\cdot\phi_{\tau,-},\sg)|\BC(\pi_{1,i}'^k))\otimes\pi_{2,i}'^k\right) = \tr\left(\tau|\rho(\pi')\right)\pi'
\end{align*}
of $\GL_{n_1}(\cO)$. Since we defined $\pi'$ to be $\pi_1[\textstyle\frac{n_2}2]$, taking $\Ind_{P_{n_1}(F)}^{\GL_n(F)}(-\otimes\pi'')$ on both sides shows that
\begin{align}\label{eq:finaleq1}
\sum_{k=1}^{n_1}q^{(n-k)r}\sum_{i=1}^{t_k'}\Ind_{P_{k,n_1}(\cO)}^{\GL_n(\cO)}\left(\tr((\bf1_{B_k}\cdot\phi_{\tau,-},\sg)|\BC(\pi_{1,i}'^k))\otimes\pi_{2,i}'^k\otimes\pi''\right) = \tr\left(\tau|\rho(\pi_1)(\textstyle\frac{n_2}2)\right)\pi
\end{align}
as $\bC$-virtual representations of $\GL_n(\cO)$, where we used Proposition \ref{ss:unramifiedtwistsofrho} to identify $\rho(\pi')$ with $\rho(\pi_1)(\frac{n_2}2)$.

Next, we turn to the second term on the right hand side of Equation (\ref{eq:jacquetrestriction}). As with the first term, we have an equality of virtual representations
\begin{align*}
\pi''_{N_{k_2}} = \sum_{i=1}^{t_k''}\pi''^k_{1,i}\otimes\pi''^k_{2,i},
\end{align*}
where $t_k''$ is a non-negative integer, and the $\pi''^k_{1,i}$ and $\pi''^k_{2,i}$ are irreducible smooth representations of $\GL_k(F)$ and $\GL_{n_2-k}(F)$, respectively. Write $P_{k,n_2}$ for the standard parabolic subgroup of $\GL_n$ with block sizes $(k,n_2-k,n_1)$. By applying $\Ind_{P_{n_2}(F)\cap(\GL_k(F)\times\GL_{n-k}(F))}^{\GL_k(F)\times\GL_{n-k}(F)}((-)[n_1]\otimes\pi'[-n_2])$ to the above decomposition of $\pi''_{N_{k_2}}$, we obtain the equality of virtual representations
\begin{align}\label{eq:virtual2}
\Ind_{P_{n_2}(F)\cap(\GL_k(F)\times\GL_{n-k}(F))}^{\GL_k(F)\times\GL_{n-k}(F)}(\pi''_{N_{k_2}}[n_1]\otimes\pi'[-n_2]) = \sum_{i=1}^{t_k''}\pi_{1,i}''^k[n_1]\otimes\Ind_{P_{k,n_2}(F)\cap\GL_{n-k}(F)}^{\GL_{n-k}(F)}(\pi_{2,i}''^k[n_1]\otimes\pi'[-n_2]),
\end{align}
giving an alternate description of the second term on the right hand side of Equation (\ref{eq:jacquetrestriction}).

Now Proposition \ref{ss:jacquetmoduletraceidentity} for $\pi''$ indicates that
\begin{align*}
\sum_{k=1}^{n_2}q^{(n_2-k)r}\sum_{i=1}^{t_k''}\tr\left(h_2|\Ind_{P_k(\cO)\cap\GL_{n_2}(\cO)}^{\GL_{n_2}(\cO)}\left(\tr((\bf1_{B_k}\cdot\phi_{\tau,-},\sg)|\BC(\pi_{1,i}''^k))\otimes\pi_{2,i}''^k\right)\right) = \tr(f_{\tau,h_2}|\pi'')
\end{align*}
for all functions $h_2$ in $C^\infty_c(\GL_{n_2}(\cO))$. Theorem A.(i) for $\pi''$ implies that $\tr(f_{\tau,h_2}|\pi'')=\tr\left(\tau|\rho(\pi'')\right)\tr(h_2|\pi'')$, so \ref{ss:peterweyl} yields an equality of $\bC$-virtual representations
\begin{align*}
\sum_{k=1}^{n_2}q^{(n_2-k)r}\sum_{i=1}^{t_k''}\Ind_{P_k(\cO)\cap\GL_{n_2}(\cO)}^{\GL_{n_2}(\cO)}\left(\tr((\bf1_{B_k}\cdot\phi_{\tau,-},\sg)|\BC(\pi_{1,i}''^k))\otimes\pi_{2,i}''^k\right) = \tr\left(\tau|\rho(\pi'')\right)\pi''
\end{align*}
of $\GL_{n_2}(\cO)$. Since we defined $\pi''$ to be $\pi_2[-\textstyle\frac{n_1}2]$, taking $\Ind_{P_{n_2}(F)}^{\GL_n(F)}((-)[n_1]\otimes\pi'[-n_2])$ on both sides shows that
\begin{align}\label{eq:finaleq2}
\sum_{k=1}^{n_2}q^{(n_2-k)r}\sum_{i=1}^{t_k''}\Ind_{P_{k,n_2}(\cO)}^{\GL_n(\cO)}\left(\tr((\bf1_{B_k}\cdot\phi_{\tau,-},\sg)|\BC(\pi_{1,i}''^k[n_1]))\otimes\pi_{2,i}''^k[n_1]\otimes\pi'[-n_2]\right) = \tr\left(\tau|\rho(\pi_2)(\textstyle\frac{n_1}2)\right)\pi
\end{align}
as $\bC$-virtual representations of $\GL_n(\cO)$, where we used Proposition \ref{ss:unramifiedtwistsofrho} to identify $\rho(\pi'')$ with $\rho(\pi_2)(-\frac{n_2}2)$.

We piece together the above work as follows. Equation (\ref{eq:jacquetrestriction}), Equation (\ref{eq:virtual1}), and Equation (\ref{eq:virtual2}) indicate that
\begin{align*}
\pi_{N_k} &= \sum_{i=1}^{t_k'}\pi_{1,i}'^k\otimes\Ind_{P_{k,n_1}(F)\cap\GL_{n-k}(F)}^{\GL_{n-k}(F)}(\pi_{2,i}'^k\otimes\pi'')\\
&+\sum_{i=1}^{t_k''}\pi_{1,i}''^k[n_1]\otimes\Ind_{P_{k,n_2}(F)\cap\GL_{n-k}(F)}^{\GL_{n-k}(F)}(\pi_{2,i}''^k[n_1]\otimes\pi'[-n_2])+\sum_{j=1}^{u_k}\tau_{1,j}^k\otimes\tau_{2,j}^k,
\end{align*}
and we take this decomposition for our $\pi^k_{1,i}$ and $\pi^k_{2,i}$. Because the $\tau_{1,j}^k$ are properly induced, Lemma \ref{ss:properlyinducedsupport} shows that the $\tr((\bf1_{B_k}\cdot\phi_{\tau,-},\sg)|\BC(\tau_{1,j}^k))$ terms vanish. Therefore the sum of Equation (\ref{eq:finaleq1}) and Equation (\ref{eq:finaleq2}) is precisely Equation (\ref{eq:goal}), and this concludes the proof of Proposition \ref{ss:conditionb}.
\end{proof}

\section{The Lubin--Tate tower}\label{s:lubintatetower}
In this section, our goal is to prove that condition (d) in Lemma \ref{ss:firstinductivelemma} holds, under the assumption that condition (c) holds. Although we postpone the proof of condition (c) for later, our proof of condition (d) does not use anything from the construction of condition (c)---we may safely black box the latter. Along the way, we also prove the $n=1$ case of Theorem A, which serves as the base case for applying Lemma \ref{ss:firstinductivelemma}.

We begin by using connected local shtukas to prove, in the same vein as Lemma \ref{ss:normbijection} and Lemma \ref{ss:etalelocalshtukatransfer}, the existence of transfers for certain conjugation-invariant functions on division algebras over $F$. We then introduce the \emph{Lubin--Tate tower}, which is closely related to our deformation spaces from \S\ref{s:deformationspaces}. By studying this relation, along with the fact that it gives the Lubin--Tate proof of local class field theory, we prove the $n=1$ case of Theorem A. We also use this relation to compare the cohomologies of the Lubin--Tate tower and $\br\fX_{\de,m}$. This comparison is stated in terms of the \emph{local Jacquet--Langlands correspondence}, which relates representations of $\GL_n(F)$ with representations of division algebras. Finally, we use results from Jacquet--Piatetski-Shapiro--Shalika's theory of \emph{new-vectors}, which are certain elements in irreducible generic representations of $\GL_n(F)$, to conclude the proof of condition (d).

\subsection{}
We start by using two descriptions of isomorphism classes of connected local shtukas over $\Spec\ka_r$ to prove the following proposition. Write $B$ for the central division algebra over $F$ of Hasse invariant $\frac1n$, and write $\cO_B$ for its ring of integers. By abuse of notation, we denote the normalized valuation on $B$ using $v$. Write $B_r$ for the subset of $b$ in $B$ satisfying $v(b)=r$.
\begin{lem*}\label{ss:noncommutativenormbijection}
There exists a bijection
\begin{align*}
\N:\{\GL_n(\cO_r)\mbox{-}\sg\mbox{-conjugacy classes in }B_n\}\rar^\sim\{\cO_B^\times\mbox{-conjugacy classes in }B_r\}
\end{align*}
such that the characteristic polynomial of $\N\de$ equals that of $\de\sg(\de)\dotsm\sg^{r-1}(\de)$. By abuse of notation, we call this the \emph{norm map}.
\end{lem*}
This result and its proof are entirely analogous to those of Lemma \ref{ss:normbijection}.
\begin{proof}
We shall construct bijections between both the left-hand as well as right-hand sides and the set
\begin{align*}
\{\mbox{isomorphism classes of connected local shtukas over }\Spec\ka_r\mbox{ of rank }n\},
\end{align*}
and we denote the composed bijection using $\N$. On the left-hand side, \ref{ss:effectiveminusculelocalshtukaconjugacy} yields the desired bijection. 

We now turn to the right-hand side. Let $\sM$ be a connected local shtuka over $\Spec\ka_r$. Completed unramified descent implies that the isomorphism class of $\sM$ is determined by the isomorphism class of $\br\sM\deq\sM_{\ov\ka}$ along with a descent isomorphism $f:\sg^{-r,*}\br\sM\rar^\sim\br\sM$ corresponding to the action of the geometric $q^r$-Frobenius map $\sg^{-r}$. As $\br\sM$ must be isomorphic to the local shtuka associated with the Lubin--Tate module, after choosing such an isomorphism we may identify $\End(\br\sM)$ with $\cO_B$. Write $\br\sF:\sg^*\br\sM\rar\br\sM$ for the Frobenius of $\br\sM$, which corresponds to an element of valuation $1$ in $\cO_B$, and write $\ga$ for the endomorphism
\begin{align*}
\ga\deq\br\sF\circ\dotsb\circ\sg^{r-1,*}\br\sF\circ\sg^{r,*}f
\end{align*}
of $\br\sM$, viewed as an element of $\cO_B$. As $\sg^{r,*}f$ is an isomorphism and hence corresponds to an element of valuation $0$ in $\cO_B$, this shows that $v(\ga)=r$. Since any other choice of isomorphism between $\br\sM$ and the local shtuka associated with the Lubin--Tate module preserves the $\cO_B^\times$-conjugacy class of $\ga$, this yields a bijection $\ga\leftrightarrow\sM$ between $\cO_B^\times$-conjugacy classes in $B_r$ and isomorphism classes of connected local shtukas over $\Spec\ka_r$ of rank $n$.

Recall that $\br\sF$ corresponds to $\de\circ\sg^{\oplus n}$ under \ref{ss:effectiveminusculelocalshtukaconjugacy}. Since $f$ corresponds to $(\sg^{\oplus n})^{-r}$, we see that $\ga$ corresponds to
\begin{align*}
\underbrace{(\de\circ\sg^{\oplus n})\circ\dotsb\circ(\de\circ\sg^{\oplus n})}_{r\,\text{times}}\circ (\sg^{\oplus n})^{-r} = \left(\de\sg(\de)\dotsm\sg^{r-1}(\de)\right)\circ(\sg^{\oplus n})^r\circ (\sg^{\oplus n})^{-r} = \de\sg(\de)\dotsm\sg^{r-1}(\de),
\end{align*}
as desired.
\end{proof}

\subsection{}\label{ss:connectedlocalshtukatransfer}
We will now construct certain functions with matching twisted orbital integrals, that is, certain transfers of functions. Let $h$ be an $\cO_B^\times$-conjugation-invariant function in $C^\infty_c(B_r)$, and form the function $\phi:B_n\rar\bC$ by sending $\de\mapsto h(\N\de)$. One can show that $\phi$ is locally constant \cite[Corollary 2.3]{Sch13b}\footnote{The proof given here is stated for mixed characteristic and $\GL_n$, but it only uses Lang's lemma and the relationship between inner twists of $\GL_n$ and central division algebras. In particular, it adapts to equal characteristic and any inner twist of $\GL_n$.}.
\begin{lem*}
The functions $h$ and $\phi$ have matching twisted orbital integrals.
\end{lem*}
This result and its proof are entirely analogous to those of Lemma \ref{ss:etalelocalshtukatransfer}.
\begin{proof}
We need to prove that $\O_{\N\de}(h)=\TO_{\de,\sg}(\phi)$ for all regular $\de$ in $B_n$, where $\O_{\N\de}$ and $\TO_{\de,\sg}$ denote the orbital integral on $\N\de$ and twisted orbital integral on $\de$, respectively, with respect to compatible Haar measures. We can prove this as follows: write $\sM$ for the connected local shtuka corresponding to $\de$ under \ref{ss:effectiveminusculelocalshtukaconjugacy}, and write $X$ for the set of isomorphism classes of quasi-isogenies $\be:\sM\dasharrow\sM'$ between connected local shtukas. Note that the group $\Ga$ of self-quasi-isogenies $\sM\dasharrow\sM$ has a left action on $X$ given by sending $\be$ to $\be\circ g^{-1}$ for any $g$ in $\Ga$.

One can equate both $\O_{\N\de}(h)$ as well as $\TO_{\de,\sg}(\phi)$ to the sum
\begin{align*}
\sum_{(\sM',\be)}\wt{h}(\sM',\be),
\end{align*}
where $(\sM',\be)$ ranges over all elements in $\Ga\bs X$, and $\wt{h}$ is the function sending $(\sM',\be)$ to $h(\de')$, where $\de'$ corresponds to $\sM'$ under \ref{ss:effectiveminusculelocalshtukaconjugacy}. In a line of reasoning similar to the proof of Lemma \ref{ss:noncommutativenormbijection}, we relate $\TO_{\de,\sg}(\phi)$ to the above sum by using \ref{ss:effectiveminusculelocalshtukaconjugacy}, and we relate $\O_{\N\de}(h)$ to the above sum by using completed unramified descent, c.f. \cite[Proposition 4.7]{Sch13}.
\end{proof}

\subsection{}
At this point, we introduce the \emph{Lubin--Tate tower}. Let $\de$ be in $B_n$. By the Dieudonn\'e--Manin classification \cite[(2.4.5)]{Lau96}, the isomorphism class of $\br{H}_\de$ is independent of our choice of $\de$, and we may take $\de$ to lie in $\GL_n(F)$. We write $\br{H}$ for $\br{H}_\de$. The inverse system $(\br\fX_{\de,m})_m$ does not depend on $\de$ either, so we shall rewrite it as
\begin{align*}
\dotsm\rar\br\fX_m\rar\dotsm\rar\br\fX_1\rar\br\fX_0.
\end{align*}
This is the \emph{Lubin--Tate tower}. Of course, all our statements from \S\ref{s:deformationspaces} concerning $\br{H}_\de$ also apply to the Lubin--Tate tower:
\begin{enumerate}[$\bullet$]
\item this inverse system has a right action of $\GL_n(\cO)$,
\item we can form the cohomology group $R^i\psi_m\deq H^i(\br\fX_{m,\bC_\vpi},\ov\bQ_\ell)$, which is isomorphic to $H^0(\br\fX_{m,\ov\ka},R^i\Psi_{\br\fX_{m,\bC_\vpi}},\ov\bQ_\ell)$,
\item we can form the direct limit $R^i\psi\deq\dirlim_mR^i\psi_m$,
\item we can take the alternating sum $[R\psi]\deq\sum_{i=0}^\infty(-1)^iR^i\psi$.
\end{enumerate}
We shall now explain how the left action of $I_F$ on $R^i\psi_m$ actually extends to an action of $(B^\times\times W_F)_0$, where $(B^\times\times W_F)_0$ is the subgroup 
\begin{align*}
(B^\times\times W_F)_0\deq\{b\times\tau\in B^\times\times W_F|v(b)+v(\tau)=0\}
\end{align*}
of $B^\times\times W_F$. We start by describing the action for $b\times \tau$ with non-negative $v(\tau)=-v(b)$. Let $r=v(\tau)=v(b^{-1})$, and note that $b^{-1}$ lies in $\cO_B$. By abuse of notation, write $b^{-1}:\br{H}\rar\br{H}$ for the corresponding endomorphism of $\br{H}$, write $\br{F}:\sg^*\br{H}\rar\br{H}$ for the Frobenius of $\br{H}$, and write $f_\tau:\sg^{-r,*}\br{H}\rar^\sim\br{H}$ for the isomorphism from \ref{ss:deformationunramifiedextension}. Using the isomorphism $\sg^*\br{H}\cong\br{H}$ obtained from the fact that $\de$ lies in $\GL_n(F)$, the map $\br{F}$ corresponds to an element of valuation $1$ in $\cO_B$. Therefore the quasi-isogeny
\begin{align*}
b\times\tau\deq f_\tau\circ(\sg^{-r,*}\br{F}^{-1}\circ\dotsb\circ\sg^{-1,*}\br{F}^{-1})\circ b^{-1}
\end{align*}
corresponds to an element of valuation $0$ and hence equals an automorphism of $\br{H}$.

As for $b\times\tau$ with negative $v(\tau)$, we define $b\times\tau$ as the inverse of $d^{-1}\times\tau^{-1}$. Since we view $\cO_B$ as acting from the right on $\br{H}$, we see that in all cases $(b\times\tau)\circ(b'\times\tau') = bb'\times\tau\tau'$.

By sending the triple $(H',\al',\io')$ to $(H',\al',\io'\circ(d\times\tau))$, we obtain a right action on $\br\fX_m$ and hence its base change $\br\fX_{m,\cO_{\bC_\vpi}}$. Taking cohomology gives us the desired left action on $R^i\psi_m$. Forming the direct limit over all $m$ yields a $(\GL_n(\cO)\times\cO_B^\times)\times I_F$-admissible/continuous representation as in \cite[p.~24]{HT01} of $\GL_n(\cO)\times(B^\times\times W_F)_0$ over $\ov\bQ_\ell$.

\subsection{}\label{ss:lubintatecomparison}
Now return to the situation of an arbitrary $\de$ in $B_n$. We shall begin studying the relationship between $\br\fX_{\de,m}$ and the Lubin--Tate tower by considering the virtual representations $[R\psi_\de]$ and $[R\psi]$.
\begin{prop*}
We have an equality of traces
\begin{align*}
\tr(\tau\times h|[R\psi_\de]) = \tr(h\times(\N\de)^{-1}\times\tau|[R\psi]).
\end{align*}
\end{prop*}
\begin{proof}
Under our $\GL_n(\cO)$-equivariant isomorphism $\br{H}_\de\cong\br{H}$, we see that $\br{F}$ corresponds to $\de\circ\sg^{\oplus n}$. Therefore work from the proof of Lemma \ref{ss:noncommutativenormbijection} shows that
\begin{align*}
(\N\de,\tau) = f_\tau\circ(\sg^{-r,*}\br{F}^{-1}\circ\dotsb\circ\sg^{-1,*}\br{F}^{-1})\circ\left((\br{F}\circ\dotsb\circ\sg^{r-1,*}\br{F}\circ\sg^{r,*}f)^{-1}\right)^{-1} = f_\tau\circ\sg^{r,*}f,
\end{align*}
where $f:\sg^{-r,*}\br{H}\rar^\sim\br{H}$ corresponds to the action of $\sg^{-r}$. The description of $\fX_{\de,m}$ in \ref{ss:deformationunramifiedextension} shows that precomposing with $f_\tau\circ\sg^{r,*}f$ induces the action of $\tau$ on $R^i\psi_{\de,m}$. Thus the action of $\ga\times(\N\de)^{-1}\times\tau$ on $R^i\psi_m$ corresponds to the action of $\tau\times\ga$ on $R^i\psi_{\de,m}$ for all $\ga$ in $\GL_n(\cO)$. Taking the direct limit over all $m$, forming the alternating sum over all $i$, and taking traces yields the desired result.
\end{proof}

\subsection{}\label{ss:lubintatelocalclassfieldtheory}
Before turning to the proof of the $n=1$ case of Theorem A, we take a brief interlude to describe the cohomology of the Lubin--Tate tower in terms of the Lubin--Tate proof of local class field theory. Note that in the $n=1$ case, $\GL_1(\cO)=\cO^\times$ and $(B^\times\times W_F)_0=(F^\times\times W_F)_0$. Recall that $\Art:F^\times\rar^\sim W^\ab_F$ denotes the local reciprocity isomorphism that sends uniformizers to geometric $q$-Frobenii. We denote the maximal unramified extension of $F$ using $F^\text{nr}$, and we denote the extension of $F^\text{nr}$ corresponding to $(\cO/\vpi^m)^\times$ using $F^{\text{nr},\,m}$.

In our calculations for the $n=1$ case, we take $\de=\vpi$.
\begin{prop*}
In the $n=1$ case, we have an equality of representations
\begin{align*}
[R\psi] = C^\infty_c(\cO^\times),
\end{align*}
where $\cO^\times$ acts by inverse left multiplication, and $(F^\times\times W_F)_0$ acts via right multiplication by $\Art^{-1}(\tau)^{-1}b$.
\end{prop*}
We only distinguish between left and right multiplication to maintain the situation for general $n\geq1$---there is no difference between them in the $n=1$ case, because all groups involved only act through their abelianizations.
\begin{proof}[Proof of Proposition \ref{ss:lubintatelocalclassfieldtheory}]
Recall from \S\ref{s:deformationspaces} that $\br\fX_m=\Spf\br{R}_m$ for some regular complete local Noetherian $\br\cO$-algebra $\br{R}_m$. The Lubin--Tate proof of local class field theory shows that $\br{R}_m$ is isomorphic to the completion of $\cO^{\text{nr},\,m}$, where $\cO^{\text{nr},\,m}$ is the ring of integers of $F^{\text{nr},\,m}$. Therefore $\br\fX_{m,\bC_\vpi}=\Spf\br{R}_m\wh\otimes_{\br\cO}\bC_\vpi$ consists of one point for each element of $\Gal(\br{R}_m[\frac1\vpi]/\br{F})=(\cO/\vpi^m)^\times$.

Under this identification, the right action of $\cO^\times$ on $\br\fX_{m,\bC_\vpi}$ is given by sending $g$ to $\ga^{-1}g$ for any $\ga$ in $\cO^\times$, and the Lubin--Tate proof of local class field theory implies that the right action of $(F^\times\times W_F)_0$ is given by sending $g$ to $g\Art^{-1}(\tau)b^{-1}$ for any $\tau\times b$ in $(F^\times\times W_F)_0$. Taking cohomology shows that
\begin{align*}
R^0\psi_m = \ov\bQ_\ell[(\cO/\vpi^m)^\times],
\end{align*}
where $\cO^\times$ acts by inverse left multiplication, and $(F^\times\times W_F)_0$ acts via right multiplication by $\Art^{-1}(\tau)^{-1}b$. Finally, as $R^i\psi_m$ vanishes for $i>n-1=0$, we have $[R\psi]=R^0\psi$. Forming the direct limit over all $m$ yields the desired result, where we have identified $\ov\bQ_\ell$ and $\bC$.
\end{proof}

\subsection{}\label{ss:theoremabasecase}
With Proposition \ref{ss:lubintatecomparison} and Proposition \ref{ss:lubintatelocalclassfieldtheory} in hand, we can now prove the $n=1$ case of Theorem A. This amounts to rewriting the Lubin--Tate tower in terms of the Lubin--Tate proof of local class field theory.
\begin{prop*}
Let $\chi$ be an irreducible smooth representation of $\GL_1(F)$, that is, a smooth character $\chi:F^\times\rar\bC^\times$. Then Theorem A holds for $\chi$, that is, there exists a unique continuous character $\rho(\chi):W_F\rar\bC$ satisfying the following property:
\begin{align*}
\mbox{for all }\tau\mbox{ in }W_F\mbox{ with }v(\tau)>0\mbox{ and }h\mbox{ in }C^\infty_c(\GL_n(\cO))\mbox{, we have }\tr(f_{\tau,h}|\chi) = \tr\left(\tau|\rho(\chi)\right)\tr(h|\chi).
\end{align*}
This $\rho(\chi)$ is given by $\rho(\chi)=\chi\circ\Art^{-1}$.
\end{prop*}
Note that we need not consider Theorem A.(ii) for $n=1$, since $\GL_1$ has no proper parabolic subgroups. As before, we always take compatible Haar measures, but we generally omit their precise description.
\begin{proof}[Proof of Proposition \ref{ss:theoremabasecase}]
We first use \ref{ss:twistedcharacters} to show that
\begin{align*}
\tr(f_{\tau,h}|\chi) = \tr((\phi_{\tau,h},\sg)|\BC(\chi)) = \int_{\vpi\cO_r^\times}\!\dif\de\,\tr(\tau\times h|[R\psi_\de])\chi(\N\de),
\end{align*}
and then Proposition \ref{ss:lubintatecomparison} implies that the above integral is equal to
\begin{align*}
\int_{\vpi\cO^\times_r}\!\dif\de\,\tr(h\times(\N\de)^{-1}\times\tau|[R\psi])\chi(\N\de).
\end{align*}
Since $\tr(h\times(-)^{-1}\times\tau|[R\psi])$ and $\tr(h\times(\N-)^{-1}\times\tau|[R\psi])$ have matching twisted orbital integrals by Lemma \ref{ss:connectedlocalshtukatransfer}, making the change of variables $b=\N\de$ gives us
\begin{align*}
\int_{\vpi^r\cO^\times}\!\dif b\,\tr(h\times b^{-1}\times\tau|[R\psi])\chi(b).
\end{align*}
Using Proposition \ref{ss:lubintatelocalclassfieldtheory} and applying the Peter--Weyl theorem to $\cO^\times$ (which is really just Pontryagin duality, since $\cO^\times$ is commutative) shows that the above integral is equal to
\begin{align*}
\tr(h|\chi)\tr(\Art^{-1}(\tau)|\chi) = \tr(h|\chi)\tr(\tau|\chi\circ\Art^{-1}).
\end{align*}
Therefore $\rho(\chi)=\chi\circ\Art^{-1}$ satisfies the desired property. Since $\chi$ is a character, it equals its trace, so it is also characterized uniquely by this property.
\end{proof}

\subsection{}
Now return to the arbitrary $n\geq1$ case. We introduce some notation on the multiplicity of $\la$ in $R^i\psi$, where $\la$ is an irreducible smooth representation of $B^\times$. Write $R^i\psi(\la)$ for the vector space
\begin{align*}
R^i\psi(\la)\deq\Hom_{\cO_B^\times}(\res\la_{\cO_B^\times},R^i\psi).
\end{align*}
This has a left action of $\GL_n(\cO)\times W_F$ as follows. For any $f$ in $R^i\psi(\la)$ and $(\ga,\tau)$ in $\GL_n(\cO)\times W_F$, we set
\begin{align*}
((\ga\times\tau)f)(v) = (\ga\times b\times\tau)f(b^{-1}v)
\end{align*}
for all $v$ in $\la$, where $b$ is any element of $B^\times$ satisfying $v(b)+v(\tau)=0$. Since $f$ commutes with $\cO_B^\times$, this action is independent of our choice of $b$. We see that $R^i\psi(\la)$ is a $\GL_n(\cO)\times I_F$-admissible/continuous representation of $\GL_n(\cO)\times W_F$ over $\ov\bQ_\ell$. Write $[R\psi](\la)$ for the virtual representation $\sum_{i=0}^\infty(-1)^iR^i\psi(\la)$.

\subsection{}\label{ss:localjacquetlanglands}
In order to state our next result, we recall the \emph{local Jacquet--Langlands correspondence}. Let $B$ be a central division algebra over $F$ of dimension $n^2$. Since every irreducible smooth representation $\la$ of $B^\times$ is finite-dimensional, its character distribution is represented by the function $\Te_\la:B^\times\rar\bC$ that sends $b\mapsto\tr(b|\la)$. This finite-dimensionality also implies that every such representation is essentially $L^2$.

There exists a unique bijection \cite[Th. 1.1]{Bad00}
\begin{align*}
\JL:\left\{
  \begin{tabular}{c}
    isomorphism classes of irreducible essentially \\
    $L^2$ representations of $\GL_n(F)$
  \end{tabular}\right\}\rar^\sim  \left\{\begin{tabular}{c}
    isomorphism classes of irreducible \\
    representations of $B^\times$
  \end{tabular}\right\}
\end{align*}
such that, for all regular elliptic elements $\ga$ in $\GL_n(F)$ and $b$ in $B^\times$ whose characteristic polynomials are the same, we have the equality
\begin{align*}
\Te_\pi(\ga) = (-1)^{n-1}\Te_{\JL(\pi)}(b).
\end{align*}
Furthermore, the central characters of $\pi$ and $\JL(\pi)$ are equal.

\begin{prop}\label{prop:lubintatecohomology}
Assume that condition (c) in Lemma \ref{ss:firstinductivelemma} holds, and suppose that $\pi$ is an irreducible cuspidal representation of $\GL_n(F)$. Then we have an equality of virtual $\GL_n(\cO)\times I_F$-admissible/continuous representations
\begin{align*}
(-1)^{n-1}\res\pi_{\GL_n(\cO)}\otimes\rho(\pi) = [R\psi](\JL(\pi)).
\end{align*}
\end{prop}
\begin{proof}
We begin by using condition (c) in Lemma \ref{ss:firstinductivelemma} and \ref{ss:twistedcharacters} to see that
\begin{align*}
\tr(h\times\tau|\res\pi_{\GL_n(\cO)}\otimes\rho(\pi)) = \tr(h|\pi)\tr(\tau|\rho(\pi)) = \tr(f_{\tau,h}|\pi) = \tr((\phi_{\tau,h},\sg)|\BC(\pi)).
\end{align*}
Then, Proposition \ref{ss:lubintatecomparison} shows that this twisted trace is equal to the integral
\begin{align*}
\int_{\GL_n(\cO_r)\diag(\vpi,1,\dotsc,1)\GL_n(\cO_r)}\!\dif\de\,\tr(h\times(\N\de)^{-1}\times\tau|[R\psi])\Te_\pi^{\sg}(\de).
\end{align*}
To narrow down our domain of integration, we shall use the following lemma.
\begin{lem}\label{lem:cuspidaltwistedcharacter}
The twisted character $\Te_\pi^{\sg}$ vanishes on $\GL_n(\cO_r)\diag(\vpi,1,\dotsc,1)\GL_n(\cO_r)\ssm B_n$.
\end{lem}
\begin{proof}
Recall that we have the decomposition
\begin{align*}
\GL_n(\cO_r)\diag(\vpi,1,\dotsc,1)\GL_n(\cO_r) = \coprod_{k=1}^nB_{n,k},
\end{align*}
and recall the description of $B_{n,k}$ from \ref{ss:localshtukaunramifiedextension} in terms of $\de=\de^\circ\oplus\de^\et$. Write $P_\de$ for the parabolic subgroup of $\GL_n$ over $F_r$ associated with $\de$. This shows that $P_\de$ is a proper subgroup if and only if $\de$ does not lie in $B_n$, and because $\pi$ is cuspidal, Casselman's theorem \cite[5.2]{Cas77} yields the desired result.
\end{proof}
Return to the proof of Proposition \ref{prop:lubintatecohomology}. Lemma \ref{lem:cuspidaltwistedcharacter} implies that our integral becomes
\begin{align*}
\int_{B_n}\!\dif\de\,\tr(h\times(\N\de)^{-1}\times\tau|[R\psi])\Te_\pi(\N\de).
\end{align*}
As $\tr(h\times(-)^{-1}\times\tau|[R\psi])$ and $\tr(h\times(\N-)^{-1}\times\tau|[R\psi])$ have matching twisted orbital integrals by Lemma \ref{ss:connectedlocalshtukatransfer}, making the change of variables $b=\N\de$ and applying the local Jacquet--Langlands correspondence indicate that the above expression equals
\begin{align*}
(-1)^{n-1}\int_{B_r}\!\dif b\,\tr(h\times b^{-1}\times\tau|[R\psi])\Te_{\JL(\pi)}(b) = (-1)^{n-1}\tr(h\times\tau|[R\psi](\JL(\pi))).
\end{align*}
Since virtual $\GL_n(\cO)\times I_F$-admissible/continuous representations of $\GL_n(\cO)\times W_F$ are determined by their traces, this yields the desired result.
\end{proof}

\subsection{}\label{ss:conditiond}
We conclude this section by verifying that, under the assumption that condition (c) from Lemma \ref{ss:firstinductivelemma} holds, condition (d) holds as well.
\begin{prop*}
Assume that condition (c) in Lemma \ref{ss:firstinductivelemma} holds, and suppose that $\pi$ is an irreducible cuspidal representation of $\GL_n(F)$. Then the $\bQ_{\geq0}$-virtual continuous representation $\rho(\pi)$ is actually a $\bZ$-virtual continuous representation of $W_F$.
\end{prop*}
\begin{proof}
Proposition \ref{prop:lubintatecohomology} shows that the $\bZ$-virtual representation $[R\psi](\JL(\pi))$ equals $\res\pi_{\GL_n(\cO)}\otimes\rho(\pi)$ as virtual $\GL_n(\cO)\times I_F$-admissible/continuous representations of $\GL_n(\cO)\times W_F$ over $\ov\bQ_\ell$. Thus its $\la$-isotypic component is also a $\bZ$-virtual representation for any irreducible smooth representation $\la$ of $\GL_n(\cO)$. If we could find such a $\la$ that is contained in $\pi$ with multiplicity $1$, then the $\la$-isotypic component of $[R\psi](\JL(\pi))$ would equal $\rho(\pi)$. Therefore, once we prove the following lemma, the result follows.
\end{proof}
\begin{lem}
Let $\pi$ be an irreducible cuspidal representation of $\GL_n(F)$. Then there exists an irreducible smooth representation $\la$ of $\GL_n(\cO)$ such that $\pi$ contains $\la$ with multiplicity $1$.
\end{lem}
\begin{proof}
Since $\pi$ is cuspidal and hence generic, the theory of new-vectors \cite[Theorem (5.1).(ii)]{JPSS81} shows that there exists a non-negative number $a$ such that $\dim\pi^{K(a)}=1$, where $K(a)$ is the subgroup
\begin{align*}
K(a)\deq\left\{\matr{a}{b}{c}{d}\in\GL_n(\cO)\middle|
a\in\GL_{n-1}(\cO),\,c\equiv0\pmod{\vpi^a},\,d\equiv1\pmod{\vpi^a}\right\}
\end{align*}
for positive $a$, and $K(a)\deq\GL_n(\cO)$ for $a=0$. In other words, $\dim\Hom_{K(a)}(\bC,\res\pi_{K(a)})=1$, where $\bC$ denotes the trivial representation. Frobenius reciprocity yields
\begin{align*}
\dim\Hom_{\GL_n(\cO)}(\cInd_{K(a)}^{\GL_n(\cO)}\bC,\res\pi_{\GL_n(\cO)}) = \dim\Hom_{K(a)}(\bC,\res\pi_{K(a)})=1.
\end{align*}
The Peter--Weyl theorem then implies that any irreducible smooth subrepresentation $\la$ of $\cInd_{K(a)}^{\GL_n(\cO)}\bC$ is contained in $\pi$ with multiplicity $1$, as desired.
\end{proof}

\section{Moduli spaces of $\sD$-elliptic sheaves}\label{s:modulispaces}
At this point, we shift our focus from local considerations to global ones. In this section, we begin by introducing \emph{$\sD$-elliptic sheaves}, which are the equi-characteristic analog of abelian varieties equipped with certain endomorphism structures. Afterwards, we introduce moduli spaces of $\sD$-elliptic sheaves, which therefore correspond to certain Shimura varieties and their integral models. The cohomology of these moduli spaces plays an important role in the Langlands correspondence. We can also obtain local shtukas from $\sD$-elliptic sheaves, and this is the equi-characteristic version of taking the $p$-divisible group of an abelian variety. We conclude this section by introducing a \emph{Serre--Tate theorem}, which relates deformations of $\sD$-elliptic sheaves to deformations of their associated local shtukas. This will eventually allow us to prove condition (c) in Lemma \ref{ss:firstinductivelemma}.

\subsection{}
We start by switching our notation to a global context. In these next few sections, let $\ka$ be a finite field of characteristic $p$ and cardinality $q$, and fix a separable closure $\ov\ka$ of $\ka$. We view all separable extensions of $\ka$ as lying in $\ov\ka$. Let $C$ be a geometrically connected proper smooth curve over $\ka$, and write $\bf{F}$ for its field of rational functions. Denote the adele ring of $C$ using $\bA$, denote the ring of integers of $\bA$ using $\bO$, and for any finite closed subscheme $I$ of $C$, write $\bK_I$ for the ideal of $\bO$ corresponding to $I$.

For any place $x$ of $C$, write $\cO_x$ for the local ring given by the completion of $C$ at $x$, and write $\bf{F}_x$ for the fraction field of $\cO_x$. Then $\bf{F}_x$ is a completion of $\bf{F}$ at $x$. Choose a uniformizer $\vpi_x$ of $x$ in $\bf{F}$, write $\ka_x$ for $\cO_x/\vpi_x$, write $q_x$ for $\#\ka_x$, and write $\deg{x}$ for $[\ka_x:\ka]$. By abuse of notation, we write $x$ for the normalized valuation corresponding to $x$. We also write $\deg:\bA^\times\rar\bZ$ for the function given by sending
\begin{align*}
(a_x)_x\mapsto-\sum_xx(a_x)\deg{x}.
\end{align*}
We fix a closed point $\infty$ in $C$, which shall serve as a replacement for the archimedean places at infinity.

\subsection{}\label{ss:divisionalgebras}
Since division algebras are fundamental to our constructions, we recall some facts about them here. For any central division algebra $D$ over $\bf{F}$, write $D_x$ for $D\otimes_\bf{F}\bf{F}_x$, and write $\inv_x(D)$ for the Hasse invariant of $D_x$. Recall that the Brauer group of $\bf{F}$ fits into an exact sequence
\begin{align*}
0\rar\Br(\bf{F})\rar\bigoplus_x\Br(\bf{F}_x)=\bigoplus_x\bQ/\bZ\rar\bQ/\bZ\rar0,
\end{align*}
where $x$ runs over all places of $C$, the first map sends $D\mapsto(D_x)_x$, and the second map sums the Hasse invariants of any element in $\bigoplus_x\Br(\bf{F}_x)$. In particular, we have $\inv_x(D)=0$ for cofinitely many $x$, that is, $D$ is split at $x$ for cofinitely many $x$. For any central division algebra $D$ in $\Br(\bf{F})$, its dimension equals the square of the least common denominator of $\inv_x(D)$ as $x$ runs over all places of $\bf{F}$.

We now fix a central division algebra $D$ over $\bf{F}$ of dimension $n^2$ such that $\inv_\infty(D)=0$, and we write $\text{Bad}$ for the set of places $x$ of $C$ where $\inv_x(D)\neq0$, that is, where $D$ ramifies. Then $\text{Bad}$ is finite, and $\infty$ does not lie in $\text{Bad}$. We write $C'$ for the open subscheme $C\ssm\text{Bad}$ of $C$.

\subsection{}
Next, we introduce \emph{orders} of $D$, which shall provide some sort of integrality structure on $D$.
\begin{defn*}
Let $\sD$ be a locally free (not necessarily commutative) $\sO_C$-algebra of rank $n^2$. We say $\sD$ is an \emph{order} of $D$ if its generic fiber $\sD_\bf{F}$ is isomorphic to $D$. For an order $\sD$ of $D$, we say $\sD$ is \emph{maximal} if it is maximal with respect to injective $\sO_C$-algebra morphisms $\sD'\inj{}\sD$ between orders of $D$.
\end{defn*}
We denote the completion of $\sD$ at $x$ using $\sD_x$, which is a free $\cO_x$-algebra of rank $n^2$, and we identify its generic fiber $\sD_x[\frac1{\vpi_x}]$ with $D_x$.

\subsection{}\label{ss:orders}
Let $\sD$ be an order of $D$. Let $U$ be an affine open subset of $C$, and let $M$ be an $\bf{F}$-basis of $D$ contained in $\sD(U)$. Because $D$ splits at cofinitely many places, checking valuations shows that $\cO_x\cdot M=\sD_x$ for cofinitely many $x$. Conversely, let $(\sD_x)_x$ be a collection of $\cO_x$-orders (where $x$ ranges over all places of $C$) such that there exists an $\bf{F}$-basis $M$ of $D$ for which $\cO_x\cdot M=\sD_x$ for cofinitely many $x$. One can use the Riemann--Roch theorem to show that the Zariski sheaf
\begin{align*}
U\mapsto\bigcap_{x'}\sD_{x'}\cap D,
\end{align*}
where $x'$ runs over all closed points in $U$, is an order of $D$.

This construction yields a bijection between isomorphism classes of orders of $D$ and collections $(\sM_x)_x$ of $\cO_x$-orders such that there exists an $\bf{F}$-basis $M$ of $D$ satisfying $\cO_x\cdot M=\sM_x$ for cofinitely many $x$. Note that $\sD$ is maximal if and only if every $\sM_x$ is maximal.

\subsection{}
From now on, fix a maximal order $\sD$ of $D$. We may now introduce $\sD$-elliptic sheaves. For any scheme $S$ over $\ka$, write $\sg$ for the absolute $q$-Frobenius on $S$, and write ${}^\sg$ for $(\id_C\times_\ka\sg)^*$. 

\begin{defn*}
Let $S$ be a scheme over $\ka$. A \emph{$\sD$-elliptic sheaf} over $S$ is a commutative diagram of sheaves on $C\times_\ka S$
\begin{align*}
\xymatrixcolsep{5pc}
\xymatrixrowsep{3pc}
\xymatrix{\dotsm\ar[r]^-{j_{i-2}}& \sE_{i-1}\ar[r]^-{j_{i-1}} &\sE_i\ar[r]^-{j_i} &\sE_{i+1}\ar[r]^-{j_{i+1}} &\dotsm\\
\dotsm\ar[r]^-{{}^\sg\!j_{i-2}} \ar[ur]^-{t_{i-2}}& {}^\sg\!\sE_{i-1}\ar[r]^-{{}^\sg\!j_{i-1}} \ar[ur]^-{t_{i-1}}&{}^\sg\!\sE_i\ar[r]^-{{}^\sg\!j_i} \ar[ur]^-{t_i}&{}^\sg\!\sE_{i+1}\ar[r]^-{{}^\sg\!j_{i+1}} \ar[ur]^-{t_{i+1}}&\dotsm
}
\end{align*}
where the $\sE_i$ are locally free right $\sD\boxtimes\sO_S$-modules of rank $1$ (and hence vector bundles over $S$ of rank $n^2$), and the $t_i$ and $j_i$ are injective morphisms of right $\sD\boxtimes\sO_S$-modules satisfying the following conditions:
\begin{enumerate}[$\bullet$]
\item for all $i$, $\sE_{i+n\deg\infty}$ is isomorphic to $\sE_i(\infty)$, and this isomorphism identifies $\sE_i\inj{}\sE_{i+n\deg\infty}$ with the canonical injection $\sE_i\inj{}\sE_i(\infty)$,
\item there exists a morphism $i_\infty:S\rar\infty$ such that $\coker j_1$ is supported on the image of its graph $\Ga_\infty:S\rar\infty\times_\ka S$,
\item there exists a morphism $i_o:S\rar C'\ssm\infty$ such that $\coker t_1$ is supported on the image of its graph $\Ga_o:S\rar C\times_\ka S$,
\item when viewed as $\sO_S$-modules, the $\coker{j_i}$ and $\coker{t_i}$ are locally free of rank $n$.
\end{enumerate}
We denote $\sD$-elliptic sheaves by $(\sE_i,t_i,j_i)_i$. We say that $i_\infty$ is the \emph{pole} of $(\sE_i,t_i,j_i)_i$, and we say that $i_o$ is the \emph{zero} of $(\sE_i,t_i,j_i)_i$.
\end{defn*}
Since the image of $i_o$ is disjoint from $\infty$, we see $t_i$ induces an isomorphism ${}^\sg\!(\sE_i/\sE_{i-1})\rar^\sim\sE_{i+1}/\sE_i$. Thus $\coker j_{i+1}$ is supported on the image of the graph of $i_\infty\circ\sg^i$. And as the image of $i_o$ is disjoint from $\infty$, we see that $\coker t_i$ is supported on the image of $\Ga_o$ for all $i$. 

\subsection{}\label{ss:modulispacedellipticsheavesnolevel}
We proceed to describe the \emph{moduli space} of $\sD$-elliptic sheaves. For any scheme $S$ over $\ka$, write $\cE\ell\ell_\sD(S)$ for the category whose objects are $\sD$-elliptic sheaves and whose morphisms are isomorphisms of $\sD$-elliptic sheaves. Then $\cE\ell\ell_\sD$ forms an fppf stack over $\ka$ \cite[(2.4)]{LRS93}, and the assignment sending a $\sD$-elliptic sheaf over $S$ to its zero $i_o:S\rar C'\ssm\infty$ yields a morphism $\cE\ell\ell_\sD\rar C'\ssm\infty$. We have the following result of Laumon--Rapoport--Stuhler.
\begin{prop*}[{\cite[(4.1)]{LRS93}}] Our $\cE\ell\ell_\sD$ is Deligne--Mumford stack, and the morphism $\cE\ell\ell_\sD\rar C'\ssm\infty$ is smooth of relative dimension $n-1$.
\end{prop*}
This morphism $\cE\ell\ell_\sD\rar C'\ssm\infty$ is the equi-characteristic analog of the structure morphism from an integral model of a Shimura variety to $\Spec\bZ$. Given that $\sD$-elliptic sheaves $(\sE_i,t_i,j_i)_i$ correspond to abelian schemes $A\rar S$ in our analogy, the zero of $(\sE_i,t_i,j_i)_i$ therefore corresponds to the characteristic of $S$.

\subsection{}
Just as with abelian varieties, we have a notion of \emph{level structure} for $\sD$-elliptic sheaves. The situation is more complicated when the characteristic divides the level, so we begin by excluding this case. Let $I$ be a finite closed subscheme of $C\ssm\infty$.
\begin{defn*}
Let $(\sE_i,t_i,j_i)_i$ be $\sD$-elliptic sheaf over $S$ such that the image of $i_o$ does not meet $I$. A \emph{level-$I$ structure} on $(\sE_i,t_i,j_i)_i$ is an isomorphism of right $\res\sD_I\boxtimes\sO_S$-modules
\begin{align*}
\io:\res\sD_I\boxtimes\sO_S\rar^\sim\res\sE_{I\times_\ka S}
\end{align*}
such that the diagram
\begin{align*}
\xymatrixcolsep{3pc}
\xymatrix{
{}^\sg\!(\res\sD_I\boxtimes\sO_S)\ar[r]^-{f_\sg}\ar[d]^-{{}^\sg\!\io} & \res\sD_I\boxtimes\sO_S\ar[d]^-\io\\
{}^\sg\!\res\sE_{I\times_\ka S}\ar[r]^-{\res{j}_{I\times_\ka S}} & \res\sE_{I\times_\ka S}
}
\end{align*}
commutes, where $f_\sg$ denotes the linearization of $\id\times_\ka\sg:\res\sD_I\boxtimes\sO_S\rar\res\sD_I\boxtimes\sO_S$. For any closed subscheme $I'$ of $I$ and level-$I$ structure $\io$, its restriction $\res\io_{I'}$ is a level-$I'$ structure.
\end{defn*}

\subsection{}\label{ss:modulispacedellipticsheaveslevel}
We now incorporate the data of level-$I$ structures in order to introduce more moduli spaces of $\sD$-elliptic sheaves. Write $\cE\ell\ell_{\sD,I}(S)$ for the category whose
\begin{enumerate}[$\bullet$]
\item objects are pairs $((\sE_i,t_i,j_i)_i,\io)$, where $(\sE_i,t_i,j_i)_i$ is a $\sD$-elliptic sheaf over $S$, and $\io$ is a level-$I$ structure on $(\sE_i,t_i,j_i)_i$,
\item morphisms are isomorphisms of pairs.
\end{enumerate}
Then $\cE\ell\ell_{\sD,I}$ forms a Deligne--Mumford stack over $\ka$ \cite[(4.1)]{LRS93}. We naturally identify $\cE\ell\ell_{\sD,\varnothing}=\cE\ell\ell_\sD$, and as in \ref{ss:modulispacedellipticsheavesnolevel}, we have a morphism $\cE\ell\ell_{\sD,I}\rar C'\ssm(I\cup\infty)$ given by sending any $\sD$-elliptic sheaf with level-$I$ structure to its zero.

By abuse of notation, we shall write $\sD\otimes\bO/\bK_I$ for the product ring $\prod_x\sD_x/\bK_{I,x}\sD_x$, where $x$ runs over all places in $C$. Since $\bK_{I,x}=\cO_x$ for cofinitely many $x$, we see that $\sD\otimes\bO/\bK_I$ is finite. As $I$ varies, the stacks $\cE\ell\ell_{\sD,I}$ are related as follows.
\begin{prop*}[{\cite[(4.1)]{LRS93}}] For any closed subscheme $I'$ of $I$, the restriction morphism $\cE\ell\ell_{\sD,I}\rar\res{\cE\ell\ell_{\sD,I'}}_{C'\ssm(I\cup\infty)}$ is finite representable Galois, and the right action of
    \begin{align*}
      \ker\left((\sD\otimes\bO/\bK_I)^\times\rar(\sD\otimes\bO/\bK_{I'})^\times\right)
    \end{align*}
on $\cE\ell\ell_{\sD,I}\rar\res{\cE\ell\ell_{\sD,I'}}_{C'\ssm(I\cup\infty)}$ given by multiplication on the level-$I$ structure yields the Galois action.
\end{prop*}
By combining this with Proposition \ref{ss:modulispacedellipticsheavesnolevel}, we see that $\cE\ell\ell_{\sD,I}\rar C'\ssm(I\cup\infty)$ is smooth of relative dimension $n-1$.

\subsection{}\label{ss:dellipticsheavestwistingbybundle}
In anticipation for studying the cohomology of $\cE\ell\ell_{\sD,I}$, we present some operations that one can perform on $\cE\ell\ell_{\sD,I}$. First, write $\Pic_I(C)$ for the group
\begin{align*}
\Pic_I(C)\deq\{\mbox{pairs }(\sL,\be)\}/\!\sim,
\end{align*}
where $\sL$ is a line bundle on $C$, and $\be$ is an isomorphism $\be:\sO_I\rar^\sim\res\sL_I$ of $\sO_I$-modules, under tensor product. As $\sD$-elliptic sheaves consist of vector bundles, we can twist them by line bundles. That is, for any $\sD$-elliptic sheaf $(\sE_i,t_i,j_i)_i$ over $S$ with level-$I$ structure $\io$, we obtain another $\sD$-elliptic sheaf over $S$ with level-$I$ structure as follows:
\begin{align*}
(\sL,\be)\cdot((\sE_i,t_i,j_i)_i,\io) \deq ((\sE_i\otimes(\sL\boxtimes\sO_S),t_i\otimes\id,j_i\otimes\id)_i,\io\otimes\be).
\end{align*}
This yields an action of $\Pic_I(C)$ on $\cE\ell\ell_{\sD,I}$ over $C'\ssm(I\cup\infty)$.

\subsection{}
As $I$ varies, the moduli spaces of $\sD$-elliptic sheaves with level-$I$ structure form a system of stacks $(\cE\ell\ell_{\sD,I})_I$, which corresponds to how Shimura varieties form systems of schemes. We package this system into one object as follows. First, let $T$ be a finite set of places of $\bf{F}$. Write $\bA^T$ for the ring of adeles away from $T$, and more generally for any subset $X$ of $\bA$, write $X^T$ for the projection of $X$ to $\bA^T$. Similarly, write $\bA_T$ for the ring of adeles at $T$, and write $X_T$ for the projection of $X$ to $\bA_T$. We write $\cE\ell\ell_\sD^\infty$ for the inverse limit
\begin{align*}
\cE\ell\ell_\sD^\infty \deq \invlim_I\cE\ell\ell_{\sD,I},
\end{align*}
where $I$ ranges over all finite closed subschemes of $C$ that do not meet $\infty$. This inverse limit yields a morphism $\cE\ell\ell_\sD^\infty\rar\Spec{\bf{F}}$. By abuse of notation, write $\bO^\infty$ for the $\sO_C$-algebra given by
\begin{align*}
U\mapsto\prod_{x'}\cO_{x'},
\end{align*}
where $x'$ runs over all closed points in $U\ssm\infty$. We see that $\cE\ell\ell_\sD^\infty(S)$ is the category whose \cite[(7.1)]{LRS93}
\begin{enumerate}[$\bullet$]
\item objects are pairs $((\sE_i,t_i,j_i)_i,\io^\infty)$, where $(\sE_i,t_i,j_i)_i$ is a $\sD$-elliptic sheaf over $S$ and $\io^\infty$ is a $(\sD\otimes_{\sO_C}\bO^\infty)\boxtimes\sO_S$-linear isomorphism
\begin{align*}
\io^\infty:(\sD\otimes_{\sO_C}\bO^\infty)\boxtimes\sO_S\rar^\sim\sE\otimes(\bO^\infty\boxtimes\sO_S)
\end{align*}
such that the diagram
\begin{align*}
\xymatrixcolsep{3pc}
\xymatrix{
{}^\sg\!((\sD\otimes_{\sO_C}\bO^\infty)\boxtimes\sO_S)\ar[r]^-{f_\sg}\ar[d]^-{{}^\sg\!\io^\infty} & (\sD\otimes_{\sO_C}\bO^\infty)\boxtimes\sO_S\ar[d]^-{\io^\infty}\\
{}^\sg\!\sE\otimes(\bO^\infty\boxtimes\sO_S)\ar[r]^-{j\boxtimes\id} & \sE\otimes(\bO^\infty\boxtimes\sO_S)
}
\end{align*}
commutes, where $f_\sg$ denotes the linearization of $\id\times_\ka\sg:(\sD\otimes_{\sO_C}\bO^\infty)\boxtimes\sO_S\rar(\sD\otimes_{\sO_C}\bO^\infty)\boxtimes\sO_S$,
\item morphisms are isomorphisms of pairs.
\end{enumerate}

\subsection{}
The action of twisting by line bundles is compatible with the inverse limit, so $\cE\ell\ell_\sD^\infty$ has an action of 
\begin{align*}
\invlim_I\Pic_I(C) = \bf{F}^\times\bs\bA^\times/\bf{F}_\infty^\times = \bf{F}^{\infty,\times}\bs\bA^{\infty,\times}
\end{align*}
over $\bf{F}$, where we identify the idele $(a_x)_x$ with the line bundle $\sO_C(-\sum_xx(a_x))$ along with the trivialization induced via multiplication by $(a_x)_x$. In particular, $\cE\ell\ell_\sD^\infty$ has an action of $\bA^{\infty,\times}$. We see that $\cE\ell\ell_\sD^\infty$ also has a right action of
\begin{align*}
\invlim_I(\sD\otimes\bO/\bK_I)^\times = (\sD\otimes\bO^\infty)^\times
\end{align*}
over $\bf{F}$ given by multiplication on the level structure, and both of these actions actually arise from a right action of $(D\otimes\bA^\infty)^\times$ on $\cE\ell\ell_\sD^\infty$ \cite[(7.3)]{LRS93}. This is the \emph{Hecke} action, analogous to the Hecke action on Shimura varieties. Writing $\sK_I$ for the subgroup
\begin{align*}
\sK_I\deq\ker\left((\sD\otimes\bO)^\times\rar(\sD\otimes\cO/\bK_I)^\times\right),
\end{align*}
we see that Proposition \ref{ss:modulispacedellipticsheaveslevel} implies $\cE\ell\ell_\sD^\infty/\sK_I^\infty=\cE\ell\ell_{\sD,I}$. 

\subsection{}
Since $\sD$-elliptic sheaves are essentially vector bundles, they have a notion of \emph{degree}.
\begin{defn*}
Let $(\sE_i,t_i,j_i)_i$ be a $\sD$-elliptic sheaf. We say its \emph{degree} is the locally constant function
\begin{align*}
\frac{\deg\sE_1-\deg\sD}n
\end{align*}
on the base scheme $S$, where we use $1$ instead of $0$ to avoid confusion with $o$. One can show that the degree is integer-valued \cite[p.~49]{Laf97}.
\end{defn*}
By partitioning the base scheme $S$ by the degree of $(\sE_i,t_i,j_i)_i$, we can write $\cE\ell\ell_{\sD,I}$ as a disjoint union
\begin{align*}
\cE\ell\ell_{\sD,I} = \coprod_d\cE\ell\ell_{\sD,I,d},
\end{align*}
where $d$ ranges over all integers, and $\cE\ell\ell_{\sD,I,d}$ is the open substack of $\sD$-elliptic sheaves with level-$I$ structure whose degree equals $d$. Note that the restriction morphisms $\cE\ell\ell_{\sD,I}\rar\cE\ell\ell_{\sD,I'}$ preserve degree, so we obtain a similar decomposition
\begin{align*}
\cE\ell\ell_\sD^\infty = \coprod_d\cE\ell\ell_{\sD,d}^\infty,
\end{align*}
where $\cE\ell\ell_{\sD,d}^\infty=\invlim_I\cE\ell\ell_{\sD,I,d}$ for $I$ running over all finite closed subschemes of $C\ssm\infty$.

\subsection{}\label{ss:dellipticsheavestranslation}
We shall define an action of $\bZ$ on $\cE\ell\ell_{\sD,I}$ over $C'\ssm(I\cup\infty)$ as follows. For any integer $l$ and $\sD$-elliptic sheaf $(\sE_i,t_i,j_i)_i$, write
\begin{align*}
[l](\sE_i,t_i,j_i)_i\deq(\sE_{i+l},t_{i+l},j_{i+l})_i,
\end{align*}
that is, $[l]$ acts via translation by $l$ on the index $i$. Then $[l](\sE_i,t_i,j_i)_i$ also forms a $\sD$-elliptic sheaf, so this yields an action of $\bZ$ on $\cE\ell\ell_{\sD,I}$. The definition of a $\sD$-elliptic sheaf implies that $\deg\sE_{i+l}=\deg\sE_i+nl$, so the degree of $[l](\cE_i,t_i,j_i)_i$ is $l$ plus that of $(\cE_i,t_i,j_i)_i$. This allows us to identify the quotient stack $\cE\ell\ell_{\sD,I}/\bZ$ with
\begin{align*}
\cE\ell\ell_{\sD,I}/\bZ = \cE\ell\ell_{\sD,I,0}
\end{align*}
as stacks over $\ka$. The translation action of $\bZ$ commutes with restriction morphisms as well as the Hecke action, so we obtain a similar description of $\cE\ell\ell_\sD^\infty/\bZ$. The Hecke action also descends to $\cE\ell\ell_{\sD,I}/\bZ$. By passing to this quotient stack (or equivalently, by restricting the degree), we obtain the following finite-type representability result.
\begin{prop*}[{\cite[(6.2)]{LRS93}}]
Suppose $I$ is nonempty. Then $\cE\ell\ell_{\sD,I}/\bZ$ is actually a projective scheme over $C'\ssm(I\cup\infty)$.
\end{prop*}
Thus $\cE\ell\ell_\sD^\infty/\bZ=\invlim_I\cE\ell\ell_{\sD,I}/\bZ$ is also representable by a scheme.

\subsection{}
Before we proceed to the cohomology of our moduli spaces, we introduce a covering of $\cE\ell\ell_{\sD,I}$ which we shall use to construct our coefficient sheaves. Write $\ov{B}$ for the central division algebra over $\bf{F}_\infty$ of Hasse invariant $-\frac1n$.
\begin{prop*}[{\cite[(8.11)]{LRS93}}]
There exists a pro-Galois covering
\begin{align*}
\wt{\cE\ell\ell}_{\sD,I}\rar\cE\ell\ell_{\sD,I},
\end{align*}
where $\wt{\cE\ell\ell}_{\sD,I}$ is a scheme, whose Galois group is given by a right action of $\ov{B}^\times/\vpi_\infty^\bZ$. This covering is compatible with restriction morphisms $\cE\ell\ell_{\sD,I}\rar\cE\ell\ell_{\sD,I'}$, so taking the inverse limit yields an analogous pro-Galois covering
\begin{align*}
\wt{\cE\ell\ell}_\sD^\infty\rar\cE\ell\ell_\sD^\infty
\end{align*}
whose Galois group is given by a right action of $\ov{B}^\times/\vpi_\infty^\bZ$.
\end{prop*}
Briefly, the covering $\wt{\cE\ell\ell}_{\sD,I}\rar\cE\ell\ell_{\sD,I}$ is given as follows. Given a $\sD$-elliptic sheaf $(\sE_i,t_i,j_i)_i$, we can construct an object resembling a local shtuka at the place $\infty$ (in a fashion similar to the construction to be given in \ref{ss:dellipticsheavestolocalshtuka}). This ``local shtuka at $\infty$'' will always be isomorphic to a fixed object analogous to the local shtuka of slope $-\frac1n$ under the Dieudonn\'e--Manin classification \cite[(2.4.5)]{Lau96}, and the covering $\wt{\cE\ell\ell}_{\sD,I}\rar\cE\ell\ell_{\sD,I}$ parametrizes isomorphisms between our ``local shtuka at $\infty$'' and this fixed object. The Galois action is given by composition with this isomorphism, as the endomorphism ring of a local shtuka of slope $-\frac1n$ at $\infty$ is $\ov{B}$.

The stack $\wt{\cE\ell\ell}_{\sD,I}$ also has an action of $\bZ$ by translation, and it is preserved under the morphism $\wt{\cE\ell\ell}_{\sD,I}\rar\cE\ell\ell_{\sD,I}$. Thus we obtain a pro-Galois morphism $\wt{\cE\ell\ell}_{\sD,I}/\bZ\rar \cE\ell\ell_{\sD,I}/\bZ$ whose Galois group is given by a right action of $\ov{B}^\times/\vpi_\infty^\bZ$. This yields a similar pro-Galois morphism $\wt{\cE\ell\ell}_\sD^\infty/\bZ\rar \cE\ell\ell_\sD^\infty/\bZ$ as well.

The space $\wt{\cE\ell\ell}_\sD^\infty$ is the equi-characteristic analog of the Hermitian symmetric space covering a Shimura variety.

\subsection{}\label{ss:dellipticsheavescohomology}
At this point, we can finally introduce the cohomology of our moduli spaces. Fix a separable closure $\bf{F}^\sep$ of $\bf{F}$, and view all separable extensions of $\bf{F}$ as lying in $\bf{F}^\sep$. For every place $x$ of $C$, choose a separable closure $\bf{F}^\sep_x$ of $\bf{F}_x$, fix an embedding $\bf{F}^\sep\inj{}\bf{F}^\sep_x$, and form the absolute Galois groups $G_\bf{F}$ and $G_x$ of $\bf{F}$ and $\bf{F}_x$, respectively, with respect to these separable closures. Write $\bC_x$ for the completion of $\bf{F}^\sep_x$. We view $G_x$ as a subgroup of $G_\bf{F}$ via our embedding, and we denote the arithmetic $q_x$-Frobenius at $x$ using $\sg_x$.

Fix an irreducible smooth representation
\begin{align*}
\xi:\ov{B}^\times/\vpi_\infty^\bZ\rar\GL_N(\ov\bQ_\ell).
\end{align*}
Since $\ov{B}^\times/\vpi_\infty^\bZ$ is compact, we see $\xi$ has finite image and is therefore defined over a finite extension of $\bQ_\ell$ in $\ov\bQ_\ell$. Write $\cL_{\xi,I}$ for the $\ell$-adic sheaf on $\cE\ell\ell_{\sD,I}/\bZ$ induced from $\xi$ and the pro-Galois morphism $\wt{\cE\ell\ell}_{\sD,I}/\bZ\rar \cE\ell\ell_{\sD,I}/\bZ$ via monodromy, and write $\cL_{\xi}$ for the analogous $\ell$-adic sheaf on $\cE\ell\ell^\infty_\sD/\bZ$. Form the $\ov\bQ_\ell$-vector space
\begin{align*}
H^i_{\xi,\eta} \deq H^i(\cE\ell\ell^\infty_{\sD,\bf{F}^\sep}/\bZ,\cL_{\xi}) = \dirlim_IH^i(\cE\ell\ell_{\sD,I,\bf{F}^\sep}/\bZ,\cL_{\xi,I,\bf{F}^\sep}),
\end{align*}
where $I$ runs over all finite closed subschemes of $C\ssm\infty$. This has a left action of $(D\otimes\bA^\infty)^\times\times G_\bf{F}$, and it vanishes for sufficiently large $i$. Proposition \ref{ss:modulispacedellipticsheaveslevel} implies that 
\begin{align*}
(H^i_{\xi,\eta})^{\sK_I^\infty} = H^i(\cE\ell\ell_{\sD,I,\bf{F}^\sep}/\bZ,\cL_{\xi,I,\bf{F}^\sep}),
\end{align*}
which shows that $H^i_{\xi,\eta}$ is an admissible/continuous representation as in \cite[p.~24]{HT01} of $(D\otimes\bA^\infty)^\times\times G_\bf{F}$ over $\ov\bQ_\ell$. Furthermore, for any $g$ in $(D\otimes\bA^\infty)^\times$, the action of
\begin{align*}
\frac1{\vol(\sK_I^\infty)}\bf1_{\sK_I^\infty g\sK_I^\infty}\in C^\infty_c((D\otimes\bA^\infty)^\times)
\end{align*}
on $H^i_{\xi,\eta}$ is induced by a correspondence \cite[(7.5)]{LRS93}
\begin{align*}
  \xymatrix{& \ar[ld]_-{c_{1,\bf{F}}}\cE\ell\ell_{\sD,\bf{F}}^\infty/(\bZ\times(\sK_I^\infty\cap g^{-1}\sK_I^\infty g))\ar[rd]^-{c_{2,\bf{F}}} & \\
\cE\ell\ell_{\sD,\bf{F}}^\infty/(\bZ\times\sK_I^\infty) & & \cE\ell\ell_{\sD,\bf{F}}^\infty/(\bZ\times\sK_I^\infty)
}
\end{align*}
over $\bf{F}$, where $c_{1,\bf{F}}$ is induced by further quotienting by $\sK_I^\infty$, and $c_{2,\bf{F}}$ is induced by the right action of $g^{-1}$ followed by quotienting by $\sK_I^\infty$. This is the equi-characteristic analog of \emph{Hecke correspondences}.

Proposition \ref{ss:modulispacedellipticsheaveslevel} implies that $c_{1,\bf{F}}$ and $c_{2,\bf{F}}$ are finite \'etale. We write $c_\bf{F}:\cE\ell\ell_{\sD,I,\bf{F}}/\bZ\dashrightarrow\cE\ell\ell_{\sD,I,\bf{F}}/\bZ$ for the correspondence formed by $c_{1,\bf{F}}$ and $c_{2,\bf{F}}$, and we write $[H_\xi]$ for the virtual representation $[H_\xi]\deq\sum_{i=0}^\infty(-1)^iH^i_{\xi,\eta}$.

\subsection{}\label{ss:dellipticsheavestolocalshtukazero}
Now we introduce a construction that takes any $\sD$-elliptic sheaf over a certain base and yields an object resembling a local shtuka. Let $o$ be a closed point in $C'\ssm\infty$, let $S$ be a scheme over $\Spec\cO_o$, and let $(\sE_i,t_i,j_i)_i$ be a $\sD$-elliptic sheaf over $S$. 

Let $\Ga_o$ be the image of the graph of $i_o:S\rar C\times_\ka S$, and denote completions along $\Ga_o$ using $(-)^\wedge_{\Ga_o}$. Then $(\sE_1)^\wedge_{\Ga_o}$ is a vector bundle on $(C\times_\ka S)^\wedge_{\Ga_o}$ of rank $n^2$ with a right action of $\sD_o=\M_n(\cO_o)$, and $(j_1)^\wedge_{\Ga_o}$ is an isomorphism because $\Ga_o$ lies away from $\infty$. Thus we may form the composed morphism
\begin{align*}
(j_1)^{\wedge,-1}_{\Ga_o}\circ(t_1)^\wedge_{\Ga_o}:({}^\sg\!\sE_1)^\wedge_{\Ga_o}\rar(\sE_1)^\wedge_{\Ga_o}.
\end{align*}
We see that the adic pullback of $(j_1)^{\wedge,-1}_{\Ga_o}\circ(t_1)^\wedge_{\Ga_o}$ as in \cite[p.~370]{Ber96} to $\bf{F}_o$ is an isomorphism, and its cokernel is a vector bundle on $S$ of rank $n$.

We may identify $(C\times_\ka S)^\wedge_{\Ga_o}$ with the locally ringed space whose support is $\abs{S}$ and whose structure sheaf is $\sO_S\llb{\vpi_o}$ \cite[Lemma 5.3]{EH14}. From this point of view, the pair
\begin{align*}
((\sE_1)^\wedge_{\Ga_o},(j_1)^{\wedge,-1}_{\Ga_o}\circ(t_1)^\wedge_{\Ga_o})
\end{align*}
corresponds to a pair $(\sM_o,\sF_o)$, where $\sM_o$ is a locally free $\sO_S\llb{\vpi_o}$-module of rank $n^2$ with a right $\sO_S\llb{\vpi_o}$-linear action of $\M_n(\cO_o)$, and $\sF_o:\sg^*_o\sM_o\rar\sM_o$ is an $\sO_S\llb{\vpi_o}$-linear morphism such that $\coker\sF_o$ is a locally free $\sO_S$-module of rank $n$. Note that $\sF_o[\frac1{\vpi_o}]$ is an isomorphism.

Now $\sM_o$ is a free right $\M_n(\cO_o)$-module of rank $1$. Applying Morita equivalence to the right $\M_n(\cO_o)$-action shows that there exists pair $(\sM_o',\sF_o')$ satisfying 
\begin{align*}
(\sM_o,\sF_o) = (\sM_o'^{\oplus n},\sF_o'^{\oplus n}),
\end{align*}
where $\sM_o'$ is a locally free $\sO_S\llb{\vpi_o}$-module of rank $n$, and $\sF_o':\sg^*_o\sM_o'\rar\sM_o'$ is an $\sO_S\llb{\vpi_o}$-linear morphism whose cokernel is a locally free $\sO_S$-module of rank $1$ and whose generic fiber is an isomorphism. Note that if $S$ is actually a scheme over $\Spf\cO_o$, then $(\sM_o',\sF_o')$ actually forms an effective minuscule local shtuka over $S$ of rank $n$ and dimension $1$.

\subsection{}
With the above construction, we can treat level structures in the case when characteristics divide the level. As with abelian varieties, this requires a notion of \emph{Drinfeld level structure}. 

Begin by recalling some material from \S\ref{s:deformationspaces}. We have a notion of finite $\ka_o$-shtukas, which are truncated versions of local shtukas, and they correspond to certain module schemes called strict $\ka_o$-modules under a Dieudonn\'e equivalence. Now the quotient module $\sM'_o/\vpi_o^m$ is a finite $\ka_o$-shtuka, and we denote the corresponding finite $\ka_o$-strict module using $\Dr(\sM'_o/\vpi_o^m)$. As with local shtukas, we define Drinfeld level structures via using $\Dr(-)$ to pass to module schemes.
\begin{defn*}
Let $S$ be a scheme over $\Spec\cO_o$, and let $(\sE_i,t_i,j_i)_i$ be a $\sD$-elliptic sheaf over $S$. We say a \emph{Drinfeld level-$m$ structure} on $(\sE_i,t_i,j_i)_i$ is a Drinfeld level-$m$ structure on $\Dr(\sM'_o/\vpi_o^m)$, that is, an $\cO_o/\vpi_o^m$-module morphism
\begin{align*}
\al:(\vpi^{-m}\cO_o/\cO_o)^n\rar\Dr(\sM'_o/\vpi^m_o)(S)
\end{align*}
such that the collection of all $\al(x)$ for $x$ in $(\vpi^{-m}\cO_o/\cO_o)^n$ forms a full set of sections of $\Dr(\sM'_o/\vpi^m_o)$ as in \cite[(1.8.2)]{KM85}. For any non-negative integer $m'\leq m$ and Drinfeld level-$m$ structure $\al$, its restriction to $(\vpi^{-m'}_o\cO_o/\cO_o)^n$ is a Drinfeld level-$m'$ structure.
\end{defn*}

\subsection{}
We may finally define moduli spaces of $\sD$-elliptic sheaves at bad reduction. For any scheme $S$ over $\Spec\cO_o$, write $\cM_{I}^*(S)$ for the category whose
\begin{enumerate}[$\bullet$]
\item objects are triples $((\sE_i,t_i,j_i)_i,\io,\al)$, where $(\sE_i,t_i,j_i)_i$ is a $\sD$-elliptic sheaf over $S$, $\io$ is a level $(I\ssm o)$-structure on $(\sE_i,t_i,j_i)_i$, and $\al$ is a Drinfeld level-$m$ structure on $(\sE_i,t_i,j_i)_i$, where $m$ is the multiplicity of $o$ in $I$,
\item morphisms are isomorphisms of pairs.
\end{enumerate}
Then $\cM_{I}^*$ forms an fppf stack over $\ka$, and when $I$ does not contain $o$, we see that $\cM_{I}^*$ is simply the pullback of $\cE\ell\ell_{\sD,I}$ to $\Spec\cO_o$. 

By sending a $\sD$-elliptic sheaf to its zero, we obtain a morphism $\cM_{I}^*\rar\Spec\cO_o$. For any closed subscheme $I'$ of $I$, restriction of level structures gives us a morphism $\cM_{I}^*\rar\cM_{I'}^*$ over $\Spec\cO_o$, and partitioning by degrees allows us to obtain a decomposition
\begin{align*}
\cM_{I}^* = \coprod_d\cM_{I,d}^*
\end{align*}
that respects restriction morphisms, where $d$ runs over all integers.

We see that the translation action of $\bZ$ from \ref{ss:dellipticsheavestranslation} naturally extends to an action on $\cM_I^*$, and we write $\cM_{I}$ for the quotient stack $\cM_{I}\deq\cM_I/\bZ$. We have the following representability result, which follows from work of Katz--Mazur.
\begin{prop*}[{\cite[(1.10.13)]{KM85}}]
The restriction morphism $\cM_{I}^*\rar\cM_{I'}^*$ is finite and representable.
\end{prop*}
By combining this with Proposition \ref{ss:dellipticsheavestranslation}, we see that $\cM_{I}$ is a projective scheme of dimension $n-1$ over $\Spec\cO_o$ whenever $I\ssm o$ is nonempty. The scheme $\cM_{I}$ shall be our focus in the next few sections.

\subsection{}\label{ss:badreductionheckecorrespondence}
In this subsection, we extend the Hecke action to $\cM_{I}$. We use $o^m$ to denote the finite closed subscheme of $C$ supported on $o$ with multiplicity $m$. Since Drinfeld level-$m$ structures over $\bf{F}_o$ are equivalent to level-$o^m$ structures over $\bf{F}_o$ \cite[Proposition 7.1.3]{Boy99}, the generic fiber $\cM_{I,\bf{F}_o}$ is isomorphic to the pullback $\cE\ell\ell_{\sD,I,\bf{F}_o}/\bZ$. Under this identification, the right action of $(D\otimes\bA^\infty)^\times$ on $\cE\ell\ell_{\sD,\bf{F}_o}^\infty/a^\bZ$ extends \cite[p.~599]{Boy99} to a right action on the inverse limit
\begin{align*}
\cM^\infty \deq \invlim_I\cM_{I},
\end{align*}
where $I$ runs over all finite closed subschemes of $C\ssm\infty$. This action satisfies $\cM^\infty/\sK_I^\infty=\cM_{I}$. Furthermore, for any $g$ in $D\otimes\bA^\infty$, the correspondence from \ref{ss:dellipticsheavescohomology} extends to a correspondence \cite[p.~600]{Boy99}
\begin{align*}
\xymatrix{& \ar[ld]_-{c_{1,\cO_o}}\cM^\infty/(\sK_I^\infty\cap g^{-1}\sK_I^\infty g)\ar[rd]^-{c_{2,\cO_o}} & \\
\cM^\infty/\sK_I^\infty & & \cM^\infty/\sK_I^\infty
}
\end{align*}
where $c_{1,\cO_o}$ and $c_{2,\cO_o}$ are finite morphisms over $\Spec\cO_o$. We write $c_{\cO_o}:\cM_{I}\dashrightarrow\cM_{I}^\infty$ for the correspondence formed by $c_{1,\cO_o}$ and $c_{2,\cO_o}$.

\subsection{}\label{ss:boyerserretate}
Next, we present the \emph{Serre--Tate theorem} for $\sD$-elliptic sheaves, which relates deformations of $\sD$-elliptic sheaves with deformations of their associated local shtukas. Let $m$ be the multiplicity of $o$ in $I$, let $z$ be a $\ov\ka$-point of $\cM_{I}$, and let $(\sE_i,t_i,j_i)_i$ be the corresponding $\sD$-elliptic sheaf over $\ov\ka$ with level $(I\ssm o)$-structure $\io$ and Drinfeld level-$m$ structure $\al$. Then $\al$ yields a Drinfeld level-$m$ structure for the local shtuka $\sM'_o$ formed in \ref{ss:dellipticsheavestolocalshtukazero}. Recall from \ref{ss:effectiveminusculelocalshtukaconjugacy} that $\sM'_o$ corresponds to some $\de_o$ in
\begin{align*}
\GL_n(\br\cO_o)\diag(\vpi_o,1,\dotsc,1)\GL_n(\br\cO_o)
\end{align*}
uniquely up to $\GL_n(\br\cO_o)$-$\sg_o$-conjugacy, and recall that the deformation space of $(\sM'_o,\al)$ is $\br\fX_{\de_o,\al}$. Write $(\cM_{I,\br\cO_o})_z^\wedge$ for the completion of $\cM_{I,\br\cO_o}$ at $z$. Boyer proves the following comparison result for deformation spaces.
\begin{prop*}[{\cite[Theorem 7.4.4]{Boy99}}]
Our $((\sE_i,t_i,j_i)_i,\io,\al)\mapsto(\sM'_o,\al)$ induces an isomorphism $(\cM_{I,\br\cO_o})_z^\wedge\rar^\sim\br\fX_{\de_o,\al}$ over $\Spf\br\cO_o$ that preserves the right action of $\GL_n(\cO_o)$ as well as the Weil action.
\end{prop*}

\subsection{}\label{ss:extendedserretate}
We extend the above results to $\br\fX_{\de_o,m}$, where $\br\fX_{\de_o,m}$ is the deformation space of $\sM'_o$ along with a Drinfeld level-$m$ structure as in \ref{ss:deformationalgebraicclosure}. Write $\pi:\cM_{I,\br\cO_o}\rar\cM_{I\ssm o,\br\cO_o}$ for the restriction morphism. We see that $\pi^{-1}(\pi(z))$ corresponds to all possible Drinfeld level-$m$ structures on $(\sE_i,t_i,j_i)_i$. Therefore the isomorphism in Proposition \ref{ss:boyerserretate} also induces an isomorphism $(\cM_{I,\br\cO_o})^\wedge_{\pi^{-1}(\pi(z))}\rar^\sim\br\fX_{\de_o,m}$ over $\Spf\br\cO_o$, where $(\cM_{I,\br\cO_o})^\wedge_{\pi^{-1}(\pi(z))}$ is the completion of $\cM_{I,\br\cO_o}$ along $\pi^{-1}(\pi(z))$, and this isomorphism preserves the right action of $\GL_n(\cO_o)$ as well as the Weil action.

\subsection{}\label{ss:dellipticsheavestolocalshtuka}
We conclude this section by discussing a variant of the construction from \ref{ss:dellipticsheavestolocalshtukazero}. Let $S$ be a scheme over $\Spec\cO_o$, let $(\sE_i,t_i,j_i)_i$ be a $\sD$-elliptic sheaf over $S$, let $x$ be a closed point in $C\ssm o$, and let $f:S\rar C$ be a morphism over $\ka$ whose image is $x$.

Write $\Ga_f$ for the image of the graph of $f$ in $C\times_\ka S$, and denote completions along $\Ga_f$ using $(-)^\wedge_{\Ga_f}$. Then $(\sE_1)_{\Ga_f}^\wedge$ is a vector bundle on $(C\times_\ka S)_{\Ga_f}^\wedge$ of rank $n^2$ with a right action of $\sD_x$. The adic generic fibers $(j_1)^\wedge_{\Ga_f,\bf{F}_x}$ and $(t_1)^\wedge_{\Ga_f,\bf{F}_x}$ as in \cite[p.~370]{Ber96} are isomorphisms, so the composed morphism
\begin{align*}
(j_1)^{\wedge,-1}_{\Ga_f,\bf{F}_x}\circ(t_1)^\wedge_{\Ga_f,\bf{F}_x}:({}^\sg\!\sE_1)^\wedge_{\Ga_f,\bf{F}_x}\rar^\sim(\sE_1)^\wedge_{\Ga_f,\bf{F}_x}
\end{align*}
is also an isomorphism. Similarly to \ref{ss:dellipticsheavestolocalshtukazero}, we may identify $(C\times_\ka S)^\wedge_{\Ga_f}$ with the locally ringed space whose support is $\abs{S}$ and whose structure sheaf is $\sO_S\llb{\vpi_x}$ \cite[Lemma 5.3]{EH14}. From this point of view, the pair
\begin{align*}
((\sE_1)_{\Ga_f}^\wedge,(j_1)^{\wedge,-1}_{\Ga_f,\bf{F}_x}\circ(t_1)^\wedge_{\Ga_f,\bf{F}_x})
\end{align*}
corresponds to a pair $(\sM_x,\sF_x)$, where $\sM_x$ is a locally free $\sO_S\llb{\vpi_x}$-module of rank $n^2$ with a right $\sO_S\llb{\vpi_x}$-linear action of $\sD_x$, and $\sF_x:\sg^*_x\sM_x[\frac1{\vpi_x}]\rar^\sim\sM_x[\textstyle\frac1{\vpi_x}]$ is an $\sO_S\llp{\vpi_x}$-linear isomorphism that commutes with the right $\sD_x$-action. Note that $\sM_x$ is a free right $\sD_x$-module of rank $1$.

When $x$ does not equal $\infty$, the graph $\Ga_f$ lies away from both $o\times_\ka S$ and $\infty\times_\ka S$, so the completions $(t_1)_{\Ga_f}^\wedge$ and $(j_1)_{\Ga_f}^\wedge$ are isomorphisms. In this case, the morphism $\sF_x$ is the localization of an $\sO_S\llb{\vpi_x}$-module isomorphism $\sg^*_x\sM_x\rar^\sim\sM_x$, which we also refer to using $\sF_x$ by abuse of notation.

Finally, in the case when $x$ equals $\infty$, we have $\sD_\infty=\M_n(\cO_\infty)$ because $D$ splits at $\infty$. Applying Morita equivalence to the right $\M_n(\cO_\infty)$-action on $\sM_\infty$ provides a pair $(\sM_\infty',\sF_\infty')$ that satisfies
\begin{align*}
(\sM_\infty,\sF_\infty) = (\sM_\infty'^{\oplus n},\sF_\infty'^{\oplus n}),
\end{align*}
where $\sM_\infty'$ is a locally free $\sO_S\llb{\vpi_\infty}$-module of rank $n$, and $\sF_\infty':\sg_\infty^*\sM_\infty'[\frac1{\vpi_\infty}]\rar^\sim\sM_\infty'[\textstyle\frac1{\vpi_\infty}]$ is an $\sO_S\llp{\vpi_\infty}$-linear isomorphism. Furthermore, if $S$ is the spectrum of $\ov\ka$, then the isogeny class of $\sM_\infty'$ has slope $-\frac1n$ \cite[(9.8).(i)]{LRS93} under the Dieudonn\'e--Manin classification \cite[(2.4.5)]{Lau96}.

\section{A nearby cycles calculation and semisimple trace}\label{s:nearbycycles}
Write $\bf{F}_{o,r}$ for the $r$-th degree unramified extension of $\bf{F}_o$, and recall our test function $\phi_{\tau,h}$ in $C^\infty_c(\GL_n(\bf{F}_{o,r}))$ as in \S\ref{s:deformationspaces}. Our goal in this section is to express the integral
\begin{align*}
\int_{\sg_o^{-r}I_{\bf{F}_o}}\!\dif\tau\,\phi_{\tau,h}
\end{align*}
in terms of representations of $\GL_n(\cO_o)$. We shall use this description in \S\ref{s:secondinductivelemma} to show that the preimage of any unramified representation under $\pi\mapsto\rec\pi$ remains unramified, which plays an important role in our proof that $\rec$ is bijective.

To begin, we present a geometric calculation of the inertia invariants of nearby cycles, due to Scholze. This calculation relies on a case of \emph{Grothendieck's purity conjecture} as proved by Thomason, which describes relative cohomology in the \'etale setting. By using results from \S\ref{s:deformationspaces}, we verify that the hypotheses of this calculation apply to $\cM_I$ as well as explicitly compute the inertia invariants in terms of representations of $\GL_n(\cO_o/\vpi_o^m)$, where $m$ denotes the multiplicity of $o$ in $I$. We then relate nearby cycles to \emph{semisimple trace}, which is a better-behaved variant of Frobenius traces for ramified representations of $W_{\bf{F}_o}$. Finally, we conclude by applying results from \S\ref{s:deformationspaces} once more to describe our integral in terms of semisimple traces.

\subsection{}\label{ss:inertiainvariantspullback}
We start by introducing notation for the calculation of nearby cycles. Let $\cX$ be a scheme over $\cO_o$ of finite type, and write
\begin{align*}
\ov\j:\cX_{\br{\bf{F}}_o}\rar\cX_{\br\cO_o}\mbox{ and }\ov\i:\cX_{\ov\ka}\rar\cX_{\br\cO_o}
\end{align*}
for the inclusion morphisms of the generic and special fibers of $\cX_{\br\cO_o}$, respectively. Denote the derived $I_{\bf{F}_o}$-invariants functor by $R_{I_{\bf{F}_o}}$, and denote the derived nearby cycles functor on $\cX_{\br\cO_o}$ by $R\Psi_{\cX_{\br{\bf{F}}_o}}$. The Galois description of \'etale sheaves implies that
\begin{align*}
(R_{I_{\bf{F}_o}}R\Psi_{\cX_{\br{\bf{F}}_o}})(-) = \ov\i^*R\ov\j_*((-)_{\br{\bf{F}}_o})
\end{align*}
as functors from $D^b_c(\cX_{\bf{F}_o},\ov\bQ_\ell)$ to $D^b_c(\cX_{\ka_o},\ov\bQ_\ell)$, since the non-derived versions of both sides are equal.

\subsection{}\label{ss:stratificationcomplex}
Next, we use the combinatorics of the geometry of $\cX$ to construct a certain family of $\ov\bQ_\ell$-vector spaces involved in the calculation of nearby cycles. Assume that $\cX$ is regular, and suppose that the morphism $\cX\rar\Spec\cO_o$ is flat of relative dimension $d$. Suppose we have a \emph{stratification} of $\cX_{\ka_o}$, i.e. assume that $\cX_{\ka_o}$ can be written as
\begin{align*}
\cX_{\ka_o}=\bigcup_j\mathring\cZ_j,
\end{align*}
where $j$ ranges over some finite indexing set $\bJ$, and the $\mathring\cZ_j$ are disjoint irreducible locally closed subsets of $\cX_{\ka_o}$ whose closures $\cZ_j\deq\ov{\mathring\cZ_j}$ are regular and equal to unions of $\mathring\cZ_k$ for some $k$. For any $j$ in $\bJ$, write $c(j)$ for the codimension of $\cZ_j$ in $\cX$, which is positive because $\cZ_j$ lies in $\cX_{\ka_o}$. We write $k\succ j$ if $\cZ_k$ strictly contains $\cZ_j$.

By inducting on $c(j)$, we shall assign a finite-dimensional $\ov\bQ_\ell$-vector space $W_j$ to any $j$ in $\bJ$ as follows. In the $c(j)=1$ case, set $W_j\deq\ov\bQ_\ell$. In the $c(j)=2$ case, we let
\begin{align*}
W_j\deq \ker\left(\bigoplus_kW_k\rar\ov\bQ_\ell\right),
\end{align*}
where $k$ runs over elements of $\bJ$ satisfying $k\succ j$ and $c(k)=1$, and the maps in the direct sum are the identity morphisms. Thus $W_j$ is just the kernel of the summation map. Note that we have a morphism $W_j\rar W_k$ given by
\begin{align*}
W_j\inj{}\bigoplus_kW_k\surj{}W_k
\end{align*}
for all $k\succ j$ with $c(k)=1$. Finally, in the $c(j)\geq3$ case, set
\begin{align*}
W_j\deq \ker\left(\bigoplus_kW_k\rar\bigoplus_lW_l\right),
\end{align*}
where now $k$ runs over all elements of $\bJ$ satisfying $k\succ j$ and $c(k)=c(j)-1$, $l$ runs over all elements of $\bJ$ satisfying $l\succ j$ and $c(l)=c(j)-2$, and the maps in the direct sum are of the form $W_k\rar W_l$ for $l\succ k$. Note that we have a morphism $W_j\rar W_k$ as before via inclusion and projection, so we may indeed inductively continue this construction.

Let $G$ be a group, and suppose it acts on $\cX_{\ka_o}$ in a manner preserving the stratification. Then the stabilizer of $\cZ_j$ in $G$ acts from the left on $W_j$ via permuting the $\cZ_k$ and hence $W_k$ for which $k\succ j$.

\subsection{}\label{ss:scholzenearbycycles}
We now present Scholze's calculation of the inertia invariants of nearby cycles. Let $j$ be in $\bJ$. Consider the chain complex
\begin{align*}
0\rar W_j\rar\bigoplus_{k_1} W_{k_1}\rar\bigoplus_{k_2}W_{k_2}\rar\dotsb\rar\bigoplus_{k_{c(j)-1}}W_{k_{c(j)-1}}\rar\ov\bQ_\ell\rar0,
\end{align*}
where $k_s$ runs over elements of $\bJ$ satisfying $k_s\succ j$ and $c(k_s)=c(j)-s$, and the maps in the direct sum are of the form $W_{k_s}\rar W_{k_{s+1}}$ for $k_{s+1}\succ k_s$. Note that this complex is exact at $W_j$ by construction. For any $\ka_o$-point $z$ of $\cX$, write $\ov{z}$ for the corresponding $\ov\ka$-point of $\cX$. 
\begin{prop*}[{\cite[Theorem 5.3]{Sch13b}}]
If the above chain complex is exact for all $j$ in $\bJ$, then we have a canonical isomorphism
\begin{align*}
(\ov\i^*R^i\ov\j_*\ov\bQ_\ell)_{\ov{z}}\rar^\sim\bigoplus_kW_k(-i)
\end{align*}
for all $\ka_o$-points $z$ of $\cX$ and non-negative integers $i$, where $k$ ranges over all $k$ in $\bJ$ such that $z$ lies in $\cZ_k$ and $c(k)=i$.
\end{prop*}
The proof of this proposition uses a case of \emph{Grothendieck's purity conjecture} as proved by Thomason \cite[Corollary 3.9]{Tho84}, which says that relative \'etale cohomology is concentrated in the expected degree according to codimension.

\subsection{}
In order to apply Proposition \ref{ss:scholzenearbycycles}, we must first describe a stratification of $\cM_{I,\ka_o}$ as in \ref{ss:stratificationcomplex}. This stratification will be defined in terms of Drinfeld level-$m$ structures, and it is called the \emph{Newton stratification}. 

Let $V$ be an $\cO_o/\vpi_o^m$-linear direct summand of $(\vpi_o^{-m}\cO_o/\cO_o)^n$, and write $\cM_I^V$ for the subfunctor of $\cM_I$ consisting of triples $((\sE_i,t_i,j_i)_i,\io,\al)$ that satisfy
\begin{align*}
\sum_v[\al(v)] = \#V\cdot[0],
\end{align*}
where $v$ runs over all elements of $V$, $[0]$ denotes the zero section of $\Dr(\sM'_o/\vpi_o^m)$, the sum is taken as closed subschemes of $\Dr(\sM'_o/\vpi_o^m)$, and $\#V\cdot[0]$ denotes $\#V$-fold sum of $[0]$ as a closed subscheme of $\Dr(\sM'_o/\vpi_o^m)$. Here, $\sM'_o$ denotes the local shtuka of $(\sE_i,t_i,j_i)_i$ at $o$ as in \ref{ss:dellipticsheavestolocalshtukazero}, and $\Dr(\sM'_o/\vpi^m_o)$ denotes the strict $\ka_o$-module associated with the finite $\ka_o$-shtuka $\sM'_o/\vpi^m_o$ under the Dieudonn\'e equivalence \cite[Theorem 5.2]{HS16}.

\begin{prop}\label{ss:initialstratification}\hfill
\begin{enumerate}[(i)]
\item The subfunctor $\cM_I^V$ is a regular closed subscheme of $\cM_I$ that is equidimensional of dimension $n-\rk_{\cO_o/\vpi^m_o}(V)$.
\item The special fiber $\cM_{I,\ka_o}$ equals $\bigcup_V\cM_I^V$, where $V$ ranges over all nonzero $\cO_o/\vpi_o^m$-linear direct summands of $(\vpi_o^{-m}\cO_o/\cO_o)^n$.
\item For any $\cO_o/\vpi^m_o$-linear direct summand $V'$ of $(\vpi^{-m}_o\cO_o/\cO_o)^n$, if any irreducible component of $\cM^{V'}_I$ lies in $\cM^{V}_I$, then $V'$ contains $V$. Conversely, if $V'$ contains $V$, then $\cM^{V'}_I$ lies in $\cM^{V}_I$.
\end{enumerate}
\end{prop}
Hence we obtain a decomposition
\begin{align*}
\cM_{I,\ka_o} = \bigcup_V\mathring\cM_I^V,
\end{align*}
where $V$ runs over all nonzero $\cO_o/\vpi_o^m$-linear direct summands of $(\vpi_o^{-m}\cO_o/\cO_o)^n$, and the $\mathring\cM_I^V$ are the disjoint locally closed subschemes
\begin{align*}
\mathring\cM_I^V \deq \cM^V_I\ssm\bigcup_{V'}\cM^{V'}_I,
\end{align*}
where $V'$ runs over all $\cO_o/\vpi_o^m$-linear direct summands of $(\vpi_o^{-m}\cO_o/\cO_o)^n$ that strictly contain $V$. Note that the only condition in \ref{ss:stratificationcomplex} that this decomposition  of $\cM_{I,\ka_o}$ does not satisfy is that the $\cM_I^V$ might not be connected. Furthermore, we see that the right action of $\GL_n(\cO_o/\vpi_o^m)$ on $\cM_{I,\ka_o}$ given by composition with Drinfeld level-$m$ structures preserves this decomposition.
\begin{proof}[Proof of Proposition \ref{ss:initialstratification}]
The second part of (iii) is immediate. And when $V$ is zero, $\cM^V_I$ equals all of $\cM_I$, so the whole proposition for this $V$ follows immediately. Now suppose $V$ is nonzero. Let $\ov{z}$ be a geometric point of $\cM_I^V$, and write $((\sE_i,t_i,j_i)_i,\io,\al)$ for the corresponding triple. If $\ov{z}$ lay inside the generic fiber $\cM_{I,\bf{F}_o}$, then $\Dr(\sM'_o/\vpi_o^m)$ would be \'etale \cite[Proposition 7.1.3]{Boy99} and hence could not possibly satisfy the relation required for $\ov{z}$ to lie in $\cM_I^V$. Therefore $\cM_I^V$ is contained in $\cM_{I,\ka_o}$. As the $\cM_I^V$ cover $\cM_I$ as $V$ varies, this proves part (ii).

We can check parts (i) and (iii) by passing to the completion of closed points, so now assume that $\ov{z}$ is a $\ov\ka$-point. Write $\de_o$ for the element of
\begin{align*}
\GL_n(\br\cO_o)\diag(\vpi_o,1,\dotsc,1)\GL_n(\br\cO_o),
\end{align*}
unique up to $\GL_n(\br\cO_o)$-$\sg_o$-conjugacy, which corresponds to $\sM'_o$. Then the completion of $\cM_I$ at $\ov{z}$ is isomorphic to $\br\fX_{\de_o,\al}$ by Proposition \ref{ss:boyerserretate}. Recall from Proposition \ref{ss:connecteddeformationalgebraicclosure} that the restriction morphism induces an isomorphism $\br{R}_{\de_o^\circ,\al^\circ}\llb{T_1,\dotsc,T_{n-k}}\rar^\sim\br{R}_{\de_o,\al}$, where $k$ is the height of the connected part $\sM_o'^\circ$ of $\sM_o'$, $\de_o^\circ$ is an element of $\GL_k(\br\cO_o)\diag(\vpi_o,1,\dotsc,1)\GL_k(\br\cO_o)$ corresponding to $\sM_o'^\circ$,\footnote{We apologize for the notation.} and $\al^\circ$ is the restriction of $\al$ to $\ker\al$. The condition defining $\cM_I^V$ depends only on $\sM_o'^\circ$, so we may further narrow our focus to the local deformation ring $\br{R}_{\de_o^\circ,\al^\circ}$.

Let $e_1,\dotsc,e_s$ be an $\cO_o/\vpi_o^m$-basis of $V$. If $\ov{z}$ lies in $\cM_I^V$, then $\ker\al$ contains $V$, so here we may extend this to an $\cO_o/\vpi_o^m$-basis $e_1,\dotsc,e_k$ of $\ker\al$. For any local Artinian $\br\cO_o$-algebra $A$ with residue field $\ov\ka$ and local $\br\cO_o$-algebra morphism $f:\br{R}_{\de_o^\circ,\al^\circ}\rar A$, write $(H',\al',\io')$ for the corresponding triple. Recall from Proposition \ref{ss:drinfeldparameters} that $\br{R}_{\de_o^\circ,\al^\circ}$ has a choice of local parameters $X_1,\dots,X_k$ such that, for all such $A$ and $f$, the image of $X_i$ under $f$ equals the element of $\fm_A$ corresponding to $\al'(e_i)$. Now the condition defining $\cM_I^V$ is equivalent to
\begin{align*}
\prod_v(T-\al'(v)) = (T)^{\#V}
\end{align*}
as ideals in the formal power series ring $A\llb{T}$, where $v$ runs over all elements in $V$. This is equivalent to the condition that $X_1=\dotsb=X_j=0$. This proves the first part of (iii), and as $X_1,\dotsc,X_k$ is a regular sequence, this also proves part (i).
\end{proof}

\subsection{}
The decomposition of $\cM_{I,\ka_o}$ from Proposition \ref{ss:initialstratification} consists of closed subschemes $\cM_I^V$ which might not be connected. We can immediately rectify this issue (which is necessary for applying Proposition \ref{ss:scholzenearbycycles}) by refining this decomposition into its irreducible components. To simplify the exposition, we proceed as follows.

Let $z$ be a $\ka_o$-point of $\cM_I$, and write $\ov{z}$ for the corresponding $\ov\ka$-point of $\cM_I$. Then $z$ corresponds to a $\sD$-elliptic sheaf $(\sE_i,t_i,j_i)_i$ with Drinfeld level-$m$ structure $\al$. Note that, for any $\cO_o/\vpi_o^m$-linear direct summand $V$ of $(\vpi_o^{-m}\cO_o/\cO_o)^n$, the closed subscheme $\cM^V_I$ contains $z$ if and only if $V$ lies in $\ker\al$.

Given such a $V$, write $\mathring\cM^{V,0}_I,\dotsc,\mathring\cM^{V,l_V}_I$ for the irreducible components of $\mathring\cM^V_I$, and write $\cM^{V,0}_I,\dotsc,\cM^{V,l_V}_I$ for their closures, respectively. These are the irreducible components of $\cM_I^V$, and regularity implies that they are disjoint. We label them such that $\cM^{V,0}_I$ is the one containing $z$. For any $\cO_o/\vpi_o^m$-linear direct summand $V'$ lying in $V$, the disjointness of the $\cM^{V',0}_I,\dotsc,\cM^{V',l_{V'}}_I$ implies that precisely one of them contains $\cM_I^{V,a}$, where $a$ is any integer $0\leq a\leq l_V$. Therefore
\begin{align*}
\cM_{I,\ka_o}^z\deq\bigcup_{(V,a)}\mathring\cM_I^{V,a},
\end{align*}
where $(V,a)$ runs over all pairs for which
\begin{enumerate}[$\bullet$]
\item $V$ is a nonzero $\cO_o/\vpi^m_o$-linear direct summand of $(\vpi_o^{-m}\cO_o/\cO_o)^n$ such that $\cM_I^V$ contains $z$,
\item $a$ is an integer $0\leq a\leq l_V$,
\end{enumerate}
is an open subscheme of $\cM_{I,\ka_o}$. Furthermore, the $\mathring\cM_I^{V,a}$ now form a stratification of $\cM_{I,\ka_o}^z$ as in \ref{ss:stratificationcomplex}, where the indexing set equals the collection of such pairs $(V,a)$. Proposition \ref{ss:initialstratification} shows that $c(V,a)=\rk_{\cO_o/\vpi_o^m}(V)$. Note that the right action of $\GL_n(\cO_o/\vpi_o^m)$ on $\cM_{I,\ka_o}$ restricts to an action on $\cM_{I,\ka_o}^z$ that preserves this stratification.

\subsection{}\label{ss:modulispacecomplex}
At this point, we specialize the general constructions from \ref{ss:stratificationcomplex} to our specific situation of $\cM_{I,\ka_o}^z$. Let $W_{V,a}$ be the $\ov\bQ_\ell$-vector space associated with $\cM_I^{V,a}$ as in \ref{ss:stratificationcomplex}, and make a change of basis in $(\vpi_o^{-m}\cO_o/\cO_o)^n$ to identify $V$ with $(\vpi_o^{-m}\cO_o/\cO_o)^k$, where $k\leq n$ is rank of $V$ as a free $\cO_o/\vpi_o^m$-module. Then the stabilizer of $\cM_I^{V,a}$ under the $\GL_n(\cO_o/\vpi_o^m)$-action is $P(\cO_o/\vpi_o^m)$, where $P$ denotes the standard parabolic subgroup of $\GL_n$ with block sizes $(k,n-k)$, so $W_{V,a}$ has a left action by the finite group $P(\cO_o/\vpi_o^m)$.

Proposition \ref{ss:initialstratification}.(iii) implies that the $\cM^{V',a'}_I$ containing $\cM^{V,a}_I$ are precisely those for which $V'$ lies in $V$. Furthermore, after choosing such a $V'$, the integer $a'$ is uniquely determined by this containment condition. By making another change of basis, we may identify $V'$ with $(\vpi_o^{-m}\cO_o/\cO_o)^{k'}$, where $k'\leq k$ is the rank of $V'$ as a free $\cO_o/\vpi_o^m$-module. Then $Q(\cO_o/\vpi_o^m)$ is the stabilizer of $\cM_I^{V',a'}$  under the action of $\GL_n(\cO_o/\vpi_o^m)$, where $Q$ denotes the standard parabolic subgroup of $\GL_n$ with block sizes $(k',n-k')$. Note that $Q$ contains the unipotent radical of $P$ as well as the Levi factor $\GL_{n-k}$ of $P$, so we see that $P(\cO_o/\vpi_o^m)$ acts on $W_{V,a}$ through the quotient $\GL_k(\cO_o/\vpi_o^m)$.

\subsection{}\label{ss:finitesteinberg}
Now \ref{ss:modulispacecomplex} shows that $\cM^{V,a}_I$ is contained in a unique irreducible component of $\cM_I^{V'}$ whenever $V'$ lies in $V$. Therefore, by inducting on $k$ and examining the construction of $W_{V,a}$, we see that $W_{V,a}$ is isomorphic to the Steinberg representation
\begin{align*}
\St_k^{\cO_o/\vpi_o^m}\deq\ker\left(\Ind_{B(\cO_o/\vpi_o^m)}^{\GL_k(\cO_o/\vpi_o^m)}\ov\bQ_\ell\rar\bigoplus_P\Ind_{P(\cO_o/\vpi_o^m)}^{\GL_k(\cO_o/\vpi_o^m)}\ov\bQ_\ell\right)
\end{align*}
of $\GL_k(\cO_o/\vpi_o^m)$, where $B$ denotes the standard Borel subgroup of $\GL_k$ of upper triangular matrices, $\ov\bQ_\ell$ denotes the trivial representation, and $P$ ranges over all standard parabolic subgroups of $\GL_k$ that do not equal $B$. Under this identification, the chain complex from Proposition \ref{ss:scholzenearbycycles} formed from the $\cM_I^{V,a}$ becomes
\begin{align*}
0\rar\St_k^{\cO_o/\vpi_o^m}\rar&\Ind_{P_{k-1}(\cO_o/\vpi_o^m)}^{\GL_k(\cO_o/\vpi_o^m)}\left(\St_{k-1}^{\cO_o/\vpi_o^m}\otimes\,\ov\bQ_\ell\right)\rar  \Ind_{P_{k-2}(\cO_o/\vpi_o^m)}^{\GL_k(\cO_o/\vpi_o^m)}\left(\St_{k-2}^{\cO_o/\vpi_o^m}\otimes\,\ov\bQ_\ell\right)\rar\dotsb\\
\dotsb&\rar\Ind_{P_1(\cO_o/\vpi_o^m)}^{\GL_k(\cO_o/\vpi_o^m)}\left(\ov\bQ_\ell\otimes\ov\bQ_\ell\right)\rar\ov\bQ_\ell\rar0,
\end{align*}
where $P_s$ denotes the standard parabolic subgroup of $\GL_k$ with block sizes $(s,k-s)$.

\subsection{}\label{ss:exacthypothesis}
We proceed to verify that the chain complex in \ref{ss:finitesteinberg} is exact, that is, the hypotheses of Proposition \ref{ss:scholzenearbycycles} hold for the stratification $\cM_{I,\ka_o}^z=\bigcup_{(V,a)}\mathring\cM_I^{V,a}$.
\begin{lem*}
For all positive integers $k$, the chain complex
\begin{align*}
0\rar\St_k^{\cO_o/\vpi_o^m}\rar&\Ind_{P_{k-1}(\cO_o/\vpi_o^m)}^{\GL_k(\cO_o/\vpi_o^m)}\left(\St_{k-1}^{\cO_o/\vpi_o^m}\otimes\,\ov\bQ_\ell\right)\rar  \Ind_{P_{k-2}(\cO_o/\vpi_o^m)}^{\GL_k(\cO_o/\vpi_o^m)}\left(\St_{k-2}^{\cO_o/\vpi_o^m}\otimes\,\ov\bQ_\ell\right)\rar\dotsb\\
\dotsb&\rar\Ind_{P_1(\cO_o/\vpi_o^m)}^{\GL_k(\cO_o/\vpi_o^m)}\left(\ov\bQ_\ell\otimes\ov\bQ_\ell\right)\rar\ov\bQ_\ell\rar0,
\end{align*}
is exact.
\end{lem*}
In the proof, we will frequently pass between induced representations of $\GL_k(\bf{F}_o)$ and $\GL_k(\cO_o)$. For this, one can use the fact that $P(\cO_o)\bs\GL_k(\cO_o)$ equals $P(\bf{F}_o)\bs\GL_k(\bf{F}_o)$ to show that
\begin{align*}
\Ind_{P(\bf{F}_o)}^{\GL_k(\bf{F}_o)}\la = \Ind_{P(\cO_o)}^{\GL_k(\cO_o)}\la
\end{align*}
as representations of $\GL_k(\cO_o)$, where $P$ is a parabolic subgroup of $\GL_k$ over $\cO_o$, and $\la$ is a smooth representation of $P(\bf{F}_o)$.
\begin{proof}[Proof of Lemma \ref{ss:exacthypothesis}]
We reduce this to a fact about representations of $\GL_k(\bf{F}_o)$ as follows. Suppose we could show that the chain complex
\begin{align*}
0\rar\St_k\rar&\nInd_{P_{k-1}(\bf{F}_o)}^{\GL_k(\bf{F}_o)}\left(\St_{k-1}\otimes\,\ov\bQ_\ell\right)\rar  \nInd_{P_{k-2}(\bf{F}_o)}^{\GL_k(\bf{F}_o)}\left(\St_{k-2}\otimes\,\ov\bQ_\ell\right)\rar\dotsb\\
\dotsb&\rar\nInd_{P_1(\bf{F}_o)}^{\GL_k(\bf{F}_o)}\left(\ov\bQ_\ell\otimes\ov\bQ_\ell\right)\rar\ov\bQ_\ell\rar0,
\end{align*}
is exact, where $\St_s$ is the Steinberg representation of $\GL_s(\bf{F}_o)$, that is, the representation $\St_s(\ov\bQ_\ell)$ in the terminology of Definition \ref{ss:generalizedsteinbergspeh}. As the modulus character $\de_{P_s}$ of $P_s(\bf{F}_o)$ vanishes on $P_s(\cO_o)$, we see that the above chain complex is isomorphic as representations of $\GL_n(\cO_o)$ to 
\begin{align*}
0\rar\St_k\rar&\Ind_{P_{k-1}(\cO_o)}^{\GL_k(\cO_o)}\left(\St_{k-1}\otimes\,\ov\bQ_\ell\right)\rar  \Ind_{P_{k-2}(\cO_o)}^{\GL_k(\cO_o)}\left(\St_{k-2}\otimes\,\ov\bQ_\ell\right)\rar\dotsb\\
\dotsb&\rar\Ind_{P_1(\cO_o)}^{\GL_k(\cO_o)}\left(\ov\bQ_\ell\otimes\ov\bQ_\ell\right)\rar\ov\bQ_\ell\rar0,
\end{align*}
which is therefore also exact.

Subgroups of $\cO_o$-points are compact, so modulus characters vanish on them. Therefore we have
\begin{align*}
\St_s=\ker\left(\nInd_{B(\bf{F}_o)}^{\GL_s(\bf{F}_o)}\de_B^{1/2}\rar\bigoplus_P\nInd_{P(\bf{F}_o)}^{\GL_s(\bf{F}_o)}\de_P^{1/2}\right) = \ker\left(\Ind_{B(\cO_o)}^{\GL_s(\cO_o)}\ov\bQ_\ell\rar\bigoplus_P\Ind_{P(\cO_o)}^{\GL_s(\cO_o)}\ov\bQ_\ell\right)
\end{align*}
as representations of $\GL_s(\cO_o)$, where $B$ denotes the standard Borel subgroup of $\GL_s$, $P$ runs over all standard parabolic subgroups of $\GL_b$ that do not equal $B$, and $\de_B$ and $\de_P$ denote the modulus characters, where we have identified $\ov\bQ_\ell$ with $\bC$. Thus taking $(1+\vpi_o^m\M_s(\cO_o/\vpi_o^m))$-invariants of $\St_s$ yields $\St_s^{\cO_o/\vpi_o^m}$. Because the Peter--Weyl theorem implies that taking $(1+\vpi_o^m\M_k(\cO_o))$-invariants preserves exactness, we conclude that
\begin{align*}
0\rar\St_k^{\cO_o/\vpi_o^m}\rar&\Ind_{P_{k-1}(\cO_o/\vpi_o^m)}^{\GL_k(\cO_o/\vpi_o^m)}\left(\St_{k-1}^{\cO_o/\vpi_o^m}\otimes\,\ov\bQ_\ell\right)\rar  \Ind_{P_{k-2}(\cO_o/\vpi_o^m)}^{\GL_k(\cO_o/\vpi_o^m)}\left(\St_{k-2}^{\cO_o/\vpi_o^m}\otimes\,\ov\bQ_\ell\right)\rar\dotsb\\
\dotsb&\rar\Ind_{P_1(\cO_o/\vpi_o^m)}^{\GL_k(\cO_o/\vpi_o^m)}\left(\ov\bQ_\ell\otimes\ov\bQ_\ell\right)\rar\ov\bQ_\ell\rar0
\end{align*}
is exact, as desired.

Therefore it suffices to show that our chain complex of representations of $\GL_k(\bf{F}_o)$ is exact. A routine calculation \cite[Lemma I.3.2]{HT01} using the graph-theoretic description of the Jordan--H\"older factors of 
\begin{align*}
\Ind_{P_s(\bf{F}_o)}^{\GL_k(\bf{F}_o)}\left(\St_s\otimes\,\ov\bQ_\ell\right)
\end{align*}
shows that there exists irreducible smooth representations $\pi_0,\dotsc,\pi_{k-1}$ of $\GL_k(\bf{F}_o)$ such that the above representation has $\pi_s$ as a subrepresentation and $\pi_{s-1}$ as a quotient. Because the maps in our chain complex are nonzero, this fact indicates that our chain complex of representations of $\GL_n(\bf{F}_o)$ is exact, which concludes the proof.
\end{proof}

\subsection{}\label{ss:modulinearbycycles}
With Lemma \ref{ss:exacthypothesis} in hand, we may apply Proposition \ref{ss:scholzenearbycycles} to our situation. This finally allows us to compute the inertia invariants of nearby cycles on $\cM_I$.
\begin{cor*}
Write $k$ for the rank of $\ker\al$ as a free $\cO_o/\vpi_o^m$-module. We have a canonical isomorphism
\begin{align*}
(R_{I_{\bf{F}_o}}^iR\Psi_{\cM_{I,\br{\bf{F}}_o}}\ov\bQ_\ell)_{\ov{z}} \rar^\sim
\begin{cases}
\Ind_{P_i(\cO_o/\vpi_o^m)}^{\GL_k(\cO_o/\vpi_o^m)}\left(\St_i^{\cO_o/\vpi_o^m}\otimes\,\ov\bQ_\ell\right)(-i) & \mbox{ if }0\leq i \leq k,\\
0 & \mbox{otherwise},
\end{cases}
\end{align*}
as representations of $\GL_k(\cO_o/\vpi_o^m)\times W_{\bf{F}_o}$, where $P_i$ denotes the standard parabolic subgroup of $\GL_k$ with block sizes $(i,k-i)$.
\end{cor*}
\begin{proof}
Because $\cM_{I,\ka_o}^z$ is an open subscheme of $\cM_{I,\ka_o}$ and stalks are local, we may replace $\cM_I$ with its open subscheme $\cM_{I,\ka_o}^z\cup\cM_{I,\bf{F}_o}$. In \ref{ss:inertiainvariantspullback}, we saw that
\begin{align*}
R_{I_{\bf{F}_o}}^iR\Psi_{\cM_{I,\br{\bf{F}}_o}}\ov\bQ_\ell = \ov\i^*R^i\ov\j_*\ov\bQ_\ell,
\end{align*}
and Lemma \ref{ss:exacthypothesis} implies that we may compute $\ov\i^*R^i\ov\j_*\ov\bQ_\ell$ by means of Proposition \ref{ss:scholzenearbycycles}. Recall that $z$ lies in $\cM_I^V$ if and only if $V$ lies in $\ker\al$. Therefore the same argument as in \ref{ss:finitesteinberg} shows that the right hand side of Proposition \ref{ss:scholzenearbycycles} yields
\begin{align*}
\Ind_{P_i(\cO_o/\vpi_o^m)}^{\GL_k(\cO_o/\vpi_o^m)}\left(\St_i^{\cO_o/\vpi_o^m}\otimes\,\ov\bQ_\ell\right)(-i),
\end{align*}
if $0\leq i\leq k$, as desired, while we get zero otherwise because $\ker\al$ contains no $\cO_o/\vpi_o^m$-linear direct summands of rank $i>k$.
\end{proof}

\subsection{}
At this point, we turn to \emph{semisimple traces}. Let $H$ be a finite group, and let $Y$ be a finite-dimensional continuous representation of $H\times W_{\bf{F}_{o,r}}$ over $\ov\bQ_\ell$. Because taking $I_{\bf{F}_o}$-invariants is not exact, the operation $Y\mapsto\tr(\sg_o^{-r}|Y^{I_{\bf{F}_o}})$ is not additive. One way of rectifying this starts by passing to certain filtrations of $Y$.
\begin{defn*}
An exhaustive filtration
\begin{align*}
0=Y_0\subseteq Y_1\subseteq\dotsb\subseteq Y_d=Y
\end{align*}
of $H\times W_{\bf{F}_{o,r}}$-subrepresentations is \emph{admissible} if $I_{\bf{F}_o}$ acts through a finite quotient on the associated graded representation
\begin{align*}
\gr Y_\bullet\deq\bigoplus_{i=1}^dY_i/Y_{i-1}.
\end{align*}
\end{defn*}
Note that refinements of admissible filtrations remain admissible, as do their sums and intersections.

\subsection{}\label{ss:ladicmonodromy}
To ensure that admissible filtrations of $Y$ exist, we shall use the following version of Grothendieck's $\ell$-adic monodromy theorem.
\begin{lem*}
There exists an admissible filtration of $Y$.
\end{lem*}
\begin{proof}
One can prove this precisely as in \cite[Lemma 7.2]{Sch11}: the argument proceeds as in the usual proof of Grothendieck's $\ell$-adic monodromy theorem, except we use $W_{\bf{F}_{o,r}}$ in place of $G_o$ and carry around the extra commuting action of $H$ throughout.

\end{proof}
\begin{defn}
Let $r$ be a non-negative integer, and let $h$ be in $H$. The \emph{semisimple trace} of $h\times \sg^{-r}_o$ on $Y$ is
\begin{align*}
\tr^\text{ss}\left(h\times \sg^{-r}_o|Y\right)\deq\tr\left(h\times \sg^{-r}_o|(\gr Y_\bullet)^{I_{\bf{F}_o}}\right),
\end{align*}
where $Y_\bullet$ is any admissible filtration of $Y$. 
\end{defn}
Note that the common refinement of two admissible filtrations remains admissible. Taking $I_{\bf{F}_o}$-invariants on the associated graded representation is exact because $I_{\bf{F}_o}$ acts through a finite quotient, so we see that $\tr^\text{ss}(h\times \sg^{-r}_o|Y)$ is independent of the choice of $Y_\bullet$. This same observation shows that $\tr^\text{ss}(h\times \sg^{-r}_o|-)$ itself is additive. Therefore semisimple trace descends to the derived category of finite-dimensional continuous representations of $H\times W_{\bf{F}_o,r}$ over $\ov\bQ_\ell$, and it is additive on exact triangles.

\subsection{}\label{ss:inertiainvariantssemisimpletrace}
Taking derived $I_{\bf{F}_o}$-invariants is also additive on exact triangles, so we can consider the additive function $\tr(h\times\sg_o^{-r}|-)\circ R_{I_{\bf{F}_o}}$ as well. It has the following relationship with semisimple trace.
\begin{lem*}
We have
\begin{align*}
\tr\left(h\times\sg_o^{-r}|-\right)\circ R_{I_{\bf{F}_o}} = (1-q_o^r)\tr^\text{ss}\left(h\times\sg_o^{-r}|-\right)
\end{align*}
as functions on the derived category of finite-dimensional continuous representations of $H\times W_{\bf{F}_{o,r}}$ over $\ov\bQ_\ell$.
\end{lem*}
\begin{proof}
One can prove this exactly as in \cite[Lemma 7.5]{Sch11}: first, reduce to the case of a complex concentrated in one degree, and take wild inertia invariants. Since wild inertia is a pro-$p$ group, while the modules in question are $\ell$-torsion, Maschke's theorem indicates that this is exact. Next, we take tame inertia invariants: as tame inertia is pro-cyclic, we can take the standard resolution for pro-cyclic groups to conclude the proof.  
\end{proof}

\subsection{}\label{ss:semisimpletraceintegral}
Semisimple traces have the following interpretation in terms of integrals of traces.
\begin{lem*}
We have an equality
\begin{align*}
\tr^\text{ss}\left(h\times\sg_o^{-r}|Y\right) = \int_{\sg_o^{-r}I_{\bf{F}_o}}\!\dif\tau\,\tr\left(h\times\tau|Y\right),
\end{align*}
where $\dif\tau$ is the Haar measure on $W_{\bf{F}_o}$ that gives $I_{\bf{F}_o}$ volume $1$.
\end{lem*}
\begin{proof}
Because both sides are additive in $Y$, it suffices to prove this when $Y$ is irreducible. Lemma \ref{ss:ladicmonodromy} implies that $I_{\bf{F}_o}$ acts on $Y$ through a finite quotient, so we obtain
\begin{align*}
\tr^\text{ss}\left(h\times\sg_o^{-r}|Y\right) = \tr\left(h\times\sg_o^{-r}|Y^{I_{\bf{F}_o}}\right).
\end{align*}
Furthermore, $Y$ is a smooth representation of $W_{\bf{F}_{o,r}}$, and the action of $h\times\sg_o^{-r}$ on $Y^{I_{\bf{F}_o}}$ is given by the action of the function
\begin{align*}
\bf1_{h\times\sg_o^{-r}I_{\bf{F}_o}}\in C^\infty_c(H\times W_{\bf{F}_{o,r}})
\end{align*}
on $Y$, where $\bf1_{h\times\sg_o^{-r}I_{\bf{F}_o}}$ is the indicator function on $h\times\sg_p^{-r}I_{\bf{F}_o}$. The trace of this function is given by the above integral, so we obtain the desired equality.
\end{proof}

\subsection{}\label{ss:sectionsixfinalresult}
Let $h$ be a function in $C^\infty_c(\GL_n(\cO_o))$. We conclude by combining the results in this section to write the integral of $\phi_{\tau,h}$, as $\tau$ ranges over elements of $\sg_o^{-r}I_{\bf{F}_o}$, in terms of representations of $\GL_n(\cO_o)$. Write $\cO_{o,r}$ for the ring of integers of $\bf{F}_{o,r}$, and write $\ka_{o,r}$ for the residue field of $\cO_{o,r}$. Let $\de_o$ be an element of $\GL_n(\bf{F}_{o,r})$, and if $\de_o$ lies in $\GL_n(\cO_{o,r})\diag(\vpi_o,1,\dotsc,1)\GL_n(\cO_{o,r})$, write $k$ for the rank of the connected part of the associated effective minuscule local shtuka $H_{\de_o}$ over $\Spec\ka_{o,r}$ as in \S\ref{s:deformationspaces}.
\begin{cor*}
The integral
\begin{align*}
\int_{\sg_o^{-r}I_{\bf{F}_o}}\!\dif\tau\,\phi_{\tau,h}(\de_o)
\end{align*}
vanishes if $\de_o$ does not lie in $\GL_n(\cO_{o,r})\diag(\vpi_o,1,\dotsc,1)\GL_n(\cO_{o,r})$. On the other hand, if $\de_o$ lies in this double coset, then the above integral equals
\begin{align*}
\fC_k\deq\frac1{1-q_o^r}\sum_{i=0}^k(-1)^iq_o^{ir}\tr\left(h|\Ind_{P_{i,k}(\cO_o)}^{\GL_n(\cO_o)}\left(\St_i\otimes\,\bC\right)\right),
\end{align*}
where $P_{i,k}$ is the standard parabolic subgroup of $\GL_n$ with block sizes $(i,k-i,n-k)$, and $\bC$ denotes the trivial representation of $\GL_{k-i}(\cO_o)\times\GL_{n-k}(\cO_o)$.
\end{cor*}
\begin{proof}
Because $\phi_{\tau,h}(\de_o)$ vanishes for $\de_o$ not in $\GL_n(\cO_{o,r})\diag(\vpi_o,1,\dotsc,1)\GL_n(\cO_{o,r})$ by definition, it suffices to prove this when $\de_o$ actually does lie in this double coset. Recalling the definition of $\phi_{\tau,h}(\de_o)$ for such $\de_o$ yields
\begin{align*}
\int_{\sg_o^{-r}I_{\bf{F}_o}}\!\dif\tau\,\phi_{\tau,h}(\de_o) = \int_{\sg_o^{-r}I_{\bf{F}_o}}\!\dif\tau\,\tr\left(\tau\times h|[R\psi_{\de_o}]\right),
\end{align*}
where $[R\psi_{\de_o}]$ is the virtual $\GL_n(\cO_o)\times I_{\bf{F}_o}$-admissible/continuous representation as in \cite[p.~24]{HT01} of $\GL_n(\cO_o)\times W_{\bf{F}_o}$ over $\ov\bQ_\ell$ from \ref{ss:fundamentalrepresentation}. Let $m$ be a positive integer for which $h$ descends to a function in $\ov\bQ_\ell[\GL_n(\cO_o/\vpi_o^m)]$, where we have identified $\ov\bQ_\ell$ with $\bC$. We denote this element in $\ov\bQ_\ell[\GL_n(\cO_o/\vpi_o^m)]$ using $h$ by abuse of notation. Then this integral of traces becomes
\begin{align*}
\sum_{i=0}^\infty(-1)^i\int_{\sg_o^{-r}I_{\bf{F}_o}}\!\dif\tau\,\tr\left(\tau\times h|R^i\psi_{\de_o,m}\right),
\end{align*}
where $R^i\psi_{\de_o,m}$ is the finite-dimensional continuous representation of $\GL_n(\cO_o/\vpi_o^m)\times W_{\bf{F}_o}$ from \ref{ss:nearbycyclesalgebraization}, since
\begin{align*}
\sum_{i=0}^\infty(-1)^iR^i\psi_{\de_o,m}
\end{align*}
are the $(1+\vpi_o^m\M_n(\cO_o))$-invariants of $[R\psi_{\de_o}]$. Lemma \ref{ss:semisimpletraceintegral} shows that our above sum of integrals equals
\begin{align*}
\sum_{i=0}^\infty(-1)^i\tr^\text{ss}\left(\sg_o^{-r}\times h|R^i\psi_{\de_o,m}\right).
\end{align*}
Next, Lemma \ref{ss:inertiainvariantssemisimpletrace} implies that this sum of traces equals
\begin{align*}
\frac1{1-q_o^r}\sum_{i=0}^\infty(-1)^i\tr\left(\sg_o^{-r}\times h|R_{I_{\bf{F}_o}}R^i\psi_{\de_o,m}\right).
\end{align*}
Using $\cM_{o^m,\br\cO_o}$ as an algebraization of $\br\fX_{\de_o,m}$ as in the proof of Lemma \ref{ss:nearbycyclesalgebraization}, write $z$ for the point of $\cM_{\varnothing,\br\cO_o}$ corresponding to $\de_o$, and write $\pi:\cM_{o^m,\br\cO_o}\rar\cM_{\varnothing,\br\cO_o}$ for the restriction morphism. Then Berkovich's nearby cycles comparison theorem \cite[Theorem 3.1]{Ber96} shows that the above sum becomes
\begin{align*}
\frac1{1-q_o^r}\sum_{i=0}^\infty(-1)^i\tr\left(\sg_o^{-r}\times h|R^i_{I_{\bf{F}_o}}\res{R\Psi_{\cM_{o^m,\br{\bf{F}}_o}}\ov\bQ_\ell}_{\pi^{-1}(z)}\right),
\end{align*}
where now we use the alternating product to expand $R_{I_{\bf{F}_o}}$ instead of $R\Psi_{\cM_{o^m,\br{\bf{F}}_o}}$. Applying the decomposition of $\br\fX_{\de_o,m}$ in Proposition \ref{ss:deformationalgebraicclosure} alongside the calculation of inertia invariants in Corollary \ref{ss:modulinearbycycles} yields the desired expression, where we have identified $\ov\bQ_\ell$ with $\bC$.
\end{proof}

\section{Langlands--Kottwitz counting and the Serre--Tate trick}\label{s:langlandskottwitz}
In this section, all representations shall be taken over $\ov\bQ_\ell$, and we view $\bC$-valued functions as $\ov\bQ_\ell$-valued ones via our fixed identification $\ov\bQ_\ell=\bC$. Let $f_{\tau,h}$ be our test function in $C^\infty_c(\GL_n(\bf{F}_o))$ from \S\ref{s:deformationspaces}, and let $f^{\infty,o}$ be any function in $C^\infty((D\otimes\bA^{\infty,o})^\times)$. Our goal is to relate the trace of $f^{\infty,o}\times h\times\tau$ to the trace of $f^{\infty,o}\times f_{\tau,h}$ on the virtual representation $[H_\xi]$. We begin by using the nearby cycles spectral sequence to pass from the generic fiber to the special fiber. Then, we use \emph{Deligne's conjecture} as proven by Fujiwara, which describes this trace in terms of manageable local terms, to convert this into a sum of terms indexed by geometric points in the special fiber. These points correspond to isomorphism classes of $\sD$-elliptic sheaves with extra structure.

We express this sum in terms of orbital integrals by using the equi-characteristic analog of \emph{Langlands--Kottwitz counting}, which describes isomorphism classes of $\sD$-elliptic sheaves in terms of the algebraic group-theoretic data associated with their local shtukas. Here, the Serre--Tate theorem from \S\ref{s:modulispaces} allows us to convert contributions from $h\times\tau$ into contributions from $f_{\tau,h}$. We then use the \emph{Selberg trace formula} to rewrite this expression in terms of automorphic representations. Finally, results of Laumon--Rapoport--Stuhler on the cohomology of $\cE\ell\ell_{\sD,I}$ allow us to conclude.

\subsection{}\label{ss:galoistohecke}
The following identity is the main result of this section. Write $v_o:W_{\bf{F}_o}\rar\bZ$ for the canonical surjection that sends geometric $q_o$-Frobenius at $o$ to $1$ in $\bZ$. 
\begin{prop*}
Let $\tau$ be an element $W_{\bf{F}_o}$ with $v_o(\tau)>0$, let $h$ be a function in $C^\infty_c(\GL_n(\cO_o))$, and let $f^{\infty,o}$ be a function in $C^\infty_c((D\otimes\bA^{\infty,o})^\times)$. Then we have an equality of traces
\begin{align*}
\tr(f^{\infty,o}\times h\times\tau|[H_\xi]) = \textstyle\frac1n\tr(f^{\infty,o}\times f_{\tau,h}|[H_\xi]),
\end{align*}
where $f_{\tau,h}$ in $C^\infty_c(\GL_n(\bf{F}_o))$ is the function defined as in \ref{ss:ftauhdefinition}.
\end{prop*}
\begin{proof}
We start with some immediate reductions. Note that both sides are $\ov\bQ_\ell$-linear with respect to the function $f^{\infty,o}\times h$ in $C^\infty_c((D\otimes\bA^\infty)^\times)$. Therefore it suffices to prove this equality for
\begin{align*}
f^{\infty,o} = \frac1{\vol(\sK_I^{\infty,o})}\bf1_{\sK^{\infty,o}_Ig^o\sK^{\infty,o}_I}\mbox{ and }h = \frac1{\vol(\sK_{I,o}^\infty)}\bf1_{\sK_{I,o}^\infty g_o\sK_{I,o}^\infty},
\end{align*}
where $I$ can be any sufficiently large finite closed subscheme of $C\ssm\infty$, and $g$ is an element of $(D\otimes\bA^\infty)^\times$ for which $g_o$ lies in $\GL_n(\cO_o)$. With these choices of $f^{\infty,o}$ and $h$, our trace becomes
\begin{align*}
\tr(f^{\infty,o}\times h\times\tau|[H_\xi]) = \tr(c_{\bf{F}^\sep}\times\tau|[H_\xi]),
\end{align*}
where $c_{\bf{F}^\sep}=(c_{1,\bf{F}^\sep},c_{2,\bf{F}^\sep})$ is the pullback \cite[1.1.8]{Var07} to $\bf{F}^\sep$ of the correspondence from \ref{ss:dellipticsheavescohomology} that induces the action of $f^{\infty,o}\times h$. Next, we pass to $\bC_o$ and apply nearby cycles, which shows that this trace equals
\begin{align*}
\sum_{i=0}^\infty(-1)^i\tr\left(c_{\ov\ka}\times\tau|H^i(\cM_{I,\ov\ka},R\Psi_{\cM_{I,\bC_o}}\cL_{\xi,I,\bC_o})\right),
\end{align*}
where $R\Psi_{\cM_{I,\bC_o}}$ denotes the derived nearby cycles functor \cite[XIII (1.3.2.3)]{DK73} on $\cM_{I,\cO_{\bC_o}}$, and $c_{\ov\ka}=(c_{1,\ov\ka},c_{2,\ov\ka})$ is the pullback to $\ov\ka$ of the correspondence $c_{\cO_o}$ from \ref{ss:badreductionheckecorrespondence} that extends $c_\bf{F}$. Writing $\pi:\cM_{I}\rar\cM_{I\ssm o}$ for the restriction morphism, we see that the above sum equals
\begin{align*}
\sum_{i=0}^\infty(-1)^i\tr\left(\pi_{\ov\ka,*}c_{\ov\ka}\times\tau|H^i(\cM_{I\ssm o,\ov\ka},\pi_{\ov\ka,*}R\Psi_{\cM_{I,\bC_o}}\cL_{\xi,I,\bC_o})\right),
\end{align*}
where $\pi_{\ov\ka,*}c_{\ov\ka}=(\pi_{\ov\ka,*}c_{1,\ov\ka},\pi_{\ov\ka,*}c_{1,\ov\ka})$ is the pushforward correspondence \cite[1.1.6]{Var07} of $c_{\ov\ka}$. Next, write $r=v_o(\tau)$. The independence of characters implies that it suffices to prove our desired equalities for sufficiently large $r$. With this reduction in hand, we may now apply Deligne's conjecture as proven by Fujiwara \cite[Corollary 5.4.5]{Fuj97} to obtain the sum
\begin{align*}
\sum_y\tr\left((\pi_{\ov\ka,*}c_{\ov\ka}\times\tau)_y|(\pi_{\ov\ka,*}R\Psi_{\cM_{I,\bC_o}}\cL_{\xi,I,\bC_o})_{c_{2,\ov\ka}(y)}\right),
\end{align*}
where $y$ runs over all $\ov\ka$-points of $\Fix(\sg_o^r\circ\pi_{\ov\ka,*}c_{\ov\ka})$, and $\Fix(\sg_o^r\circ\pi_{\ov\ka,*}c_{\ov\ka})$ in turn denotes the fiber product
\begin{align*}
\xymatrix{\Fix(\sg_o^r\circ c_{\ov\ka})\ar[r]\ar[d] & \cM^\infty_{\ov\ka}/(\sK_I^{\infty,o}\cap g^{o,-1}\sK_I^{\infty,o} g^o)\ar[d]^-{(\sg_o^r\circ\pi_{\ov\ka,*}c_{1,\ov\ka},\pi_{\ov\ka,*}c_{2,\ov\ka})}\\
\cM_{I\ssm o,\ov\ka}\ar[r]^-{\De_{\cM_{I\ssm o},\ov\ka}} & \cM_{I\ssm o,\ov\ka}\times_{\ov\ka}\cM_{I\ssm o,\ov\ka}}
\end{align*}
At this point, we shall use the following lemma to decompose the $\ell$-adic sheaf whose cohomology we are studying.
\begin{lem}\label{lem:derivediso}
We have an isomorphism in the derived category of constructible $\ell$-adic sheaves
\begin{align*}
(\pi_{\ov\ka,*}R\Psi_{\cM_{I,\bC_o}}\ov\bQ_\ell)\otimes\cL_{I\ssm o,\xi,\ov\ka}\rar^\sim\pi_{\ov\ka,*}R\Psi_{\cM_{I,\bC_o}}\cL_{\xi,I,\bC_o}
\end{align*}
that preserves the actions of $I_{\bf{F}_o}$ and $(D\otimes\bA^\infty)^\times$.
\end{lem}
For the proof of this lemma, we will need the following notation. Write
\begin{align*}
\ov\j:\cM_{I\ssm o,\bC_o}\rar\cM_{I\ssm o,\cO_{\bC_o}}\mbox{ and }\ov\i:\cM_{I\ssm o,\ov\ka}\rar\cM_{I\ssm o,\cO_{\bC_o}}
\end{align*}
for the inclusion morphisms of the generic and special fibers of $\cM_{I\ssm o,\cO_{\bC_o}}$, respectively. In the proof of this lemma, we shall freely use the fact that
\begin{align*}
R\Psi_{\cM_{I\ssm o,\bC_o}}\pi_{\bC_o,*} = \pi_{\ov\ka,*}R\Psi_{\cM_{I,\bC_o}},
\end{align*}
which follows \cite[XIII (1.3.6)]{DK73} from finite base change.
\begin{proof}[Proof of Lemma \ref{lem:derivediso}]
We start by identifying
\begin{align*}
(\pi_{\ov\ka,*}R\Psi_{\cM_{I,\bC_o}}\ov\bQ_\ell)\otimes\cL_{I\ssm o,\xi,\ov\ka} = (R\Psi_{\cM_{I\ssm o,\bC_o}}\pi_{\bC_o,*}\ov\bQ_\ell)\otimes\ov\i^*\cL_{I\ssm o,\xi,\cO_{\bC_o}}.
\end{align*}
Next, applying the projection formula to canonical adjunction morphisms yields a morphism \cite[(0.2)]{Var07}
\begin{align}\label{eq:morphism1}
(\pi_{\ov\ka,*}R\Psi_{\cM_{I,\bC_o}}\ov\bQ_\ell)\otimes\cL_{I\ssm o,\xi,\ov\ka}\rar\ov\i^*(R\ov\j_*\pi_{\bC_o,*}\ov\bQ_\ell\otimes R\ov\j_*\ov\j^*\cL_{I\ssm o,\xi,\cO_{\bC_o}}),\tag{$\star$}
\end{align}
and taking $\ov\i^*$ of a similar morphism \cite[(0.1)]{Var07}
\begin{align*}
R\ov\j_*\pi_{\bC_o,*}\ov\bQ_\ell\otimes R\ov\j_*\ov\j^*\cL_{I\ssm o,\xi,\cO_{\bC_o}}\rar R\ov\j_*(\pi_{\bC_o,*}\ov\bQ_\ell\otimes\ov\j^*\cL_{I\ssm o,\xi,\cO_{\bC_o}})
\end{align*}
provides another morphism
\begin{align}\label{eq:morphism2}
\ov\i^*(R\ov\j_*\pi_{\bC_o,*}\ov\bQ_\ell\otimes R\ov\j_*\ov\j^*\cL_{I\ssm o,\xi,\cO_{\bC_o}})\rar R\Psi_{\cM_{I\ssm o,\bC_o}}(\pi_{\bC_o,*}\ov\bQ_\ell\otimes\cL_{I\ssm o,\xi,\bC_o}).\tag{$\star\star$}
\end{align}
By reducing to finite coefficients and then checking on an \'etale trivialization, we see that the composition of Equation (\ref{eq:morphism1}) and Equation (\ref{eq:morphism2}) is an isomorphism. Applying the projection formula \cite[XVII (5.2.9)]{AGV73} to $\ov\bQ_\ell$ and $\cL_{I\ssm o,\xi,\bC_o}$ implies that
\begin{align*}
R\Psi_{\cM_{I\ssm o,\bC_o}}(\pi_{\bC_o,*}\ov\bQ_\ell\otimes\cL_{I\ssm o,\xi,\bC_o}) =R\Psi_{\cM_{I\ssm o,\bC_o}}\pi_{\bC_o,*}\pi^*_{\bC_o}\cL_{I\ssm o,\xi,\bC_o} = \pi_{\ov\ka,*}R\Psi_{\cM_{I,\bC_o}}\cL_{\xi,I,\bC_o},
\end{align*}
as desired.
\end{proof}
Return to the proof of Proposition \ref{ss:galoistohecke}. With this tensor product decomposition, our sum of traces becomes
\begin{align*}
\sum_y\tr\left((\pi_{\ov\ka,*}c_{\ov\ka}\times\tau)_y|(\pi_{\ov\ka,*}R\Psi_{\cM_{I,\bC_o}}\ov\bQ_\ell)_{c_{2,\ov\ka}(y)}\right)\tr\left((\pi_{\ov\ka,*}c_{\ov\ka}\times\tau)_y|(\cL_{I\ssm o,\xi,\ov\ka})_{c_{2,\ov\ka}(y)}\right).
\end{align*}
Write $z$ for $c_{2,\ov\ka}(y)$, write $(\sE_i,t_i,j_i)_i$ for the $\sD$-elliptic sheaf over $\ov\ka$ corresponding to $z$, write $\sM'_o$ for its local shtuka at $o$ as in \ref{ss:dellipticsheavestolocalshtukazero}, and write $\de_o$ for the corresponding element of the double coset
\begin{align*}
\ang{\vpi_o}\deq\GL_n(\br\cO_o)\diag(\vpi_o,1,\dotsc,1)\GL_n(\br\cO_o),
\end{align*}
which is unique up to $\GL_n(\br\cO_o)$-$\sg_o$-conjugacy. Because $y$ lies in $\Fix(\sg_o^r\circ\pi_{\ov\ka,*}c_{\ov\ka})$, we see $z$ is defined over $\ka_r$, so $\de_o$ can naturally be chosen to lie in $\ang{\vpi_o}\cap\GL_n(\bf{F}_{o,r})$.

Writing $m$ for the multiplicity of $o$ in $I$, we see that finite base change implies that the above sum equals
\begin{align*}
\sum_y\tr\left(c_{\ov\ka}\times\tau|H^i(\pi^{-1}(z),\res{R\Psi_{\cM_{I,\bC_o}}\ov\bQ_\ell}_{\pi^{-1}(z)})\right)\tr\left((\pi_{\ov\ka,*}c_{\ov\ka}\times\tau)_y|(\cL_{I\ssm o,\xi,\ov\ka})_{z}\right).
\end{align*}
Because $\br\fX_{\de_o,m}$ is the completion of $\cM_{I,\br\cO_o}$ along $\pi^{-1}(z)$ by \ref{ss:extendedserretate}, Berkovich's nearby cycles comparison theorem \cite[Theorem 3.1]{Ber96} indicates that this equals
\begin{align*}
&\sum_y\sum_{i=0}^\infty\tr\left(g_o\times\tau|H^i(\fX_{\de_o,m,\ov\ka},R^i\Psi_{\fX_{\de_o,m,\bC_o}}\ov\bQ_\ell)\right)\tr\left((\pi_{\ov\ka,*}c_{\ov\ka}\times\tau)_y|(\cL_{I\ssm o,\xi,\ov\ka})_{z}\right)\\
=&\sum_y\tr(\tau\times h|[R\psi_{\de_o}])\tr\left((\pi_{\ov\ka,*}c_{\ov\ka}\times\tau)_y|(\cL_{I\ssm o,\xi,\ov\ka})_{z}\right) = \sum_y\phi_{\tau,h}(\de_o)\tr\left((\pi_{\ov\ka,*}c_{\ov\ka}\times\tau)_y|(\cL_{I\ssm o,\xi,\ov\ka})_{z}\right),
\end{align*}
where $[R\psi_{\de_o}]$ is the virtual $\GL_n(\cO_o)\times I_{\bf{F}_o}$-admissible/continuous representation of $\GL_n(\cO_o)\times W_{\bf{F}_o}$ defined in \ref{ss:fundamentalrepresentation}, and $\phi_{\tau,h}$ in $C_c^\infty(\GL_n(\bf{F}_{o,r}))$ is the function defined in \ref{ss:sigmaconjugationisopen}.
\subsection{}
In order to express the above sum in terms of orbital integrals, we now introduce \emph{Langlands--Kottwitz counting}. This technique counts isomorphism classes of $\sD$-elliptic sheaves over $\ov\ka$ whose zero lies over $o$ by
\begin{enumerate}[(1)]
\item first describing the \emph{isogeny classes} of $\sD$-elliptic sheaves in terms of conjugacy classes in certain groups,
\item then counting the number of isomorphism classes in each isogeny class by taking orbital integrals.
\end{enumerate}
Since $\sD$-elliptic sheaves consist of vector bundles, it is quite natural to reinterpret them in terms of linear algebraic groups. This incarnation of the Langlands--Kottwitz method is heavily based on work of Drinfeld and Laumon. 

\subsection{}
To define isogeny classes of $\sD$-elliptic sheaves, we first introduce \emph{$\vp$-spaces}.
\begin{defn*}
A \emph{$\vp$-space} is a pair $(V,\vp)$, where
\begin{enumerate}[$\bullet$]
\item $V$ is a finite-dimensional $\bf{F}\otimes_\ka\ov\ka$-vector space,
\item $\vp:{}^\sg\!V\rar^\sim V$ is an $\bf{F}\otimes_\ka\ov\ka$-linear isomorphism,
\end{enumerate}
where ${}^\sg$ denotes $(\id_\bf{F}\otimes_\ka\sg)^*$.
\end{defn*}
Note that $\vp$-spaces resemble local shtukas and $\sD$-elliptic sheaves, with an important difference: $\vp$-spaces lie over the generic fiber of the curve $C$, rather than completions of $C$ at closed points in the case of local shtukas or (a specific open subset of) the entire curve $C$ in the case of $\sD$-elliptic sheaves.

\subsection{}\label{ss:vpguys}
For any $\sD$-elliptic sheaf $(\sE_i,t_i,j_i)_i$ over $\ov\ka$, the generic fiber $\sE_{1,\bf{F}\otimes_\ka\ov\ka}$ is a finite-dimensional $\bf{F}\otimes_\ka\ov\ka$-vector space, and the map $(j_{1,\bf{F}\otimes_\ka\ov\ka})^{-1}\circ t_{1,\bf{F}\otimes_\ka\ov\ka}$ is a $\bf{F}\otimes_\ka\ov\ka$-linear isomorphism. Therefore the pair
\begin{align*}
(V,\vp)\deq\left(\sE_{1,\bf{F}\otimes_\ka\ov\ka},(j_{1,\bf{F}\otimes_\ka\ov\ka})^{-1}\circ t_{1,\bf{F}\otimes_\ka\ov\ka}\right)
\end{align*}
 forms a $\vp$-space, and it has a right action of $D$ for which $V$ is a free right $D$-module of rank $1$. We call this the \emph{generic fiber} of $(\sE_i,t_i,f_i)_i$, and we say two $\sD$-elliptic sheaves over $\ov\ka$ are \emph{isogenous} if they have isomorphic generic fibers. Thus isogeny classes of $\sD$-elliptic sheaves correspond to certain isomorphism classes of $\vp$-spaces.

Suppose that the zero of $(\sE_i,t_i,j_i)_i$ lies over $o$. For every place $x$ of $\bf{F}$, fix a map $\ka_x\rar\ov\ka$ over $\ka$, and write $\sM_x$ for the resulting local shtuka of $(\sE_i,t_i,j_i)_i$ at $x$ as constructed in \ref{ss:dellipticsheavestolocalshtukazero} for $x=o$ or \ref{ss:dellipticsheavestolocalshtuka} for $x\neq o$. Since the set of all maps $\ka_x\rar\ov\ka$ over $\ka$ are cyclically permuted via composition with $\sg$, we can view $\sM_x$ as a free module over
\begin{align*}
\cO_x\wh\otimes_\ka\ov\ka = \prod_\io\cO_x\wh\otimes_{\ka_x,\io}\ov\ka
\end{align*}
equipped with a $\sg$-semilinear automorphism after inverting $\vpi_x$, where the $\io$ runs over all maps $\ka_x\rar\ov\ka$ over $\ka$ \cite[p.~33]{Dri88}. Under this perspective, we may naturally identify $(\sM_x[\frac1{\vpi_x}],\sF_x)$ with $(V\wh\otimes_\bf{F}\bf{F}_x,\vp\wh\otimes_\bf{F}\id_{\bf{F}_x})$.

\subsection{}
Next, we introduce \emph{$\vp$-pairs}, which will provide an alternative way of describing $\vp$-spaces.
\begin{defn*}
A \emph{$\vp$-pair} is a pair $(\wt{\bf{F}},\wt\Pi)$, where
\begin{enumerate}[$\bullet$]
\item $\wt{\bf{F}}$ is a finite separable (ring) extension of $\bf{F}$,
\item $\wt\Pi$ is an element of $\wt{\bf{F}}^\times\otimes_\bZ\bQ$ that is not contained in $\bf{F}'^\times\otimes_\bZ\bQ$ for any proper $\bf{F}$-subalgebra $\bf{F}'$ of $\wt{\bf{F}}$.
\end{enumerate}
Note that $\wt{\bf{F}}$ is a product of field extensions of $\bf{F}$, so it has a well-defined notion of places. For any place $\wt{x}$ of $\wt{\bf{F}}$, the $\bQ$-valued valuation $\wt{x}(\wt\Pi)$ is well-defined, and we write $d(\wt\Pi)$ for the least common denominator of $\wt{x}(\wt\Pi)\deg\wt{x}$ as $\wt{x}$ ranges over all places of $\wt{\bf{F}}$, where $\deg\wt{x}$ is taken with respect to $\ka$.
\end{defn*}
Therefore $\vp$-pairs roughly correspond to elements in extensions (up to roots of unity) that cannot be obtained from subextensions by taking rational powers.

\subsection{}\label{ss:phispacesphipairs}
The notions of a $\vp$-pair and a $\vp$-space are related as follows.
\begin{prop*}[{\cite[Proposition 2.1]{Dri88}}]
We have a canonical map
\begin{align*}
\{\mbox{isomorphism classes of }\vp\mbox{-spaces}\}&\rar\{\mbox{isomorphism classes of }\vp\mbox{-pairs}\}\\
(V,\vp)&\longmapsto(\wt{\bf{F}}_{(V,\vp)},\wt\Pi_{(V,\vp)})
\end{align*}
that induces a bijection from isomorphism classes of irreducible $\vp$-spaces to isomorphism classes of $\vp$-pairs for which $\wt{\bf{F}}$ is a field. In that situation, the dimension of $V$ over $\bf{F}\otimes_\ka\ov\ka$ equals $[\wt{\bf{F}}_{(V,\vp)}:\bf{F}]d(\wt\Pi_{(V,\vp)})$, and the endomorphism ring of $(V,\vp)$ is isomorphic to the central division algebra $\wt{D}$ over $\wt{\bf{F}}_{(V,\vp)}$ with Hasse invariants
\begin{align*}
\inv_{\wt{x}}(\wt{D}) = -\wt{x}(\wt\Pi_{(V,\vp)})\deg\wt{x}.
\end{align*}
\end{prop*}
The above map is constructed by replacing $\ov\ka$ with a finite extension of $\ka$ of degree $a$ and then considering the $\bf{F}$-algebra generated by $\vp^a$, where taking $a$ to be sufficiently divisible ensures that the resulting construction is independent of the choice of $a$ \cite[p.~31]{Dri88}. This is the equi-characteristic analogue of Honda--Tate theory.

\subsection{}\label{ss:dellipticsheavesphipairs}
We can describe which $\vp$-spaces occur as generic fibers of $\sD$-elliptic sheaves in terms of their associated $\vp$-pairs. This completes our description of the isogeny classes of $\sD$-elliptic sheaves.
\begin{prop*}[{\cite[(9.13)]{LRS93}}]
The isomorphism classes of $\vp$-spaces that arise from $\sD$-elliptic sheaves over $\ov\ka$ with zero lying over $o$ map precisely to $\vp$-pairs $(\wt{\bf{F}},\wt\Pi)$ for which
\begin{enumerate}[$\bullet$]
\item $\wt{\bf{F}}$ is a field, and $[\wt{\bf{F}}:\bf{F}]$ divides $n$,
\item there exists only one place $\wt\infty$ of $\wt{\bf{F}}$ lying over $\infty$, and it satisfies $\wt\infty(\wt\Pi)\deg\wt\infty=-[\wt{\bf{F}}:\bf{F}]/n$,
\item there exists only one other place $\wt{o}\neq\wt\infty$ of $\wt{\bf{F}}$ satisfying $\wt{o}(\wt\Pi)\neq0$, and it lies over $o$.
\end{enumerate}
For any $\sD$-elliptic sheaf $(\sE_i,t_i,j_i)_i$ lying in the isogeny class corresponding to $(\wt{\bf{F}},\wt\Pi)$, the height of the connected part of $\sM_o'$ equals $n[\wt{\bf{F}}_{\wt{o}}:\bf{F}_o]/[\wt{\bf{F}}:\bf{F}]$. Furthermore, if we write $(W,\psi)$ for the irreducible $\vp$-space corresponding to $(\wt{\bf{F}},\wt\Pi)$, then $(V,\vp)$ is isomorphic to $(W,\psi)^{\oplus n}$.
\end{prop*}
By checking Hasse invariants of the endomorphism ring $\wt{D}$ of $(W,\psi)$ and applying Proposition \ref{ss:phispacesphipairs}, we also deduce that $\wt{o}(\wt\Pi)\deg\wt{o}=[\wt{\bf{F}}:\bf{F}]/n$.

Return to the proof of Proposition \ref{ss:galoistohecke}. By gathering terms in the same isogeny class, we can rewrite our sum as
\begin{align*}
\sum_{(\wt{\bf{F}},\wt\Pi)}\sum_y\phi_{\tau,h}(\de_o)\tr\left((\pi_{\ov\ka,*}c_{\ov\ka}\times\tau)_y|(\cL_{I\ssm o,\xi,\ov\ka})_{z}\right),
\end{align*}
where $(\wt{\bf{F}},\wt\Pi)$ runs over all $\vp$-pairs satisfying the conditions in Proposition \ref{ss:dellipticsheavesphipairs}, and $y$ runs over $\ov\ka$-points of $\Fix(\sg^r_o\circ c_{\ov\ka})$ that lie in the isogeny class corresponding to $(\wt{\bf{F}},\wt\Pi)$.

\subsection{}\label{ss:dellipticsheavesisomorphismclasses}
At this point, we initiate the second part of the Langlands--Kottwitz method: counting the number of isomorphism classes in each isogeny class. We begin by describing \emph{all} isomorphism classes---we shall refine our description by isogeny class afterwards.
\begin{prop*}[{\cite[(9.4)]{LRS93}}]
The map $(\sE_i,t_i,j_i)_i\mapsto((V,\vp),(\sM_x)_x)$, where
\begin{enumerate}[$\bullet$]
\item $(V,\vp)$ is the generic fiber of $(\sE_i,t_i,j_i)_i$,
\item $x$ runs through all places of $\bf{F}$, and $\sM_x$ is as in \ref{ss:vpguys},
\end{enumerate}
yields a bijection from isomorphism classes of $\sD$-elliptic sheaves over $\ov\ka$ with zero lying over $o$ to isomorphism classes of pairs $((V,\vp),(\sM_x)_x)$, where
\begin{enumerate}[$\bullet$]
\item $(V,\vp)$ is a $\vp$-space with a right action of $D$ such that $V$ is a free right $D\otimes_\ka\ov\ka$-module of rank $1$,
\item $x$ runs over all places in $\bf{F}$, and $\sM_x$ is a local shtuka over $\ov\ka$ corresponding to a right $\sD_x\wh\otimes_\ka\ov\ka$-submodule of rank $1$ in $V\wh\otimes_\bf{F}\bf{F}_x$ that is stable under $\vp\wh\otimes_\bf{F}\bf{F}_x$,
\end{enumerate}
that satisfy the following conditions:
\begin{enumerate}[(i)]
\item the $\vp$-pair corresponding to $(V,\vp)$ satisfies the conditions in Proposition \ref{ss:dellipticsheavesphipairs},
\item the Morita reduction $\sM'_o$ of $\sM_o$ is effective minuscule of dimension $1$ and rank $n$,
\item the isogeny class of the Morita reduction $\sM'_\infty$ of $\sM_\infty$ has slope $-\frac1n$ under the Dieudonn\'e--Manin classification,
\item $\sM_x$ is \'etale for all closed points $x$ in $C\ssm\{o,\infty\}$, 
\item there exists a generator $b$ of $V$ as a free right $D\otimes_\ka\ov\ka$-module such that, for cofinitely many $x$, we have
  \begin{align*}
\sM_x = b(\sD_x\wh\otimes_\ka\ov\ka)
  \end{align*}
as right $\sD_x$-modules equipped with a $\sg$-semilinear automorphism, where we view $\sM_x$ as the corresponding free $\sD_x\wh\otimes_\ka\ov\ka$-submodule of $V\wh\otimes_\bf{F}\bf{F}_x$ of rank $1$ equipped with a $\sg$-semilinear automorphism.
\end{enumerate}
\end{prop*}

\subsection{}\label{ss:dellipticsheavesisogenyclass}
Next, we use Proposition \ref{ss:dellipticsheavesisomorphismclasses} to obtain a description of the isomorphism classes in a given isogeny class. Fix a $\vp$-pair $(\wt{\bf{F}},\wt\Pi)$ satisfying the conditions of Proposition \ref{ss:dellipticsheavesphipairs}, and write $\wt{D}$ for the endomorphism ring of $(V_{(\wt{\bf{F}},\wt\Pi)},\vp_{(\wt{\bf{F}},\wt\Pi)})$. Write $\cM_I(\ov\ka)_{(\wt{\bf{F}},\wt\Pi)}$ for the set of points in $\cM_I(\ov\ka)$ lying in the isogeny class corresponding to $(\wt{\bf{F}},\wt\Pi)$. Our goal is to describe $\cM_I(\ov\ka)_{(\wt{\bf{F}},\wt\Pi)}$ using algebraic group-theoretic data.

Let $k=n[\wt{\bf{F}}_{\wt{o}}:\bf{F}_o]/[\wt{\bf{F}}:\bf{F}]$, write $\ve_o^{\wt{o}}$ for the identity matrix in $\GL_{n-k}(\bf{F}_o)$, write $\ve_{\wt{o}}$ for the $k$-by-$k$ matrix
\begin{align*}
\ve_{\wt{o}}\deq
\begin{bmatrix}
0 & 1 & & \\
 &\ddots & \ddots & \\
 & & \ddots& 1 \\
\vpi_o & & &0
\end{bmatrix}\in\GL_k(\bf{F}_o),
\end{align*}
and write $\ve_o$ for the block matrix $\ve_o^{\wt{o}}\oplus\ve_{\wt{o}}$ in $\GL_n(\bf{F}_o)$. Note that, for any $\sD$-elliptic sheaf in $\cM_I(\ov\ka)_{(\wt{\bf{F}},\wt\Pi)}$, Proposition \ref{ss:dellipticsheavesisomorphismclasses} says that its local shtuka at $o$ corresponds to $\ve_o$ via \ref{ss:localshtukaconjugacy}, by the Dieudonn\'e--Manin classification.

Write $\sY_o$ for the subset
\begin{align*}
\sY_o\deq\{h_o\in\GL_n(\br{\bf{F}}_o)|h_o^{-1}\ve_o\sg_o(h_o)\in\ang{\vpi_o}\},
\end{align*}
and write $\br\sK_{o,m}$ for the subgroup
\begin{align*}
\br\sK_{o,m}\deq\ker\left(\GL_n(\br\cO_o)\rar\GL_n(\br\cO_o/\vpi_o^m)\right).
\end{align*}
\begin{prop*}
We have a bijection
\begin{align*}
\cM_I(\ov\ka)_{(\wt{\bf{F}},\wt\Pi)}\rar^\sim \wt{D}^\times\bs\left((D\otimes\bA^{\infty,o})^\times/\sK^{\infty,o}_I\times \sY_o/\br\sK_{o,m}\right)
\end{align*}
for which $g^{\infty,o}$ acts on the right-hand side via left multiplication on $(D\otimes\bA^{\infty,o})^\times$, and $\sg_o^{-1}$ acts on the right-hand side via sending $h_o\mapsto\ve_o\sg_o(h)$ on $\sY_o$.
\end{prop*}
\begin{proof}
We use the description of $\cM_I(\ov\ka)_{(\wt{\bf{F}},\wt\Pi)}$ given in Proposition \ref{ss:dellipticsheavesisomorphismclasses}. Under this description, the translation action of $\bZ$ sends $((V,\vp),(\sM_x)_x)$ to $((V,\vp),(\vp^l(\sM_x))_x)$ for all integers $l$. The characterization of $\sM_\infty'$ given in Proposition \ref{ss:dellipticsheavesisomorphismclasses}.(iii) implies that we may fix the position of $\sM'_\infty$ in $V\wh\otimes_\bf{F}\bf{F}_\infty$ via translating by $\bZ$.

Next, fix a generator $b$ of $V$ as a free right $D\otimes_\ka\ov\ka$-module. For all closed points $x$ in $C\ssm\infty$, let $h_x$ be an element of $D_x$ such that $bh_x$ generates $\sM_x$ over $\sD_x\wh\otimes_\ka\ov\ka$. Parts (iv) and (v) of Proposition \ref{ss:dellipticsheavesisomorphismclasses} show that $(h_x)_x$ is an element of $(D\otimes\bA^{\infty,o})^\times$ as $x$ runs over all closed points in $C\ssm\{o,\infty\}$, and the level-$I$ structures shows that $(h_x)_x$ is well-defined up to right multiplication by $\sK_I^{\infty,o}$. Similarly, $h_o$ yields a well-defined element of $\sY_o/\br\sK_{o,m}$.

Finally, the left action of $\wt{D}^\times$ on $(V,\vp)$ yields an embedding
\begin{align*}
\wt{D}^\times\inj{}(D\otimes\bA^{\infty,o})^\times\times\sY_o
\end{align*}
that preserves the action of $\wt{D}^\times$ on $(V,\vp)$, and taking the quotient of
\begin{align*}
(D\otimes\bA^{\infty,o})^\times/\sK^{\infty,o}_I\times \sY_o/\br\sK_{o,m}
\end{align*}
with respect to this action concludes the proof, by Proposition \ref{ss:dellipticsheavesisomorphismclasses}.
\end{proof}

\subsection{}\label{ss:dellipticsheavesadelicisogenyclass}
With Proposition \ref{ss:dellipticsheavesisogenyclass} in hand, we shall give an analogous description of $\Fix(\sg_o^r\circ\pi_*c_{\ov\ka})$. Write $\Fix(\sg_o^r\circ\pi_*c_{\ov\ka})_{(\wt{\bf{F}},\wt\Pi)}$ for the set of $y$ in $\Fix(\sg_o^r\circ\pi_*c_{\ov\ka})$ lying in the isogeny class corresponding to $(\wt{\bf{F}},\wt\Pi)$. Our goal is to describe $\Fix(\sg_o^r\circ\pi_*c_{\ov\ka})_{(\wt{\bf{F}},\wt\Pi)}$ in terms of the bijection from Proposition \ref{ss:dellipticsheavesisogenyclass}.

Write $\sK_{r,m}$ for the subgroup
\begin{align*}
\sK_{r,m}\deq\ker\left(\GL_n(\cO_{o,r})\rar\GL_n(\cO_{o,r}/\vpi_o^m)\right).
\end{align*}
For any element $\wt{d}$ in $\wt{D}^\times$, we say $\wt{d}$ is \emph{$r$-admissible} if there exists some $h_{\wt{d}}$ in $\GL_n(\br{\bf{F}}_o)$ such that
\begin{align*}
h^{-1}_{\wt{d}}\cdot\wt{d}^{-1}\cdot\N_o\ve_o\cdot\sg_o^r(h_{\wt{d}})=1,
\end{align*}
where we view $\wt{d}^{-1}$ as an element of $(\wt{D}\otimes_{\wt{\bf{F}}}\br{\bf{F}}_o)^\times=\GL_n(\br{\bf{F}}_o)$, and $\N_o$ denotes the norm map. Note that $r$-admissibility descends to a property on $\wt{D}^\times$-conjugacy classes in $\wt{D}^\times$.

For any group $G$ and element $g$ in $G$, we denote the centralizer of $g$ in $G$ by $G_g$. We write $\sY_{\wt{d}}$ for the double quotient space
\begin{align*}
\sY_{\wt{d}}\deq\wt{D}^\times_{\wt{d}}\bs\left((D\otimes\bA^{\infty,o})^\times/\sK_I^{\infty,o}\times\GL_n(\bf{F}_{o,r})/\sK_{r,m}\right).
\end{align*}
\begin{prop*}\label{prop:bijdisjoint}
We have a bijection from $\Fix(\sg_o^r\circ\pi_*c_{\ov\ka})_{(\wt{\bf{F}},\wt\Pi)}$ to the disjoint union
\begin{align*}
\coprod_{\wt{d}}\left\{\wt{D}^\times_{\wt{d}}\left(h^{\infty,o}\sK_I^{\infty,o},h_{\wt{d}}h_r\sK_{r,m}\right)\in\sY_{\wt{d}}\middle|(h^{\infty,o})^{-1}\wt{d}h^{\infty,o}\in\sK^{\infty,o}_Ig^o\sK_I^{\infty,o}\mbox{ and }h_r^{-1}\wt\ga\sg_o(h_r)\in\ang{\vpi_o}\cap\GL_n(\cO_{o,r})\right\},
\end{align*}
where $\wt{d}$ runs over all $r$-admissible $\wt{D}^\times$-conjugacy classes in $\wt{D}^\times$, and $\wt\ga$ is defined to be $h_{\wt{d}}^{-1}\ve_o\sg_o^r(h_{\wt{d}})$.
\end{prop*}
Note that $\wt\ga$ is fixed by $\sg_o^r$ and hence lies in $\GL_n(\bf{F}_{o,r})$.
\begin{proof}[Proof of Proposition \ref{prop:bijdisjoint}]
Begin by applying Proposition \ref{ss:dellipticsheavesisomorphismclasses} to $g^{o,-1}\sK_I^{\infty,o}g^o$ instead of $\sK_I^{\infty,o}$ to obtain an adelic description of $\Fix(\sg_o^r\circ\pi_*c_{\ov\ka})_{(\wt{\bf{F}},\wt\Pi)}$as in \cite[p.~53]{Lau96}. Next, simplify the resulting description by using our largeness hypothesis on $I$: this ensures that $\sK_I^{\infty,o}$ is small enough to apply the same argument as in \cite[(3.2.6)]{Lau96}. Conclude as in \cite[(3.2.7)]{Lau96}.
\end{proof}
Return to the proof of Proposition \ref{ss:galoistohecke}. Since $\de_o$ corresponds to $h_r^{-1}\wt\ga\sg_o(h_r)$ under Proposition \ref{ss:dellipticsheavesadelicisogenyclass}, our sum equals
\begin{align*}
\sum_{(\wt{\bf{F}},\wt\Pi)}\sum_{\wt{d}}\vol\left(\wt{D}^\times_{\wt{d}}\bs\left((D\otimes\bA^{\infty,o})^\times_{\wt{d}}\times\GL_{n,\wt\ga}^{\sg_o}(\bf{F}_o)\right)\right)\O_{\wt{d}}(f^{\infty,o})\TO_{\wt\ga,\sg_o}(\phi_{\tau,h})\tr\left((\pi_{\ov\ka,*}c_{\ov\ka}\times\tau)_y|(\cL_{I\ssm o,\xi,\ov\ka})_{z}\right),
\end{align*}
where $\wt{d}$ runs over all $r$-admissible $\wt{D}^\times$-conjugacy classes in $\wt{D}^\times$, and $\GL_{n,\wt\ga}^{\sg_o}$ denotes the $\sg_o$-centralizer of $\wt\ga$ in $\GL_n(\bf{F}_{o,r})$. As before, we do not explicate our Haar measures, but we choose them compatibly whenever possible. We may rewrite the volume factors in the above sum via the following lemma.
\begin{lem}[{\cite[(11.7)]{LRS93}}]\label{lem:volumelemma}
The embeddings
\begin{align*}
(\wt{D}\otimes\bA^{\infty,o})^\times_{\wt{d}}\inj{}(D\otimes\bA^{\infty,o})^\times_{\wt{d}}\mbox{ and }(\wt{D}\otimes \bf{F}_o)^\times\inj{}\GL_{n,\wt\ga}^{\sg_o}(\bf{F}_o)
\end{align*}
induced by the left action of $\wt{D}^\times$ as in the proof of Proposition \ref{ss:dellipticsheavesisogenyclass} are isomorphisms.
\end{lem}
Return to the proof of Proposition \ref{ss:galoistohecke}. Lemma \ref{lem:volumelemma} indicates that we can modify our volume terms to make our expression of interest equal
\begin{align*}
\sum_{(\wt{\bf{F}},\wt\Pi)}\sum_{\wt{d}}\vol\left(\wt{D}^\times_{\wt{d}}\bs(\wt{D}\otimes\bA^\infty)^\times_{\wt{d}}\right)\O_{\wt{d}}(f^{\infty,o})\TO_{\wt\ga,\sg_o}(\phi_{\tau,h})\tr\left((\pi_{\ov\ka,*}c_{\ov\ka}\times\tau)_y|(\cL_{I\ssm o,\xi,\ov\ka})_{z}\right).
\end{align*}
At this point, we have two goals: to rewrite this as a sum over certain conjugacy classes of $D^\times$ instead of $\wt{D}^\times$, and to rewrite the trace on $(\cL_{I\ssm o,\xi})_z$ in terms of algebraic group-theoretic data. For this, we shall use the following lemma, which transfers both conjugacy classes as well as volumes of their stabilizers.

\subsection{}\label{ss:dtildetodtransfer}
Write $\rn:D\rar \bf{F}$ for the reduced norm of $D$, and let $\ga$ be an element of $D^\times$. Write $\bf{F}'$ for the finite extension $\bf{F}[\ga]$ of $\bf{F}$. We say that $\ga$ is \emph{$r$-admissible} if $o(\rn\ga)=r$ and there exists a place $o'$ of $\bf{F}'$ above $o$ such that, for all other places $x'\neq o'$ of $\bf{F}'$ lying above $o$, we have $x'(\ga)=0$. Note that $r$-admissibility descends to a property on $D^\times$-conjugacy classes in $D^\times$. 
\begin{lem*}[{\cite[(11.9)]{LRS93}, \cite[(3.5.4)]{Lau96}}]
We have a bijection
\begin{align*}
&\{D^\times\mbox{-conjugacy classes in }D^\times\mbox{ that are }r\mbox{-admissible and elliptic in }D_\infty^\times\}\\
&\rar^\sim\coprod_{(\wt{\bf{F}},\wt\Pi)}\{\wt{D}^\times\mbox{-conjugacy classes in }\wt{D}^\times\mbox{ that are }r\mbox{-admissible}\},
\end{align*}
where $(\wt{\bf{F}},\wt\Pi)$ ranges over all $\vp$-pairs satisfying the conditions of Proposition \ref{ss:dellipticsheavesphipairs}. If we denote this bijection by $\ga\mapsto\wt{d}$, then the $D^\times_o=\GL_n(\bf{F}_o)$-conjugacy class of $\N_o\wt\ga$ equals that of $\ga$, and we have an equality of traces
\begin{align*}
\tr\left((\pi_{\ov\ka,*}c_{\ov\ka}\times\tau)_y|(\cL_{I\ssm o,\xi,\ov\ka})_{z}\right) = \Te_\xi(\ov\ga),
\end{align*}
where $\ov\ga$ is any elliptic element of $\ov{B}^\times$ with the same characteristic polynomial as $\ga$, and $\Te_\xi$ is the character of $\xi$. Furthermore, we have an equality of volumes
\begin{align*}
 \vol\left(\wt{D}^\times_{\wt{d}}\bs(\wt{D}\otimes\bA^\infty)^\times_{\wt{d}}\right) = a(\ga)\frac1{n\cdot\vol(\vpi_\infty^\bZ\bs\wt{D}_{\wt\infty}^\times)},
\end{align*}
where $a(\ga)$ is another volume factor
\begin{align*}
a(\ga) \deq \vol\left(\vpi_\infty^\bZ D^\times_\ga\bs(D\otimes\bA)^\times_\ga\right).
\end{align*}
\end{lem*}
Return to the proof of Proposition \ref{ss:galoistohecke}. The above lemma implies that our sum becomes
\begin{align*}
\sum_\ga a(\ga)\frac{\ve(\ga)}{n\cdot\vol(\vpi_\infty^\bZ\bs\wt{D}_{\wt\infty}^\times)}\O_\ga(f^{\infty,o})\ve(\ga)\TO_{\wt\ga,\sg_o}(\phi_{\tau,h})\Te_\xi(\ov\ga),
\end{align*}
where $\ga$ runs over all $D^\times$-conjugacy classes in $D^\times$ that are $r$-admissible and elliptic in $D^\times_\infty$, and $\ve(\ga)$ is the Kottwitz sign $\ve(\ga)\deq(-1)^{n/[\bf{F}_\infty[\ga]:\bf{F}_\infty]-1}$. Since $\phi_{\tau,h}$ and $f_{\tau,h}$ have matching orbital integrals, applying Lemma \ref{ss:dtildetodtransfer} further changes our sum to
\begin{align*}
\sum_\ga a(\ga)\frac{\ve(\ga)}{n\cdot\vol(\vpi_\infty^\bZ\bs\wt{D}_{\wt\infty}^\times)}\O_\ga(f^{\infty,o})\O_\ga(f_{\tau,h})\Te_\xi(\ov\ga) = \sum_\ga a(\ga)\frac{\ve(\ga)}{n\cdot\vol(\vpi_\infty^\bZ\bs\wt{D}_{\wt\infty}^\times)}\Te_\xi(\ov\ga)\O_\ga(f^{\infty,o}\times f_{\tau,h}).
\end{align*}
We now want to absorb the $\ve(\ga)\vol(\vpi_\infty^\bZ\bs\wt{D}_{\wt\infty}^\times)^{-1}\Te_\xi(\ov\ga)$ term into our orbital integral. In order to do so, we shall introduce the following special function on $D^\times_\infty=\GL_n(\bf{F}_\infty)$.

\subsection{}
We now introduce \emph{pseudo-coefficients}, which are certain functions in $C^\infty_c(\GL_n(\bf{F}_\infty))$. Recall that $\vpi_\infty$ acts trivially under $\xi$, which implies that $\xi$ has unitary central character. Therefore the local Jacquet--Langlands correspondence as in \ref{ss:localjacquetlanglands} yields an irreducible $L^2$ representation $\JL^{-1}(\xi)$ of $\GL_n(\bf{F}_\infty)$. Let $f_{\xi,\infty}$ in $C^\infty_c(\GL_n(\bf{F}_\infty))$ be the corresponding pseudo-coefficient of $\JL^{-1}(\xi)$ as in \cite[Section 5]{BR17}.
\begin{lem*}[{\cite[p.~2203]{BR17}, \cite[(13.8)]{LRS93}}]
Let $\ga$ be an element of $\GL_n(\bf{F}_\infty)$. Then the function $f_{\xi,\infty}$ satisfies
\begin{align*}
\O_\ga(f_{\xi,\infty}) =
\begin{cases}
\displaystyle\frac{\ve(\ga)}{\vol(\vpi_\infty^\bZ\bs\wt{D}_{\wt\infty}^\times)}\Te_\xi(\ov\ga) &\mbox{if }\ga\mbox{ is elliptic,}\\
0 & \mbox{otherwise.}
\end{cases}
\end{align*}
\end{lem*}

Return to the proof of Proposition \ref{ss:galoistohecke}. The pseudo-coefficient $f_{\xi,\infty}$ allows us to rewrite our sum as
\begin{align*}
\frac1n\sum_\ga a(\ga)\O_\ga(f_{\xi,\infty}\times f^{\infty,o}\times f_{\tau,h}).
\end{align*}
\subsection{}
Before using the Selberg trace formula, we first need some notation regarding automorphic representations. Write $\cA(\vpi_\infty^\bZ D^\times\bs(D\otimes\bA)^\times)$ for the $\ov\bQ_\ell$-vector space
\begin{align*}
\cA(\vpi_\infty^\bZ D^\times\bs(D\otimes\bA)^\times)\deq\{f:\vpi_\infty^\bZ D^\times\bs(D\otimes\bA)^\times\rar\ov\bQ_\ell\mid f\mbox{ is locally constant}\}.
\end{align*}
Then $\cA(\vpi_\infty^\bZ D^\times\bs(D\otimes\bA)^\times)$ has a left action of $\vpi_\infty^\bZ\bs(D\otimes\bA)^\times$ via right multiplication, and we see that this is a smooth representation of $\vpi_\infty^\bZ\bs(D\otimes\bA)^\times$. Now $\vpi_\infty^\bZ D^\times\bs(D\otimes\bA)^\times$ is compact because $D$ is a division algebra, so we obtain a decomposition
\begin{align*}
\cA(\vpi_\infty^\bZ D^\times\bs(D\otimes\bA)^\times) = \bigoplus_{\wt\Pi}\wt\Pi^{\oplus m(\wt\Pi)},
\end{align*}
where $\wt\Pi$ ranges over all irreducible admissible representations of $\vpi^\bZ_\infty\bs(D\otimes\bA)^\times$, and $m(\wt\Pi)$ is a non-negative integer.  If $m(\wt\Pi)$ is nonzero, we say $\wt\Pi$ is \emph{automorphic}. The weak multiplicity one theorem \cite[Theorem 3.3.(a)]{BR17} indicates that $m(\wt\Pi)$ is at most $1$.

Return to the proof of Proposition \ref{ss:galoistohecke}. Applying the Selberg trace formula to $\cA(\vpi_\infty^\bZ D^\times\bs(D\otimes\bA)^\times)$ shows that our sum of orbital integrals equals the sum of traces
\begin{align*}
\frac1n\sum_{\wt\Pi} m(\wt\Pi)\tr(f_{\xi,\infty}\times f^{\infty,o}\times f_{\tau,h}|\wt\Pi),
\end{align*}
where $\wt\Pi$ runs over all irreducible automorphic representations of $\vpi_\infty^\bZ\bs(D\otimes\bA)^\times$. Therefore the following proposition of Laumon--Rapoport--Stuhler concludes the proof of Proposition \ref{ss:galoistohecke}.
\end{proof}
\begin{prop}
We have an equality of traces
\begin{align*}
\sum_{\wt\Pi} m(\wt\Pi)\tr(f_{\xi,\infty}\times f^{\infty,o}\times f_{\tau,h}|\wt\Pi) = \tr(f^{\infty,o}\times f_{\tau,h}|[H_\xi]),
\end{align*}
where $\wt\Pi$ ranges over all irreducible automorphic representations of $\vpi_\infty^\bZ\bs(D\otimes\bA)^\times$.
\end{prop}
\begin{proof}
This is the generalization of \cite[(13.6)]{LRS93} as described in \cite[(13.8)]{LRS93}. This generalization relies on the existence of the pseudo-coefficient $f_{\xi,\infty}$, which we have already seen exists as in \cite[Section 5]{BR17}.
\end{proof}
We remark that the proof of this proposition uses similar point-counting methods as we do, except that Laumon--Rapoport--Stuhler immediately pass to the level of $(D\otimes\bA^\infty)^\times$, rather than remaining at the level of $\GL_n(\bf{F}_{o,r})$ at the place $o$. The latter is necessary for incorporating the contributions of $\phi_{\tau,h}$ and hence $f_{\tau,h}$.

\section{Local-global compatibility}\label{s:localglobal}
Our goal in this section is to prove that condition (c) in Lemma \ref{ss:firstinductivelemma} holds, which completes our proof of Theorem A. Along the way, we shall prove Theorem B. Returning to the local notation of \S\ref{s:deformationspaces}--\S\ref{s:lubintatetower} for a moment, we shall construct the desired virtual representation $\rho(\pi)$ of $W_F$ from our cohomology representation $[H_\xi]$.

Now revert back to global notation conventions. We start by recalling results of Laumon--Rapoport--Stuhler that specify the automorphic representations of $(D\otimes\bA^\infty)^\times$ occurring in $[H_\xi]$. Next, we use the \emph{strong multiplicity one theorem}, which says that automorphic representations of $(D\otimes\bA)^\times$ are determined by their local components at cofinitely many places, to convert Proposition \ref{ss:galoistohecke} into a global analog of condition (c) in Lemma \ref{ss:firstinductivelemma}. Finally, we introduce the \emph{global Jacquet--Langlands correspondence}, which enables us to pass between automorphic representations of $(D\otimes\bA)^\times$ and $\GL_n(\bA)$. This finishes the proof of Theorem B, our titular local-global compatibility result.

At this point, we want to use Theorem B to prove condition (c) in Lemma \ref{ss:firstinductivelemma}. To do so, we need to embed local representations of $\GL_n$ into global representations. We start by presenting such an embedding for $L^2$ representations, which uses the \emph{Deligne--Kazhdan simple trace formula}. Trace formula methods also allow us to embed cuspidal representations globally. From here, we use M\oe glin--Waldspurger's description of the discrete automorphic spectrum of $\GL_n$ in order to realize Speh modules in a global setting, and this enables us to complete the proof of condition (c) in Lemma \ref{ss:firstinductivelemma}.

\subsection{}
First, we introduce some notation on multiplicity spaces. Let $\wt\Pi^\infty$ be an irreducible admissible representation of $(D\otimes\bA^\infty)^\times$, and write $H_{\xi,\eta}^i(\wt\Pi^\infty)$ for the $\ov\bQ_\ell$-vector space
\begin{align*}
H^i_{\xi,\eta}(\wt\Pi^\infty) \deq \Hom_{(D\otimes\bA^\infty)^\times}(\wt\Pi^\infty,H^{i,\text{ss}}_{\xi,\eta}),
\end{align*}
that is, the multiplicity of $\wt\Pi^\infty$ in $H^{i,\text{ss}}_{\xi,\eta}$, where $H^{i,\text{ss}}_{\xi,\eta}$ denotes the semisimplification of $H^{i}_{\xi,\eta}$. We see that $H^{i}_{\xi,\eta}(\wt\Pi^\infty)$ is a continuous finite-dimensional representation of $G_\bf{F}$ over $\ov\bQ_\ell$, where $G_\bf{F}$ denotes the Galois group of the global field $\bf{F}$. Write $[H_\xi(\wt\Pi^\infty)]$ for the virtual representation $\sum_{i=0}^\infty(-1)^i H^i_{\xi,\eta}(\wt\Pi^\infty)$, which we see equals the multiplicity of $\wt\Pi^\infty$ in $[H_\xi]$. Therefore $\wt\Pi^\infty$ occurs in $[H_\xi]$ if and only if $[H_\xi(\wt\Pi^\infty)]$ is nonzero.

\subsection{}\label{ss:laumonrapoportstuhlerisotypiccomponents}
The question of whether $\wt\Pi^\infty$ occurs in $[H_\xi]$ is connected to automorphic representations of $\vpi_\infty^\bZ\bs(D\otimes\bA)^\times$ as well as the representation $\xi$ of $\ov{B}^\times$. Since $\vpi_\infty$ acts trivially under $\xi$, we see $\xi$ has unitary central character. Therefore the representation $\JL^{-1}(\xi)$ as in \ref{ss:localjacquetlanglands} is an irreducible $L^2$ representation of $\GL_n(\bf{F}_\infty)$ and hence isomorphic to a Steinberg module
\begin{align*}
\JL^{-1}(\xi) = \St_t(\pi_{0,\xi})
\end{align*}
as in \ref{ss:generalizedsteinbergspeh}, where $t$ is a positive divisor of $n$, $d=\frac{n}t$, and $\pi_{0,\xi}$ is an irreducible cuspidal representation of $\GL_d(\bf{F}_\infty)$ with unitary central character. Recall that we can also form the Speh module $\Sp_t(\pi_{0,\xi})$ as in \ref{ss:generalizedsteinbergspeh}, which is an irreducible smooth representation of $\GL_n(\bf{F}_\infty)$ The following result of Laumon--Rapoport--Stuhler determines precisely when $\wt\Pi^\infty$ occurs in $[H_\xi]$.
\begin{prop*}\hfill
  \begin{enumerate}[(i)]
  \item The virtual representation $[H_\xi(\wt\Pi^\infty)]$ of $G_\bf{F}$ is nonzero if and only if either $\wt\Pi^\infty\otimes\St_t(\pi_{0,\xi})$ or $\wt\Pi^\infty\otimes\Sp_t(\pi_{0,\xi})$ is an irreducible automorphic representation of $\vpi_\infty^\bZ\bs(D\otimes\bA)^\times$. Furthermore, we have $\dim[H_\xi(\wt\Pi^\infty)]=n$.
\item In addition, we have $H^i_{\xi,\eta}(\wt\Pi^\infty)=0$ either for all odd $i$ or all even $i$.
  \end{enumerate}
\end{prop*}
\begin{proof}
Part (ii) and the first statement of part (i) comprise the generalization of \cite[(14.7)]{LRS93} as described in \cite[(14.21)]{LRS93}. This generalization follows from the global Jacquet--Langlands correspondence between $D^\times$ and $\GL_n(\bf{F})$ (which is now known \cite[Theorem 3.2]{BR17}), by the remark made in \cite[(14.24)]{LRS93}. As for the dimension statement in part (i), this follows from Deligne's purity theorem, c.f. \cite[(14.11)(ii)]{LRS93}.
\end{proof}

\subsection{}\label{ss:globalconditionc}
Next, we use Proposition \ref{ss:galoistohecke} to obtain the following result, which resembles condition (c) in Lemma \ref{ss:firstinductivelemma}.
\begin{prop*}
Let $\tau$ be an element of $W_{\bf{F}_o}$ with $v_o(\tau)>0$, and let $h$ be a function in $C^\infty_c(\GL_n(\cO_o))$. Let $\wt\Pi^\infty$ be an irreducible admissible representation of $(D\otimes\bA^\infty)^\times$ for which $[H_\xi(\wt\Pi^\infty)]$ is nonzero. Then we have an equality of traces
\begin{align*}
\tr(f_{\tau,h}|\wt\Pi_o^\infty) = \tr(\tau|[H_\xi(\wt\Pi^\infty)])\tr(h|\wt\Pi_o^\infty),
\end{align*}
where $\wt\Pi_o^\infty$ denotes the component of $\wt\Pi^\infty$ at $o$.
\end{prop*}
For any irreducible admissible representation $\wt\Pi'^\infty$ of $(D\otimes\bA^\infty)^\times$, we write $\wt\Pi'^{\infty,o}$ for the component of $\wt\Pi'^\infty$ away from $o$.
\begin{proof}
Let $I$ be a sufficiently large finite closed subscheme of $C\ssm\infty$ such that $\wt\Pi^{\infty,o}$ has nonzero $\sK_I^{\infty,o}$-invariants and $h$ is invariant under $\sK_{I,o}^\infty$. Because $[H_\xi]$ is admissible, we see that $[H_\xi]^{\sK_I^\infty}$ is finite-dimensional, so $[H_\xi]$ can only contain finitely many non-isomorphic irreducible admissible representations $\wt\Pi'^\infty$ of $(D\otimes\bA^\infty)^\times$ with nonzero $\sK_I^\infty$-invariants. By applying the Chinese remainder theorem to these left $C^\infty_c((D\otimes\bA^{\infty,o})^\times\!/\!/\sK_I^{\infty,o})$-modules $\wt\Pi'^\infty$, we obtain a function $f^{\infty,o}$ in $C^\infty_c((D\otimes\bA^{\infty,o})^\times)$ satisfying the following properties:
\begin{enumerate}[(i)]
\item $\tr(f^{\infty,o}|\wt\Pi^{\infty,o})=1$,
\item if $\tr(f^{\infty,o}|\wt\Pi'^\infty)\neq0$ for one of our aforementioned $\wt\Pi'^\infty$, then $\wt\Pi'^{\infty,o}$ is isomorphic to $\wt\Pi^{\infty,o}$.
\end{enumerate}
In order to strengthen property (ii) to show that $\wt\Pi'^\infty$ is isomorphic to $\wt\Pi^\infty$ (hence extending our isomorphism to the component at $o$), we use the \emph{strong multiplicity one theorem}. We recall its statement below.

\subsection{}\label{ss:strongmultiplicityone}
Let $T$ be a finite set of places of $\bf{F}$. For any irreducible admissible representation $\wt\Pi$ of $\vpi^\bZ_\infty\bs(D\otimes\bA)^\times$, we write $\wt\Pi^T$ for the component of $\wt\Pi$ away from $T$, and we write $\wt\Pi_T$ for the component of $\wt\Pi$ at $T$.
\begin{thm*}[{\cite[Theorem 3.3.(b)]{BR17}}]
Let $\wt\Pi$ and $\wt\Pi'$ be two irreducible automorphic representations of $\vpi_\infty^\bZ\bs(D\otimes\bA)^\times$, and let $T$ be a finite set of places of $\bf{F}$. If $\wt\Pi^T$ is isomorphic to $\wt\Pi'^T$, then $\wt\Pi$ is isomorphic to $\wt\Pi'$.
\end{thm*}
We remark that \cite{BR17} deduces the strong multiplicity one theorem by using the global Jacquet--Langlands correspondence \cite[Theorem 3.2]{BR17} to reduce to the case of $\GL_n(\bf{F})$.

Return to the proof of Proposition \ref{ss:globalconditionc}. The strong multiplicity one theorem and Proposition \ref{ss:laumonrapoportstuhlerisotypiccomponents}.(i) imply that, if the hypothesis of property (ii) holds, then $\wt\Pi'^\infty$ is actually isomorphic to $\wt\Pi^\infty$. Therefore plugging $f^{\infty,o}$ into Proposition \ref{ss:galoistohecke} shows that
\begin{align*}
\tr(h|\wt\Pi_o^\infty)\tr(\tau|[H_\xi(\wt\Pi^\infty)]) = \tr(f^{\infty,o}\times h\times\tau|[H_\xi]) = \textstyle\frac1n\tr(f^{\infty,o}\times f_{\tau,h}|[H_\xi]) = \frac1n\dim[H_\xi(\wt\Pi^\infty)]\tr(f_{\tau,h}|\wt\Pi_o^\infty),
\end{align*}
and applying the second statement in Proposition \ref{ss:laumonrapoportstuhlerisotypiccomponents}.(i) concludes the proof of Proposition \ref{ss:globalconditionc}.
\end{proof}

\subsection{}\label{ss:globaljacquetlanglands}
To convert automorphic representations of $\GL_n(\bA)$ into those of $(D\otimes\bA)^\times$, we use the \emph{global Jacquet--Langlands correspondence}. Recall from \ref{ss:divisionalgebras} that $\text{Bad}$ denotes the set of places where $D$ ramifies.

From this point onwards, assume that $D_x$ is a division algebra for all places $x$ in $\text{Bad}$. We say that an irreducible discrete automorphic representation $\Pi$ of $\GL_n(\bA)$ is \emph{$D$-admissible} if, for all places $x$ in $\text{Bad}$, the local component $\Pi_x$ is a Steinberg module or Speh module, as a representation of $\GL_n(\bf{F}_x)$.

There exists a unique bijection \cite[Theorem 3.2]{BR17}
\begin{align*}
\JL:\left\{
  \begin{tabular}{c}
    isomorphism classes of irreducible\\
$D$-admissible representations of $\GL_n(\bA)$
  \end{tabular}\right\}\rar^\sim  \left\{\begin{tabular}{c}
    isomorphism classes of irreducible\\
    automorphic representations of $(D\otimes\bA)^\times$
  \end{tabular}\right\}
\end{align*}
such that, for all places $x$ of $\bf{F}$,
\begin{enumerate}[(i)]
\item if $x$ does not lie in $\text{Bad}$, then $\JL(\Pi)_x$ is isomorphic to $\Pi_x$,
\item if $x$ lies in $\text{Bad}$, and $\Pi_x$ is isomorphic to either $\St_t(\pi_0)$ or $\Sp_t(\pi_0)$ as in Definition \ref{ss:generalizedsteinbergspeh}, then $\JL(\Pi)_x$ is isomorphic to $\JL(\St_t(\pi_0))$, where the latter $\JL$ denotes the local Jacquet--Langlands correspondence from \ref{ss:localjacquetlanglands}.
\end{enumerate}

\subsection{}\label{ss:choiceofdivisionalgebra}
Before using the global Jacquet--Langlands correspondence to deduce Theorem B, let us fix our choice of division algebra. Fix three distinct places $x_1$, $x_2$, and $\infty$ of $\bf{F}$, and let $D$ be the central division algebra over $\bf{F}$ of dimension $n^2$ defined by
\begin{align*}
\inv_x(D) =
\begin{cases}
\frac1n &\mbox{ if }x=x_1,\,\\
-\frac1n &\mbox{ if }x=x_2,\,\\
0 & \mbox{otherwise,}
\end{cases}
\end{align*}
as in \ref{ss:divisionalgebras}. Let $\sD$ be a maximal order of $D$, which can be constructed using \ref{ss:orders} because division algebras split at cofinitely many places.

\subsection{}\label{prop:localglobalcompatibility}
We now proceed to prove Theorem B, using Proposition \ref{ss:globalconditionc} and the global Jacquet--Langlands correspondence.
\begin{prop*}
Let $\Pi$ be an irreducible discrete automorphic representation of $\GL_n(\bA)$ whose components at $x_1$, $x_2$, and $\infty$ are either irreducible $L^2$ representations or Speh modules of $\GL_n(\bf{F}_{x_1})$, $\GL_n(\bf{F}_{x_2})$, and $\GL_n(\bf{F}_\infty)$, respectively. Then there exists a unique $n$-dimensional semisimple continuous representation $R(\Pi)$ of $G_\bf{F}$ over $\ov\bQ_\ell$ such that, for all places $o$ of $\bf{F}$ not lying in $\{x_1,x_2,\infty\}$, the restriction of $R(\Pi)$ to $W_{\bf{F}_o}$ satisfies
\begin{align*}
\res{R(\Pi)}_{W_{\bf{F}_o}} = \rho(\Pi_o),
\end{align*}
where we identify $\ov\bQ_\ell$ with $\bC$.
\end{prop*}
\begin{proof}
The Chebotarev density theorem implies that, as $o$ varies over all places of $\bf{F}$ not lying in $\{x_1,x_2,\infty\}$, the conjugacy classes of geometric $q_o$-Frobenius elements at $o$ in $G_\bf{F}$ are dense. Now $n$-dimensional continuous semisimple representations of $G_\bf{F}$ are determined, up to isomorphism, by their characteristic polynomials, and said polynomials are continuous in $G_\bf{F}$, so we see that the above condition determines $R(\Pi)$ uniquely.

We turn to the existence of $R(\Pi)$. Our hypotheses indicate that $\Pi$ is $D$-admissible, so we can form the automorphic representation $\wt\Pi\deq\JL(\Pi)$ of $(D\otimes\bA)^\times$. Now $\wt\Pi_\infty=\Pi_\infty$ is either a Steinberg module or a Speh module, so Proposition \ref{ss:laumonrapoportstuhlerisotypiccomponents}.(i) ensures that there exists an irreducible smooth representation of $\ov{B}^\times/\vpi_\infty^\bZ$ such that $[H_\xi(\wt\Pi^\infty)]$ has dimension $n$. Finally, Proposition \ref{ss:globalconditionc} indicates that $\res{[H_\xi(\wt\Pi^\infty)]}_{W_{\bf{F}_o}}$ satisfies the defining condition Theorem A.(i) of $\rho(\Pi_o)$, so taking $R(\Pi)=[H_\xi(\wt\Pi^\infty)]$ yields the desired result.
\end{proof}

\subsection{}
Briefly return to the local notation of \S\ref{s:deformationspaces}--\S\ref{s:lubintatetower}. In order to convert Theorem B into a proof of condition (c) in Lemma \ref{ss:firstinductivelemma}, we must find some global automorphic representation $\wt\Pi$ of $\GL_n(\bA)$ such that our local representation $\pi$ is isomorphic to $\wt\Pi_o$ as a representation of $\GL_n(\bf{F}_o)$. That is, we must embed local representations into global ones. Recall that the two classes of possibilities for $\pi$ from condition (c) of Lemma \ref{ss:firstinductivelemma} that we must consider are
\begin{enumerate}[(1)]
\item essentially $L^2$ representations, which we will reduce to the case of $L^2$ representations via Proposition \ref{ss:unramifiedtwistsofrho},
\item Speh modules, which we reduce to the case of cuspidal representations by M\oe glin--Waldspurger's classification of the discrete automorphic spectrum of $\GL_n$. 
\end{enumerate}

\subsection{}\label{ss:embedl2}
Return to our global notation. We start by embedding local $L^2$ representations into global representations.
\begin{lem*}
Let $\pi$ be an irreducible $L^2$ representation of $\GL_n(\bf{F}_o)$. Then there exists an irreducible automorphic representation $\Pi$ of $\GL_n(\bA)$ whose
\begin{enumerate}[(i)]
\item component at $o$ is isomorphic to $\pi$,
\item components at $x_1$, $x_2$, and $\infty$ are irreducible $L^2$ representations of $\GL_n(\bf{F}_{x_1})$, $\GL_n(\bf{F}_{x_2})$, and $\GL_n(\bf{F}_\infty)$, respectively, with unitary central characters.
\end{enumerate}
\end{lem*}
\begin{proof}
This is proven precisely as in \cite[Corollary VI.2.5]{HT01} using the simple trace formula \cite[A.1.d Theorem]{DKV84}.
\end{proof}

\subsection{}\label{ss:embedcuspidal}
As a first step towards embedding Speh modules into a global setting, we first realize certain local cuspidal representations in a global context.
\begin{lem*}
Let $\pi$ be an irreducible cuspidal representation of $\GL_n(\bf{F}_o)$ with unitary central character. Then there exists an irreducible cuspidal automorphic representation $\Pi$ of $\GL_n(\bA)$ whose
\begin{enumerate}[(i)]
\item component at $o$ is isomorphic to $\pi$,
\item components at $x_1$, $x_2$, and $\infty$ are irreducible cuspidal representations of $\GL_n(\bf{F}_{x_1})$, $\GL_n(\bf{F}_{x_2})$, and $\GL_n(\bf{F}_\infty)$, respectively, with unitary central characters.
\end{enumerate}
\begin{proof}
This follows from the proof of \cite[(15.2)]{LRS93}, where we use $\infty$ instead of $x_3$ and $o$ instead of $x_0$.
\end{proof}
\end{lem*}
\subsection{}\label{ss:moeglinwaldspurger}
We want to use M\oe glin--Waldspurger's classification of the discrete automorphic spectrum of $\GL_n$. To present this classification, let us recall some notation from \ref{ss:bernsteinzelevinsky} on local representations. Briefly return to the local notation of \S\ref{s:deformationspaces}--\S\ref{s:lubintatetower}, and let $\{\De_1,\dotsc,\De_t\}$ be a collection of segments as in \cite[3.1]{Zel80} such that $\De_i$ does not precede $\De_j$ as in \cite[4.1]{Zel80} for $i<j$. Then each associated $Q(\De_i)$ is an irreducible essentially $L^2$ representation of $\GL_{n_i}(F)$, and recall that $Q(\De_1,\dotsc,\De_t)$ denotes the unique irreducible quotient of
\begin{align*}
\nInd_{P(F)}^{\GL_n(F)}\left(Q(\De_1)\otimes\dotsb\otimes Q(\De_t)\right),
\end{align*}
where $n=n_1+\dotsb+n_t$, and $P$ is the standard parabolic subgroup of $\GL_n$ with block sizes $(n_1,\dotsc,n_t)$.

Now return to the global notation. Work of M\oe glin--Waldspurger implies the following.
\begin{lem*}[{\cite[Theorem]{MW89}}]
Let $d$ be a positive divisor of $n$, write $d\deq\frac{n}t$, and let $\Pi_0$ be an irreducible automorphic cuspidal representation of $\GL_d(\bA)$. Then the restricted tensor product
\begin{align*}
\Pi\deq\sideset{}{'}\bigotimes_xQ\left(\{\Pi_{0,x}[\textstyle\frac{t-1}2]\},\dotsc,\{\Pi_{0,x}[\frac{1-t}2]\}\right),
\end{align*}
where $x$ ranges over all places of $\bf{F}$, is an irreducible discrete automorphic representation of $\GL_n(\bA)$.
\end{lem*}

\subsection{}\label{ss:embedspeh}
We apply Lemma \ref{ss:moeglinwaldspurger} to embed Speh modules into global representations.
\begin{lem*}
Let $\pi$ be a Speh module, as a representation of $\GL_n(\bf{F}_o)$. Then there exists an irreducible discrete automorphic representation $\Pi$ of $\GL_n(\bA)$ whose
\begin{enumerate}[(i)]
\item component at $o$ is isomorphic to $\pi$,
\item components at $x_1$, $x_2$, and $\infty$ are Speh modules, as representations of $\GL_n(\bf{F}_{x_1})$, $\GL_n(\bf{F}_{x_2})$, and $\GL_n(\bf{F}_\infty)$, respectively.
\end{enumerate}
\end{lem*}
\begin{proof}
Write $\pi$ as $\pi=\Sp_t(\pi_0)$, where $t$ is some positive divisor of $n$, and $\pi_0$ is an irreducible cuspidal representation of $\GL_{n/t}(\bf{F}_o)$ with unitary central character. Now Lemma \ref{ss:embedcuspidal} provides an irreducible cuspidal automorphic representation $\wt\Pi_0$ of $\GL_{n/t}(\bA)$ whose component at $o$ is isomorphic to $\pi_0$ and whose components at $x_1$, $x_2$, and $\infty$ are irreducible cuspidal representations of $\GL_n(\bf{F}_{x_1})$, $\GL_n(\bf{F}_{x_2})$, and $\GL_n(\bf{F}_\infty)$, respectively, with unitary central characters. Then restricted tensor product $\Pi$ associated with $\Pi_0$ as in Lemma \ref{ss:moeglinwaldspurger} has components
\begin{align*}
\Pi_o  = \Sp_t(\pi_0) = \pi,\, \Pi_{x_1}= \Sp_t(\Pi_{0,x_1}),\, \Pi_{x_2}= \Sp_t(\Pi_{0,x_2}),\mbox{ and }\Pi_\infty = \Sp_t(\Pi_{0,\infty}),
\end{align*}
as desired.
\end{proof}

\subsection{}\label{s:conditionc}
Return to the local notation of \S\ref{s:deformationspaces}--\S\ref{s:lubintatetower}. Thus $F$ is a local field of positive characteristic, $\cO$ denotes its ring of integers, $\vpi$ is a fixed uniformizer, and $\ka$ denotes $\cO/\vpi$. We will finally complete our proof of condition (c) in Lemma \ref{ss:firstinductivelemma} by explicitly embedding our local situation into the global situation of $\sD$-elliptic sheaves.
\begin{prop*}
Let $\pi$ be a Speh module or an irreducible essentially $L^2$ representation of $\GL_n(F)$. Then there exists an $n$-dimensional $\bQ_{\geq0}$-virtual continuous representation $\rho(\pi)$ of $W_F$ satisfying the trace condition
  \begin{align*}
    \tr(f_{\tau,h}|\pi) = \tr\left(\tau|\rho(\pi)\right)\tr(h|\pi)
  \end{align*}
for all $\tau$ in $W_F$ with $v(\tau)>0$ and $h$ in $C^\infty_c(\GL_n(\cO))$.
\end{prop*}
\begin{proof}
Our goal is to apply Theorem B. Let $C=\bP^1_\ka$ be our curve of interest, choose $o$ in \ref{ss:choiceofdivisionalgebra} to be a $\ka$-point of $C$, and choose $x_1$, $x_2$, and $\infty$ in \ref{ss:choiceofdivisionalgebra} to be distinct closed points in $C\ssm o$. Note that $\bf{F}_o$ is isomorphic to $F$.

Let us first consider the case when $\pi$ has unitary central character, that is, $\pi$ is either a Speh module or $L^2$. Then Lemma \ref{ss:embedl2} or Lemma \ref{ss:embedspeh}, respectively, yields an irreducible discrete automorphic representation $\Pi$ of $\GL_n(\bA)$ satisfying the hypotheses of Theorem B such that $\Pi_o$ is isomorphic to $\pi$. Proceeding to apply Theorem B to $\Pi$ shows that we may take $\rho(\pi)=\res{R(\Pi)}_{W_{\bf{F}_o}}$.

In the last remaining case when $\pi$ is an arbitrary irreducible essentially $L^2$ representation $\pi$ of $\GL_n(F)$, note that $\pi$ is isomorphic to an unramified twist of an $L^2$ representation. Thus the above work, along with Proposition \ref{ss:unramifiedtwistsofrho}, concludes the proof in this case.
\end{proof}

By using Proposition \ref{ss:conditionb}, Proposition \ref{s:conditionc}, and Proposition \ref{ss:conditiond} to verify that conditions (b)--(d) of Lemma \ref{ss:firstinductivelemma} hold and using Proposition \ref{ss:theoremabasecase} to check the $n=1$ base case, this concludes our proof by induction of Theorem A.

\section{The second inductive lemma: bijectivity of the correspondence}\label{s:secondinductivelemma}
Retain the local notation of \S\ref{s:deformationspaces}--\S\ref{s:lubintatetower}. The purpose of this section is to prove Theorem C, i.e. that $\pi\mapsto\rec\pi$ yields a bijection from isomorphism classes of irreducible cuspidal representations of $\GL_n(F)$ to isomorphism classes of $n$-dimensional irreducible continuous representations of $W_F$. We start by collecting some facts on restricting representations of $W_F$ to representations of $W_E$ for cyclic extensions $E/F$. Next, we use these facts to motivate \emph{automorphic base change}, an analogous operation that turns representations of $\GL_n(F)$ into representations of $\GL_n(E)$. Automorphic base change yields section's first main result: a lemma which allows us to prove Theorem C by inducting on $n$, provided that we verify certain conditions.

We then proceed to verify the conditions needed for this inductive lemma, thus completing the proof of Theorem C. The main ingredients are Theorem A, Theorem B, and our nearby cycles calculation from \S\ref{s:nearbycycles}. More precisely, we apply results from \S\ref{s:nearbycycles} to show that the preimage of any unramified representation under $\rec$ remained unramified, and we use Theorem B to show that $\rec$ is compatible with automorphic base change as well as twisting by characters.

\subsection{}\label{ss:galoissidebasechange}
We begin by establishing some notation on field extensions. Let $E$ be a finite extension of $F$ inside $F^\sep$. For any finite-dimensional semisimple continuous representation $\rho'$ of $W_E$ and any map $\al$ in $G_F$, write $\rho'^\al$ for the finite-dimensional semisimple continuous representation of $W_{\al^{-1}(E)}$ given by $\tau\mapsto\rho'(\al\circ\tau\circ\al^{-1})$ for all $\tau$ in $W_{\al^{-1}(E)}$.

Now assume that $E/F$ is a cyclic extension. Hence $\al^{-1}(E)=E$. Write $r$ for the degree of $E/F$, and fix a generator $\sg$ of $\Gal(E/F)=W_F/W_E$, which identifies it with $\bZ/r\bZ$. Write $\fK(E/F)$ for the set of group homomorphisms $\Gal(E/F)\rar\bC^\times$, and interpret $\fK(E/F)$ as a set of characters of $W_F$. Note that $\fK(E/F)$ acts on the set of isomorphism classes of irreducible continuous finite-dimensional representations of $W_F$ via twisting.

When $r$ is prime, one can use Frobenius reciprocity and the Mackey formula to verify the following assertions:
\begin{enumerate}[(i)]
\item Let $\rho'$ be an irreducible continuous finite-dimensional representation of $W_E$. Then $\rho'$ extends to a representation $\rho$ of $W_F$ if and only if $\rho'=\rho'^\sg$. Furthermore, if this is the case, then there are $r$ isomorphism classes of such $\rho'$, and any two differ by a twist of a character in $\fK(E/F)$.
\item Let $\rho$ be an irreducible continuous finite-dimensional representation of $W_F$. Then $\res\rho_{W_E}$ is reducible if and only if the stabilizer of $\rho$ in $\fK(E/F)$ is nontrivial (and hence all of $\fK(E/F)$). Furthermore, if this is the case, then
  \begin{align*}
    \res\rho_{W_E} = \rho'\oplus\dotsb\oplus\rho'^{\sg^{r-1}}
  \end{align*}
  for some irreducible continuous finite-dimensional representation $\rho'$ of $W_E$ satisfying $\rho'\neq\rho'^\sg$, and $\rho$ is the unique finite-dimensional semisimple continuous representation of $W_F$ with this property.
\end{enumerate}
We remark that the same argument works in the global setting, for which we shall use entirely analogous notation.

\subsection{}\label{ss:basechangeliftcuspidal}
Next, we recall \emph{automorphic base change} in the local setting. For any finite extension $E$ of $F$ inside $F^\sep$, any irreducible smooth representation $\pi'$ of $\GL_n(E)$, and any map $\al$ in $G_F$, write $\pi'^\al$ for the irreducible smooth representation of $\GL_n(\al^{-1}(E))$ given by $g\mapsto\pi'(\al(g))$ for all $g$ in $\GL_n(\al^{-1}(E))$. Note that $\Te_{\pi'^\al}=\Te_{\pi'}\circ\al$.

As in \ref{ss:galoissidebasechange}, specialize to cyclic extensions $E$ of $F$, and adopt the notation of \ref{ss:galoissidebasechange} as well. Since $\Art$ induces an isomorphism $F^\times/\Nm_{E/F}(E^\times)\rar^\sim\Gal(E/F)$, we can view $\fK(E/F)$ as a set of characters of $\GL_n(F)$ via precomposition with $\det\circ\Art$. Thus $\fK(E/F)$ acts on the set of isomorphism classes of irreducible smooth representations of $\GL_n(F)$ by twisting.

First, let $\pi$ be an irreducible cuspidal representation of $\GL_n(F)$. Write $u$ for the cardinality of the stabilizer of $\pi$ in $\fK(E/F)$. Then $u$ divides $n$, and there exists an irreducible cuspidal representation $\pi'$ of $\GL_{n/u}(E)$ such that
\begin{align*}
Q\left(\{\pi'\},\dotsc,\{\pi'^{\sg^{u-1}}\}\right)
\end{align*}
is an irreducible generic representation of $\GL_n(E)$ and is isomorphic to the base change lift of $\pi$ to $\GL_n(E)$ \cite[(II.4.12, prop.)]{HL11}. In particular, the description of genericity from \ref{ss:bernsteinzelevinsky} indicates that $\pi'$ is isomorphic to $\pi'^\sg$ if and only if $u=1$. Note that this is the automorphic analogue of \ref{ss:galoissidebasechange}.

\subsection{}\label{ss:basechangeliftlocal}
We now turn to local automorphic base change for generic representations. For any irreducible cuspidal representation $\tau$ of $\GL_d(F)$, write $\De(\tau,m)$ for the segment $\{\tau,\dotsc,\tau[m-1]\}$ as in \cite[3.1]{Zel80}. Recall from \ref{ss:bernsteinzelevinsky} that every irreducible generic representation $\pi$ of $\GL_n(F)$ is of the form
\begin{align*}
\nInd_{P(F)}^{\GL_n(F)}\left(Q(\De(\pi_1,m_1))\otimes\dotsb\otimes Q(\De(\pi_t,m_t))\right),
\end{align*}
where $n=d_1m_1+\dotsb+d_tm_t$, the $\pi_i$ are irreducible cuspidal representations of $\GL_{d_i}(F)$, and $P$ is the standard parabolic subgroup of $\GL_n$ with block sizes $(d_1m_1,\dotsc,d_tm_t)$. Furthermore, none of the $\De(\pi_i,m_i)$ are linked.

At this point, suppose that $\pi$ also has $\fK(E/F)$-regular segments as in \cite[(II.3.4)]{HL11}. Write $u_i$ for the divisor of $d_i$ and $\pi'_i$ for the irreducible cuspidal representation of $\GL_{d_i/u_i}(E)$ associated with $\pi_i$ as in \ref{ss:basechangeliftcuspidal}. Then the base change lift $\pi_E$ of $\pi$ to $\GL_n(E)$ is isomorphic to \cite[(II.4.4)]{HL11}, \cite[(II.4.12, cor.)]{HL11}
\begin{align*}
Q\left(\De(\pi'_1,m_1),\dotsc,\De(\pi_1'^{\sg^{u_1-1}},m_1),\dotsc,\De(\pi'_t,m_t),\dotsc,\De(\pi_t'^{\sg^{u_t-1}},m_t)\right).
\end{align*}
Note that $\pi_E$ is an irreducible generic representation of $\GL_n(E)$.

Next, let $\xi$ be any irreducible generic representation of $\GL_n(F)$ with $\fK(E/F)$-regular segments such that the base change lift of $\xi$ to $\GL_n(E)$ is isomorphic to $\pi_E$. Then $\xi$ must be isomorphic to
\begin{align*}
\nInd_{P(F)}^{\GL_n(F)}\left(Q(\De(\chi_1\cdot\pi_1,m_1))\otimes\dotsb\otimes Q(\De(\chi_t\cdot\pi_t,m_t))\right)
\end{align*}
for some $\chi_1,\dotsc,\chi_t$ in $\fK(E/F)$ \cite[(II.4.13, lem.)]{HL11}. Finally, every $\sg$-stable irreducible generic representation $\pi'$ of $\GL_n(E)$ is the base change lift of some irreducible generic representation of $\GL_n(F)$ with $\fK(E/F)$-regular segments \cite[(II.1.4)]{HL11}.

\subsection{}\label{ss:basechangeliftglobal}

We will also need automorphic base change in the global setting, for which we briefly return to our global notation. Hence $\bf{F}$ is a global function field, and $\bA$ denotes its ring of adeles. Let $\bf{E}$ be a cyclic extension of $\bf{F}$, write $\bA_\bf{E}$ for the ring of adeles of $\bf{E}$, and fix a generator $\sg$ of $\Gal(\bf{E}/\bf{F})=W_{\bf{F}}/W_{\bf{E}}$. Then $\sg$ acts on the set of isomorphism classes of irreducible discrete automorphic representations of $\GL_n(\bA_\bf{E})$ via precomposition.

Let $\Pi$ be an irreducible cuspidal automorphic representation of $\GL_n(\bA)$. Then there exists a unique irreducible discrete automorphic representation $\Pi_{\bf{E}}$ of $\GL_n(\bA_{\bf{E}})$ such that, for all places $o$ of $\bf{F}$ and places $o'$ of $\bf{E}$ lying above $o$, the base change lift of $\Pi_o$ to $\GL_n(\bf{E}_{o'})$ is isomorphic to $\Pi_{\bf{E},o'}$ \cite[(IV.1.3)]{HL11}. Next, let $\Xi$ be any irreducible cuspidal automorphic representation of $\GL_n(\bA)$ such that $\Xi_\bf{E}$ is isomorphic to $\Pi_\bf{E}$. Then $\Xi$ must be isomorphic to $\Pi\otimes(X\circ\det)$ for some character $X:\bA^\times/(\bf{F}^\times\Nm_{\bf{E}/\bf{F}}(\bA_\bf{E}^\times))\rar\bC^\times$ \cite[Theorem III.3.1]{AC89}.\footnote{The proof given here is stated for number fields, but it only uses the relationship between unramified local $L$-functions and Satake parameters, as well as the fact that $L(\Pi\times\Pi',s)$ has a pole at $s=1$ if and only if $\Pi\cong\Pi'^\vee$. In particular, it carries over to arbitrary global fields.} Finally, every $\sg$-stable irreducible cuspidal representation $\Pi'$ of $\GL_n(\bA_\bf{E})$ is isomorphic to $\Pi_{\bf{E}}$ for some irreducible cuspidal automorphic representation $\Pi$ of $\GL_n(\bA)$. 

\subsection{}\label{ss:automorphismscompatibility}
Return to the local notation of \S\ref{s:deformationspaces}--\S\ref{s:lubintatetower}. We first verify that $\rec$ is compatible with automorphisms.
\begin{lem*}
For any $\al$ in $G_F$ and any irreducible smooth representation $\pi'$ of $\GL_n(E)$, we have an isomorphism
  \begin{align*}
    \rec(\pi'^\al) = \rec(\pi')^\al.
  \end{align*}

\end{lem*}
\begin{proof}
  We immediately see $\al$ preserves the Tate twist $(\textstyle\frac{1-n}2)$, so it suffices to show that $\rho(\pi'^\al) = \rho(\pi')^\al$. We will do so by verifying that $\rho(\pi')^\al$ satisfies the defining property of $\rho(\pi'^\al)$ as in Theorem A.(i). Write $\cO_E$ for the ring of integers of $E$. Let $\tau$ be an element of $W_{\al^{-1}(E)}$ satisfying $r\deq v(\tau)>0$, and let $h$ be a function in $C^\infty_c(\GL_n(\al^{-1}(\cO_E)))$. Now \ref{ss:twistedcharacters} yields
  \begin{align*}
    \tr\left(f_{\tau,h}|\pi'^\al\right) = \tr\left((\phi_{\tau,h},\sg)|\pi'^\al/E_r\right) = \int_{\GL_n(\al^{-1}(E))}\!\dif\de'\,\phi_{\tau,h}(\de')\Te_{\pi'^\al}(\N\de'),
  \end{align*}
  where $\phi_{\tau,h}$ is our test function in $C^\infty_c(\GL_n(\al^{-1}(E)_r))$ from \S\ref{s:deformationspaces}. Fixing a uniformizer $\vpi_E$ of $E$ and expanding the definition of $\phi_{\tau,h}$, we see that this integral equals
  \begin{align*}
    \int_{\GL_n(\al^{-1}(\cO_E))\diag(\al^{-1}(\vpi_E),1,\dotsc,1)\GL_n(\al^{-1}(\cO_E))}\!\dif\de'\,\tr(\tau\times h|[R\psi_{\de'}])\Te_{\pi'}(\al(\N\de')).
  \end{align*}
  Upon making the change of variables $\de=\al(\de')$, our integral becomes
  \begin{align*}
    &\int_{\GL_n(\cO_E)\diag(\vpi_E,1,\dotsc,1)\GL_n(\cO_E)}\!\dif\de\,\tr(\tau\times h|[R\psi_{\al^{-1}(\de)}])\Te_{\pi'}(\N(\al(\de')))\\
    =\,&\int_{\GL_n(\cO_E)\diag(\vpi_E,1,\dotsc,1)\GL_n(\cO_E)}\!\dif\de\,\tr((\al\circ\tau\circ\al^{-1})\times h\circ\al^{-1}|[R\psi_\de])\Te_{\pi'}(\N\de)\\
    =\,&\int_{\GL_n(E)}\!\dif\de\,\phi_{\al\circ\tau\circ\al^{-1},h\circ\al^{-1}}(\de)\Te_{\pi'}(\N\de) = \tr\left((\phi_{\al\circ\tau\circ\al^{-1},h\circ\al^{-1}},\sg)|\pi'/E_r\right).
  \end{align*}
  Applying \ref{ss:twistedcharacters} once more indicates that this equals $\tr\left(f_{\al\circ\tau\circ\al^{-1},h\circ\al^{-1}}|\pi'\right)$. From here, Theorem A.(i) shows that
  \begin{align*}
    \tr\left(f_{\al\circ\tau\circ\al^{-1},h\circ\al^{-1}}|\pi'\right) = \tr\big(\al^{-1}\circ\tau\circ\al|\rho(\pi')\big)\tr\big(h\circ\al^{-1}|\pi'\big) = \tr\left(\tau|\rho(\pi')^\al\right)\tr\big(h\circ\al^{-1}|\pi'\big).
  \end{align*}
Thus all that remains is to prove that $\tr\big(h\circ\al^{-1}|\pi'\big)=\tr\big(h|\pi'^\al\big)$. Now $\tr\big(h|\pi'^\al\big)$ is the trace of the operator
  \begin{align*}
    v\mapsto\int_{\GL_n(\al^{-1}(E))}\!\dif g'\,h(g')\pi'(\al(g'))v = \int_{\GL_n(E)}\!\dif g\,h(\al^{-1}(g))\pi'(g)v,
  \end{align*}
  where we have made the change of variables $g=\al(g')$. This shows that $\tr(h|\pi'^\al) = \tr(h\circ\al^{-1}|\pi')$, concluding our proof of Lemma \ref{ss:automorphismscompatibility}.
\end{proof}

\subsection{}\label{ss:secondinductivelemma}
Finally, we can introduce our inductive lemma.
\begin{lemsecond}
Assume that the following conditions hold for all irreducible smooth representations $\pi$ of $\GL_n(F)$:
\begin{enumerate}[(a)]
 \item Theorem C is true for $n'<n$,
  \item if $n=1$, then $\rec\pi$ is isomorphic to $\pi\circ\Art^{-1}$,
  \item if $\pi$ is isomorphic to a subquotient of the normalized parabolic induction of
    \begin{align*}
      \pi_1\otimes\dotsb\otimes\pi_t,
    \end{align*}
    where the $\pi_i$ are irreducible smooth representations of $\GL_{n_i}(F)$ for which $n_1+\dotsb+n_t=n$. Then
    \begin{align*}
      \rec(\pi) = \rec(\pi_1)\oplus\dotsb\oplus\rec(\pi_t),
    \end{align*}
  \item for all smooth characters $\chi:F^\times\rar\bC^\times$, we have
    \begin{align*}
      \rec(\pi\otimes(\chi\circ\det)) = \rec(\pi)\otimes\rec(\chi),
    \end{align*}
  \item if $\pi$ is generic, then for all cyclic extensions $E/F$ of prime degree, we have
    \begin{align*}
      \rec(\pi_E) = \res{\rec(\pi)}_{W_E},
    \end{align*}
    where $\pi_E$ denotes the base change lift of $\pi$ to $\GL_n(E)$,
  \item if $\rec(\pi)$ is an unramified representation of $W_F$, then $\pi$ is an unramified representation of $\GL_n(F)$.
  \end{enumerate}
Then Theorem C is true for $n$.
\end{lemsecond}
In the first part of this section, our goal is to prove the second inductive lemma. To this end, starting from this point, assume that conditions (a)--(f) hold. We have already proved condition (b) in Proposition \ref{ss:theoremabasecase} and condition (c) in Theorem A. After proving Lemma \ref{ss:secondinductivelemma}, we shall prove conditions (d)--(f) in the remainder of this section, using results from \S\ref{s:nearbycycles} and \S\ref{s:localglobal}.

\subsection{}\label{ss:solvablebasechange}
We begin by immediately upgrading condition (e) to finite solvable extensions $E/F$.
\begin{prop*}
  Let $\pi$ be generic. Then, for all finite solvable extensions $E/F$, we have
\begin{align*}
      \rec(\pi_E) = \res{\rec(\pi)}_{W_E},
\end{align*}
where $\pi_E$ denotes the base change lift of $\pi$ to $\GL_n(E)$.
\end{prop*}
\begin{proof}
  We induct on the degree of $E$ over $F$. The result is immediate for $E=F$, and condition (e) ensures that it holds for prime $[E:F]$. In general, the solvability of $E/F$ yields a Galois subextension $E\supset E'\supseteq F$ such that $E/E'$ is cyclic of prime degree and $E'/F$ is solvable. Writing $\pi_{E'}$ for the base change lift of $\pi$ to $\GL_n(E')$, we see that $\rec(\pi_{E'}) = \res{\rec(\pi)}_{W_{E'}}$ by the inductive hypothesis. Transitivity of base change indicates that $\pi_E$ is also the base change lift of $\pi_{E'}$ to $\GL_n(E)$, so condition (e) gives us
  \begin{gather*}
\rec(\pi_E) = \res{\rec(\pi_{E'})}_{W_E} = \res{\rec(\pi)}_{W_E}.\qedhere
  \end{gather*}
\end{proof}

\subsection{}\label{ss:unramifiedsolvablebasechange}
We make the following observation. Let $\rho$ be a finite-dimensional semisimple continuous representation of $W_F$. Then $\rho$ is smooth, so $\res\rho_{I_E}$ is trivial for some finite Galois extension $E$ of $F$. As $F$ is a nonarchimedean local field, $E/F$ is solvable. Altogether, $\rho$ becomes unramified after passing to a solvable extension.

In particular, by letting $\rho=\rec\pi$ for any irreducible generic representation $\pi$ of $\GL_n(F)$, Proposition \ref{ss:solvablebasechange} and condition (f) imply that the base change lift $\pi_E$ of $\pi$ to $\GL_n(E)$ is unramified. 

\subsection{}\label{ss:cuspidaltoirreducible}
We now check that $\rec$ sends irreducible cuspidal representations to irreducible ones.
\begin{prop*}
Suppose that $\pi$ is irreducible cuspidal. Then $\rec\pi$ is an irreducible representation of $W_F$.
\end{prop*}
\begin{proof}
  This is immediate for $n=1$, so suppose that $n\geq2$. Now \ref{ss:unramifiedsolvablebasechange} yields a finite solvable extension $E/F$ for which the base change lift $\pi_E$ of $\pi$ to $\GL_n(E)$ is unramified. As $n\geq2$, we see that $\pi_E$ cannot be cuspidal. Therefore, by replacing $E/F$ with a subextension if necessary, we can find a Galois subextension $E\supset E'\supseteq F$ satisfying the following properties:
  \begin{enumerate}[$\bullet$]
  \item $E/E'$ is cyclic of prime degree $r$,
  \item $E'/F$ is solvable,
  \item the base change lift $\pi_{E'}$ of $\pi$ to $\GL_n(E')$ remains cuspidal.
  \end{enumerate}
  Since $\pi_E$ is also the base change lift of $\pi_{E'}$ to $\GL_n(E)$, we can apply \ref{ss:basechangeliftcuspidal} to obtain a positive integer $u$ and an irreducible cuspidal representation $\pi'$ of $\GL_{n/u}(E)$ such that $\pi_E$ is isomorphic to
  \begin{align*}
    Q\left(\{\pi'\},\dotsc,\{\pi'^{\sg^{u-1}}\}\right),
  \end{align*}
  where $\sg$ is a generator of $\Gal(E/E')$. Thus condition (e) and condition (c) indicate that $\rec(\pi_{E'})$ restricts to 
  \begin{align*}
    \res{\rec(\pi_{E'})}_{W_E} = \rec(\pi_E) = \rec(\pi')\oplus\dotsb\oplus\rec(\pi')^{\sg^{r-1}}.
  \end{align*}
  Now $u$ is the cardinality of a subgroup of $\fK(E/E')$, so $u$ divides $r$. But $\pi_E$ is not cuspidal, which forces $u=r$. Hence $\pi'$ is not isomorphic to $\pi'^\sg$. From here, Theorem C for $\pi'$ and Lemma \ref{ss:automorphismscompatibility} show that $\rec(\pi')$ is not isomorphic to $\rec(\pi'^\sg) = \rec(\pi')^\sg$, so \ref{ss:galoissidebasechange}.(ii) indicates that $\rec(\pi_{E'})$ is irreducible.

  Finally, because Proposition \ref{ss:solvablebasechange} shows that $\rec(\pi_{E'}) = \res{\rec(\pi)}_{W_{E'}}$, we see that $\rec(\pi)$ is irreducible as well.
\end{proof}

\subsection{}\label{ss:lemmas=1step}
Next, we establish a few lemmas on the bijectivity of $\rec$ in special cases.
\begin{lem*}
Let $\rho$ be an $n$-dimensional irreducible continuous representation of $W_F$, and let $E/F$ be a cyclic extension of prime degree. If $\res\rho_{W_E}$ is reducible, then there exists a unique irreducible cuspidal representation $\pi$ of $\GL_n(E)$ such that $\rec\pi=\rho$.
\end{lem*}
\begin{proof}
  Write $r$ for the degree of $E/F$, and fix a generator $\sg$ of $\Gal(E/F)$. Then \ref{ss:galoissidebasechange}.(ii) yields an irreducible continuous finite-dimensional representation $\rho'$ of $W_E$ satisfying $\res\rho_{W_E}=\rho'\oplus\dots\oplus\rho'^{\sg^{r-1}}$ and $\rho'\neq\rho'^\sg$. Now Theorem C for $\rho'$ provides a unique irreducible cuspidal representation $\pi'$ of $\GL_{n/r}(E)$ such that $\rec(\pi')$ is isomorphic to $\rho'$.

Define $\wt\pi$ to be the irreducible smooth representation
\begin{align*}
  \wt\pi\deq Q\left(\{\pi'\},\dotsc,\{\pi'^{\sg^{r-1}}\}\right)
\end{align*}
of $\GL_n(E)$. We see that $\wt\pi$ is generic, and van Dijk's formula \cite[Theorem 5.9]{Lem16} implies that $\wt\pi$ is $\sg$-stable. Therefore \ref{ss:basechangeliftlocal} shows that $\wt\pi$ is the base change lift of some irreducible generic representation $\pi$ of $\GL_n(F)$. Furthermore, because the segments $\{\pi'\},\dotsc,\{\pi'^{\sg^{r-1}}\}$ in $\wt\pi$ have length $1$, our description of possibilities for $\pi$ in \ref{ss:basechangeliftlocal} indicates $\pi$ is cuspidal.

Condition (e), condition (c), and Lemma \ref{ss:automorphismscompatibility} yield
\begin{align*}
\res{\rec(\pi)}_{W_E} = \rec(\wt\pi) = \rec(\pi')\oplus\dotsb\oplus\rec(\pi'^{\sg^{r-1}}) = \rho'\oplus\dotsb\oplus\rho'^{\sg^{r-1}}.
\end{align*}
The uniqueness of $\rho$ from \ref{ss:galoissidebasechange}.(ii) therefore indicates that $\rec\pi=\rho$. This takes care of the existence of $\pi$.

We now tackle uniqueness. Suppose $\xi$ is any irreducible cuspidal representation of $\GL_n(F)$ with $\rec(\xi)=\rho$, and write $\xi_E$ for the base change lift of $\xi$ to $\GL_n(E)$. Condition (e) and condition (c) show that $\rec(\xi_E)=\res\rho_{W_E}$ is reducible, so Proposition \ref{ss:cuspidaltoirreducible} implies that $\xi_E$ is not cuspidal. Therefore our description of base change in \ref{ss:basechangeliftcuspidal} gives
\begin{align*}
  \xi_E = Q\left(\{\xi'\},\dotsc,\{\xi'^{\sg^{r-1}}\}\right)
\end{align*}
for some irreducible cuspidal representation $\xi'$ of $\GL_{n/r}(E)$. Applying condition (e), condition (c), and Lemma \ref{ss:automorphismscompatibility} once more yields
\begin{align*}
\rho'\oplus\dotsb\oplus\rho'^{\sg^{r-1}} = \res{\rec(\pi)}_{W_E} = \rec(\xi_E) = \rec(\xi')\oplus\dotsb\oplus\rec(\xi'^{\sg^{r-1}}) = \rec(\xi')\oplus\dotsb\oplus\rec(\xi')^{\sg^{r-1}}.
\end{align*}
Theorem C for $\xi'$ tells us that $\rec(\xi')$ is irreducible. As $\rho'$ is irreducible too, we see that $\rho'$ is isomorphic to $\rec(\xi')^{\sg^i}$ for some $0\leq i\leq r-1$. After replacing $\xi'$ with $\xi'^{\sg^i}$, we get $\rho' = \rec(\xi')$. Similarly replacing $\pi'$ with $\pi'^{\sg^i}$ for some $0\leq i\leq r-1$ yields $\rho'=\rec(\pi')$, and hence $\pi'=\xi'$ via Theorem C for $\rho'$. This in turn yields $\wt\pi=\xi_E$. Since $\wt\pi$ is the base change lift of $\pi$ to $\GL_n(E)$, we see from \ref{ss:basechangeliftlocal} that $\xi$ must be isomorphic to $\chi\cdot\pi$ for some $\chi$ in $\fK(E/F)$. But the stabilizer of $\pi$ in $\fK(E/F)$ has cardinality $r$ by \ref{ss:basechangeliftcuspidal}. Thus we finally obtain $\xi=\chi\cdot\pi=\pi$, as desired.
\end{proof}
\begin{lem}\label{lem:lemmas=2step}
Let $\rho$ be an $n$-dimensional irreducible continuous representation of $W_F$, and let $E/F$ be a cyclic extension of prime degree. Suppose that there exists a unique irreducible cuspidal representation $\wt\pi$ of $\GL_n(E)$ such that $\rec\wt\pi = \res\rho_{W_E}$. If $\res\rho_{W_E}$ is irreducible, then there also exists a unique irreducible cuspidal representation $\pi$ of $\GL_n(F)$ such that $\rec\pi=\rho$.
\end{lem}
\begin{proof}
Write $r$ for the degree of $E/F$, and fix a generator $\sg$ of $\Gal(E/F)$. As \ref{ss:galoissidebasechange}.(i) indicates that
  \begin{align*}
    (\res\rho_{W_E})^\sg = \res\rho_{W_E},
  \end{align*}
  our uniqueness hypothesis on $\wt\pi$ and \ref{ss:automorphismscompatibility} yield $\wt\pi^\sg = \wt\pi$. Thus $\wt\pi$ is $\sg$-stable, so \ref{ss:basechangeliftlocal} shows that it is the base change lift of some irreducible generic representation $\pi$ of $\GL_n(F)$. Furthermore, because $\wt\pi$ is cuspidal, our description of possibilities for $\pi$ implies that $\pi$ is also cuspidal. Consequently, Proposition \ref{ss:cuspidaltoirreducible} indicates that $\rec\pi$ is irreducible. Now condition (e) yields
  \begin{align*}
    \res{\rec(\pi)}_{W_E} = \rec(\wt\pi) = \res\rho_{W_E},
  \end{align*}
  so \ref{ss:galoissidebasechange}.(i) shows that $\rec(\pi)$ is isomorphic to $\chi\cdot\rho$ for some $\chi$ in $\fK(E/F)$. After replacing $\pi$ with $\chi^{-1}\cdot\pi$, condition (d) tells us that $\rec\pi=\rho$. This takes care of the existence of $\pi$.

  As for uniqueness, suppose $\xi$ is any irreducible cuspidal representation of $\GL_n(F)$ satisfying $\rec(\xi)=\rho$. Writing $\xi_E$ for the base change lift of $\xi$ to $\GL_n(E)$, we see that condition (e) gives us
  \begin{align*}
    \rec(\xi_E) = \res\rho_{W_E} = \rec(\wt\pi).
  \end{align*}
  Condition (c) and the irreducibility of $\res\rho_{W_E}$ imply that $\xi_E$ is cuspidal, so our uniqueness hypothesis on $\wt\pi$ indicates that $\xi_E = \wt\pi$. As $\xi_E$ is the base change lift of $\xi$ to $\GL_n(E)$, we see from \ref{ss:basechangeliftlocal} that $\xi$ must be isomorphic to $\chi\cdot\pi$ for some $\chi$ in $\fK(E/F)$. From here, condition (d) yields
  \begin{align*}
    \rho = \rec\xi = \rec(\chi\cdot\pi) = \chi\cdot\rec\pi = \chi\cdot\rho,
  \end{align*}
  and \ref{ss:galoissidebasechange} along with the irreducibility of $\res\rho_{W_E}$ imply that $\chi=1$. Hence $\xi=\pi$, as desired.
\end{proof}
We shall finally wrap up the proof of Lemma \ref{ss:secondinductivelemma} itself, that is, prove Theorem C holds for $n$.
\begin{proof}[Proof of Lemma \ref{ss:secondinductivelemma}] Proposition \ref{ss:cuspidaltoirreducible} shows that $\pi\mapsto\rec\pi$ indeed yields a map from isomorphism classes of irreducible cuspidal representations of $\GL_n(F)$ to isomorphism classes of $n$-dimensional irreducible continuous representations of $W_F$.

Next, we proceed towards bijectivity. Let $\rho$ be an $n$-dimensional irreducible continuous representation of $W_F$. Bijectivity is immediate for $n=1$, so suppose that $n\geq2$. Now \ref{ss:unramifiedsolvablebasechange} yields a finite solvable extension $E/F$ for which $\res\rho_{W_E}$ is unramified. As $n\geq2$, this implies that $\res\rho_{W_E}$ cannot be irreducible. Thus, by replacing $E/F$ with a subextension if necessary, we obtain a tower of field extensions $E=E_s\supset\dotsb\supset E_0 = F$ such that
  \begin{enumerate}[$\bullet$]
  \item $E_{j+1}/E_j$ is cyclic of prime degree for all $0\leq j\leq s-1$,
  \item $\res\rho_{E_{s-1}}$ is irreducible,
  \item $\res\rho_{E_s}$ is reducible.
  \end{enumerate}
Lemma \ref{ss:lemmas=1step} provides a unique irreducible cuspidal representation $\pi_{s-1}$ of $\GL_n(E_{s-1})$ satisfying $\rec\pi_{s-1}=\res\rho_{E_{s-1}}$. From here, repeated applications of Lemma \ref{lem:lemmas=2step} yield unique irreducible cuspidal representations $\pi_j$ of $\GL_n(E_j)$ such that $\rec\pi_j=\res\rho_{E_j}$, and the $j=0$ case is precisely the desired result.
\end{proof}
Now that we have Lemma \ref{ss:secondinductivelemma}, we turn our attention towards verifying that conditions (b)--(f) hold. Recall that we already proved condition (b) in Proposition \ref{ss:theoremabasecase} and condition (c) in Theorem A.

\subsection{}\label{lem:unramifiedlocallanglands}
To prove condition (d) and condition (e), we use the following explicit description of $\rec\pi$ for unramified $\pi$.
\begin{lem*}
  Let $\pi$ be isomorphic to $Q\left(\{\chi_1\},\dotsc,\{\chi_n\}\right)$, where the $\chi_i:F^\times\rar\bC^\times$ are unramified characters sending $\vpi$ to $z_i$ in $\bC^\times$. Then $\rec\pi$ is isomorphic to the  $n$-dimensional unramified representation of $W_F$ where geometric $q$-Frobenius acts via $\diag(z_1,\dotsc,z_n)$. In particular, for all irreducible unramified representations $\pi$ and $\pi'$ of $\GL_n(F)$, we have $L(\pi\times\pi',s)=L(\rec(\pi)\otimes\rec(\pi'),s)$.
\end{lem*}
\begin{proof}
  Theorem A shows that $\rec\pi$ is isomorphic to $\rec(\chi_1)\oplus\dotsb\oplus\rec(\chi_n)$, and Proposition \ref{ss:theoremabasecase} allows us to conclude $\rec\pi$ has the desired form. Next, any irreducible unramified representation $\pi$ of $\GL_n(F)$ is isomorphic to $Q\left(\{\chi_1\},\dotsc,\{\chi_n\}\right)$ for some unramified characters $\chi_i:F^\times\rar\bC^\times$ \cite{Cas80}. Write $z_i$ for $\chi_i(\vpi)$, and form the analogous $\chi'$ and $z_i'$ for $\pi'$ too. Then we have
  \begin{gather*}
    L(\pi\times\pi',s) = \prod_{i,j=1}^nL(\chi_i\times\chi'_j,s) = \prod_{i,j=1}^n(1-z_iz_jq^{-s})^{-1} = L(\rec(\pi)\otimes\rec(\pi'),s).\qedhere
  \end{gather*}
\end{proof}

\subsection{}
From here, we shall prove condition (d) and condition (e) via embedding into the global situation. More precisely, we will use Theorem B along with the Chebotarev density theorem. Let us recall some global notation: $C$ denotes a geometrically connected proper smooth curve over $\ka$, $\bf{F}$ denotes its field of rational functions, $\bA$ denotes its ring of adeles, and $W_{\bf{F}}$ denotes the Weil group of $\bf{F}$. By abuse of notation, we write $\Art:\bA^\times/\bf{F}^\times\rar^\sim W_{\bf{F}}^\ab$ for the global Artin isomorphism normalized by sending uniformizers to geometric Frobenii. Fix distinct places $x_1$, $x_2$, and $\infty$ of $\bf{F}$. 
\begin{prop}\label{prop:compatwithtwisting}
  Let $\pi$ be an irreducible smooth representation of $\GL_n(F)$. For all smooth characters $\chi:F^\times\rar\bC^\times$, we have
      \begin{align*}
      \rec(\pi\otimes(\chi\circ\det)) = \rec(\pi)\otimes\rec(\chi).
    \end{align*}
  \end{prop}
  \begin{proof}
    We begin with some reductions. Theorem A indicates that both sides are compatible with parabolic induction, so it suffices to prove this for cuspidal $\pi$. As $\rec(\pi)=\rho(\pi)(\frac{1-n}2)$ and $\rec(\chi)=\chi\circ\Art^{-1}$, this is equivalent to
      \begin{align*}
      \rho(\pi\otimes(\chi\circ\det)) = \rho(\pi)\otimes(\chi\circ\Art^{-1}).
    \end{align*}
Finally, note that $\chi$ is the product of a finite order character with $\abs\cdot^s$ for some complex number $s$. The $\chi=\abs\cdot^s$ case follows immediately from Proposition \ref{ss:unramifiedtwistsofrho}, so we need only consider the case when $\chi$ has finite order.

Assume this is the case, and let $C=\bP^1_\ka$ be our curve of interest. Choose $o$ to be a $\ka$-point of $C$, and choose $x_1$, $x_2$, and $\infty$ in \ref{ss:choiceofdivisionalgebra} to be distinct closed points in $C\ssm o$. Note that $\bf{F}_o$ is isomorphic to $F$. Since $\chi$ has finite order, there exists a finite order smooth character $X:\bA^\times/\bf{F}^\times\rar\bC^\times$ such that $\res{X}_{\bf{F}_o^\times}=\chi$ \cite[Theorem 5 of Chapter X]{AT09}. Furthermore, as $\pi$ is cuspidal, Lemma \ref{ss:embedcuspidal} provides an irreducible cuspidal representation $\Pi$ of $\GL_n(\bA)$ whose component at $o$ is isomorphic to $\pi$ and whose components at $x_1$, $x_2$, and $\infty$ are irreducible cuspidal representations of $\GL_n(\bf{F}_{x_1})$, $\GL_n(\bf{F}_{x_2})$, and $\GL_n(\bf{F}_\infty)$, respectively, with unitary central characters.

Because the smooth character $X\circ\Art^{-1}:W_{\bf{F}}\rar\bC^\times$ has finite image, it extends to a continuous character $\Xi:G_{\bf{F}}\rar\ov\bQ_\ell^\times$, where we identify $\ov\bQ_\ell$ with $\bC$ \cite[(IV.2.2)]{HL11}. Note also that $\Pi\otimes(X\circ\det)$ is an irreducible cuspidal representation of $\GL_n(\bA)$. Therefore we may apply Theorem B to obtain the $n$-dimensional semisimple continuous representations
\begin{align*}
R_1\deq R(\Pi\otimes(X\circ\det))\mbox{ and } R_2\deq R(\Pi)\otimes \Xi
\end{align*}
of $G_{\bf{F}}$ over $\ov\bQ_\ell$. Write $T$ for the set of places $x$ of $\bf{F}$ such that
\begin{enumerate}[$\bullet$]
\item $\res{X}_{\bf{F}_x^\times}$ is unramified,
\item $\Pi_x$ is unramified,
\item $x$ does not lie in $\{x_1,x_2,\infty\}$,
\end{enumerate}
and note that $T$ is cofinite. For $x$ in $T$, Lemma \ref{lem:unramifiedlocallanglands} indicates that $\res{R_1}_{W_{\bf{F}_x}}$ and $\res{R_2}_{W_{\bf{F}_x}}$ are isomorphic.

The Chebotarev density theorem implies that, as $x$ varies over all places of $\bf{F}$ not lying in $T$, the conjugacy classes in $G_{\bf{F}}$ of arithmetic $q_x$-Frobenius elements at $x$ are dense. Because $n$-dimensional continuous semisimple representations of $G_{\bf{F}}$ are determined, up to isomorphism, by their characteristic polynomials (and said polynomials are continuous in $G_{\bf{F}}$), we see that $R_1=R_2$ as representations of $G_{\bf{F}}$. Restricting both sides to $W_{\bf{F}_o}$ and applying Theorem B again yields the desired result.
\end{proof}

\begin{prop}\label{prop:compatwithbasechange}
  Let $\pi$ be an irreducible generic representation of $\GL_n(F)$. For all cyclic extensions $E/F$ of prime degree, we have
  \begin{align*}
    \rec(\pi_E) = \res{\rec(\pi)}_{W_E},
  \end{align*}
  where $\pi_E$ denotes the base change lift of $\pi$ to $\GL_n(E)$.
\end{prop}
\begin{proof}
  Applying Theorem A to the decomposition of $\pi_E$ from \ref{ss:basechangeliftlocal} shows that it suffices to consider cuspidal $\pi$. With this reduction in hand, first set $C=\bP^1_\ka$ as our curve of interest, and choose a $\ka$-point $o$ of $C$. We may identify $\bf{F}_o$ with $F$. By writing $E=F[t]/f(t)$ and approximating $f(t)$ using a polynomial with entries in $\bf{F}$, Krasner's lemma yields a separable extension $\bf{E}/\bf{F}$ such that $o$ is inert in $\bf{E}$, and $\bf{E}_{o'}/\bf{F}_o$ can be identified with $E/F$, where $o'$ is the unique place of $\bf{E}$ dividing $o$. By replacing $\bf{E}$ with its Galois closure and replacing $\bf{F}$ (and changing $C$ accordingly) with the subfield corresponding to the decomposition group of $o'$, we may assume that $\bf{E}/\bf{F}$ is Galois and hence cyclic as well.

  Since $\bf{E}$ is a global function field, it corresponds to a geometrically connected proper smooth curve $C'$ over a finite field $\ka'$. Thus we can apply the results of \S\ref{s:modulispaces}--\S\ref{s:localglobal} to $\bf{E}$ and $C'$. By the Chebotarev density theorem, we may choose $x_1$, $x_2$, and $\infty$ in \ref{ss:choiceofdivisionalgebra} for $\bf{F}$ to be distinct closed points in $C\ssm o$ such that they split completely in $C'$. Choose $x_1'$, $x_2'$, and $\infty'$ in \ref{ss:choiceofdivisionalgebra} for $\bf{E}$ lying above $x_1$, $x_2$, and $\infty$, respectively. This allows us to identify $\bf{E}_{x_1'}$ with $\bf{F}_{x_1}$, $\bf{E}_{x_2'}$ with $\bf{F}_{x_2}$, and $\bf{E}_{\infty'}$ with $\bf{F}_\infty$.

  Because $\pi$ is cuspidal, Lemma \ref{ss:embedcuspidal} gives an irreducible cuspidal representation $\Pi$ of $\GL_n(\bA)$ whose component at $o$ is isomorphic to $\pi$ and whose components at $x_1$, $x_2$, and $\infty$ are irreducible cuspidal representations of $\GL_n(\bf{F}_{x_1})$, $\GL_n(\bf{F}_{x_2})$, and $\GL_n(\bf{F}_\infty)$, respectively, with unitary central characters. As $x_1$, $x_2$, and $\infty$ split in $\bf{E}$, the irreducible discrete automorphic representation $\Pi_{\bf{E}}$ formed in \ref{ss:basechangeliftglobal} has components
  \begin{align*}
    \Pi_{\bf{E},x_1'}=\Pi_{x_1},\,\Pi_{\bf{E},x_2'}=\Pi_{x_2},\mbox{ and }\Pi_{\bf{E},\infty'}=\Pi_\infty.
  \end{align*}
  In particular, these components are irreducible cuspidal with unitary central characters. Hence we can apply Theorem B to $\Pi_{\bf{E}}$ to form the $n$-dimensional semisimple continuous representations
  \begin{align*}
    R_1\deq R(\Pi_{\bf{E}})\mbox{ and }R_2\deq \res{R(\Pi)}_{G_{\bf{E}}}
  \end{align*}
  of $G_{\bf{E}}$ over $\ov\bQ_\ell$.

  Write $T$ for the set of places $x$ of $\bf{F}$ such that $\Pi_x$ is unramified and $x$ does not lie in $\{x_1,x_2,\infty\}$, which is a cofinite set of places. For all $x$ in $T$, Lemma \ref{lem:unramifiedlocallanglands} and the description of base change lifts in \ref{ss:basechangeliftcuspidal} and \ref{ss:basechangeliftglobal} indicate that $\res{R_1}_{W_{\bf{F}_x}}$ is isomorphic to $\res{R_2}_{W_{\bf{F}_x}}$. From here, we conclude using the Chebotarev density theorem and Theorem B, as in the proof of Proposition \ref{prop:compatwithtwisting}.
\end{proof}

\subsection{}\label{ss:compatwithunramified}
Finally, we prove condition (f) using our results from \S\ref{s:nearbycycles}.
\begin{prop*}
Let $\pi$ be an irreducible smooth representation of $\GL_n(F)$. If $\rec(\pi)$ is an unramified representation of $W_F$, then $\pi$ is an unramified representation of $\GL_n(F)$.
\end{prop*}
\begin{proof}
  Write $\pi_1,\dotsc,\pi_t$ for the cuspidal support of $\pi$, where the $\pi_i$ are irreducible cuspidal representations of $\GL_{n_i}(F)$ such that $n=n_1+\dotsb+n_t$. Now Theorem A.(ii) yields 
  \begin{align*}
    \rec(\pi) = \rec(\pi_1)\oplus\dotsb\oplus\rec(\pi_t).
  \end{align*}
  If $\pi$ is not unramified, then there exists some $i$ such that either
  \begin{enumerate}[(1)]
  \item $n_i=1$ and $\pi_i:F^\times\rar\bC^\times$ is not an unramified character,
  \item $n_i\geq2$.
  \end{enumerate}
  In case (1), Proposition \ref{ss:theoremabasecase} shows that $\rec(\pi_i)$ and hence $\rec(\pi)$ is not unramified, so we need only tackle case (2). By replacing $\pi$ with $\pi_i$, it suffices to assume that $\pi$ is cuspidal and $n\geq2$. Finally, because $\rec(\pi)$ is an unramified twist of $\rho(\pi)$, we need only show that $\rho(\pi)$ is not unramified.

We begin by using Schur orthogonality to obtain a function $h$ in $C^\infty_c(\GL_n(\cO))$ that satisfies $\tr(h|\pi)=1$. Then we have
  \begin{align*}
    \tr\left(\sg^{-r}|\rho(\pi)^{I_F}\right) = \tr\left(\bf1_{\sg^{-r}I_F}|\rho(\pi)\right)\tr(h|\pi) = \int_{\sg^{-r}I_F}\!\dif\tau\,\tr\left(\tau|\rho(\pi)\right)\tr(h|\pi),
  \end{align*}
  where $\bf1_{\sg^{-r}I_F}$ is the indicator function on $\sg^{-r}I_F$. Applying Theorem A.(i), we obtain
\begin{align*}
\int_{\sg^{-r}I_F}\!\dif\tau\,\tr\left(f_{\tau,h}|\pi\right),
\end{align*}
and \ref{ss:twistedcharacters} indicates that this integral becomes
\begin{align*}
\int_{\sg^{-r}I_F}\!\dif\tau\,\tr\left((\phi_{\tau,h},\sg)|\pi/F_r\right) = \int_{\sg^{-r}I_F}\!\dif\tau\,\int_{\GL_n(\cO_r)\diag(\vpi,1,\dotsc,1)\GL_n(\cO_r)}\!\dif\de\,\phi_{\tau,h}(\de)\Te_\pi^\sg(\de).
\end{align*}
Recall the subset $B_n\subseteq\GL_n(\cO_r)\diag(\vpi,1,\dotsc,1)\GL_n(\cO_r)$, which was defined in \ref{ss:localshtukaunramifiedextension} using the decomposition $\de=\de^\circ\oplus\de^\et$. As $\pi$ is cuspidal, we can use Lemma \ref{lem:cuspidaltwistedcharacter} to convert our integral into
\begin{align*}
\int_{\sg^{-r}I_F}\!\dif\tau\,\int_{B_n}\!\dif\de\,\phi_{\tau,h}(\de)\Te_\pi^\sg(\de).
\end{align*}
Now let $C=\bP^1_\ka$ be our curve of interest, and choose a $\ka$-point $o$ of $C$. Since we may identify $\bf{F}_o$ with $F$, the results of \S\ref{s:nearbycycles} apply to our situation. Namely, Fubini's theorem and Corollary \ref{ss:sectionsixfinalresult} allow us to rewrite this integral as
\begin{align*}
\int_{\sg^{-r}I_F}\!\dif\tau\,\phi_{\tau,h}(\de)\int_{B_n}\!\dif\de\,\Te_\pi^\sg(\de) = \fC_n\int_{B_n}\!\dif\de\,\Te_\pi^\sg(\de),
\end{align*}
where $\fC_n$ is as in Corollary \ref{ss:sectionsixfinalresult}. The definition of $\Te_\pi^\sg$ turns the above expression into
\begin{align*}
\fC_n\int_{B_n}\!\dif\de\,\Te_\pi(\N\de),
\end{align*}
and from here Lemma \ref{ss:noncommutativenormbijection} and the local Jacquet--Langlands correspondence yield
\begin{align*}
\fC_n(-1)^{n-1}\int_{B_r}\!\dif b\,\Te_{\JL(\pi)}(b) = \fC_n(-1)^{n-1}\tr\left(\bf1_{B_r}|\JL(\pi)\right) = \fC_n(-1)^{n-1}\tr\left(\bf1_{B_r}|\JL(\pi)^{\cO_B^\times}\right),
\end{align*}
where $B$ is the central division algebra over $F$ of Hasse invariant $\frac1n$, and $B_r$ is the subset of valuation $r$ elements in $B$. Because $\pi$ has no $\GL_n(\cO)$-invariants, its image under the local Jacquet--Langlands correspondence has no $\cO_B^\times$-invariants. Hence the above trace vanishes.

Altogether, we have shown that $\tr\left(\sg^{-r}|\rho(\pi)^{I_F}\right)$ vanishes for any positive integer $r$. As semisimple representations are determined by their traces, we have $\rho(\pi)^{I_F}=0$. In particular, $\rho(\pi)$ is not unramified.
\end{proof}
By using Theorem A, Proposition \ref{prop:compatwithtwisting}, Proposition \ref{prop:compatwithbasechange}, and Proposition \ref{ss:compatwithunramified} to verify that conditions (b)--(f) of Lemma \ref{ss:secondinductivelemma} hold and using Proposition \ref{ss:theoremabasecase} to check the $n=1$ base case, this concludes our proof by induction of Theorem C.

\section{Duals, $L$-functions, and $\eps$-factors}\label{s:llc}
In this section, our goal is to prove Theorem D, i.e. that $\pi\mapsto\rec\pi$ satisfies Henniart's properties \cite[Theorem 1.2]{Hen85} characterizing the local Langlands correspondence for $\GL_n$ over $F$. We begin by collecting facts on inducing representations of $W_\bf{E}$ to representations of $W_\bf{F}$ for separable extensions $\bf{E}/\bf{F}$. Similar statements hold for extensions $E/F$ of local fields. We use these facts to motivate \emph{automorphic induction}, an analogous operation that turns representations of $\GL_{n}(\bA_\bf{E})$ (respectively $\GL_{n}(E)$) into representations of $\GL_{n[\bf{E}:\bf{F}]}(\bA)$ (respectively $\GL_{n[E:F]}(F)$) for cyclic extensions. Combining this with Theorem B allows us to prove automorphic induction for some non-Galois extensions $\bf{E}/\bf{F}$.

We use this \emph{non-Galois automorphic induction} to show that $\rec$ is compatible with central characters. More precisely, we apply Brauer induction to reduce to the case of induced representations, embed into the global setting, and then invoke our non-Galois automorphic induction. Afterwards, we use compatibility with central characters to prove that $\rec$ preserves $L$-functions and $\eps$-factors, by twisting with highly ramified characters. Finally, compatibility with duals follows from the decomposition of $L$-functions of pairs in terms of $L$-functions of characters. This concludes our proof of the local Langlands correspondence for $\GL_n$ over $F$.

\subsection{}\label{ss:inductionongroups}
Let us recall some global notation: $C$ denotes a geometrically connected proper smooth curve over $\ka$, $\bf{F}$ denotes its field of rational functions, $\bA$ denotes its ring of adeles, and $W_{\bf{F}}$ denotes the Weil group of $\bf{F}$. 

Let $\bf{E}$ be a finite extension of $\bf{F}$ inside $\bf{F}^\sep$. Now write $\wt{\bf{E}}$ for the Galois closure of $\bf{E}$ over $\bf{F}$, and let $R:W_\bf{E}\rar\bC^\times$ be a smooth character. One can use Frobenius reciprocity and the Mackey formula to show that $\Ind_{W_\bf{E}}^{W_\bf{F}}R$ is irreducible if and only if the stabilizer of $\res{R}_{W_{\wt{\bf{E}}}}$ in $\Gal(\wt{\bf{E}}/\bf{F})$ equals $\Gal(\wt{\bf{E}}/\bf{E})$. The same argument works in the local setting: let $E$ be a finite extension of $F$ in $F^\sep$, and let $\rho:W_E\rar\bC^\times$ be a smooth character. Writing $\wt{E}$ for the Galois closure of $E$ over $F$, we see that $\Ind_{W_E}^{W_F}\rho$ is irreducible if and only if the stabilizer of $\res\rho_{W_{\wt{E}}}$ in $\Gal(\wt{E}/F)$ equals $\Gal(\wt{E}/E)$.

\subsection{}\label{ss:automorphicinductionlocal}
We begin by recalling \emph{automorphic induction} in the local case. Let $E$ be a cyclic $F$-algebra, and write $\fK(E/F)$ for the set of group homomorphisms $\Gal(E/F)\rar\bC^\times$. By reducing to the case when $E$ is a field, one can show that $\Art$ induces an isomorphism $F^\times/\Nm_{E/F}(E^\times)\rar^\sim\Gal(E/F)$, so $\fK(E/F)$ acts on isomorphism classes of smooth representations of $\GL_{n[E:F]}(F)$ as in \ref{ss:basechangeliftcuspidal}.

Let $\pi$ be an irreducible tempered representation of $\GL_n(E)$. Then there exists a unique irreducible tempered representation $I_{E/F}(\pi)$ of $\GL_{n[E:F]}(F)$ that is fixed by $\fK(E/F)$ and satisfies a certain character identity involving $\Te_\pi$ \cite[Theorem 1.3]{HH95}. This representation satisfies $L(I_{E/F}(\pi),s) = L(\pi,s)$ \cite[Theorem 1.4.(a)]{HH95}. Furthermore, $I_{E/F}(\pi)$ is cuspidal if and only if the stabilizer of $\pi$ in $\Gal(E/F)$ is trivial \cite[Proposition 5.5]{HH95}. Note that this is the automorphic analog of \ref{ss:inductionongroups}.

\subsection{}\label{ss:automorphicinductionglobal}
Next, we introduce automorphic induction in the global setting. Let $\bf{E}/\bf{F}$ be a finite cyclic extension, and write $\bA_\bf{E}$ for the ring of adeles of $\bf{E}$. Then for any irreducible cuspidal automorphic representation $\Pi$ of $\GL_n(\bA_\bf{E})$, there exists a unique irreducible automorphic representation $I_{\bf{E}/\bf{F}}(\Pi)$ of $\GL_{n[\bf{E}:\bf{F}]}(\bA)$ such that, for cofinitely many places $o$ of $\bf{F}$ for which $\Pi$ is unramified at every place $o'$ of $\bf{E}$ above $o$, we have $L(I_{\bf{E}/\bf{F}}(\Pi)_o,s)=\prod_{o'\mid o}L(\Pi_{o'},s)$ \cite[(IV.1.8)]{HL11}. In addition, for any place $o$ of $\bf{F}$ where the irreducible smooth representation $\Pi_o$ of $\GL_n(\bf{E}\otimes_{\bf{F}}\bf{F}_o)$ is tempered, we have $I_{\bf{E}/\bf{F}}(\Pi)_o=I_{\bf{E}\otimes_{\bf{F}}\bf{F}_o/\bf{F}_o}(\Pi_o)$ \cite[(IV.1.9)]{HL11}.

\subsection{}\label{defn:associatedreps}
We establish some terminology that reflects the connection between automorphic and Galois representations.
\begin{defn*}
  Let $R$ be an $n$-dimensional semisimple continuous representation of $G_\bf{F}$ over $\ov\bQ_\ell$, and let $\Pi$ be an irreducible automorphic representation of $\GL_n(\bA)$. We say $R$ and $\Pi$ are \emph{associated} if there exists a finite set $T$ of places of $\bf{F}$ such that, for all places $x$ of $\bf{F}$ not in $T$,
  \begin{enumerate}[$(a)$]
  \item $\res{R}_{W_{\bf{F}_x}}$ and $\Pi_x$ are unramified,
  \item $\rec\Pi_x$ is isomorphic to $\res{R}_{W_{\bf{F}_x}}$, where we identify $\ov\bQ_\ell$ with $\bC$.
  \end{enumerate}
\end{defn*}
Note that either one of $R$ or $\Pi$ determines the other, by the strong multiplicity one theorem and the Chebotarev density theorem. Furthermore, when there exist three places $\{x_1,x_2,\infty\}$ of $\bf{F}$ where $\Pi$ is irreducible $L^2$ or a Speh module, Theorem B shows that $R$ is isomorphic to $R(\Pi)(\frac{1-n}2)$ and that condition (b) is true for all $x$ not in $\{x_1,x_2,\infty\}$, even if $\res{R}_{W_{\bf{F}_x}}$ or $\Pi_x$ are ramified.

\subsection{}\label{ss:associatedreps}
We view automorphic induction within the framework of Definition \ref{defn:associatedreps} as follows. Let $\bf{E}/\bf{F}$ be a finite cyclic extension. Let $R'$ be an $n$-dimensional continuous semisimple continuous representation of $G_\bf{E}$ over $\ov\bQ_\ell$, and let $\Pi'$ be an irreducible automorphic representation of $\GL_n(\bA_\bf{E})$. If $R'$ and $\Pi'$ are associated, then Lemma \ref{lem:unramifiedlocallanglands}, \ref{ss:automorphicinductionglobal}, and \ref{ss:automorphicinductionlocal} show that $\Ind_{G_{\bf{E}}}^{G_\bf{F}}R'$ and $I_{\bf{E}/\bf{F}}(\Pi')$ are associated.

Next, let $X:\bA^\times/\bf{F}^\times\rar\bC^\times$ be a finite order smooth character. The smooth character $X\circ\Art^{-1}:W_{\bf{F}}\rar\bC^\times$ has finite image, so it extends uniquely to a continuous character $\Xi:G_{\bf{F}}\rar\bC^\times$ \cite[(IV.2.2)]{HL11}. Proposition \ref{ss:theoremabasecase} and the local-global compatibility of the Artin map imply that $\Xi$ and $X$ are associated. 

Finally, let $R$ be an $n$-dimensional semisimple continuous representation of $G_\bf{F}$ over $\ov\bQ_\ell$, and let $\Pi$ be an irreducible automorphic representation of $\GL_n(\bA_\bf{F})$ such that $R$ and $\Pi$ are associated. Since any irreducible unramified representation of $\GL_n(\bf{F}_x)$ is isomorphic to $Q\left(\{\chi_1\},\dotsc,\{\chi_n\}\right)$ for some unramified characters $\chi_i:\bf{F}^\times_x\rar\bC^\times$  \cite{Cas80}, Lemma \ref{ss:compatwithunramified} shows that $\rec\om_{\Pi,x}=\det\res{R}_{W_{\bf{F}_x}}$ for cofinitely many places $x$ of $\bf{F}$. Then the Chebotarev density theorem, Proposition \ref{ss:theoremabasecase}, and the local-global compatibility of the Artin map indicate that $\det{R}$ is the unique extension of $\om_\Pi\circ\Art^{-1}$ to a continuous character $G_\bf{F}\rar\ov\bQ_\ell^\times$, where we identify $\ov\bQ_\ell^\times$ with $\bC$.

\subsection{}\label{ss:nongaloisinduction} We apply Theorem B via \ref{defn:associatedreps} to prove automorphic induction in a non-Galois setting.
\begin{prop*}
  Let $\bf{E}/\bf{F}$ be a finite separable extension, and let $X:\bA_\bf{E}^\times/\bf{E}^\times\rar\bC^\times$ be a finite order smooth character. Assume that
  \begin{enumerate}[(a)]
  \item the Galois closure $\wt{\bf{E}}$ of $\bf{E}/\bf{F}$ is solvable,
  \item there exist three places $\{x_1,x_2,\infty\}$ of $\bf{F}$ inert in $\wt{\bf{E}}$ such that, for all $x$ in $\{x_1,x_2,\infty\}$, the stabilizer of $\res{(X\circ\Art^{-1})}_{W_{\wt{\bf{E}}_{\wt{e}}}}$ in $\Gal(\wt{\bf{E}}_{\wt{e}}/\bf{F}_x)$ equals $\Gal(\wt{\bf{E}}_{\wt{e}}/\bf{E}_e)$, where $\wt{e}$ (respectively $e$) is the unique place of $\wt{\bf{E}}$ (respectively $\bf{E}$) lying above $x$.
  \end{enumerate}
Then there exists an irreducible cuspidal automorphic representation $I_{\bf{E}}^{\bf{F}}(X)$ of $\GL_{[\bf{E}:\bf{F}]}(\bA)$ associated with $\Ind_{G_{\bf{E}}}^{G_{\bf{F}}}(\wh{X})$, where $\Xi:G_{\bf{E}}\rar\bC^\times$ is the character associated with $X$ as in \ref{ss:associatedreps}. Furthermore, the components of $I^\bf{F}_\bf{E}(X)$ at $x_1$, $x_2$, and $\infty$ are irreducible cuspidal representations with unitary central characters.
\end{prop*}

\begin{proof}
  We induct on the degree of $\bf{E}$ over $\bf{F}$, where the result is immediate for $\bf{E}=\bf{F}$. In general, the solvability of $\bf{E}/\bf{F}$ yields a Galois subextension $\wt{\bf{E}}\supseteq\bf{K}\supset\bf{F}$ such that $\bf{K}/\bf{F}$ is cyclic of prime degree. Write $k$ for the place of $\bf{K}$ above $x$, and write $k'$ for the place of $\bf{K}\bf{E}$ above $x$. Then the stabilizer of
\begin{align*}
    \res{(X\circ\Nm_{\bf{K}\bf{E}/\bf{E}}\circ\Art^{-1})}_{W_{\wt{\bf{E}}_{\wt{e}}}} = \res{\Big(\res{(X\circ\Art^{-1})}_{W_{\bf{K}\bf{E}}}\Big)}_{W_{\wt{\bf{E}}_{\wt{e}}}} = \res{(X\circ\Art^{-1})}_{W_{\wt{\bf{E}}_{\wt{e}}}}
\end{align*}
in $\Gal(\wt{\bf{E}}_{\wt{e}}/\bf{K}_k)$ equals $\Gal(\wt{\bf{E}}_{\wt{e}}/\bf{E}_e)\cap\Gal(\wt{\bf{E}}_{\wt{e}}/\bf{K}_k)=\Gal(\wt{\bf{E}}_{\wt{e}}/\bf{K}\bf{E}_{k'})$. Also, note that the unique extension of $X\circ\Nm_{\bf{K}\bf{E}/\bf{E}}\circ\Art^{-1}$ to a continuous character of $G_{\bf{K}\bf{E}}$ is $\res{\Xi}_{G_{\bf{K}\bf{E}}}$. Finally, as $[\wt{\bf{E}}:\bf{K}]<[\wt{\bf{E}}:\bf{F}]$, we can apply the inductive hypothesis to obtain an irreducible cuspidal automorphic representation $I^\bf{K}_{\bf{K}\bf{E}}(X\circ\Nm_{\bf{K}\bf{E}/\bf{E}})$ of $\GL_{[\bf{K}\bf{E}:\bf{K}]}(\bA_\bf{K})$ associated with $\Ind_{G_{\bf{K}\bf{E}}}^{G_\bf{K}}(\res{\Xi}_{G_{\bf{K}\bf{E}}})$. Furthermore, $I^\bf{K}_{\bf{K}\bf{E}}(X\circ\Nm_{\bf{K}\bf{E}/\bf{E}})$ is cuspidal with unitary central character at $k$. We now casework:
\begin{enumerate}[(1)]
\item Suppose that $\bf{E}$ contains $\bf{K}$. Then $\bf{K}\bf{E}=\bf{E}$, and we define $I^\bf{F}_\bf{E}(X)$ to be $I_{\bf{K}/\bf{F}}(I^\bf{K}_{\bf{E}}(X))$. Now \ref{ss:associatedreps} indicates that $I^\bf{F}_\bf{E}(X)$ and $\Ind^{G_{\bf{F}}}_{G_{\bf{K}}}\Ind_{G_{\bf{E}}}^{G_\bf{K}}(\Xi)=\Ind_{G_\bf{E}}^{G_\bf{F}}(\Xi)$ are associated. Because $x$ is inert in $\bf{K}$ and $I^\bf{K}_\bf{E}(X)$ is cuspidal at $k$, we see from \ref{ss:automorphicinductionglobal} that $I^\bf{F}_\bf{E}(X)_x = I_{\bf{K}_k/\bf{F}_x}(I^\bf{K}_\bf{E}(X)_k)$. Note that the stabilizer of $I^\bf{K}_\bf{E}(X)_k$ in $\Gal(\bf{K}_k/\bf{F}_x)$ equals the image of $\Gal(\wt{\bf{E}}_{\wt{e}}/\bf{E}_e)$. Since $\bf{E}_e$ contains $\bf{K}_k$, this is trivial, and hence \ref{ss:automorphicinductionlocal} shows that $I_{\bf{K}_k/\bf{F}_x}(I^\bf{K}_\bf{E}(X)_k)$ is cuspidal. 

\item Suppose that $\bf{E}$ does not contain $\bf{K}$. Then $\bf{K}\cap\bf{E} = \bf{F}$ and $[\bf{E}:\bf{F}]=[\bf{K}\bf{E}:\bf{K}]$, so $\Ind^{G_\bf{K}}_{G_{\bf{K}\bf{E}}}(\res{\Xi}_{G_{\bf{K}\bf{E}}})$ is isomorphic to $\res{\Ind^{G_{\bf{F}}}_{G_{\bf{E}}}(\Xi)}_{G_{\bf{K}}}$. Next, fix a generator $\sg$ of $\Gal(\bf{K}/\bf{F})$. As $I^\bf{K}_{\bf{K}\bf{E}}(X\circ\Nm_{\bf{K}\bf{E}/\bf{E}})$ and $\res{\Ind^{G_{\bf{F}}}_{G_{\bf{E}}}(\Xi)}_{G_{\bf{K}}}$ are associated, there exist cofinitely many places $v$ of $\bf{K}$ where both are unramified and
  \begin{align*}
    \rec(I^\bf{K}_{\bf{K}\bf{E}}(X\circ\Nm_{\bf{K}\bf{E}/\bf{E}})_v^\sg) = \rec(I^\bf{K}_{\bf{K}\bf{E}}(X\circ\Nm_{\bf{K}\bf{E}/\bf{E}})_v)^\sg = \Big(\res{\Ind^{G_{\bf{F}}}_{G_{\bf{E}}}(\Xi)}_{W_{\bf{K}_v}}\Big)^\sg = \res{\Ind^{G_{\bf{F}}}_{G_{\bf{E}}}(\Xi)}_{W_{\bf{K}_v}},
  \end{align*}
by Lemma \ref{ss:automorphismscompatibility} and \ref{ss:galoissidebasechange}.(i). Then Lemma \ref{lem:unramifiedlocallanglands} indicates that the local representations $I^\bf{K}_{\bf{K}\bf{E}}(X\circ\Nm_{\bf{K}\bf{E}/\bf{E}})_v^\sg$ and $I^\bf{K}_{\bf{K}\bf{E}}(X\circ\Nm_{\bf{K}\bf{E}/\bf{E}})_v$ are isomorphic. Hence the strong multiplicity one theorem shows that the global representations $I^\bf{K}_{\bf{K}\bf{E}}(X\circ\Nm_{\bf{K}\bf{E}/\bf{E}})^\sg$ and $I^\bf{K}_{\bf{K}\bf{E}}(X\circ\Nm_{\bf{K}\bf{E}/\bf{E}})$ are isomorphic, so \ref{ss:basechangeliftglobal} yields an irreducible discrete automorphic representation $I^\bf{F}_\bf{E}(X)$ of $\GL_{[\bf{E}:\bf{F}]}(\bA)$ such that $I^\bf{F}_\bf{E}(X)_\bf{K}$ is isomorphic to $I^\bf{K}_{\bf{K}\bf{E}}(X\circ\Nm_{\bf{K}\bf{E}/\bf{E}})$. Because $\Gal(\bf{K}_{k_\infty}/\bf{F}_\infty)=\Gal(\bf{K}/\bf{F})$, where $k_\infty$ is the place of $\bf{K}$ above $\infty$, Proposition \ref{prop:compatwithbasechange} and \ref{ss:galoissidebasechange}.(i) show that we may choose $I^\bf{F}_\bf{E}(X)$ such that $\rec I^\bf{F}_\bf{E}(X)_\infty=\res{\Ind_{G_\bf{E}}^{G_\bf{F}}(\Xi)}_{W_{\bf{F}_\infty}}$.

Now the base change lift of $I^\bf{F}_\bf{E}(X)_x$ to $\GL_{[\bf{E}:\bf{F}]}(\bf{K}_k)$ is cuspidal with unitary central character, so \ref{ss:basechangeliftlocal} indicates $I^\bf{F}_\bf{E}(X)_x$ is also cuspidal with unitary central character. Thus we may apply Theorem B to obtain a semisimple continuous representation $\Sg$ of $G_\bf{F}$ over $\ov\bQ_\ell$ associated with $I^\bf{F}_\bf{E}(X)$. Proposition \ref{prop:compatwithbasechange} implies that $\res{\Sg}_{G_\bf{K}}$ is isomorphic to $\res{\Ind^{G_{\bf{F}}}_{G_{\bf{E}}}(\Xi)}_{G_{\bf{K}}}$, so \ref{ss:galoissidebasechange}.(i) shows that $\Sg$ is isomorphic to $\Ind^{G_{\bf{F}}}_{G_{\bf{E}}}(\Xi)$ tensored with a character $\Gal(\bf{K}/\bf{F})\rar\bC^\times$. By restricting to $W_{\bf{F}_\infty}$, we see that this character is trivial. Altogether, $I_\bf{E}^\bf{F}(X)$ is associated with $\Sg=\Ind^{G_\bf{F}}_{G_\bf{E}}(\Xi)$.
\end{enumerate}
Finally, as $I^\bf{F}_\bf{E}(X)$ is cuspidal at one place, we see that $I^\bf{F}_\bf{E}(X)$ itself is cuspidal.
\end{proof}

\subsection{}\label{ss:galoisgrothendieckgroup}
We now introduce some notation for $\bZ$-virtual representations of $W_F$. Write $\sG_F$ for the Grothendieck group of finite-dimensional continuous representations of $W_F$. Then $\sG_F$ is free over $\bZ$, with a $\bZ$-basis given by isomorphism classes of finite-dimensional irreducible continuous representations $\rho$ of $W_F$. As every such $\rho$ is of the form $\sg(s)$ for some complex number $s$ and continuous representation $\sg$ of $G_F$, Brauer induction shows that $\sG_F$ is $\bZ$-spanned by elements of the form $\Ind_{W_E}^{W_F}\chi$, where $E$ runs over finite separable extensions of $F$, and $\chi:W_E\rar\bC^\times$ runs over smooth characters \cite[4.10]{Del73}.

The assignments $\rho\mapsto\det\rho$ and $\rho\mapsto\rho^\vee$ for finite-dimensional continuous representations $\rho$ of $W_F$ are additive. Therefore they extend to homomorphisms
\begin{align*}
\det:\sG_F\rar\Hom_{\text{cts}}(W_F,\bC^\times)\mbox{ and }(-)^\vee:\sG_F\rar\sG_F.
\end{align*}
Write $\fM$ for the field of meromorphic functions on $\bC$, and fix a nontrivial continuous homomorphism $\psi:F\rar\bC^\times$. We also define $\bZ$-bilinear maps
\begin{align*}
L(-\otimes-,s):\sG_F\times\sG_F\rar\fM^\times\mbox{ and }\eps(-\otimes-,\psi,s):\sG_F\times\sG_F\rar\fM^\times
\end{align*}
by extending via $\bZ$-linearity from their values on pairs $(\rho,\rho')$ of isomorphism classes of finite-dimensional irreducible continuous representations of $W_F$.

\subsection{}\label{ss:automorphicgrothendieckgroup}
The following forms an automorphic analogue of $\sG_F$. Write $\sA_F$ for the free $\bZ$-basis with generators given by isomorphism classes of irreducible cuspidal representations $\pi$ of $\GL_n(F)$, where $n$ ranges over all positive integers. The assignment $\pi\mapsto\rec(\pi)$ extends to a map $\rec:\sA_F\rar\sG_F$, and Theorem C implies that this is an isomorphism.

Write $\om_\pi:F^\times\rar\bC^\times$ for the central character of $\pi$. We define $\bZ$-linear maps
\begin{gather*}
  \om_\pi:\sA_F\rar\Hom_\text{cts}(F^\times,\bC^\times)\mbox{ and }(-)^\vee:\sA_F\rar\sA_F,\\
  L(-\times-,s):\sA_F\otimes_\bZ\sA_F\rar\fM^\times\mbox{ and }\eps(-\times-,\psi,s):\sA_F\otimes_\bZ\sA_F\rar\fM^\times.
\end{gather*}
by extending via $\bZ$-linearity from their values on isomorphism classes of irreducible cuspidal representations of $\GL_n(F)$, or pairs thereof.

\subsection{}
We now prove Theorem D, by using Proposition \ref{ss:nongaloisinduction} and embedding into the global situation.
\begin{prop*}\hfill
\begin{enumerate}[(i)]
\item For any smooth character $\chi:F^\times\rar\bC^\times$, we have $\rec(\chi) = \chi\circ\Art^{-1}$.

\item For any irreducible cuspidal representation $\pi$ of $\GL_n(F)$ and smooth character $\chi:F^\times\rar\bC^\times$, we have
  \begin{align*}
    \rec(\pi\otimes(\chi\circ\det)) = \rec(\pi)\otimes\rec(\chi).
  \end{align*}

\item For any irreducible cuspidal representation $\pi$ of $\GL_n(F)$, we have
\begin{align*}
\rec(\om_\pi) = \det\circ\rec(\pi)\mbox{ and }\rec(\pi^\vee) = \rec(\pi)^\vee.
\end{align*}
\item For any irreducible cuspidal representations $\pi$ of $\GL_n(F)$ and $\pi'$ of $\GL_{n'}(F)$, we have
\begin{align*}
L(\pi\times\pi',s) = L(\rec(\pi)\otimes\rec(\pi'),s)\mbox{ and } \eps(\pi\times\pi',\psi,s) = \eps(\rec(\pi)\otimes\rec(\pi'),\psi,s).
\end{align*}
\end{enumerate}
\end{prop*}

\begin{proof}
  We have already proved (i) in Proposition \ref{ss:theoremabasecase} and (ii) in Proposition \ref{prop:compatwithtwisting}. Now both sides of (iii) are $\bZ$-linear in $\pi$, and both sides of (iv) are $\bZ$-bilinear in $(\pi,\pi')$. Therefore \ref{ss:galoisgrothendieckgroup} and \ref{ss:automorphicgrothendieckgroup} show that (iii) and (iv) follow from considering the same equations for $\pi=\rec^{-1}(\Ind_{W_E}^{W_F}\rho)$ and $\pi'=\rec^{-1}(\Ind_{W_{E'}}^{W_F}\rho')$ in $\sA_F$, where $E$ and $E'$ are separable extensions of $F$, and $\rho:W_E\rar\bC^\times$ and $\rho':W_{E'}\rar\bC^\times$ are smooth characters.

  Note that $\rho$ is the product of a finite order character with a Tate twist $(s)$ for some complex number $s$. The projection formula and (ii) indicate that twisting $\rho$ by $(-s)$ results in twisting $\pi$ by $(-s)$. Now (ii) implies that both sides of (iii) and (iv) are compatible with Tate twists, so we need only consider the case when $\rho$ has finite order. In the same way, we may assume $\rho'$ also has finite order.

  With this reduction in hand, write $\wt{E}$ for the Galois closure of $E$. The proof of Proposition \ref{prop:compatwithbasechange} yields a Galois extension $\wt{\bf{E}}/\bf{F}$ of global function fields and a place $z$ of $\wt{\bf{E}}$ such that $z$ is inert over $\bf{F}$ and $\wt{\bf{E}}_z/\bf{F}_o$ can be identified with $\wt{E}/F$, where $o$ is the place of $\bf{F}$ below $z$. Note then that $\Gal(\wt{\bf{E}}/\bf{F})$ is identified with $\Gal(\wt{E}/F)$. Writing $\bf{E}$ for the subfield of $\wt{\bf{E}}$ corresponding to $\Gal(\wt{E}/E)$, we see that this identifies $\bf{E}_y$ with $E$, where $y$ is the place of $\bf{E}$ below $z$. We apply this to $E'$ to similarly obtain extensions $\wt{\bf{E}}'\supseteq\bf{E}'\supseteq\bf{F}'$ with places $z'$, $y'$, and $o'$.

Because $\wt{\bf{E}}/\bf{F}$ is inert at one place $o$, the Chebotarev density theorem provides three more places $\{x_1,x_2,\infty\}$ of $\bf{F}$ that are inert in $\wt{\bf{E}}$. For all $x$ in $\{x_1,x_2,\infty\}$, there exists a finite order smooth character $\chi_x:\bf{E}_e^\times\rar\bC^\times$ such that the stabilizer of $\chi_x\circ\Nm_{\wt{\bf{E}}_{\wt{e}}/\bf{E}_e}$ in $\Gal(\wt{\bf{E}}_{\wt{e}}/\bf{F}_x)$ equals $\Gal(\wt{\bf{E}}_{\wt{e}}/\bf{E}_e)$, where $\wt{e}$ (respectively $e$) is the unique place of $\wt{\bf{E}}$ (respectively $\bf{E}$) lying above $x$ \cite[Lemma 4.7]{Har98}. As $\rho\circ\Art$ also has finite order, there exists a finite order smooth character $X:\bA_{\bf{E}}^\times/\bf{E}^\times\rar\bC^\times$ such that $\res{X}_{\bf{E}_y^\times}=\rho\circ\Art$ and $\res{X}_{\bf{E}_x^\times}=\chi_x$ for all $x$ in $\{x_1,x_2,\infty\}$ \cite[Theorem 5 of Chapter X]{AT09}. The same discussion yields analogous places $\{x_1',x_2',\infty'\}$ of $\bf{F}'$ and an analogous finite order smooth character $X':\bA^\times_{\bf{E}'}/\bf{E}'^\times\rar\bC^\times$.
  
Since $F$ is a nonarchimedean local field, $\wt{E}/F$ and hence $\wt{\bf{E}}/\bf{F}$ are solvable. Therefore we may apply Proposition \ref{ss:nongaloisinduction} to obtain an irreducible cuspidal automorphic representation $I^\bf{F}_\bf{E}(X)$ of $\GL_{[\bf{E}:\bf{F}]}(\bA)$ associated with $\Ind_{G_\bf{E}}^{G_\bf{F}}(\Xi)$, where $\Xi:G_{\bf{E}}\rar\bC^\times$ is the character associated with $X$ as in \ref{ss:associatedreps}. Furthermore, the components of $I^\bf{F}_\bf{E}(X)$ at $x_1$, $x_2$, and $\infty$ are irreducible cuspidal representations, so \ref{defn:associatedreps} shows that
\begin{align*}
\rec(I^\bf{F}_\bf{E}(X)_v) = \res{\Ind_{G_\bf{E}}^{G_\bf{F}}(\Xi)}_{W_{\bf{F}_v}}
\end{align*}
for all places $v$ of $\bf{F}$ not in $\{x_1,x_2,\infty\}$. Taking $v=o$ yields
\begin{align*}
\rec(I^\bf{F}_\bf{E}(X)_o) = \res{\Ind_{G_\bf{E}}^{G_\bf{F}}(\Xi)}_{W_{\bf{F}_o}} = \Ind^{W_{\bf{F}_o}}_{W_{\bf{E}_y}}(\res{X}_{\bf{E}^\times_y}\circ\Art^{-1}) = \Ind^{W_F}_{W_E}\rho,
\end{align*}
so $\pi$ is the component of $I^\bf{F}_\bf{E}(X)$ at $o$. We apply this to $\bf{E}'/\bf{F}'$ and $X'$ to similarly obtain an irreducible cuspidal automorphic representation $I^{\bf{F}'}_{\bf{E}'}(X')$ associated with $\Ind^{G_{\bf{F}'}}_{G_{\bf{E}'}}(\Xi')$.

Now \ref{ss:associatedreps} implies that $\om_{I^\bf{F}_\bf{E}(X)}$ equals $\big(\det\Ind_{G_\bf{E}}^{G_\bf{F}}(\Xi)\big)\circ\Art$, so restricting to $o$ tells us that
\begin{align*}
\om_\pi = \om_{I^\bf{F}_\bf{E}(X)_o} = \det\res{\Ind^{G_\bf{F}}_{G_\bf{E}}(\Xi)}_{W_{\bf{F}_o}} = \det\Ind^{W_F}_{W_E}\rho = \det\rec\pi.
\end{align*}
This completes the first part of (iii). Additionally, because $\big(\det\Ind_{G_\bf{E}}^{G_\bf{F}}(\Xi)\big)\circ\Art$ has finite image, $\om_{I^\bf{F}_\bf{E}(X)}$ does as well. And \ref{ss:inductionongroups} shows that the restriction of $\Ind_{G_\bf{E}}^{G_\bf{F}}(\Xi)$ to $W_{\bf{F}_\infty}$ is irreducible, so $\Ind_{G_\bf{E}}^{G_\bf{F}}(\Xi)$ itself is irreducible. The same discussion holds for $I^{\bf{F}'}_{\bf{E}'}(X')$ and $\Ind^{G_{\bf{F}'}}_{G_{\bf{E}'}}(\Xi')$, so we can apply \cite[Theorem 2.4]{Hen00} to conclude that (iv) holds.

Finally, return to the situation where $\pi$ and $\pi'$ are irreducible cuspidal representations. Then
\begin{align*}
L(\pi\times\pi',s) = \prod_\chi L(\chi,s),
\end{align*}
where $\chi$ ranges over unramified characters $F^\times\rar\bC^\times$ such that $\pi'^\vee\otimes(\chi\circ\det)$ is isomorphic to $\pi$. Since $L(\chi,s)$ has a pole at $s=1$ if and only if $\chi$ is trivial, we see that $L(\pi\times\pi',s)$ has a pole at $s=1$ if and only if $\pi'^\vee=\pi$. We have an analogous decomposition of $L(\rec(\pi)\otimes\rec(\pi'),s)$, and the same argument indicates that $L(\rec(\pi)\otimes\rec(\pi'),s)$ has a pole at $s=1$ if and only if $\rec(\pi')^\vee=\rec(\pi)$. Applying (iv) completes the second part of (iii).
\end{proof}

\bibliographystyle{habbrv}
\bibliography{senior_thesis}
\end{document}